\documentclass[a4paper,10pt]{report}
\usepackage[english]{babel}
\usepackage[T1]{fontenc}
\usepackage[latin1]{inputenc}
\usepackage{amsfonts}
\usepackage{amsmath}
\usepackage{amsthm}
\usepackage{amssymb}
\usepackage{dsfont}
\usepackage{braket}
\usepackage{faktor}
\usepackage{array}
\usepackage{epigraph}
\usepackage[pdftex]{graphicx}
\usepackage{subfig}
\usepackage{indentfirst}
\usepackage{hyperref}
\usepackage{tikz}
\usetikzlibrary{matrix,arrows}
\usepackage{float}
\usepackage[colorinlistoftodos, textwidth=4cm, shadow]{todonotes}
\usepackage{amsmath}
\usepackage[displaymath, tightpage]{preview}

\usepackage{color}

\usepackage{syntonly}
\usepackage{hyperref}
\renewcommand{\geq}{\geqslant}
\renewcommand{\leq}{\leqslant}
\theoremstyle{plain}
\newtheorem{teo}{Theorem}[section]
\newtheorem*{teo*}{Theorem}
\newtheorem{lemma}[teo]{Lemma}
\newtheorem{prop}[teo]{Proposition}
\newtheorem{cor}[teo]{Corollary}
\newtheorem{remark}[teo]{Remark}
\newtheorem{axiom}{Axiom}[section]
\newtheorem{axiom*}{Axiom}
\theoremstyle{definition}
\newtheorem{defin}{Definition}[section]
\newtheorem*{defin*}{Definition}
\theoremstyle{remark}
\newtheorem{note}[teo]{Note}
\newtheorem*{note*}{Note}

\begin{document}

%titlepage
\begin{titlepage}

\addtolength{\hoffset}{24pt}

\begin{center}

\makebox[\linewidth]{
\begin{minipage}[t]{0.4\linewidth}
\centering{\textsc{\Large{Universit\`a degli Studi di Trieste}\\ 
                   \vspace{0.3cm}
                   \large{Dipartimento di Matematica e Geoscienze}}\\
                   \vspace{0.3cm}
                   Corso di Studi in Matematica \\
                   \vspace{0.5cm}
                   \includegraphics[width=0.59\textwidth]{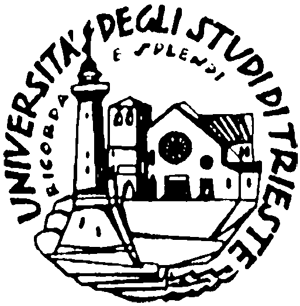}}
\end{minipage}

\hspace{0.2\linewidth}

\begin{minipage}[t]{0.4\linewidth}
\centering{\textsc{\Large{S.I.S.S.A. - Scuola Internazionale Superiore di Studi Avanzati}}\\
                          \vspace{0.3cm}
                          Percorso Formativo Comune per la Laurea Magistrale in Matematica\\
                          \vspace{0.8cm}
                          \includegraphics[width=0.68\textwidth]{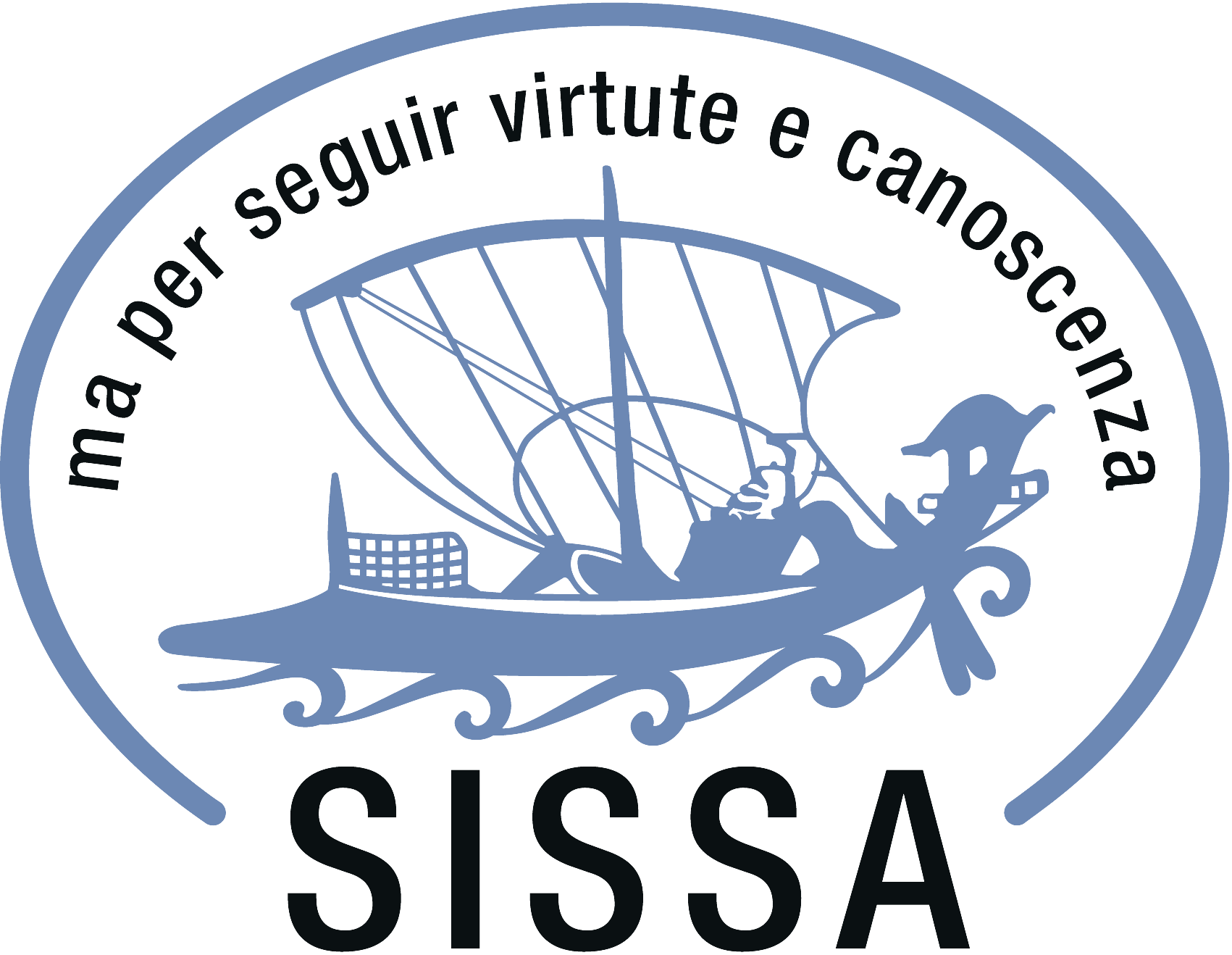}}
\end{minipage}
}

\vspace{3.5cm}
\large{\textsc{Tesi di Laurea Magistrale}}\\

\vspace{1.5cm}
\Large{\textsc{Approximate Hermitian-Yang-Mills structures on semistable Higgs bundles}}\\

\vspace{3cm}

\end{center}

\makebox[\linewidth]{%
\begin{minipage}[t]{0.4\linewidth}
\centering{\large{\textsc{Candidato}\\
                  \textbf{El\'ia Saini}}}
\end{minipage}

\hspace{0.2\linewidth}

\begin{minipage}[t]{0.4\linewidth}
\centering{\large{\textsc{Relatore}\\
                  \textbf{Prof. Ugo Bruzzo}\\
                  S.I.S.S.A.}}
\end{minipage}
}

\vfill
\begin{center}
\normalsize
\rule{8cm}{0.1mm}\\
\bigskip
ANNO ACCADEMICO 2012--2013
\end{center}

\end{titlepage}

\newpage
\hspace{1cm}\thispagestyle{empty}

%Inserisco abstract e sunto della tesi
\newpage

\thispagestyle{empty}

\vspace*{\fill}

\noindent {\sc Strutture Hermitian-Yang-Mills approssimate su fibrati di Higgs semistabili}
 
\vspace{0.5cm}

\noindent {\it Sunto} 
\vspace{0.3cm}

Si introducono le nozioni di struttura (debole) Hermitian-Yang-Mills e di struttura Hermitian-Yang-Mills 
approssimata su fibrati di Higgs, quindi si costruisce il funzionale di Donaldson per fibrati di Higgs su 
variet\`a di K\"ahler compatte e si presentano alcune propriet\`a di base di questo funzionale. In particolare, 
si prova che il suo gradiente pu\`o essere scritto in termini della curvatura media della connessione di 
Hitchin-Simpson e si studiano alcune propriet\`a dell'equazione di evoluzione associata al funzionale di 
Donaldson. Successivamente si affronta il problema dell'esistenza di strutture Hermitian-Yang-Mills approssimate 
su fibrati di Higgs e si studia la relazione tra l'esistenza di tali strutture e la nozione algebro-geometrica 
di semistabilit\`a di Mumford-Takemoto. In particolare, si prova che per un fibrato di Higgs su una superficie di 
Riemann compatta le nozioni di esistenza di strutture Hermitian-Yang-Mills approssimate e di semistabilit\`a 
sono equivalenti.

Infine, usando il flusso del calore associato al funzionale di Donaldson, si prova che la semistabilit\`a di un 
fibrato di Higgs su una variet\`a di K\"ahler compatta e l'esistenza di strutture 
Hermitian-Yang-Mills approssimate sono equivalenti in qualunque dimensione. Come conseguenza di questo fatto, si 
deduce che diversi risultati riguardanti l'esistenza di strutture Hermitian-Yang-Mills 
approssimate su fibrati di Higgs possono essere espressi in termini di semitabilit\`a.
\vspace{3cm}

\noindent {\sc Approximate Hermitian-Yang-Mills structures on semistable Higgs bundles}

\vspace{0.5cm}

\noindent { \it Abstract} 
\vspace{0.3cm}

We review the notions of (weak) Hermitian-Yang-Mills structure and approximate Hermitian-Yang-Mills structure 
for Higgs bundles. Then, we construct the Donaldson functional for Higgs bundles over compact K\"ahler manifolds 
and we present some basic properties of it. In particular, we show that its gradient flow can be written in 
terms of the mean curvature of the Hitchin-Simpson connection. We also study some properties of the evolution 
equation associated to that functional. 
Next, we study the problem of the existence of approximate Hermitian-Yang-Mills structure and its relation with 
the algebro-geometric notion of Mumford-Takemoto semistability and we
show that for a Higgs bundle over a compact Riemann surface, the notion of approximate Hermitian-Yang-Mills 
structure is in fact the differential-geometric counterpart of the notion of semistability.

Finally, using Donaldson heat flow, we show that the semistability of a Higgs bundle over a compact 
K\"ahler manifold (of every dimension) implies the existence of an approximate Hermitian-Yang-Mills structure. 
As a consequence of this we deduce that many results about Higgs 
bundles written in terms of approximate Hermitian-Yang-Mills structures can be translated in terms of 
semistability.
\vspace*{\fill}

\pagebreak

%Inserisco epigrafe
\newpage
\thispagestyle{empty}
\noindent
\beforeepigraphskip 8cm
%\epigraph{Oppa Gangnam Style}{\emph{Old Confucian Say}}

\epigraph{"Adesso \`e pi\`u normale\\ 
adesso \`e meglio,\\
adesso \`e giusto, giusto, \`e giusto\\
che io vada"\\[5pt]
E quando poi spar\'i del tutto\\
a chi diceva "\`e stato un male"\\
a chi diceva "\`e stato un bene"\\
raccomand\`o "non vi conviene\\ 
venir con me dovunque vada,\\ 
ma c'\`e amore un po' per tutti\\ 
e tutti quanti hanno un amore\\ 
sulla cattiva strada\\ 
sulla cattiva strada"}{\emph{Fabrizio de Andr\'e\linebreak\rm{La cattiva strada}}}

\newpage
\hspace{1cm}\thispagestyle{empty}

%inserisci indice
\tableofcontents

\chapter*{Introduction}

\section*{Some historical background}

The notion of holomorphic vector bundle is common to some branches of mathematics and theoretical physics. In 
particular, this notion plays a fundamental role in complex differential geomtry, algebraic geometry, conformal 
string 
and Yang-Mills theories. Moreover, the study of holomorphic vector bundles involves 
techniques from geometric analysis, partial differential equations and topology. In this thesis we study Higgs 
bundles and some of their main properties.
We restrict our study to the case when the complex manifold is compact K\"ahler. On the one hand, complex 
manifolds provide a rich class of geometric objects, which behave rather differently than real smooth manifolds.
\\

In complex geometry, the Hitchin-Kobayashi correspondence asserts that the notion of (Mumford-Takemoto) 
stability, originally introduced in algebraic geometry, has a differential-geometric equivalent in terms of 
special metrics. In its classical version, this correspondence is established for holomorphic vector bundles 
over compact K\"ahler manifolds and says that such bundles are polystable if and only if they admit a 
Hermitian-Einstein \footnote{In the literature Hermitian-Einstein, Einstein-Hermite and 
Hermitian-Yang-Mills are all synonymous. Sometimes alse the terminology Hermitian-Yang-Mills-Higgs is used.}
structure. This correspondence also holds for Higgs bundles.\\

The history of this correspondence probably starts in 1965, when Narasimhan and Seshadri \cite{N12} proved that 
a holomorphic bundle over a Riemann surface is stable if and only if it corresponds to a projective irreducible 
representation of the fundamental group of the surface. Then, in the 80's Kobayashi \cite{KO3} introduced for 
the first time the notion of Hermitian-Einstein structure in a holomorphic vector bundle, as a generalization 
of K\"ahler-Einstein metric in a tangent bundle. Shortly after, Kobayashi \cite{KO4} and L\"ubke 
\cite{L41} proved that a bundle with a irreducible Hermitian-Eintein structure must be necessarly stable. 
Donaldson \cite{D50} showed that the result of Narasimhan and Seshadri \cite{N12} can be formulated in term of 
metrics and proved that the concepts of stability and existence of Hermitian-Einstein metrics are equivalent 
for 
holomorphic vector bundles over Riemann surfaces. Then, Kobayashi and Hitchin conjectured that the equivalence 
sholud be true in general for holomorphic vector bundles over K\"ahler manifolds.\\ 

The existence of a Hermitian-Einstein structure in a stable holomorphic vector bundle was proved by Donaldson 
for 
projective algebraic surfaces in \cite{D51} and for projective algebraic manifolds in \cite{D52}. Finally, 
Uhlenbeck and Yau showed in\cite{YAU} proved  for general compact K\"ahler manifolds using some techniques from 
analysis and Yang-Mills theory. Hitchin \cite{H16}, while studying the self-duality equations over a compact 
Riemann surface, introduced the notion of Higgs field and showed that the result of Donaldson for Riemann 
surfaces could be modified to include the presence of a Higgs field. Following the result of Hitchin, Simpson
\cite{SIM} defined a Higgs bundle to be a holomorphic vector bundle together with a Higgs field and proved the 
Hitchin-Kobayashi correspondence for such objects. Actually, using some sophisticated techniques in partial 
differential equations and Yang-Mills theory, he proved the correspondence even for non-compact K\"ahler 
manifolds, if they satisfy some analytic conditions. As an application of this, Simpson \cite{S18} later studied 
in detail a one-to-one correspondence between stable Higgs bundles over a compact K\"ahler manifold with 
vanishing Chern classes and irreducible representations of the foundamental group of that K\"ahler manifold.\\

The Hitchin-Kobayashi correspondence has been extended in several directions. L\"ubke and Teleman 
\cite{L40} studied the correspondence for compact complex manifolds. Bando and Siu \cite{B14} extended the 
correspondence to torsion-free sheaves over compact K\"ahler manifolds and introduced the notion of 
admissible Hermitian metric for such objects. Following the ideas of Bando and Siu, Biswas and Schumacher 
\cite{B24} introduced the notion of admissible Hermitian-Yang-Mills metric in the Higgs case and generalized 
this extension to torsion-free Higgs sheaves.\\

In \cite{D51} and \cite{D52} Donaldson introduced a functional, which is indeed known as the Donaldson 
functional, and 
later Simpson \cite{SIM} extended this functional in his study of the Hitchin-Kobayashi correspondence for Higgs 
bundles. Kobayashi in \cite{KOB} constructed the same functional in a different form and showed that it plays a 
foundamental role in a possible extension of the Hitchin-Simpson correspondence. In fact, he proved in 
\cite{KOB} that for holomorphic vector bundles over projective algebraic manifolds, the counterpart of 
semistability is the notion of approximate Hermitian-Yang-Mills structure.\\

The correspondence between semistability and the existence of approximate Hermitian-Yang-Mills structure in the 
ordinary case has been originally proposed by Kobayashi. In \cite{KOB} he proved that for a holomorphic vector 
bundle over a compact K\"ahler manifold a boundedness property of the Donaldson functional implies the 
existence of a Hermitian-Einstein structure and that implies the semistability of the bundle. Then, using some 
properties of the Donaldson functional and the Mehta-Ramanathan Theorem, he established the boundedness
propery of the Donaldson functional for semistable holomorphic bundles over compact algebraic manifolds. 
As a consequence of this, he obtained the correspondence between semistability and the existence of 
approximate Hermitian-Einstein structures when the base manifold was projective. Then he conjectured that 
all three conditions (the boundedness property, the existence of an approximate Hermitian-Einsten structure and 
the semistability) should be equivalent in general, that is, independently from whether the manifold was
algebraic or not.\\

In the Higgs case and when the manifold is one-dimensional (a compact Riemann surface), the boundedness property 
of the Donaldson functional follows from the semistability in a similar way to the classical case, since we 
need to consider only Higgs subbundles and their quotients and we have a decomposition of the Donaldson 
functional in terms of these objects. The existence of approximate Hermitian-Einsten metrics for semistable 
holomorphic vector bundles has been recently studied in \cite{J46} using some techniques developed by Buchdahl 
\cite{B48}, \cite{B49} for the desingularization of sheaves in the case of compact complex surfaces. One of the 
main difficulties in the study of this correspondence in higher dimensions arises from the notion of stability, 
since for compact K\"ahler manifolds with dimension greater or equal than two, it is necessary to consider 
subsheaves and not only subbundles. On the other hand, properties of the Donaldson functional commonly involve 
holomorphic bundles. Although all these difficulties appears also in the Higgs case,
using Donaldson heat flow, Li and Zhanh \cite{CIN} recently showed that 
the semistability of a Higgs bundle over a compact K\"ahler manifold implies the existence of approximate 
Hermitian-Einsten structure. It remains to be proved that, in dimension greater or equal than two, the 
semistability of the Higgs bundle implies the boundedness propery of the Donaldson functional. In order to 
prove this, it seems natural to introduce first the notion of admissible Hermitian metrics on Higgs sheaves. 
Then, to define the Donaldson functional for such objects using these metrics and finally, following 
\cite{J46} to study how this 
functional defined for a semistable Higgs bunde can be decomposed in terms of Higgs subsheaves and their
quotients.
 
\section*{About this thesis}
This thesis is organized as follows. In Chapter 1 we start with some basic definitions and results, 
in particular 
we introduce complex manifolds and almost complex structures on differentiable manifolds. Then we summarize the 
main properties of complex differential forms over complex manifolds. Finally, we introduce K\"ahler manifolds.
After studying some properties of these objects, we give an exemple of complex manifolds which cannot be 
given a K\"ahlerian structure.\\

In Chapter 2 we present some definitions and results on principal fibre bundles over a differentiable manifold 
$M.$ In particular we study connections and curvatures in principal fibre bundles and present some important 
results such as the structure equation and the Bianchi identity.\\

Although our primary interest lies in holomorphic vector bundles, we begin Chapter 3 with the study of 
connections in differentiable complex bundles. Then, we introduce connections in complex vector bundles over 
complex manifolds and we characterize those complex vector bundles which admit 
holomorphic structures. Moreover, we define Hermitian structures in complex vector bundles over 
(real or complex) manifolds and we study those connections which are compatible with the holomorphic structure
and the Hermitian metric. Finally, we summarize the main properties of subbundles and quotient bundles. 
We study how  Hermitian metrics on holomorphic vector bundles (over complex manifolds) 
induce metrics on subbundles and quotient bundles.
Moreover, we give the definition and study the main properties of the second 
fundamental form of subbundles and quotient bundles.\\

In Chapter 4 we introduce Chern classes. First of all we define the first Chern class of a 
line bundle. Then, using the splitting principle, we define higher Chern classes for complex vector bundles of 
any rank. 
Moreover, we present the axiomatic approach to Chern classes: this enables us to separate differential geometry 
aspects of Chern classes from their topological aspects. Finally, via de Rahm theory, we give a  
representation of Chern classes in terms of the curvature form of a connection
in a complex vector bundle.\\

In Chapter 5 we summarize some basic properties of coherent sheaves over compact K\"ahler manifolds. Then, we 
review the definition of singularity sets for coherent sheaves and briefly comment some of their main 
properties concerning to the codimension of these singularity sets. We review the construction of the 
determinant 
bundle of a coherent sheaf and we write some facts on determinant bundles that are used through this work.
Moreover, we present some useful analytic results such as the Fredholm alternative Theorem and the Maximum 
principle for parabolic equations.\\

In Chapter 6 we study the basics results of Higgs sheaves. These results 
are important mainly because the notion of stability in higher dimension (greater than one) makes reference to 
Higgs subsheaves and not only Higgs subbundles. Moreover, we summarized some properties of metrics and 
connections on Higgs bundles and itroduce the space of Hermitian structures, which is the space where the 
Donaldson functional is defined. Finally, we construct the Donaldson functional for Higgs bundles over compact 
K\"ahler manifolds following a construction similar to that of Kobayashi and we present some basic properties of
it. In particular, we prove that the critical points of this functional are precisely the 
Hermitian-Yang-Mills structures, and we show also that its gradient flow can be written in terms of the mean 
curvature of the Hitchin-Simpson connection. We also establish some properties of the solution of the evolution 
equation associated with that functional. 
Next, we study the problem of the existence of approximate Hermitian-Yang-Mills structure and its relation with 
the algebro-geometric notion of Mumford-Takemoto semistability. We prove that if the Donaldson functional of a 
Higgs bundle over a compact K\"ahler manifold is bounded from below, then there exists an approximate 
Hermitian-Yang-Mills structure on it. This fact, together with a result of Bruzzo and Gra\~{n}a Otero 
\cite{BRU}, implies the semistability of the Higgs bundle.
Then we show that a semistable but non-stable Higgs bundle can be included into a short exact sequence with a 
stable Higgs subsheaf and a semistable Higgs quotient. We use this result in the final part of this chapter when 
we show that for a Higgs bundle over a compact Riemann surface, the notion of approximate Hermitian-Yang-Mills 
structure is in fact the differential-geometric counterpart of the notion of semistability.\\

In Chapter 7, using Donaldson heat flow, we show that the semistability of a Higgs bundle over a compact 
K\"ahler manifold (of every dimension) implies the existence of an approximate Hermitian-Yang-Mills structure;
this proof is due to Jiayu and Zhang \cite{CIN} and essentially follows from Simpson \cite{SIM}.

\section*{Acknowledgements}
I would like to thank my colleagues Umberto Lardo and Gianluca Orlando and my closest friend Alberto "Kravatz" 
Cavallo for many enlighthening discussions about mathematics. 
I wish to mention Francesco de Anna for helping me in the proof of Lemma \ref{Simpsonliminf}.

\chapter{Basics on Complex Manifolds}
 \section{Complex manifolds}
In this first chapter we start with some basic definitions and results, in particular we introduce complex 
manifolds and almost complex structures on differentiable manifolds.
 
 \begin{defin}
  Let $M$ a T2, N2 topological space. A complex $n\text{-atlas}$ $\mathcal{A}$ is a collection of local charts 
  $\{U_{\alpha},\phi_{\alpha}\}$ such that:
  \begin{enumerate}
   \item $\{U_{\alpha}\}$ is a countable open cover of $M,$
   \item $\phi_{\alpha}:U_{\alpha}\longrightarrow\mathbb{C}^n$ are homeomorphism of $U_{\alpha}$
   with an open subset of $\mathbb{C}^n,$
   \item If $U_{\alpha}\cap U_{\beta}\neq\emptyset,$ the transition functions
   \begin{equation*}
    \left.\phi_{\alpha}\circ\phi_{\beta}^{-1}\right|_{\phi_{\beta}(U_{\alpha}\cap U_{\beta})}:
   \phi_{\beta}(U_{\alpha}\cap U_{\beta})\longrightarrow\phi_{\alpha}(U_{\alpha}\cap U_{\beta})
   \end{equation*}
   are biholomorphisms.
  \end{enumerate}
 \end{defin}
 
 \begin{defin}
  Let $M$ a T2, N2 topological space. Two complex $n\text{-atlases}$ 
  \linebreak$\mathcal{A}=\{U_{\alpha},\phi_{\alpha}\}$ and 
  $\mathcal{B}=\{V_{\beta},\psi_{\beta}\}$ are equivalent if every transition function
  \begin{equation*}
   \left.\phi_{\alpha}\circ\psi_{\beta}^{-1}\right|_{\psi_{\beta}(U_{\alpha}\cap V_{\beta})}:
   \psi_{\beta}(U_{\alpha}\cap V_{\beta})\longrightarrow\phi_{\alpha}(U_{\alpha}\cap V_{\beta})
  \end{equation*}
   is a biholomorphim. A complex structure on $M$ is an equivalence class $\overline{\mathcal{A}}$ 
   of complex $n\text{-atlases}.$
 \end{defin}
 
 \begin{defin}
  Let $M$ be a T2, N2 topological space and let $\overline{\mathcal{A}}$ a complex structure on $M.$ 
  We say that the pair $(M,\overline{\mathcal{A}})$ is a complex manifold. We also define the complex 
  dimension of $M$ as $\text{dim}_{\mathbb{C}}=n.$
 \end{defin}
 
 \begin{defin}
  Let $M$ and $N$ be complex manifolds with $\text{dim}_{\mathbb{C}}M=m$ and $\text{dim}_{\mathbb{C}}N=n.$
  Let $f:M\longrightarrow N$ be an application. $f$ is holomorphic if for every point $P\in M$ and local 
  charts $(U_{\alpha},\phi_{\alpha})$ of $P$ in $M$ and
  $(V_{\beta},\psi_{\beta})$ of $f(P)$ in $N$ such that $f(U_{\alpha})\subseteq V_{\beta}$
  the representative $\psi_{\beta}\circ f\circ\phi_{\alpha}^{-1}$
  is holomorphic. 
 \end{defin}
 
Holomorphic functions have some interesting properties, see \cite{CAR} and \cite{GR} for more details.

 \begin{teo}
 \emph{(Open mapping Theorem)}
  Let $M$ and $N$ be complex manifolds and let $f:M\longrightarrow N$ be a holomorphic function. If $f$
  is not constant, then $f$ is an open mapping.
 \end{teo}
 
 \begin{defin}
 Let $(M,g)$ be a Riemannian manifold and let $f$ be a $\mathcal{C}^{\infty}$ function on $M.$ The gradient of
 $f$ is the unique vector field $\nabla f\in\chi(M)$ such that for every vector field $X\in\chi(M)$
 \begin{equation*}
  g(\nabla f,X)=\mathrm{d}f(X).
 \end{equation*}
A point $P\in M$ is said to be a critical point of $f$ if $(\nabla f)_P=0.$
\end{defin}

\begin{note}
 In local coordinates we have
\begin{equation*}
 (\nabla f)=g^{ij}\frac{\partial f}{\partial x^i}\frac{\partial}{\partial x^j}.
\end{equation*}
\end{note}

\begin{defin}
 Let $(M,g)$ be a Riemannian manifold of (real) dimension $n$ and let $X\in\chi(M)$ be a vector field on $M.$
 Let $\nabla$ be the Levi-Civita connection associated with the Riemannian metric $g.$ The divergence of $X$ 
 the 
 function on $M$ obtained by contraction of the $(1,1)\text{-tensor}$ field $\nabla X,$ i.e., the function
 \begin{equation*}
  \mathrm{div}(X)=\mathcal{C}_{1}^{1}(\nabla X).
 \end{equation*}
Here $\mathcal{C}_{1}^{1}$ denote the contraction of the $(1,1)\text{-tensor}$ field $\nabla X.$
\end{defin}

\begin{note}
 In local coordinates we have
 \begin{equation*}
  \mathrm{div}(X)=\sum_{k=1}^{n}\left(\frac{\partial X^{k}}{\partial x^{k}}+
  \sum_{h=1}^{n}\Gamma_{kh}^{h}X^{h}\right),
 \end{equation*}
where $\Gamma_{jl}^{i}$ are the Christoffel's symbols of the Levi-Civita connection $\nabla$ associated with $g.$
\end{note}

 \begin{defin}
  Let $(M,g)$ be a Riemannian manifold of (real) dimension $n$ and let $f$ be a $\mathcal{C}^{\infty}$ complex
  valued function on $M.$ We define the Laplace-Beltrami operator as the divergence of the gradient:
  \begin{equation*}
   \Delta f=\mathrm{div}(\nabla f).
  \end{equation*}
 \end{defin}

 \begin{defin}
  Let $M$ be a complex manifold of complex dimension $n$ and let $f:M\longrightarrow\mathbb{C}$ be a 
  $\mathcal{C}^{\infty}$ function. 
  We say that $f$ is harmonic if, regarded as a function on the real Riemannian manifold $M$ of real 
  dimension $2n,$ $\Delta f=0.$ 
 \end{defin}

 \begin{prop}
  Let $M$ be a compact complex manifold and let $f:M\longrightarrow\mathbb{C}$ be a harmonic function. 
  Then $f$ must be a constant.
 \end{prop}

We now define and study almost complex manifolds.

 \begin{defin}
  Let $M$ be a differentiable manifold. A pair $(M,J)$ is called an almost complex
  manifold if $J$ is a tensor field of type $(1,1)$ such that at each point $P\in M,$ we have:
  $J_{P}^2=-\mathds{1}_{P}.$ The tensor field $J$ is called the almost complex structure of $M.$
 \end{defin}
 Let us look at the tangent space of a complex manifold $M$ with $\text{dim}_{\mathbb{C}}M=n$. The tangent 
 space $T_{P}M$, regarded as a $2n\text{-real}$ vector space, is spanned by $2n\text{-real}$ vectors
 \begin{equation*}
  \{\partial/\partial x^1,\dots,\partial/\partial x^n,\dots,
 \partial/\partial y^1\dots,\partial/\partial y^n\},
 \end{equation*}
 where $z^k=x^k+iy^k$ are the complex coordinates in a complex chart $(U,\phi).$
 With the same coordinates the dual space $T_{P}^*M$ as $2n\text{-dimensional}$ (real) vector space is spanned 
 by
 \begin{equation*}
  \{\mathrm{d}x^1,\dots,\mathrm{d}x^n,\mathrm{d}y^1,\dots,\mathrm{d}y^n\}.
 \end{equation*}
 Let us define $2n\text{-vectors}$
 
 \begin{equation}
 \label{basez}
  \partial/\partial z^k=\frac{1}{2}\{\partial/\partial x^k
 -i\partial/\partial y^k\},
 \end{equation}

 \begin{equation}
 \label{basezz}
   \partial/\partial\overline{z}^k=\frac{1}{2}\{\partial/\partial x^k
 +i\partial/\partial y^k\},
 \end{equation}
 where $1\leq k\leq n.$ Clearly they form a basis of the $2n\text{-dimensional}$ (complex) vector space
 $T_{P}M^{\mathbb{C}}=T_{P}M\otimes\mathbb{C}.$ Note that 
 $\overline{\frac{\partial}{\partial z^k}}=\frac{\partial}{\partial\overline{z}^k}.$ 
 Correspondingly $2n$ complex one-forms
 \begin{equation*}
   \mathrm{d}z^k=\mathrm{d}x^k+i\mathrm{d}y^k\hspace{0.5cm}\text{ and }\hspace{0.5cm}
   \mathrm{d}\overline{z}^k=\mathrm{d}x^k-id\mathrm{y}^k
 \end{equation*}

 form a complex basis of $T_{P}^*M^{\mathbb{C}}=T_{P}^{\ast}M\otimes\mathbb{C}.$ They are dual to
 $\partial/\partial z^k$ and
 $\partial/\partial\overline{z}^k.$
 
\begin{prop}
 Let $M$ be a complex manifold of (complex) dimension $n.$ Regarded as a $2n\text{-differentiable}$ 
 manifold, $M$ admits an almost complex 
 structure $J.$
\end{prop}
\begin{proof}
 Let $M$ be a complex manifold and define a linear map 
 \linebreak$J_P:T_PM\longrightarrow T_PM$ by:
 \begin{equation*}
  J_P\left(\left.\frac{\partial}{\partial x^k}\right|_P\right)=
  \left.\frac{\partial}{\partial y^k}\right|_P\hspace{0.5cm}\text{ and }\hspace{0.5cm}
  J_P\left(\left.\frac{\partial}{\partial y^k}\right|_P\right)=
  -\left.\frac{\partial}{\partial x^k}\right|_P.
 \end{equation*}

 $J_P$ is a real tensor of type $(1,1)$ and $J_{P}^2=-\mathds{1}_P$ where $\mathds{1}_P$ is the 
 identity map on $T_PM.$ 
 With respect to the basis  
 $\{\partial/\partial x^1,\dots,\partial/\partial x^n,\dots,
 \partial/\partial y^1\dots,\partial/\partial y^n\},$ $J_P$ takes the form:
 \begin{equation*}
   \left(\begin{array}{cc}
	0		     & -\mathds{1}_P\\
        \mathds{1}_P             &      0 
      \end{array}\right).
  \end{equation*}
  We only have to show that the action of $J_P$ is independent of the chart.
 In fact, let $(U,\phi)$ and $(V,\psi)$ be overlapping charts with local coordinates
 $z^k=x^k+iy^k$ and $w^k=u^k+iv^k.$
 On $U\cap V,$ the functions $z^k=z^k(w)$ satisfy the Cauchy-Riemann relations. Then we find:
 \begin{equation*}
  J\left(\frac{\partial}{\partial u^k}\right)=
 J\left(\frac{\partial x^h}{\partial u^k}\frac{\partial}{\partial x^h}+
 \frac{\partial y^h}{\partial u^k}\frac{\partial}{\partial y^h}\right)=
 \frac{\partial y^h}{\partial v^k}\frac{\partial}{\partial y^h}+
 \frac{\partial x^h}{\partial v^k}\frac{\partial}{\partial x^h}=\frac{\partial}{\partial v^k}.
 \end{equation*}
 We also find $J(\partial/\partial v^k)=-\partial/\partial u^k$ and this completes the proof.
\end{proof}

Let $(M,J)$ be an almost complex manifold. $J_P$ may be defined on $T_PM^{\mathbb{C}}$ by setting

\begin{equation*}
 J_P(X+iY)=J_PX+iJ_PY.
\end{equation*}
Then we have the identities
\begin{equation*}
 J_P(\partial/\partial z^k)=i\partial/\partial z^k\hspace{0.5cm}\text{ and }\hspace{0.5cm}
 J_P(\partial/\partial\overline{z}^k)=-i\partial/\partial\overline{z}^k.
\end{equation*}
Thus we have an expression for $J_P$ in (anti-) holomorphic bases,
\begin{equation*}
 J_P=idz^k\otimes\partial/\partial{z^k}-id\overline{z}^k\otimes\partial/\partial{z^k}
\end{equation*}
whose components are given by
\begin{equation*}
   \left(\begin{array}{cc}
	i\mathds{1}_P		     & 0\\
        0             &      -i\mathds{1}_P 
      \end{array}\right).
  \end{equation*}
Let $Z\in T_PM^{\mathbb{C}}$ be a vector of the form $Z=Z^k\partial/\partial z^k$. Then $Z$ is
an eigenvector of $J_P,$ in fact we have $J_PZ=iZ.$ In the same way if 
$W=W^k\partial/\partial\overline{z}^k,$
it satisfies $J_PW=-iW.$ In this way $T_PM^{\mathbb{C}}$ is decomposed into a direct sum
\begin{equation}
 \label{SUM}
 T_PM^{\mathbb{C}}=T_PM^{\mathbb{C}+}\oplus T_PM^{\mathbb{C}-},
\end{equation}
where
\begin{equation*}
 T_PM^{\mathbb{C}\pm}=\{Z\in T_PM^{\mathbb{C}}|J_PZ=\pm iZ\}.
\end{equation*}
Now $Z\in T_PM^{\mathbb{C}}$ is uniquely decomposed as $Z=Z^{+}+Z^-$ and $T_PM^{\mathbb{C}+}$ is
spanned by $\{\partial/\partial z^k\},$ while 
$T_PM^{\mathbb{C}-}$ by $\{\partial/\partial\overline{z}^k\}.$
In the same way \linebreak$TM^{\mathbb{C}}=TM\otimes\mathbb{C}$ is decomposed into a direct sum 
\begin{equation*}
 TM^{\mathbb{C}}=TM^{\mathbb{C}+}\oplus TM^{\mathbb{C}-}
\end{equation*}
where 
\begin{equation*}
 TM^{\mathbb{C}\pm}=\{Z\in TM^{\mathbb{C}}|JZ=\pm iZ\}.
\end{equation*}

\begin{defin}
 $Z\in T_PM^{\mathbb{C}+}$ is called a vector of type $(1,0),$ while \linebreak$W\in T_PM^{\mathbb{C}-}$ is 
 called a vector of type $(0,1).$
\end{defin}
We can easy verify that

\begin{equation*}
 T_PM^{\mathbb{C}-}=\overline{T_PM^{\mathbb{C}+}}.
\end{equation*}

\begin{note}
 We have the following identity for the dimensions of these complex vector spaces:
 \begin{equation*}
  \text{dim}_{\mathbb{C}}T_PM^{\mathbb{C}+}=\text{dim}_{\mathbb{C}}T_PM^{\mathbb{C}-}=
  \frac{1}{2}\text{dim}_{\mathbb{C}}T_PM^{\mathbb{C}}=\text{dim}_{\mathbb{C}}M.
 \end{equation*}
 \end{note}
 
 \begin{defin}
  Let $M$ be a complex manifold of (complex) dimension $n.$ We define 
  $\mathcal{X}(M)^{\mathbb{C}}=\mathcal{X}(M)\otimes\mathbb{C}.$

  Given a complex vector field $Z\in\mathcal{X}(M)^{\mathbb{C}},$ $Z$ is naturally decomposed as 
  $Z=Z^{+}+Z^-.$ $Z^+$ is called the component of type $(1,0),$ while $Z^-$ of type $(0,1).$
  Accordingly once $J$ is given, $\mathcal{X}(M)^{\mathbb{C}}$ is decomposed uniquely as
  \begin{equation}
   \mathcal{X}(M)^{\mathbb{C}}=\mathcal{X}(M)^{\mathbb{C}+}\oplus\mathcal{X}(M)^{\mathbb{C}-}
  \end{equation}
 \end{defin}
 
 \begin{defin}
  Let $(M,J)$ be an almost complex manifold. If the Lie bracket of any vector fields of type $(1,0)$
  $X,Y\in \mathcal{X}(M)^{\mathbb{C}+}$ is again of type $(1,0),$ i.e.
  $[X,Y]\in \mathcal{X}(M)^{\mathbb{C}+},$ the almost complex structure $J$ is said to be 
  integrable.
 \end{defin}
 
\begin{teo}
 Let $(M,J)$ be an almost complex manifold. If the almost complex structure $J$ is integrable then $M$ admits a
 complex structure and such complex structure is unique up to biholomorphism.
\end{teo}

\section{Complex differential forms}
In this section we define the main properties of complex differential forms over a complex manifold.

\begin{defin}
 Let $M$ be an $n\text{-dimensional}$ real differentiable manifold. We define 
 $A^p=\Gamma(M,T^{\ast}M\otimes\mathbb{C})$ the space of $\mathcal{C}^{\infty}$ 
 complex $p\text{-forms}$ over $M.$ In local coordinates $\omega\in A^p$ has the form:
 \begin{equation*}
  \omega=\omega_{j_1\ldots j_p}\mathrm{d}x^{j_1}\wedge\ldots\wedge \mathrm{d}x^{j_p},
 \end{equation*}
here $\omega_{j_1\ldots j_p}$ are $\mathcal{C}^{\infty}(U,\mathbb{C})$ functions on the local chart 
$(U;x^1,\ldots,x^n).$
\end{defin}
Now we define complex differential forms on a complex manifold $M$ and then we introduce the Dolbeault
complex and the Dolbeault cohomology groups.

\begin{defin}
 Let $M$ be a complex manifold of (complex) dimension $n$ and let $p,q\geq0.$ We define 
 \begin{equation*}
   A^{p,q}=\Gamma(M,\bigwedge_{p}T^{\ast}M^{\mathbb{C}+}\bigotimes\bigwedge_{q}T^{\ast}M^{\mathbb{C}-})
 \end{equation*}
 the space of $(p,q)\text{-forms}$
 over $M.$ As in the previous definition, in local coordinates $\omega\in A^{p,q}$ has the form:
 \begin{equation*}
  \omega=\omega_{i_1\ldots i_pj_1\ldots j_p}\mathrm{d}z^{i_1}\wedge\ldots\wedge\mathrm{d}z^{i_p}\wedge
  \mathrm{d}\overline{z}^{j_1}\wedge\ldots\wedge \mathrm{d}\overline{z}^{j_q},
 \end{equation*}
where $\omega_{i_1\ldots i_pj_1\ldots j_p}$ are $\mathcal{C}^{\infty}(U,\mathbb{C})$ functions on
$(U;z^1,\ldots,z^n,\overline{z}^1,\ldots,\overline{z}^n).$
\end{defin}

So we have $A^r=\sum_{p+q=r}A^{p,q}$ and the exterior differential $\mathrm{d}$ has the form 
$\mathrm{d}=\mathrm{d}'+\mathrm{d''},$
where $\mathrm{d}':A^{p,q}\longrightarrow A^{p+1,q}$ and $\mathrm{d''}:A^{p,q}\longrightarrow A^{p,q+1}.$
Now we can define the Dolbeault cohomology groups.

\begin{defin}
 Let $M$ be a complex manifold of (complex) dimension $n$ and let $r\geq0.$ The sequence of 
 $\mathbb{C}\text{-linear}$ maps
 
 \begin{equation*}
 \begin{tikzpicture}[node distance=2cm, auto]
  \node (A) {$A^{r,0}$};
  \node (B) [right of=A] {$A^{r,1}$};
  \node (C) [right of=B] {$\cdots$};
  \node (D) [right of=C] {$A^{r,n-1}$};
  \node (E) [right of=D] {$A^{r,n}$};
  \draw[->] (A) to node {$\mathrm{d''}$} (B);
  \draw[->] (B) to node {$\mathrm{d''}$} (C);
  \draw[->] (C) to node {$\mathrm{d''}$} (D);
  \draw[->] (D) to node {$\mathrm{d''}$} (E);
\end{tikzpicture}
\end{equation*}
 
is the Dolbeault complex. 

The set of 
$\mathrm{d''}\text{-closed}$ $(p,q)\text{-forms}$
is called the is denoted by $Z^{p,q}(M)$ and its element are called $(p,q)\text{-cocycles},$ while the set of 
$\mathrm{d''}\text{-exact}$
$(p,q)\text{-forms}$ is denoted by $B^{p,q}(M)$ and its element are called $(p,q)\text{-coboundaries}.$
Since $\mathrm{d''}\circ \mathrm{d''}=0,$ we have $B^{p,q}(M)\subseteq Z^{p,q}(M).$ 
The complex vector space
\begin{equation*}
 H^{p,q}(M)=\faktor{Z^{p,q}(M)}{B^{p,q}(M)}.
\end{equation*}
is called the $(p,q)\text{th}$ Dolbeault cohomology group of the complex manifold $M.$
\end{defin}

In the real case every closed differential form is locally exact (Poincar\'e's Lemma). 
As seen in \cite{GR} in the complex case we have a similar result due to Grothendieck.
Every $\mathrm{d''}\text{-closed}$ complex (p,q)-form, with $q\geq1,$ is locally 
$\mathrm{d''}\text{-exact}.$

\begin{defin}
 Let $r_1,\ldots,r_n\in(0,+\infty].$ A polydisc in $\mathbb{C}^n$ is a set 
 \linebreak$\Delta=\{(z^1,\ldots,z^n)\in\mathbb{C}^n||z^j|<r_j\}.$
\end{defin}

\begin{teo}
 \emph{(Dolbeault's Lemma)}
Let $\Delta\subseteq\mathbb{C}^n$ a polydisc with compact closure $\overline\Delta,$ i.e. $\Delta$ is bounded.
Let $\zeta$ a complex $(p,q)\text{-form}$ on an open neighborhood $U$ of $\overline\Delta$ and assume
$q\geq1.$ Then $\zeta$ is locally $\mathrm{d''}\text{-exact},$ i.e. there exists a complex 
$(p,q-1)\text{-form}$ $\eta$ on $\Delta$ such that $\mathrm{d''}\eta=\zeta.$
\end{teo}

\begin{note}
 $Z^{p,0}(M)$ is the set of holomorphic $p\text{-forms}.$
\end{note}

\section{K\"ahler manifolds}
Let $M$ be a complex manifold with $\text{dim}_{\mathbb{C}}M=n$ and let $g$ be a Riemann metric on $M$ 
as a differentiable manifold. Take $Z=X+iY$ and $W=U+iV$ and extend $g$ to $T_PM^{\mathbb{C}}$ so that
\begin{equation*}
\begin{split}
  g_P(Z,W)&=g_P(X+iY,U+iV)=\\
          &=g_P(X,U+iV)+g_P(iY,U+iV)=\\
          &=g_P(X,U)+ig_P(X,V)+ig_P(Y,U)-g_P(Y,V)=\\
          &=g_P(X,U)-g_P(Y,V)+i[g_P(X,V)+g_P(Y,U)].
\end{split}
\end{equation*}

The components of $g$ with respect to the bases
\eqref{basez} of $T_PM^{\mathbb{C}+}$ and \eqref{basezz} of $T_PM^{\mathbb{C}-}$ are:
\begin{equation*}
 g_{ij}(P)=g_P(\partial/\partial z^i,\partial/\partial z^j)
\end{equation*}

\begin{equation*}
 g_{i\bar\jmath}(P)=g_P(\partial/\partial z^i,\partial/\partial\overline{z}^j),
\end{equation*}

\begin{equation*}
 g_{\bar\imath j}(P)=g_P(\partial/\partial\overline{z}^i,\partial/\partial z^j),
\end{equation*}

\begin{equation*}
 g_{\bar\imath\bar\jmath}(P)=g_P(\partial/\partial\overline{z}^i,\partial/\partial\overline{z}^j).
\end{equation*}
We write $\bar\imath$ and $\bar\jmath$ to stress that $g$ is now meant to be extended to 
$T_PM^{\mathbb{C}}.$
One checks that:
\begin{equation}
\label{hermi}
 g_{ij}=g_{ji},\hspace{0.4cm}
 g_{\bar\imath\bar\jmath}=g_{\bar\jmath\bar\imath},\hspace{0.4cm}
 g_{\bar\imath j}=g_{j\bar\imath},\hspace{0.4cm}
 \overline{g_{i\bar\jmath}}=g_{\bar\imath j},\hspace{0.4cm}
 \overline{g_{ij}}=g_{\bar\imath\bar\jmath}.
\end{equation}

\begin{defin}
 If a Riemannian metric $g$ of a complex manifold $M$ satisfies
 \begin{equation}
 \label{hermitian}
 g_P(J_PX,J_PY)=g_P(X,Y)
 \end{equation}
at each point $P\in M$ and for any $X,Y\in T_PM,$ it is said to be a Hermitian metric. The pair $(M,g)$
is called a Hermitian manifold.
\end{defin}

\begin{teo}
 A complex manifold $(M,J)$ always admits a Hermitian metric.
\end{teo}
\begin{proof}
 Let $g$ be any Riemannian metric of a complex manifold $M$. Define a new metric $\tilde{g}$ by
 \begin{equation*}
  \tilde{g}_P(X,Y)=\frac{1}{2}[g_P(X,Y)+g_P(JX,JY)].
 \end{equation*}
 Clearly $\tilde{g}_P(JX,JY)=\tilde{g}_P(X,Y)$. Moreover, $\tilde{g}$ is positive definite provided 
 that $g$ is. Hence $\tilde{g}$ is a Hermitian metric on $M.$
\end{proof}
Let $g$ be a Hermitian metric on a complex maifold $M$. From \eqref{hermitian}, we find that
$g_{ij}=g_{\bar\imath\bar\jmath}.$ Thus the Hermitian metric $g$ takes the form
\begin{equation*}
 g=g_{i\bar\jmath}dz^i\otimes d\overline{z}^j+g_{\bar\imath j}d\overline{z}^i\otimes dz^j.
\end{equation*}

\begin{note}
 Take $X,Y\in T_PM^{\mathbb{C}+}.$ Define an inner product $h_P$ in $T_PM^{\mathbb{C}+}$ by 
 \begin{equation*}
  h_P(X,Y)=g_P(X,\overline{Y}).
 \end{equation*}
It is easy to see that $h_P$ is a positive-definite Hermitian form in $T_PM^{\mathbb{C}+}.$ In fact:
\begin{equation*}
 \overline{h(X,Y)}=\overline{g(X,\overline{Y})}=g(\overline{X},Y)=h(Y,\overline{X}),
\end{equation*}
and $h(X,X)=g(X,\overline{X})=g(X_1,X_1)+g(X_2,X_2)\geq0$ for $X=X_1+iX_2.$ This is why a metric $g$
satisfying \eqref{hermitian} is called Hermitian.
\end{note}

Let $(M,J,g)$ a Hermitian manifold. Define a tensor field $\omega$ whose action on 
$X,Y\in T_PM$ is 
\begin{equation}
 \label{kahler1}
 \omega_P(X,Y)=g_P(J_PX,Y).
\end{equation}
Note that $\omega$ is antisymmetric because we have identity
\begin{equation*}
 \omega(X,Y)=g(JX,Y)=g(J^2X,JY)=-g(JY,X)=-\omega(Y,X),
\end{equation*}
hence $\omega$ is a $2\text{-form}.$

\begin{defin}
 Let $(M,J,g)$ a Hermitian manifold. The $2\text{-form}$ $\omega$ is called the K\"ahler form of the 
 Hermitian metric $g.$
\end{defin}

\begin{defin}
 A K\"ahler manifold $(M,\omega)$ is a Hermitian manifold $(M,g)$ whose K\"ahler form $\omega$ is closed: 
 $\mathrm{d}\omega=0.$ The metric $g$ is called the K\"ahler metric of $M.$
\end{defin}

\begin{prop}
 Let $(M,g)$ a Hermitian manifold. The K\"ahler form $\omega$ is a complex $2\text{-form}$ of type $(1,1).$
\end{prop}
\begin{proof}
 Let $J$ be the almost complex structure and let $X,Y\in TM^{\mathbb{C}+}.$ From the definition of
 $TM^{\mathbb{C}+}$ we have $JX=iX$ and $JY=iY.$
 Then we have
 \begin{equation*}
  \begin{split}
   \omega(X,Y)&=\omega(-iJX,-iJY)=g(J(-iJX),-iJY)=\\
              &=(-i)^2g(J^2X,JY)=-g(-X,JY)=\\
              &=g(X,JY)=g(JY,X)=\omega(Y,X)=\\
              &=-\omega(X,Y).
  \end{split}
 \end{equation*}
 Hence, $\omega(X,Y)=0.$ In the same way, if $X,Y\in TM^{\mathbb{C}-},$ $\omega(X,Y)=0,$ proving that the 
 K\"ahler form $\omega$ is of type $(1,1).$
\end{proof}

\begin{note}
 Let $(M,g)$ a Hermitian manifold and let $\omega$ its K\"ahler form. In local coordinates we have
 \begin{equation}
 \label{kahlerlocal}
  \omega=ig_{i\bar\jmath}\mathrm{d}z^i\wedge \mathrm{d}\overline{z}^j.
 \end{equation}
\end{note}

\begin{prop}
 Let $(M,g)$ a Hermitian manifold. The K\"ahler form $\omega$ is real, i.e.,
 $\overline{\omega}=\omega.$
\end{prop}
\begin{proof}
 In local coordinates we have
 \begin{equation*}
 \begin{split}
   \overline{\omega}&=\overline{ig_{i\bar\jmath}\mathrm{d}z^{i}\wedge \mathrm{d}\overline{z}^j}=
 -i\overline{g_{i\bar\jmath}}\mathrm{d}\overline{z}^i\wedge \mathrm{d}z^j=\\
   &=-ig_{j\bar\imath}\mathrm{d}\overline{z}^i\wedge \mathrm{d}z^j=
   ig_{j\bar\imath}\mathrm{d}z^j\wedge \mathrm{d}\overline{z}^i=\omega.
  \end{split}
 \end{equation*}
\end{proof}

\begin{prop}
 A Hermitian manifold $(M,g)$ of complex dimension $1$ is K\"ahler.
\end{prop}
\begin{proof}
 In local coordinates the K\"ahler form is 
 $\omega=ig_{i\bar\jmath}\mathrm{d}z^i\wedge \mathrm{d}\overline{z}^j.$
 Then we have
 \begin{equation*}
  \mathrm{d}\omega=i\frac{\partial g_{i\bar\jmath}}{\partial z^k}\mathrm{d}z^k\wedge \mathrm{d}z^i\wedge
  \mathrm{d}\overline{z}^j+i\frac{\partial g_{i\bar\jmath}}{\partial\overline{z}^h}
  \mathrm{d}\overline{z}^h\wedge \mathrm{d}z^i\wedge \mathrm{d}\overline{z}^j=0.
 \end{equation*}
\end{proof}

\begin{prop}
 A Hermitian manifold $(M,g)$ of (complex) dimension $n$ regarded as
 a $2n\text{-dimensional}$ differentiable manifold is always orientable.
\end{prop}
\begin{proof}
 Let $\omega$ the K\"ahler form. Then $\omega^n$ is a $2n$ non-vanishing real form, in fact
 \begin{equation*}
  \omega^n=(i)^nn!\det(g_{i\bar\jmath})\mathrm{d}z^1\wedge \mathrm{d}\overline{z}^1\wedge\cdots
  \wedge \mathrm{d}z^n\wedge \mathrm{d}\overline{z}^n
 \end{equation*}
 and $\det(g_{i\bar\jmath})$ is nowhere vanishing.
\end{proof}

\begin{prop}
 Let $(M,g)$ be a Hermitian manifold. The following conditions are equivalent:
 \begin{enumerate}
  \item $g$ is a K\"ahler metric, i.e., $\mathrm{d}\omega=0,$
  \item $\displaystyle\frac{\partial g_{i\bar\jmath}}{\partial z^k}=
         \frac{\partial g_{k\bar\jmath}}{\partial z^i},$
  \item $\displaystyle\frac{\partial g_{i\bar\jmath}}{\partial\overline{z}^k}=
         \frac{\partial g_{i\bar{k}}}{\partial\overline{z}^j},$
  \item There exists a locally defined real function $f$ such that $\omega=i\mathrm{d}'\mathrm{d''}f,$
  i.e., 
  \begin{equation*}
   g_{i\bar\jmath}=\frac{\partial^2f}{\partial z^i\partial\overline{z}^j}.
  \end{equation*}
 \end{enumerate}
\end{prop}
\begin{proof}
 See \cite{KOB} for a detailed proof.
\end{proof}

\section{The Hopf manifold}

Let $\Delta$ be the discrete group generated by $z^k\mapsto2z^k,$ $1\leq k\leq n.$ The quotient manifold
$(\mathbb{C}^n\setminus\{0\})/\Delta$ is a complex manifold called the Hopf manifold. It is homeomorphic to
$S^1\times S^{2n-1}.$

Now, we want to prove that for $n\geq2$ the Hopf manifold cannot be given a K\"ahlerian structure. In fact compact
K\"ahler manifolds have strong topological restrictions. Perhaps the simplest among them is the following:

\begin{prop}
 The second Betti number of a compact K\"ahler manifold $(M,\omega)$ of (complex) dimension $n$ is 
 strictly positive.
\end{prop}
\begin{proof}
 Let $(M,\omega)$ be a compact K\"ahler manifold of (complex) dimension $n.$
 Since $\omega$ is closed, it determines an element $u=[\omega]\in H^{2}_{\text{de}}(M).$
 Consider now the $2n\text{-form}$ $\omega^n,$ this determines an element $u^n=[\omega]^n=[\omega^n]$ in 
 $H^{2n}_{\text{de}}(M).$ Integrating the volume form $\frac{\omega^n}{n!}$ 
we find
\begin{equation*}
 \int_{M}\frac{\omega^n}{n!}=\mathrm{Vol}(M)>0,
\end{equation*}
hence $u^n\neq0.$ 
Therefore $u\neq0,$ and then $H^2_{\text{de}}(M)\neq0.$ 
\end{proof}

\begin{cor}
 Let $n\geq2.$ The Hopf manifold $S^1\times S^{2n-1}$ cannot be given a K\"ahlerian structure.
\end{cor}
\begin{proof}
Since $n\geq2,$ $H^1_{\text{de}}(S^{2n-1})=H^2_{\text{de}}(S^{2n-1})=0.$
 From K\"unneth formula we have
 \begin{equation*}
 \begin{split}
  H^2_{\text{de}}(S^1\times S^{2n-1})&=
  \bigoplus_{p+q=2}\left[H^p_{\text{de}}(S^1)\otimes H^q_{\text{de}}(S^{2n-1})\right]=\\
  &=\left[H^2_{\text{de}}(S^1)\times H^0_{\text{de}}(S^{2n-1})\right]\oplus\\
  &\oplus\left[H^1_{\text{de}}(S^1)\times H^1_{\text{de}}(S^{2n-1})\right]\oplus\\
  &\oplus\left[H^0_{\text{de}}(S^1)\times H^2_{\text{de}}(S^{2n-1})\right]=\\
  &=\left[0\otimes\mathbb{R}\right]\oplus\left[\mathbb{R}\otimes0\right]\oplus
    \left[\mathbb{R}\otimes0\right]=0.
 \end{split}
 \end{equation*}
Then, for $n\geq2$ the second Betti number of the Hopf manifold $S^1\times S^{2n-1}$ is $0,$ 
and this shows that for $n\geq2$ the Hopf manifold cannot be given a K\"ahlerian structure.
\end{proof}

\chapter{Principal Fibre Bundles}
In this chapter we present some definitions and results on principal fibre bundles over a differentiable
manifold $M.$ See \cite{KN2} for more details.

\section{Basics on principal G-bundles}

\begin{defin}
Let $M$ be a differentiable manifold and $G$ a Lie group. A (differentiable) principal fibre bundle over
$M$ with group $G$ consists of a manifold $P$ and an action $G$ on $P$
satisfying the following conditions:
\begin{enumerate}
 \item $G$ acts freely on $P$ on the right,
 \item $M$ is the quotient space of $P$ by the equivalence relation induced by $G,$ i.e., $M=\faktor{P}{G}$
       and the canonical projection $\pi:P\longrightarrow M$ is differentiable,
 \item $P$ is locally trivial, that is, every point $x$ of $M$ has a 
       neighborhood $U$ such that $\pi^{-1}(U)$ is isomorphic with 
       $U\times G$ in the sense that there is a diffeomorphism
       $\psi:\pi^{-1}(U)\longrightarrow U\times G$ such that
       $\psi(u)=(\pi(u),\varphi(u))$ where $\varphi$ is a mapping of 
       $\pi^{-1}(U)$ into $G$ satifsfying $\varphi(ua)=(\varphi(u))a$
       for all $u\in\pi^{-1}(U)$ and $a\in G.$
\end{enumerate}
A principal fibre bundle will be denoted by $P(M,G),$ or simply $P.$ We call $P$ the total space, $M$ the
base space, $G$ the structure group and $\pi$ the projection. For each $x\in M,$ $\pi^{-1}(x)$ is a closed
submanifold of $P$ called the fibre over $x.$ If $u\in\pi^{-1}(x)$ is a point in the fibre over $x,$ then
$\pi^{-1}(x)$ is the set of point $ua$ with $a\in G.$
Every fibre is diffeomorphic to $G.$
\end{defin}

\begin{defin}
 Let $P(M,G)$ and $Q(N,H)$ be principal fibre bundles over the (differentiable)
 manifolds $M$ and $N$ respectively. A morphism of principal fibre bundles 
 consist of a $\mathcal{C}^{\infty}$ mapping $f:P\longrightarrow Q$ and a 
 homomorphism $\psi:G\longrightarrow H$ such that $f(ua)=f(u)\psi(a)$ for all $u\in P$ and $a\in G.$
 For the sake of simplicity, we shall denote $f$ and $\psi$ by the same letter $f.$
 Every morphism $f:P\longrightarrow Q$ maps each fibre of $P$ into a fibre of $Q$ and hence it induces 
 a mapping of $M$ into $N,$ which will be also denoted by $f.$
\end{defin}

\begin{defin}
A morphism $f:P(M,G)\longrightarrow Q(N,H)$ is called an imbedding if $(P,f)$ is a regular submanifold of 
$Q$ and if $f:G\longrightarrow H$ is a homomorphism of groups. If $f:P\longrightarrow Q$ is an
imbedding, then $(M,f)$ is a regular submanifold of $N.$ By identifying $P$ with $f(P),$ $G$ with $f(G)$ and
$M$ with $f(M),$ we say that $P(M,G)$ is a subbundle of $Q(N,H).$
\end{defin}

\begin{defin}
 Given a principal fibre bundle $P(M,G),$ the action of $G$ on $P$ induces a homomorphism $\sigma$ of the
 Lie algebra $\frak{g}$ of $G$ into the Lie algebra $\mathcal{X}(P)$ of vector fields on $P$ (see 
 Proposition 4.1 in \cite{KN2}). For each $A\in\frak{g},$ $A^{\ast}=\sigma(A)$ is called the foundamental
 vector field corresponding to $A.$
\end{defin}

In order to relate our intrinsic definition of a principal fibre bundle to the definition and the construction
by means of an open cover, we need the concept of transition functions. By $(3)$ for a principal fibre bundle
$P(M,G)$ it is possible to choose an open covering $\{U_{\alpha}\}$ of $M$ where each $\pi^{-1}(U_{\alpha})$
provided 
with a diffeomorphism $u\longrightarrow (\pi(u),\varphi_{\alpha}(u))$ of $\pi^{-1}(U_{\alpha})$ onto 
$U_{\alpha}\times G$ such that 
\linebreak$\varphi_{\alpha}(ua)=(\varphi_{\alpha}(u))a$ for all $u\in\pi^{-1}(U_{\alpha})$
and $a\in G.$

If $u\in\pi^{-1}(U_{\alpha}\cap U_{\beta}),$ then 
$\varphi_{\beta}(ua)(\varphi_{\alpha}(ua))^{-1}=\varphi_{\beta}(u)(\varphi_{\alpha}(u))^{-1},$ which shows that 
$\varphi_{\beta}(u)(\varphi_{\alpha}(u))^{-1}$ depends only on $\pi(u)$ not on $u.$ Hence, if 
\linebreak$U_{\alpha}\cap U_{\beta}\neq\emptyset,$ we can define a mapping
$\psi_{\beta\alpha}:U_{\alpha}\cap U_{\beta}\longrightarrow G$ by 
\linebreak$\psi_{\beta\alpha}(\pi(u))=\varphi_{\beta}(u)(\varphi_{\alpha}(u))^{-1}.$ The family of mappings 
$\psi_{\beta\alpha}$ are called transition functions of the bundle $P(M,G)$ corresponding to the open cover 
$\{U_{\alpha}\}$ of $M.$ It is easy to verify that
\begin{equation*}
 \psi_{\alpha\beta}(x)\circ\psi_{\beta\gamma}(x)\circ\psi_{\gamma\alpha}(x)=I_G\text{ for }
  x\in U_{\alpha}\cap U_{\beta}\cap U_{\gamma}.
\end{equation*}
Conversely, we have the following:

\begin{teo}
 Let $M$ be a differentiable manifold, $\{U_{\alpha}\}$ an open cover of $M$ and $G$ a Lie group.
 Given a mapping $\psi_{\alpha\beta}:U_{\alpha}\cap U_{\beta}\longrightarrow G$ for every nonempty 
 $U_{\alpha}\cap U_{\beta}$ such that
 \begin{equation*}
  \psi_{\alpha\beta}(x)\circ\psi_{\beta\gamma}(x)\circ\psi_{\gamma\alpha}(x)=I_G\text{ for }
  x\in U_{\alpha}\cap U_{\beta}\cap U_{\gamma},
 \end{equation*}
 we can construct a (differentiable) principal fibre bundle $P(M,G)$ with transition functions
 $\psi_{\alpha\beta}.$
\end{teo}

\begin{defin}
 Let $P(M,G)$ a principal fibre bundle over the differentiable manifold $M$ and let $f:N\longrightarrow M$
 be a $\mathcal{C}^{\infty}$ mapping. The pull-back of $P$ (via $f$) is the principal fibre bundle over $N$ 
 defined as the fibre product \linebreak$f^{\ast}P=N\times_{M}P=\{(n,u)\in N\times P|f(n)=\pi(u)\}.$
\end{defin}

\section{Connections in a principal fibre bundle}
\begin{defin}
 Let $P(M,G)$ be a principal fibre bundle over a manifold $M$ with structure group $G.$ For each $u\in P,$
 let $T_uP$ be the tangent space of $P$ at the point $u$ and $G_u\subseteq T_uP$ the subspace of $T_uP$
 consisting of vectors tangent to the fibre through $u.$ 
 A connection $\Gamma$ in $P$ is a $\mathcal{C}^{\infty}$ assignment of a subspace $Q_u$ of $T_uP$ for 
 each $u\in P$ such that
 \begin{enumerate}
  \item $T_uP=G_u\oplus Q_u,$
  \item $Q_{ua}=(R_a)_{\ast}Q_u$ for every $u\in P$ and $a\in G.$
 \end{enumerate}
The second condition means that the distribution $u\longrightarrow Q_u$ is invariant under the action of
$G$ over the total space $P.$
We call $G_u$ the vertical subspace and $Q_u$ the horizontal subspace of $T_uP.$ In this way a vector
$X_u\in T_uP$ can be uniquely written as direct sum of his vertical and horizontal component:
\begin{equation*}
 X_u=Y_u\oplus Z_u\hspace{0.5cm}\text{ where }\hspace{0.5cm}Y_u\in G_u,\hspace{0.5cm}Z_u\in Q_u.
\end{equation*}
\end{defin}

\begin{defin}
 Given a connection $\Gamma$ in a principal fibre bundle $P,$ we define a 1-form $\omega$ on $P$ with 
 values in the Lie algebra $\frak{g}$ of $G$ as follows. For each $X_u\in T_uP,$ we define $\omega(X)$
 to be the unique $A\in\frak{g}$ such that $(A^{\ast})_u$ is equal to the vertical component of $X_u.$
 The form $\omega$ is called the connection form of the given connection $\Gamma.$
\end{defin}

\begin{prop}
 The connection form $\omega$ of a connection $\Gamma$ in $P$ satisfies the following conditions:
 \begin{enumerate}
  \item $\omega(A^{\ast})=A$ for every $A\in\frak{g},$
  \item $(R_a)^{\ast}\omega=\mathrm{ad}_{a^{-1}}\omega,$ that is, 
        $\omega((R_a)_{\ast}X)=\mathrm{ad}_{a^{-1}}(\omega(X))$ for every $a\in G$ and every vector field
        $X$ on $P,$ where $\mathrm{ad}$ denotes the adjoint representation of $G$ on $\frak{g}.$
 \end{enumerate}
\end{prop}
The projection $\pi:P\longrightarrow M$ induces a linear mapping $\pi:T_uP\longrightarrow T_xM$ for each
$u\in P,$ where $x=\pi(u).$ When a connection is given, $\pi$ maps the horizontal subspace $Q_u$
isomorphically onto $T_xM.$

\begin{defin}
 Given a connecion $\Gamma$ in $P,$ the horizontal lift of a vector field $X$ on $M$ is the unique vector
 field $X^{\ast}$ on $P$ which is horizontal and which projects onto $X,$ i.e. 
 $\pi(X_{u}^{\ast})=X_{\pi(u)}$ for every $u\in P.$ 
\end{defin}

\begin{prop}
 Let $\Gamma$ be a connection in $P$ and let $X$ be a vector field on $M,$ there is a unique horizontal lift
 $X^{\ast}$ of $X.$ The lift $X^{\ast}$ is invariant by $R_a$ for every $a\in G.$ Conversely, every
 horizontal vector field $X^{\ast}$ on $P$ invariant by $G$ is the lift of a vector field $X$ on $M.$
\end{prop}

\begin{prop}
 Let $X^{\ast}$ and $Y^{\ast}$ be the horizontal lifts of $X$ and $Y$ respectively. Then
 \begin{enumerate}
  \item $X^{\ast}+Y^{\ast}$ is the horizontal lift of $X+Y.$
  \item For every function $f\in\mathcal{C}^{\infty}(M,\mathbb{R}),$ $f^{\ast}X^{\ast}$ is the horizontal
        lift of $fX$ where $f^{\ast}$ is the pull-back function $f^{\ast}=f\circ\pi.$
  \item The horizontal component of $[X^{\ast},Y^{\ast}]$ is the horizontal lift of $[X,Y].$
 \end{enumerate}
\end{prop}

\section{Curvature form and structure equation}
\begin{defin}
 Let $P(M,G)$ be a principal fibre bundle and $\rho$ a representation of $G$ on a finite dimensional 
 vector space $V,$ i.e., a group homomorphism \linebreak$\rho:G\longrightarrow GL(V).$ 
 A pseudotensorial form of 
 degree $r$ on $P$ of type $(\rho,V)$ is a $V\text{-valued}$ $r\text{-form}$ $\varphi$ on $P$ such that
 \begin{equation*}
  R_{a}^{\ast}\varphi=\rho(a^{-1})\cdot\varphi\hspace{0.5cm}\text{ for }\hspace{0.5cm}a\in G.
 \end{equation*}
Such a form is called a tensorial form if it is horizontal in the sense that 
\linebreak$\varphi(X_1,\ldots,X_r)=0$ whenever at least one of the tangent vectors $X_i$ is vertical, 
i.e., tangent to a fibre.
\end{defin}
By the above definition we see that if $\Gamma$ is a connection in a principal fibre bundle $P$ with 
structure group $G,$ then its connection form $\omega$ is a pseudotensorial $1\text{-form}$ of type
$(\mathrm{ad},\frak{g}).$
Now we present the main results of this section (see \cite{KN2}, p.76-78 for details). 
Let $\Gamma$ be a connection in $P$ and let $G_u$ and $Q_u$ be the vertical and the horizontal subspaces 
of $T_uP,$ respectively. Let $h:T_uP\longrightarrow Q_u$ be the projection.

\begin{prop}
 If $\varphi$ is a pseudotensorial $r\text{-form}$ on $P$ of type $(\rho,V),$ then
 \begin{enumerate}
  \item The form $\varphi h$ defined by $(\varphi h)(X_1,\ldots,X_r)=\varphi(hX_1,\ldots,hX_r)$ is a
        tensorial form of type $(\rho,V),$
  \item $\mathrm{d}\varphi$ is a pseudotensorial $(r+1)\text{-form}$ of type $(\rho,V).$
 \end{enumerate}
\end{prop}

\begin{defin}
 The $(r+1)\text{-tensorial}$ form $D\varphi=(\mathrm{d}\varphi)h$ is called the exterior covariant derivative
 of $\varphi$ and $D$ is called exterior covarian differentiation.
\end{defin}

\begin{defin}
 Let $\Gamma$ be a connection in $P$ and let $\omega$ be its connection form.
 By the above Proposition, $\Omega=D\omega$ is a tensorial $2\text{-form}$ of type $(\mathrm{ad},\frak{g}).$ 
 It is called the curvature form of $\Gamma.$ 
\end{defin}
 
 \begin{teo}
  \emph{(Structure equation)}
  Let $\Gamma$ be a connection in a principal fibre $P$ and let $\omega$ and $\Omega$ be its connection form and curvature form 
  respectively. Then
   \begin{equation*}
  \Omega(X,Y)=\mathrm{d}\omega(X,Y)+\frac{1}{2}[\omega(X),\omega(Y)].
 \end{equation*}
 \end{teo}

\begin{teo}
 \emph{(Bianchi's identity)}
 Let $\Gamma$ be a connection in a principal fibre bundle $P$ and let $\omega$ and $\Omega$ its connection
 form and its curvature form respectively. Then the exterior covariant derivative of the curvature is $0,$
 i.e.,
 \begin{equation*}
  D\Omega=0.
 \end{equation*}
\end{teo}

\chapter{Connections in Vector Bundles}
Although our primary interest lies in holomorphic vector bundles, we begin this chapter with the study 
of connections in differentiable complex vector bundles.

\begin{defin}
 Let $M$ be a differentiable manifold of (real) dimension $n.$ A $\mathcal{C}^{\infty}$ complex vector bundle
 over $M$ of rank $r$ is differentiable manifold $E$ together with a surjective $\mathcal{C}^{\infty}$ map 
 $\pi:E\longrightarrow M$ such that
 exists a countable open covering $\{U_{\alpha}\}$ satisfying:
 \begin{enumerate}
  \item There are diffeomorphisms $\phi_{\alpha}:\pi^{-1}(U_{\alpha})\longrightarrow 
        U_{\alpha}\times\mathbb{C}^r$ such that the diagrams commute 
        \begin{equation*}
 \begin{tikzpicture}[node distance=3cm, auto]
\node (A) {$\pi^{-1}(U_{\alpha})$};
  \node (B) [below of=A] {$U_{\alpha}$};
  \node (C) [right of=A] {$U_{\alpha}\times\mathbb{C}^r$};
  \draw[->] (A) to node [swap]{$\pi$} (B);
  \draw[->] (A) to node  {$\phi_{\alpha}$} (C);
  \draw[->] (C) to node  {$pr_1$} (B);
\end{tikzpicture}.
\end{equation*}
Here $pr_1$ is the projection on the first factor.
 \item When $U_{\alpha}\cap U_{\beta}\neq\emptyset,$ $\phi_{\beta}\circ\phi_{\alpha}^{-1}$ is an automorphism 
 of $(U_{\alpha}\cap U_{\beta})\times\mathbb{C}^{r}$ with the condition that 
$\forall P\in U_{\alpha}\cap U_{\beta}$
\begin{equation*}
 \phi_{\beta}\circ\phi_{\alpha}^{-1}(P,\ast):\mathbb{C}^r\longrightarrow\mathbb{C}^r
\end{equation*}
is $\mathbb{C}\text{-linear}.$
 \end{enumerate}
\end{defin}

\begin{note}
 From $(1)$ it immediately follows that if  $U_{\alpha}\cap U_{\beta}\neq\emptyset,$ then
 \linebreak$\phi_{\beta}\circ\phi_{\alpha}^{-1}(P,\ast):\mathbb{C}^r\longrightarrow\mathbb{C}^r$ is invertible.
 In fact it is a bijective $\mathbb{C}\text{-linear}$ endomorphism of $\mathbb{C}^r.$  
\end{note}

\begin{note}
 In the same manner, if $M$ is complex manifold, we can define a holomorphic vector bundle $E$ over $M$ in
 the obvious way.
\end{note}

As in the case of principal fibre bundles, we have the following result:

\begin{teo}
\label{cocycle}
 Let $M$ be a differentiable manifold, $\{U_{\alpha}\}$ an open covering of $M.$ 
 Given a $\mathcal{C}^{\infty}$ mapping
 $\psi_{\alpha\beta}:U_{\alpha}\cap U_{\beta}\longrightarrow GL(r,\mathbb{C})$ for every nonempty 
 $U_{\alpha}\cap U_{\beta}$ such that for every $x\in U_{\alpha}\cap U_{\beta}\cap U_{\gamma}$
 \begin{equation}
 \label{cocycleidentity}
  \psi_{\alpha\beta}(x)\circ\psi_{\beta\gamma}(x)\circ\psi_{\gamma\alpha}(x)=\mathds{1}_{\mathbb{C}^r}
 \end{equation}
 we can construct a complex vector bundle of rank r over the manifold $M$ with transition functions
 $\psi_{\alpha\beta}.$
\end{teo}

\begin{defin}
 Let $(E,M,\pi_E)$ and $(F,N,\pi_F)$ be complex vector bundles of rank $r$ and $p$ over the differentiable
 manifolds $M$ and $N,$ respectively. A morphism of complex vector bundles is a pair of 
 $\mathcal{C}^{\infty}$ mapping $f:E\longrightarrow F$ and $\tilde{f}:M\longrightarrow N$ such that
 \begin{enumerate}
  \item The following diagram commutes
              \begin{equation*}
                  \begin{tikzpicture}[node distance=2cm, auto]
                       \node (A) {$E$};
		        \node (B) [right of=A] {$F$};
                       \node (C) [below of=A] {$M$};
                       \node (D) [below of=B] {$N$};
                       \draw[->] (A) to node {$f$} (B);
                       \draw[->] (A) to node  [swap ]{$\pi_E$} (C);
                       \draw[->] (B) to node  {$\pi_F$} (D);
                       \draw[->] (C) to node  {$\tilde{f}$} (D);
                   \end{tikzpicture}.
            \end{equation*}
   \item  For every $P\in M$ the mapping $\left.f\right|_{E_P}:E_P\longrightarrow F_{\tilde{f}(P)}$ is
          $\mathbb{C}\text{-linear}$ on the fibres.
 \end{enumerate}
Finally, if $f$ is a diffeomorphism, we say that
the complex vector bundles $(E,M,\pi_E)$ and $(F,N,\pi_F)$ are isomorphic.
\end{defin}

\begin{defin}
 Let $(E,M,\pi_E)$ be a complex vector bundle of rank $r$ over the differentiable manifold $M.$ A subbundle
 of $E$ is a complex vector bundle $(S,M,\pi_S)$ over $M$ with a morphism $i:S\longrightarrow E$ such
 that the following diagram commutes 
            \begin{equation*}
                  \begin{tikzpicture}[node distance=2cm, auto]
                       \node (A) {$S$};
		        \node (B) [right of=A] {$E$};
                       \node (C) [below of=A] {$M$};
                       \node (D) [below of=B] {$M$};
                       \draw[->] (A) to node {$i$} (B);
                       \draw[->] (A) to node  [swap ]{$\pi_S$} (C);
                       \draw[->] (B) to node  {$\pi_E$} (D);
                       \draw[->] (C) to node  {$\mathrm{Id}_M$} (D);
                   \end{tikzpicture}.
            \end{equation*}
\end{defin}

\begin{defin}
 Let $(E,M,\pi)$ be a complex vector bundle of rank $r$ over the differentiable manifold $M$ and let 
 $f:N\longrightarrow M$ be a $\mathcal{C}^{\infty}$ mapping. The pull-back of $E$ (via $f$) is the
 complex vector bundle over $N$ defined as the fibre product
 \linebreak$f^{\ast}E=N\times_{M}E=\{(n,e)\in M\times E|f(n)=\pi(e)\}.$
\end{defin}

\begin{note}
From the above definition we can deduce that the following diagram
  \begin{equation*}
                  \begin{tikzpicture}[node distance=2cm, auto]
                       \node (A) {$f^{\ast}E$};
		        \node (B) [right of=A] {$E$};
                       \node (C) [below of=A] {$M$};
                       \node (D) [below of=B] {$N$};
                       \draw[->] (A) to node {$pr_2$} (B);
                       \draw[->] (A) to node  [swap ]{$pr_1$} (C);
                       \draw[->] (B) to node  {$\pi$} (D);
                       \draw[->] (C) to node  {$f$} (D);
                   \end{tikzpicture}
            \end{equation*}
            is commutative.
 In particular $(f^{\ast}E,M,pr_1)$ is a complex vector bundle over $M$ of rank $r$ and 
 $pr_2:f^{\ast}E\longrightarrow E$ is a morphism of complex vector bundles.
\end{note}

\begin{note}
 In the same way we can define morphisms, subbundles and pull-back bundles of holomorphic vector bundles
 over complex manifolds.
\end{note}

\section{Connections in complex vector bundles (over real manifolds)}
Let $M$ be a differentiable manifold of (real) dimension $n$ and let $E$ a $\mathcal{C}^{\infty}$ complex
vector bundle of rank $r$ over $M.$ We make use of the following notation:
\begin{enumerate}
 \item $A^p$ is the space of $\mathcal{C}^{\infty}$ complex $p\text{-forms}$ over $M,$
 \item $A^p(E)=\Gamma(M,\bigwedge_{p}T^{*}M^{\mathbb{C}}\otimes E)$ is the space of $\mathcal{C}^{\infty}$ 
       complex $p\text{-forms}$ over $M$ with values in $E.$
\end{enumerate}

\begin{defin}
 A connection $D$ in $E$ is a $\mathbb{C}\text{-linear}$ homomorphism \linebreak$D:A^0(E)\longrightarrow A^1(E)$
such that
\begin{equation}
 \label{leibniz}
 D(f\sigma)=\sigma\otimes\mathrm{d}f+fD\sigma\hspace{0.5cm}\text{ for }
 f\in A^0\hspace{0.5cm}\text{ and }\sigma\in A^0(E).
\end{equation}
By abuse of notation we omit $\otimes$ and write $D(f\sigma)=\sigma\mathrm{d}f+fD\sigma.$
\end{defin}

\begin{defin}
Let $s=(s_1,\ldots,s_r)$ be a local frame field of $E$ over an open set $U\subseteq M,$ then given a
connection $D$, we can write 
\begin{equation}
 \label{local}
 Ds_i=\sum_{j=1}^{r}s_j\omega_{i}^{j}.
\end{equation}
We call the matrix 1-form $\omega=(\omega_{i}^{j})$ the connection form of $D$ with respect to the frame
fiels $s.$
\end{defin}

Considering $s=(s_1,\ldots,s_r)$ as a row vector, we can rewrite \eqref{local} in matrix notation 
as follows:
\begin{equation*}
 Ds=s\omega.
\end{equation*}

\begin{defin}
 If $\zeta=\zeta^is_i$ is an arbitrary section of $E$ over $U,$ then \eqref{leibniz} and \eqref{local}
imply
\begin{equation*}
 D\zeta=\sum_{i=1}^{r}s_i\mathrm{d}\zeta^i+\sum_{i,j=1}^{r}s_i\zeta^j\omega_{j}^{i}.
\end{equation*}
We call $D\zeta$ the covariant derivative of $\zeta.$
\end{defin}
Evaluating $D$ on a tangent vector $X$ of $M$ at the point $x,$ we obtain an element of the fibre $E_x$
denoted by
\begin{equation*}
 D_X\zeta=(D\zeta)(X)\in E_x
\end{equation*}

\begin{defin}
 A section $\zeta$ of $E$ is said to be parallel if $D\zeta=0.$ If $c=c(t),$ $0<t<a,$ is a curve
 in $M,$ a section $\zeta$ defined along $c$ if
 \begin{equation}
 \label{curve}
  D_{c'(t)}\zeta=0\hspace{0.5cm}\text{ for }0\leq t\leq a, 
 \end{equation}
In terms of the local frame field $s,$ \eqref{curve} can be written as a system of ordinary differential
equations
\begin{equation*}
 \frac{d\zeta^i}{dt}+\sum_{j=1}^{r}\omega_{j}^{i}(c'(t))\zeta^j=0.
\end{equation*}
\end{defin}
We shall now study how the connection form $\omega$ changes when we change the local frame field $s$. Let
$s'=(s'_1,\ldots,s'_r)$ be another local frame field over $U.$ It is related to $s$ by
\begin{equation*}
 s=s'a,
\end{equation*}
where $a:U\longrightarrow GL(r,\mathbb{C})$ is a matrix-valued function on $U.$ Let $\omega'$ be the 
connection form of $D$ with respect to $s'.$ Then
\begin{equation}
\label{connectionchange}
 \omega=a^{-1}\omega'a+a^{-1}\mathrm{d}a.
\end{equation}
In fact we have,
\begin{equation*}
 s\omega=Ds=D(s'a)=(Ds')a+s'\mathrm{d}a=s'\omega'a+s'\mathrm{d}a=s(a^{-1}\omega'a+a^{-1}\mathrm{d}a).
\end{equation*}
We extend a connection $D:A^0(E)\longrightarrow A^1(E)$ to a $\mathbb{C}\text{-linear}$ map
\begin{equation*}
 D:A^p(E)\longrightarrow A^{p+1}(E),\hspace{0.5cm}\text{ for }p\geq0.
\end{equation*}
by setting
\begin{equation*}
 D(\eta\otimes s)=\mathrm{d}\eta\otimes s+(-1)^{p}\eta\otimes Ds\hspace{0.5cm}\text{ for }
 \eta\in A^{p}\hspace{0.5cm}\text{ and }s\in A^{0}(E).
\end{equation*}

\begin{defin}
 Using this extended $D,$ we define the curvature $R$ of the connection $D$ to be
 \begin{equation*}
  R=D\circ D:A^0(E)\longrightarrow A^2(E).
 \end{equation*}
Then $R$ is $A^0$-linear. In fact, if $f\in A^0$ and $\sigma\in A^0(E),$ then
\begin{equation*}
 \begin{split}
  D^{2}(fs)=D(\mathrm{d}f\otimes s+fDs)&=D(\mathrm{d}f\otimes s)+D(fDs)=\\
                                       &=-\mathrm{d}f\otimes Ds+\mathrm{d}f\otimes Ds+fD^{2}s=D^{2}s.
 \end{split}
\end{equation*}
Hence $R$ is a 2-form on $M$ with values in $\mathrm{End}(E),$ i.e., $R\in A^2(\mathrm{End}(E)).$
\end{defin}

\begin{defin}
 Using the matrix notations of \eqref{local} the curvature form $\Omega$ of $D$ with respect to the frame
 field $s$ is defined by
 \begin{equation*}
  s\Omega=D^2s.
 \end{equation*}
\end{defin}
Then
\begin{equation}
 \label{curvature}
 \Omega=d\omega+\omega\wedge\omega.
\end{equation}
In fact,
\begin{equation*}
 s\Omega=D(s\omega)=Ds\wedge\omega+s\mathrm{d}\omega=s(\omega\wedge\omega+\mathrm{d}\omega).
\end{equation*}
Exterior differentiation of \eqref{curvature} gives Bianchi identity:
\begin{equation}
 \label{bianchi}
 \mathrm{d}\Omega=\Omega\wedge\omega-\omega\wedge\Omega=[\Omega,\omega].
\end{equation}

If $\omega'$ is the connection form of $D$ relative to another frame field $s'=sa^{-1}$, the corresponing
curvature form $\Omega'$ is related to $\Omega$ by
\begin{equation}
\label{transition2}
 \Omega=a^{-1}\Omega a.
\end{equation}
In fact,
\begin{equation*}
 \begin{split}
 s\Omega&=D^2s=D^2(s'a)=D(Ds'a+s'\mathrm{d}a)=\\
        &=D^2s'a-Ds'\wedge \mathrm{d}a+Ds'\wedge \mathrm{d}a=\\
        &=s'\Omega'a=s(a^{-1}\Omega'a).
        \end{split}
\end{equation*}
Let $\{U_{\alpha}\}$ be an open cover of $M$ with local frame field $s_{\alpha}$ on each $U_{\alpha}.$ 
If $U_{\alpha}\cap U_{\beta}\neq\emptyset,$ then
\begin{equation*}
 s_{\alpha}=s_{\beta}g_{\beta\alpha}\hspace{0.5cm}\text{ on }U_{\alpha}\cap U_{\beta}.
\end{equation*}
where $g_{\beta\alpha}:U_{\alpha}\cap U_{\beta}\longrightarrow GL(r,\mathbb{C})$ is the 
$\mathcal{C}^{\infty}$ transition function.
Then we have:
\begin{equation}
\label{transition1}
 \omega_{\alpha}=g_{\beta\alpha}^{-1}\omega_{\beta}g_{\beta\alpha}+g_{\beta\alpha}^{-1}\mathrm{d}
 g_{\beta\alpha}\hspace{0.5cm}\text{ on }U_{\alpha}\cap U_{\beta}.
\end{equation}
Conversely, given a system of $\frak{gl}(r,\mathbb{C})\text{-valued}$ $1\text{-forms}$ $\omega_{\alpha}$ on 
$U_{\alpha}$ satisfying \eqref{transition1}, we obtain a connection $D$ in $E$ having $\{\omega_{\alpha}\}$ as 
connection forms.
If $\Omega_U$ is the curvature form of $D$ relative to $s_{\alpha},$ then \eqref{transition2} means
\begin{equation*}
 \Omega_{\alpha}=g_{\beta\alpha}^{-1}\Omega_{\beta}g_{\beta\alpha}\hspace{0.5cm}\text{ on }
 U_{\alpha}\cap U_{\beta}.
\end{equation*}

\begin{defin}
 Let $E$ be a $\mathcal{C}^{\infty}$ complex vector bundle over a real manifold $M.$ Let $E^{\ast}$
 be the dual vector bundle of $E.$ The duality pairing
 
 \begin{equation*}
  \langle\cdot,\cdot\rangle:E_{x}^{\ast}\times E_{x}\longrightarrow\mathbb{C}
 \end{equation*}
induces a duality pairing

\begin{equation*}
 \langle\cdot,\cdot\rangle:A^0(E^{\ast})\times A^0(E)\longrightarrow A^0.
\end{equation*}
Given a connection $D$ in $E,$ we define a connection, also denoted by $D,$ in $E^{\ast}$ by the 
following formula:

\begin{equation*}
 \mathrm{d}\langle\zeta,\eta\rangle=\langle D\zeta,\eta\rangle+\langle\zeta,D\eta\rangle
 \hspace{0.5cm}\text{ for }\zeta\in A^0(E)\hspace{0.5cm}\text{ and }\eta\in A^0(E^{\ast}).
\end{equation*}
\end{defin}

\begin{prop}
 Let $E$ and $F$ two complex vector bundles over the same manifold $M.$ Let $D_E$ and $D_F$ be connections
 in $E$ and $F,$ respectively. Then we can
 define connections
 \begin{enumerate}
  \item $D_E\oplus D_F$ in the direct sum $E\oplus F$ in the obvious way,
  \item $D_{E\otimes F}$ in the tensor product $E\otimes F.$ 
        by $D_{E\otimes F}=D_E\otimes I_F+I_E\otimes D_F$
 \end{enumerate}
If we denote the curvatures of $D_E$ and $D_F$ by $R_E$ and $R_F,$ then we have
\begin{enumerate}
  \item $D_E\oplus D_F$ has curvature $R_E\oplus R_F,$
  \item $D_{E\otimes F}$ has curvature $R_E\otimes I_F+I_E\otimes R_F.$
\end{enumerate}
\end{prop}
If $s=(s_1,\ldots,s_r)$ is a local frame field of $E$ and $t=(t_1,\ldots,t_p)$ is a local frame field of
$F$ and $\omega_E,$ $\omega_F,$ $\Omega_E,$ $\Omega_F,$ are the connecionts and the curvature forms with
respect to these frame fields, then in a natural manner the connection and curvature forms of
$D_E\oplus D_F$ are given by
\begin{equation*}
\left(\begin{array}{cc}
	\omega_E       &	   0\\
         0             &      \omega_F 
   
      \end{array}\right)
\hspace{0.5cm}\text{ and }\hspace{0.5cm}
\left(\begin{array}{cc}
	\Omega_E       &	   0\\
         0             &      \Omega_F 
   
      \end{array}\right),
\end{equation*}
while those of $D_{E\otimes F}$ are given by
\begin{equation*}
 \omega_E\otimes I_p+I_r\otimes\omega_F\hspace{0.5cm}\text{ and }\hspace{0.5cm}
 \Omega_E\otimes I_p+I_r\otimes\Omega_F.
\end{equation*}
Here $I_r$ and $I_p$ denote the identity matrices of rank $r$ and $p.$ All these formulas extend in an 
obvious way to the direct sum and the tensor product of any number of vector bundles and so they give
formulas for the connection and curvature in

\begin{equation*}
 E^{\otimes p}\otimes E^{\ast\otimes q}=E\otimes\dots\otimes E^{\ast}\otimes\dots\otimes E^{\ast}
\end{equation*}
Let $E$ be a complex vector bundle over $M$ and let $N$ be another manifold. Given a $\mathcal{C}^{\infty}$ 
mapping $f:N\longrightarrow M,$ we obtain an induced vector bundle $f^{\ast}E$ over $N.$ 

Since in the cathegory of $\mathcal{C}^{\infty}$ bundles there is an isomorphism
\begin{equation*}
 f^{\ast}E\cong f^{-1}E\otimes_{f^{-1}\mathcal{A}_{M}}\mathcal{A}_{N},
\end{equation*}
if $D$ is a connection on $E,$ there is a pull-back connection $f^{\ast}D$ on $f^{\ast}E$ defined by
\begin{equation*}
 (f^{\ast}D)(\sum\varphi_{j}s_{j})=\sum\varphi_{j}D(s_{j})+\sum\mathrm{d}\varphi_{j}\otimes s_{j},
\end{equation*}
where $\varphi_{j}\in A_{N}^{0}$ are $\mathcal{C}^{\infty}$ functions on $N$ and $s_{j}\in\Gamma(M,E)$ are 
$\mathrm{C}^{\infty}$ sections of $E.$

If $\omega$ and $\Omega$ are respectively the connection form
and the curvature form of $D$ over a local frame field $(U_{\alpha},s_{\alpha}),$ then $f^{\ast}\omega$
and $f^{\ast}\Omega$ are the connection form and the curvature form with respect
to the pull-back local frame field $(f^{-1}U,f^{\ast}s).$

\section{Connections in complex vector bundles (over complex manifolds)}

Let $M$ be a complex manifold with (complex) dimension n and $E$ a $\mathcal{C}^{\infty}$ complex vector
bundle of rank r over $M.$ In addition to the notations $A^p$ and $A^p(E)$ introduced in the previous 
section, we use the following:
\begin{itemize}
 \item $A^{p,q}=\Gamma(M,\bigwedge_{p}T^{\ast}M^{\mathbb{C}+}\otimes\bigwedge_{q}T^{\ast}M^{\mathbb{C}-})$ 
       the space of complex $(p,q)\text{-forms}$ over $M,$
 \item $A^{p,q}(E)=\Gamma(M,\bigwedge_{p}T^{\ast}M^{\mathbb{C}+}
       \otimes\bigwedge_{q}T^{\ast}M^{\mathbb{C}-}\otimes E)$ the space of complex 
       \linebreak$(p,q)\text{-forms}$ over $M$ with values in $E.$
\end{itemize}
Let $D$ be a connection in $E$. We can write $D=D'+D'',$ where
\begin{equation*}
 D':A^{p,q}(E)\longrightarrow A^{p+1,q}(E)\hspace{0.5cm}\text{ and }
 D'':A^{p,q}(E)\longrightarrow A^{p,q+1}(E).
\end{equation*}

Decomposing $D$ according to the bidegree, we have, for $\sigma\in A^0(E)$ and \linebreak$\phi\in A^{p,q},$
\begin{equation*}
 D'(\sigma\phi)=D'\sigma\wedge\phi+\sigma \mathrm{d}'\phi,
\end{equation*}

\begin{equation*}
 D''(\sigma\phi)=D''\sigma\wedge\phi+\sigma \mathrm{d''}\phi.
\end{equation*}

Let $R$ be the curvature of $D,$ i.e., $R=D\circ D\in A^2(\mathrm{End}(E)).$ Then 
\begin{equation*}
 R=D'\circ D'+(D'\circ D''+D''\circ D')+D''\circ D'',
\end{equation*}
where $D'\circ D'\in A^{2,0}(\mathrm{End}(E)),$ and $D''\circ D''\in A^{0,2}(\mathrm{End}(E)),$ while 
\linebreak$D'\circ D''+D''\circ D'\in A^{1,1}(\mathrm{End}(E)).$ 

Let $s$ be a local frame field of $E$ and let $\omega$ and $\Omega$ be the connection and the curvature
forms of $D$ with respect to $s.$ We can write
\begin{equation*}
 \omega=\omega^{1,0}+\omega^{0,1},
\end{equation*}

\begin{equation*}
 \Omega=\Omega^{2,0}+\Omega^{1,1}+\Omega^{0,2}.
\end{equation*}

Let $M$ be a complex manifold of (complex) dimension $n.$ and let $E$ be a holomorphic vector bundle of rank $r$
over $M.$ Let 
$\{U_{\alpha}\}$ be an oper cover which trivializes $E$ and let 
$s_{\alpha}=(s_{1}^{\alpha},\ldots,s_{r}^{\alpha})$ be a holomorphic frame field 
on $U_{\alpha}.$ Let $\zeta\in A^{0}(E)$ be a $\mathcal{C}^{\infty}$ section. 
On $U_{\alpha}$ we write
\begin{equation*}
 \zeta=\sum_{i}\zeta_{\alpha}^{j}s_{j}^{\alpha},
\end{equation*}
where $\zeta_{\alpha}^{j}$ are $\mathcal{C}^{\infty}(U_{\alpha},\mathbb{C})$ functions.
Then we set 
\begin{equation*}
 \mathrm{d''}_{E}(\zeta)=\sum_{j}\mathrm{d''}(\zeta_{\alpha}^{j})s_{j}^{\alpha}\hspace{0.5cm}\text{ on }
 U_{\alpha}.
\end{equation*}

\begin{prop}
$\mathrm{d''}_{E}:\Gamma(M,E)\longrightarrow \Gamma(M,E\otimes TM^{\mathbb{C}})$ is well defined and  
\begin{equation*}
 \mathrm{d''}_{E}(f\zeta)=\mathrm{d''}(f)\otimes\zeta+f\mathrm{d''}_{E}(\zeta)\hspace{0.5cm}\text{ for }
 \hspace{0.3cm}f\in A^{0},\zeta\in A^{0}(E).
\end{equation*}
 Moreover, $\mathrm{d''}_{E}\circ\mathrm{d''}_{E}=0.$
\end{prop}
\begin{proof}
 First of all we prove that $\mathrm{d''}_E$ is well defined. 
 Let $s_{\beta}=(s_{1}^{\beta},\ldots,s_{r}^{\beta})$ be a holomorphic frame field on $U_{\beta}.$ On the 
 overlapping open set $U_{\alpha}\cap U_{\beta},$ $s_{\beta}$ is related to $s_{\alpha}$ by
 \begin{equation*}
  s_{\alpha}=s_{\beta}\cdot a_{\alpha\beta},
 \end{equation*}
where $a=a_{\alpha\beta}:U_{\alpha}\cap U_{\beta}\longrightarrow GL(r,\mathbb{C})$ is a holomorphic matrix-valued 
function on $U_{\alpha}\cap U_{\beta}.$ $a$ is a holomorphic matrix-valued function, so that 
$\mathrm{d''}(a_{j}^{i})=0.$
Then we have
\begin{equation*}
 \begin{split}
  \mathrm{d''}_{E}(\zeta)&=\sum\mathrm{d''}(\zeta_{\alpha}^{i})s_{i}^{\alpha}=
  \sum\mathrm{d''}(\zeta_{\alpha}^{i})a_{i}^{j}s_{j}^{\beta}=\\
  &=\sum\mathrm{d''}(\zeta_{\alpha}^{i}a_{i}^{j})s_{j}^{\beta}=\sum\mathrm{d''}(\zeta_{\beta}^{j})s_{j}^{\beta}=\\
  &=\mathrm{d''}_{E}(\zeta),
 \end{split}
\end{equation*}
and this proves that $\mathrm{d''}_{E}$ is well defined.

\noindent In order to prove that $\mathrm{d''}_{E}(f\zeta)=\mathrm{d''}(f)\otimes\zeta+f\mathrm{d''}_{E}(\zeta),$
 let $f\in A^{0}$ and $\zeta\in A^{0}(E).$
After a straightforward computation we find
\begin{equation*}
 \begin{split}
\mathrm{d''}_{E}(f\zeta)&=\sum\mathrm{d''}(f\zeta_{\alpha}^{i})s_{i}^{\alpha}=\\
&=\sum\mathrm{d''}(f)\zeta_{\alpha}^{i}s_{i}^{\alpha}+\sum f\mathrm{d''}(\zeta_{\alpha}^{i})s_{i}^{\alpha}=\\
&=\mathrm{d''}(f)\otimes\sum\zeta_{\alpha}^{i}s_{i}^{\alpha}+
f\sum\mathrm{d''}(\zeta_{\alpha}^{i})s_{i}^{\alpha}=\\
&=\mathrm{d''}(f)\otimes\zeta+f\mathrm{d''}_{E}(\zeta).
 \end{split}
\end{equation*}
Finally, from $\mathrm{d''}\circ\mathrm{d''}=0$ we have $\mathrm{d''}_{E}\circ\mathrm{d''}_{E}=0,$ and this 
completes the proof.
\end{proof}

We shall now describe those complex vector bundles which admit holomorphic structures. Let 
$\mathcal{A}^{p,q}$ be the sheaf of complex $(p,q)\text{-forms}$ over $M$ and let 
$\mathcal{A}^{p,q}(E)=E\otimes\mathcal{A}^{p,q}$ be the sheaf of complex $(p,q)\text{-forms}$ over $M$ with 
values in $E.$ Then we have the following
\begin{teo}
 Let $M$ be a complex manifold of (complex) dimension $n$ and let $E$ be a complex vector bundle of rank $r$ 
 over $M.$ The following conditions are equivalent:
 \begin{enumerate}
  \item There exists a holomorphic vector bundle structure on $E,$
  \item There exists an operator $\mathrm{d''}_{E}:\mathcal{A}^{0,0}(E)\longrightarrow\mathcal{A}^{0,1}(E)$
        such that
        \begin{enumerate}
         \item $\mathrm{d''}_{E}(f\zeta)=\mathrm{d''}(f)\otimes\zeta+f\mathrm{d''}_{E}(\zeta)\hspace{0.5cm}
         \text{ for }\hspace{0.3cm}f\in\mathcal{A}^{0,0},\zeta\in\mathcal{A}^{0,0}(E),$
         \item $\mathrm{d''}_{E}\circ\mathrm{d''}_{E}=0.$
        \end{enumerate}
\end{enumerate}
\end{teo}

Moreover, we have the following results:

\begin{prop}
 Let $E$ be a holomorphic vector bundle over a complex manifold $M.$ Then there exists a connection $D$ such
 that
 \begin{equation*}
  D''=\mathrm{d''}_{E}
 \end{equation*}
For such a connection, the $(0,2)\text{-component}$ $D''\circ D''$ of the curvature $R$ vanishes.
\end{prop}

\begin{proof}
 Let $\{U\}$ be a locally finite open cover of $M$ and let $\{\rho_U\}$ a partition of unity subordinate
 to $\{U\}.$ Let $s_U$ be a holomorphic frame field of $E$ on $U$ and let $D_U$ be the flat connection
 in $\left.E\right|_U$ defined by $D_U(s_U)=0.$ Then $D=\sum\rho_UD_U$ is a connection in $E$ with the 
 property that $D''=\mathrm{d''}_{E}.$ The second assertion is obvious, since
 $\mathrm{d''}_{E}\circ \mathrm{d''}_{E}=0.$
\end{proof}

\begin{prop}
\label{holomorphicstructureD}
 Let $E$ be a $\mathcal{C}^{\infty}$ complex vector bundle over a complex manifold $M.$ If $D$ is a 
 connection in $E$ such that $D''\circ D''=0,$ then there is a unique holomorphic vector bundle structure
 in $E$ such that $D''=\mathrm{d''}_{E}.$
\end{prop}

\begin{proof}
 See \cite{KOB} for a detailed proof.
\end{proof}

\begin{prop}
 For a connection $D$ in a holomorphic vector bundle $E,$ the following conditions are equivalent:
 \begin{enumerate}
  \item $D''=\mathrm{d''}_{E},$
  \item For every local holomorphic section $s,$ $Ds$ is of degree $(1,0),$
  \item With respect to a local holomorphic frame field, the connection form $\omega$ is of degree $(1,0).$
 \end{enumerate}
\end{prop}

\section{Connections in Hermitian vector bundles}
\label{compatibility}

\begin{defin}
 Let $E$ be a $\mathcal{C}^{\infty}$ complex vector bundle over a (real or complex) manifold $M.$ 
 A Hermitian structure or Hermitian metric $h$ in $E$ is a $\mathcal{C}^{\infty}$ field of Hermitian
 inner products in the fibre of $E.$ Thus,
 \begin{enumerate}
  \item $h(\zeta,\eta)$ is $\mathbb{C}\text{-linear}$ in $\zeta,$ where $\zeta,\eta\in E_x,$
  \item $h(\zeta,\eta)=\overline{h(\eta,\zeta)},$
  \item $h(\zeta,\zeta)>0$ for $\zeta\neq0,$
  \item $h(\zeta,\eta)$ is a $\mathcal{C}^{\infty}$ function if $\zeta$ and $\eta$ are 
        $\mathcal{C}^{\infty}$ sections.
 \end{enumerate}
We call $(E,h)$ a Hermitian vector bundle.
\end{defin}

Given a local frame field $s_{\alpha}=(s_1,\ldots,s_r)$ of $E$ over $U_{\alpha},$ we set
\begin{equation*}
 h_{i\overline{j}}=h(s_i,s_j)\hspace{0.5cm}\text{ for }i,j=1,\ldots,r
\end{equation*}
and

\begin{equation}
 H_{\alpha}=(h_{i\overline{j}})
\end{equation}
Then $H_{\alpha}$ is a positive definite Hermitian matrix at every point of $U_{\alpha}.$ 
When we are working with a single frame field, we often drop the subscript $\alpha.$ 
We say that $s_{\alpha}$ is a unitary frame field or orthonormal frame field if $H_{\alpha}$ is the
identity matrix. Under a change of frame field given by $s_{\alpha}=s_{\beta}g_{\beta\alpha},$ we have

\begin{equation*}
 H_{\alpha}=g_{\beta\alpha}^{t}H_{\beta}g_{\beta\alpha}\hspace{0.5cm}\text{ on }U_{\alpha}\cap U_{\beta}.
\end{equation*}

\begin{defin}
 A connection $D$ in $(E,h)$ is called an $h\text{-connection}$ if it preserves $h,$ i.e., if it makes $h$ 
 parallel in the following sense:
 \begin{equation}
 \label{hconnection}
  \mathrm{d}(h(\zeta,\eta))=h(D\zeta,\eta)+h(\zeta,D\eta)\hspace{0.5cm}\text{ for }\zeta,\eta\in A^0(E).
 \end{equation}
\end{defin}
Let $\omega=(\omega_{j}^{i})$ be the connection form relative to the local frame field $s_U.$ Then setting
$\zeta=s_i$ and $\eta=s_j$ in \eqref{hconnection}, we obtain
\begin{equation}
\label{localhconnection}
 \mathrm{d}h_{i\bar\jmath}=h(Ds_i,s_j)+h(s_i,Ds_j)=\omega_{i}^{a}h_{a\bar\jmath}+
 h_{ib}\omega_{\bar\jmath}^{b}.
\end{equation}
So in matrix notation we have

\begin{equation*}
\label{matrixhconnection1}
 \mathrm{d}H=\omega^tH+H\overline{\omega}.
\end{equation*}
Applying $\mathrm{d}$ to \eqref{matrixhconnection1} we obtain

\begin{equation}
 \label{matrixhconnection2}
 \Omega^tH+H\overline{\Omega}=0.
\end{equation}
If $E$ is a holomorphic vector bundle over a complex 
manifold $M,$ then a Hermitian structure $h$ determines a natural $h$-connection satisfying 
$D''=\mathrm{d''}_{E}.$ Namely, we have

\begin{prop}
 Given a Hermitian structure $h$ in a holomorphic vector bundle $E$ over a complex manifold $M,$ there
 is a unique $h\text{-connection}$ $D$ such that $D''=\mathrm{d''}_{E}.$
\end{prop}
\begin{proof}
 \begin{enumerate}
  \item First af all we prove the uniqueness. Let $D$ be such a connection and let 
        $s_{\alpha}=(s_1,\ldots,s_r)$ be a local frame field on $U_{\alpha}.$ Since $Ds_i=D's_i,$ 
        the connection form $\omega_{\alpha}=(\omega_{j}^{i})$ is of degree $(1,0).$ 
        From \eqref{localhconnection} we obtain
        \begin{equation*}
         \mathrm{d}'h_{i\bar\jmath}=\omega_{i}^{a}h_{a\bar\jmath},
        \end{equation*}
        or, in matrix notation,
        \begin{equation*}
         \mathrm{d}'H_{\alpha}=\omega_{\alpha}^{t}H_{\alpha}.
        \end{equation*}

        This determines the connection form $\omega_{\alpha},$ i.e.,
        \begin{equation*}
         \omega_{\alpha}^{t}=\mathrm{d}'H_{\alpha}H_{\alpha}^{-1}.
        \end{equation*}
        Hence, we have proved uniqueness.
  \item Now, we want to prove existence. Let $\omega_{\alpha}^t=\mathrm{d}'H_{\alpha}H_{\alpha}^{-1}.$ 
  By a straightforward calculation we can see that the collection $\{\omega_{\alpha}\}$ satisfies
  \eqref{transition1}
  \begin{equation*}
   \omega_{\alpha}=g_{\beta\alpha}^{-1}\omega_{\beta}g_{\beta\alpha}+g_{\beta\alpha}^{-1}\mathrm{d}
   g_{\beta\alpha}\hspace{0.5cm}\text{ on }U_{\alpha}\cap U_{\beta},
  \end{equation*}
 and this completes the proof.
 \end{enumerate}
\end{proof}

\begin{defin}
 The connection given by the previous Proposition is called the Hermitian connection of the holomorphic
 Hermitian vector bundle $(E,h).$
 \end{defin}

\begin{prop}
 The curvature of the Hermitian connection in a holomorphic vector bundle is of degree $(1,1).$
 If $(E,h)$ is a $\mathcal{C}^{\infty}$ complex vector bundle over a complex manifold $M$ with an 
 Hermitian structure $h$ and $D$ is an $h\text{-connection}$ whose curvature is of degree $(1,1),$ then there is 
 a unique holomorphic structure in $E$ which makes $D$ the Hermitian connection of 
 the vector bundle $(E,h).$
\end{prop}
\begin{proof}
 Let $D$ be a Hermitian connection in the holomorphic Hermitian vector bundle $(E,h)$ over the
 complex manifold $M.$ Its connection form is given locally by
 $\omega_{\alpha}^{t}=\mathrm{d}'H_{\alpha}H_{\alpha}^{-1},$ and its curvature $R$ has no 
 $(0,2)\text{-components}$
 since 
 \linebreak$D''\circ D''=\mathrm{d''}_{E}\circ \mathrm{d''}_{E}=0.$ By \eqref{matrixhconnection2} it has no 
 $(2,0)\text{-components}$ either. So the curvature is a $(1,1)\text{-form}$ with values in $\mathrm{End}(E)$
 \begin{equation*}
  R=D'\circ D''+D''\circ D'\in A^{1,1}(E).
 \end{equation*}
 The second part of this Proposition follows \ref{holomorphicstructureD}. (See
 \cite{KOB}, p. 9 and p. 12 for more details).
\end{proof}
 With respect to a local holomorphic frame field the connection form $\omega=\omega_{j}^{i}$ is of 
 degree $(1,0).$ Since the curvature form $\Omega$ s equal to the $(1,1)\text{-component}$ of 
 $\mathrm{d}\omega+\omega\wedge\omega,$ we obtain
 
 \begin{equation*}
  \Omega=\mathrm{d''}\omega.
 \end{equation*}
From $\omega_{\alpha}^{t}=\mathrm{d}'H_{\alpha}H_{\alpha}^{-1}$ we obtain

\begin{equation*}
 \Omega^t=\mathrm{d''}\mathrm{d}'HH^{-1}+\mathrm{d}'HH^{-1}\wedge \mathrm{d''}HH^{-1}.
\end{equation*}
In local coordinates we write

\begin{equation*}
 \Omega^{i}_{j}=R^{i}_{j\alpha\overline{\beta}}\mathrm{d}z^{\alpha}\wedge \mathrm{d}\overline{z}^{\beta}.
\end{equation*}

\section{Subbundles and quotient bundles}   
Let $E$ be a holomorphic vector bundle of rank $r$ over a complex manifold $M$ of (complex) dimension
$n$ and let $S$ be a holomorphic subbundle of rank $p$ of $E.$ Then the quotient bundle 
$Q=\faktor{E}{S}$ is a
holomorphic vector bundle of rank $r-p.$ We can express this situation as an exact sequence

\begin{equation*}
 0\longrightarrow S\longrightarrow E\longrightarrow Q\longrightarrow0.
\end{equation*}
Let $h$ be a Hermitian structure in $E.$ Restricting $h$ to $S,$ we obtain a Hermitian structure $h_S$ 
in $S.$ Taking the orthogonal complement of $S$ in $E$ with respect to $h,$ we obtain a complex
subbundle $S^{\bot}$ of $E.$

\begin{note}
 The complex subbundle $S^{\bot}$ of $E$ may not be a holomorphic subbundle of $E$ in general. Thus
 
  \begin{equation*}
   E=S\oplus S^{\bot}
  \end{equation*}
is merely a $\mathcal{C}^{\infty}$ orthogonal decomposition of $E.$
\end{note}
As a $\mathcal{C}^{\infty}$ complex vector bundle, $Q$ is naturally isomorphic to $S^{\bot}.$ 
Hence, we obtain also a Hermitian structure $h_Q$ in a natural way.

\begin{defin}
Let $D$ denote the Hermitian connection in $(E,h).$ We define $D_S$ and $A$ by

\begin{equation}
\label{subconnection}
 D\zeta=D_S\zeta+A\zeta\hspace{0.5cm}\text{ for }\zeta\in A^0(S),
\end{equation}
where $D_S\zeta\in A^1(S)$ and $A\zeta\in A^1(S^{\bot}).$
\end{defin}

\begin{prop}
Under the hypothesis of the previous definition, we have the following results:
\begin{enumerate}
 \item $D_S$ is the Hermitian connection of $(S,h_S),$
 \item $A$ is a (1,0)-form with values in $\mathrm{Hom}(S,S^{\bot}),$ i.e., 
       $A\in A^{1,0}(\mathrm{Hom}(S,S^{\bot})).$
\end{enumerate}
\end{prop}
\begin{proof}
 Let $f$ be a function on $M.$ Replacing $\zeta$ by $f\zeta$ in \eqref{subconnection}, we obtain
 
 \begin{equation*}
  D(f\zeta)=D_S(f\zeta)+A(f\zeta).
 \end{equation*}
On the other hand,

\begin{equation*}
 D(f\zeta)=\mathrm{d}f\zeta+fD\zeta=\mathrm{d}f\zeta+fD_S\zeta+fA\zeta.
\end{equation*}
Comparing the components of the two decompositions of $D(f\zeta),$ we conclude

\begin{equation}
 D_S(f\zeta)=\mathrm{d}f\zeta+fD_S\zeta\hspace{0.5cm}\text{ and }\hspace{0.5cm}A(f\zeta)=fA\zeta.
\end{equation}
The first equality says that $D_S$ is a connection and the second says that $A$ is a 1-form with values in
$\mathrm{Hom}(S,S^{\bot}).$
If $\zeta$ in  $D\zeta=D_S\zeta+A\zeta$ is holomorphic, then $D\zeta$ is a (1,0)-form with values in $E.$ 
Hence, $D_S\zeta$ is a (1,0)-form with values in $S$ while $A$ is a (1,0)-form with values in
$\mathrm{Hom}(S,S^{\bot}).$ Finally, if $\zeta,\zeta'\in A^0(S),$ then

\begin{equation*}
 \begin{split}
  \mathrm{d}(h(\zeta,\zeta'))&=h(D\zeta,\zeta')+h(\zeta,D\zeta')=\\
                             &=h(D_S\zeta+A\zeta,\zeta')+h(\zeta,D_S\zeta'+A\zeta')=\\
                             &=h(D_S\zeta,\zeta')+h(\zeta,D_S\zeta'),
 \end{split}
\end{equation*}
and this proves that $D_S$ preserves $h_S.$
\end{proof}

\begin{defin}
 We call $A\in A^{1,0}(\mathrm{Hom}(S,S^{\bot}))$ the second foundamental form of $S$ in $(E,h).$ With the
 identification $Q=S^{\bot},$ we can consider $A$ as an element of $A^{1,0}(\mathrm{Hom}(S,Q)).$
\end{defin}

\begin{defin}
 Similarly, we define $D_{S^{\bot}}$ and $B$ by setting
 
 \begin{equation*}
  D\eta=B\eta+D_{S^{\bot}}\eta\hspace{0.5cm}\text{ for }\eta\in A^0(S^{\bot}),
 \end{equation*}
where $B\eta\in A^1(S)$ and $D_{S^{\bot}}\eta\in A^1(S^{\bot}).$ Under the identification $Q=S^{\bot},$ we
may consider $D_{S^{\bot}}$ as a mapping $A^0(Q)\longrightarrow A^1(Q).$ The we write $D_Q$ in place of
$D_{S^{\bot}}.$
\end{defin}

\begin{prop}
 In the hypothesis of the previous definitions and constructions we have the following results:
 
 \begin{enumerate}
  \item $D_Q$ is the Hermitian connection of $(Q,h_Q),$
  \item $B$ is a (0,1)-form with values in $\mathrm{Hom}(S^{\bot},S),$ i.e., 
        $B\in A^{0,1}(\mathrm{Hom}(S^{\bot},S)),$
  \item $B$ is the adjoint of $-A,$ i.e., 
  \begin{equation*}
   h(A\zeta,\eta)+h(\zeta,B\eta)=0\hspace{0.5cm}\text{ for }\zeta\in A^0(S)
   \hspace{0.5cm}\text{ and }\eta\in A^0(S^{\bot}).
  \end{equation*}
 \end{enumerate} 
\end{prop}
\begin{proof}
 The proof is very similar to the previous one. See \cite{KOB} for more details.
\end{proof}

\chapter{Chern Classes}

\section{Chern classes of a line bundle}
In this section we define the first Chern class of a line bundle.
We start with some useful definitions.

\begin{defin}
 Consider $\mathbb{P}^n(\mathbb{C})$ with the open covering $\{U_i\},$ where 
 \linebreak$U_i=\{z^i\neq0\}$ and 
 consider the holomorphic line bundle over $\mathbb{P}^n(\mathbb{C})$ with transition functions
 $g_{ij}=\frac{z^i}{z^j}.$ We will call it the universal line bundle over $\mathbb{P}^n(\mathbb{C}).$
\end{defin}

\begin{defin}
 Let $L$ be the universal line bundle over $\mathbb{P}^n(\mathbb{C}),$ and let 
 \linebreak$H=\{a_0z^0+\cdots a_nz^n=0\}$ be an hyperplane in $\mathbb{P}^n(\mathbb{C}).$ 
 If we take non-homogeneous coordinates we have
 \begin{equation*}
 \label{hyperplane}
  \sum a_kz^k=z^i\left(a_0\frac{z^0}{z^i}+\cdots+a_i+\cdots+a_n\frac{z^n}{z^i}\right)\hspace{0.5cm}
  \text{ on }U_i.
 \end{equation*}
So on $U_i\cap U_j$ we have
\begin{equation*}
\begin{split}
 \frac{\left(a_0\frac{z^0}{z^i}+\cdots+a_i+\cdots+a_n\frac{z^n}{z^i}\right)}
 {\left(a_0\frac{z^0}{z^j}+\cdots+a_j+\cdots+a_n\frac{z^n}{z^j}\right)}&=
 \frac{\left(a_0\frac{z^0}{z^i}+\cdots+a_i+\cdots+a_n\frac{z^n}{z^i}\right)}
 {\left(a_0\frac{z^0}{z^j}+\cdots+a_j+\cdots+a_n\frac{z^n}{z^j}\right)}\frac{z^i}{z^j}\frac{z^j}{z^i}=\\
 =\frac{\sum a_kz^k}{\sum a_kz^k}\frac{z^j}{z^i}&=\frac{z^j}{z^i}.
\end{split}
\end{equation*}
Because of this origin the latter bundle, to be denoted by $H,$ is called the hyperplane section bundle
of $\mathbb{P}^n(\mathbb{C}).$ It is the dual bundle of the universal line bundle over
$\mathbb{P}^n(\mathbb{C}).$
\end{defin}

\begin{defin}
 Let $M$ be a differentiable manifold. We define the Picard group of $M$ as
 \begin{equation*}
  Pic(M)=\{\text{isomorphism classes of complex line bundles over M}\}
 \end{equation*}
This is a group with the operation $[L_1]\otimes[L_2]=[L_1\otimes L_2].$
\end{defin}

\begin{defin}
 Let $M$ be a differentiable manifold and let $\mathcal{A}$ (resp. $\mathcal{A}^{\ast}$) be the sheaf of
 germs of $\mathcal{C}^{\infty}$ complex functions (resp. nowhere vanishing complex functions) over $M.$
 Consider the exact sequence of sheaves:
 \begin{equation*}
 0\longrightarrow\mathbb{Z}\longrightarrow\mathcal{A}\longrightarrow\mathcal{A}^{\ast}\longrightarrow1,
 \end{equation*}
where $j:\mathbb{Z}\longrightarrow\mathcal{A}$ is simply the natural injection and
$e:\mathcal{A}\longrightarrow\mathcal{A}^{\ast}$ is the exponential map
\begin{equation*}
 e(f)=\exp(2\pi if)\hspace{0.5cm}\text{ for }\hspace{0.5cm}f\in\mathcal{A}.
\end{equation*}
This induces an exact sequence of cohomology groups
 \begin{equation*}
  \begin{tikzpicture}[node distance=2.3cm, auto]
 \node (A) {$\cdots$};
 \node (B) [right of=A] {$H^1(M,\mathcal{A})$};
 \node (C) [right of=B] {$H^1(M,\mathcal{A}^{\ast})$};
 \node (D) [right of=C] {$H^2(M,\mathbb{Z})$};
 \node (E) [right of=D] {$H^2(M,\mathcal{A})$};
 \node (F) [right of=E] {$\cdots$};
 \draw[->] (A) to node {$j^{\ast}$} (B);
 \draw[->] (B) to node {$e^{\ast}$} (C);
 \draw[->] (C) to node {$\delta$} (D);
 \draw[->] (D) to node {$j^{\ast}$} (E);
 \draw[->] (E) to node {$e^{\ast}$} (F);
\end{tikzpicture}
 \end{equation*}
where $\delta$ is the connecting homomorphism.
Since $\mathcal{A}$ is a fine sheaf over the differentiable manifold $M,$ we have
\begin{equation*}
 H^p(M,\mathcal{A})=0\hspace{0.5cm}\text{ for }\hspace{0.5cm}p\geq1.
\end{equation*}
Then $\delta:H^1(M,\mathcal{A}^{\ast}\longrightarrow H^2(M,\mathbb{Z})$ 
is a group isomorphism.
Identifying \linebreak$H^1(M,\mathcal{A}^{\ast})$ with the Picard group $Pic(M)$ of $M,$ 
(see Theorem \ref{correspondence} for details), we can define the first Chern class of a complex 
line bundle $L$ over $M$ by
\begin{equation*}
 c_1(L)=\delta(L),\hspace{0.5cm}L\in H^1(M,\mathcal{A}^{\ast}).
\end{equation*}
We also set $c_0(L)=1$ and $c(L)=1+c_1(L).$
\end{defin}

\begin{remark}
 From the above definition we immediately see that 
 \begin{enumerate}
  \item Two line bundles over $M$ are isomorphic, if and only if their first Chern classes coincide,
  \item For two line bundles $L_1$ and $L_2$ over $M$ we have 
        \begin{equation*}
         c_1(L_1\otimes L_2)=c_1(L_1)+c_1(L_2).
        \end{equation*}
 \end{enumerate}
\end{remark}

\begin{note}
Let $\{U_{\alpha}\}$ be an open cover which trivializes the complex line bundle $L,$ and let be 
$g_{\alpha\beta}$ its transition functions.
 From the definition of the connecting homomorphism we can deduce an explicit formula for a \v{C}ech
 cocycle representing $c_1(L)$ with respect to the open cover $\{U_{\alpha}\}:$
 \begin{equation}
  \label{cechcocycle}
  \{c_1(L)\}_{\alpha\beta\gamma}=\frac{1}{2\pi i}
  (\ln g_{\alpha\beta}+\ln g_{\beta\gamma}+\ln g_{\gamma\alpha}).
 \end{equation}
\end{note}

\begin{note}
 If $M$ is a complex manifold, we can define Chern classes of holomorphic line bundles in a similar way.
 (See Chern \cite{CHE}, Chapter VI for more details).
\end{note}

\section{Chern classes of a complex vector bundle}

In this section we define higher Chern classes for complex vector bundles of any rank. We proceed in two
steps:
\begin{enumerate}
 \item We first define Chern classes of vector bundles that are direct sums of line bundles,
 \item We show that we can always reduce the computation of Chern classes to the previous case.
\end{enumerate}

\begin{defin}
 For $i=1,\ldots,k$ let $\sigma_i$ denote the symmetric function of order $i$ in $k$ arguments, defined as
 \begin{equation*}
  \sigma_i(x_1,\ldots,x_k)=\sum_{1\leq j_1<\cdots<j_i\leq k}x_{j_1}\ldots x_{j_i}.
 \end{equation*}
\end{defin}

\begin{defin}
 Let $M$ be a differentiable manifold and let $E$ a complex vector bundle of rank $r$ over $M.$ 
 Let us assume $E=L_1\oplus\cdots\oplus L_r$ is the direct sum of complex line bundles over $M.$ For
 $i=1,\ldots,k$ we define the $i-\text{th}$ Chern class of $E$ as
 \begin{equation*}
  c_i(E)=\sigma_i(c_1(L_1),\ldots,c_1(L_k))\in H^{2i}(M,\mathbb{Z}),
 \end{equation*}
 where
 \begin{equation*}
  \sigma_i(c_1(L_1),\ldots,c_1(L_k))=\sum_{1\leq j_1<\cdots<j_i\leq k}c_1(L_{j_1})\cup\cdots\cup c_1(L_{j_i})
 \end{equation*}
and $\cup$ denotes the cup product.
\end{defin}

Step 2 relies on the following result, sometimes called the splitting principle.
(See Bott-Tu \cite{BOT}, p. 273-278, for more details). Our proof uses the de Rham cohomology of $M,$ but with 
some little changes one can prove the splitting principle also for the cohomology with coefficient in 
$\mathbb{Z}.$

\begin{teo}
 \emph{(Splitting principle)}
 Let $E$ be a complex vector bundle of rank $r$ over a differentiable manifold $M.$ There exists a differentiable 
 manifold $N$ (also called splitting manifold) and a $\mathcal{C}^{\infty}$ mapping 
 $f:N\longrightarrow M$ such that
 \begin{enumerate}
  \item the pull-back bundle $f^{\ast}E$ is a direct sum of line bundles,
  \item the cohomology morphism $f^{\sharp}:H^{\ast}(M)\longrightarrow H^{\ast}(N)$
        is injective,
  \item the Chern classes $c_i(f^{\ast}E)$ lie in the image of the cohomology morphism $f^{\sharp}.$
 \end{enumerate}
\end{teo}

\begin{proof}
 Let $\tau:E\longrightarrow M$ be a $\mathcal{C}^{\infty}$ complex vector bundle of rank $r$ over a 
differentiable manifold $M.$ Our goal is to construct the splitting manifold $N=F(E)$ and the splitting map 
$f:F(E)\longrightarrow E.$ 
We prove this Theorem by induction on the rank of $E.$
\begin{enumerate}
 \item If $E$ has rank $1,$ there is nothing to prove.
 \item If $E$ has rank $2,$ we can take as a splitting manifold $F(E)$ the projective bundle $\mathbb{P}(E),$
       which
       is by definition the fibre bundle over $M$ whose fibre at a point $P\in M$ is the projective
       space $\mathbb{P}(E_P)$ and whose transition functions are 
       $\tilde{g}_{\alpha\beta}:U_{\alpha}\cap U_{\beta}\longrightarrow PGL(r,\mathbb{C}),$ induced 
       from the transition functions $g_{\alpha\beta}$ of $E.$
       As on the projectivization of a vector space, on $\mathbb{P}(E)$ there are several tautological
       bundles: the pull-back bundle $\pi^{\ast}E,$ the universal subbundle $S_E$ and the universal
       quotient bundle $Q_E.$
       \begin{equation*}
        \begin{tikzpicture}[node distance=1.5cm, auto]
         \node (A) {$0$};
         \node (B) [right of=A] {$S_E$};
         \node (C) [right of=B] {$\pi^{\ast}E$};
         \node (D) [right of=C] {$Q_E$};
         \node (E) [right of=D] {$0$};
         \node (F) [below of=C] {$\mathbb{P}(E)$};
         \node (G) [right of=F] {$E$};
         \node (H) [below of=G] {$M$};
         \draw[->] (A) to node {$$} (B);
         \draw[->] (B) to node {$$} (C);
         \draw[->] (C) to node {$$} (D);
         \draw[->] (D) to node {$$} (E);
         \draw[->] (C) to node [swap]{$\tau^{\ast}$} (F);
         \draw[->] (G) to node {$\tau$} (H);
         \draw[->] (F) to node {$\pi$} (H);
\end{tikzpicture}
       \end{equation*}
Here the universal subbundle $S_E$ over $\mathbb{P}(E)$ is defined by
\begin{equation*}
 S_E=\{(l_P,v)\in\pi^{\ast}E|v\in l_p\},
\end{equation*}
while the quotient bundle $Q_E$ is determined by the tautological exact sequence
\begin{equation*}
 0\longrightarrow S_E\longrightarrow\pi^{\ast}E\longrightarrow Q_E\longrightarrow0.
\end{equation*}
If $E$ has rank $2$ we have $\pi^{\ast}E=S_E\oplus Q_E,$ which is a direct sum of line bundles.

\item Now suppose $E$ has rank $3.$ Over $\mathbb{P}(E)$ the line bundle $S_E$ splits off as before.
The quotient bundle $Q_E$ over $\mathbb{P}(E)$ has rank 2 and so can be split into a direct sum of line 
bundles when pulled back to $\mathbb{P}(Q_E).$
\begin{equation*}
 \begin{tikzpicture}[node distance=1.5cm, auto]
         \node (A) {$\beta^{\ast}S_E\oplus S_E\oplus Q_{Q_E}$};
         \node (B) [below of=A] {$\mathbb{P}(Q_E)$};
         \node (C) [right of=B] {$S_E\oplus Q_E$};
         \node (D) [below of=C] {$\mathbb{P}(E)$};
         \node (E) [right of=D] {$E$};
         \node (F) [below of=E] {$M$};
         \draw[->] (A) to node {$$} (B);
         \draw[->] (C) to node {$$} (D);
         \draw[->] (B) to node {$\beta$} (D);
         \draw[->] (D) to node {$\alpha$} (F);
         \draw[->] (E) to node {$\tau$} (F);
\end{tikzpicture}
\end{equation*}
\item The pattern is now clear, we split off one subbundle at a time by pulling back to the
      projectivization of a quotient bundle (see Figure \ref{SplittingDiagram} for details). 
\end{enumerate}
Setting $F(E)=\mathbb{P}(Q_{n-2}),$  this is the splitting manifold and this completes the proof.   
\end{proof}

\begin{defin}
 Let $E$ be a complex vector bundle of rank $r$ over a differentiable manifold $M.$ 
 The $i-\text{th}$ Chern class $c_i(E)$ of $E$ is the unique class in $H^{2i}(M,\mathbb{Z})$ such that
 $f^{\sharp}(c_i(E))=c_i(f^{\ast}E).$
 We also set $c_0(E)=1$ and we define the total Chern class of $E$ as $c(E)=\sum_{i=0}^{r}c_i(E).$
\end{defin}

\begin{prop}
 The Chern classes of a complex vector bundle $E$ of rank $r$ over a differentiable manifold $M$ satisfy the
 following properties:
 \begin{enumerate}
  \item if two vector bundles $E$ and $F$ over $M$ are isomorphic, their Chern classes coincide,
  \item (Naturality): if $f:N\longrightarrow M$ is a differentiable map and $E$ is a complex vector
        bundle over $M,$ then
        \begin{equation*}
         f^{\sharp}(c_i(E))=c_i(f^{\ast}E),
        \end{equation*}
  \item (Whitney product formula): if $E$ and $F$ are complex vector bundles over $M,$ then
        \begin{equation*}
         c(E\oplus F)=c(E)\cup c(F)
        \end{equation*}
   \item (Normalization): if $L$ is the universal line bundle over $\mathbb{P}^1(\mathbb{C}),$ then $-c_1(L)$ is 
   the positive generator
 of $H^2(\mathbb{P}^1(\mathbb{C}),\mathbb{Z});$ in other words, $c_1(L)$ integrated on the foundamental
 2-cycle $\mathbb{P}^1(\mathbb{C})$ is equal to $-1.$
 \end{enumerate}
\end{prop}
\begin{proof}
 In view of the splitting principle, it is enough to prove the properties
 $(1)$ and $(2)$ when $E$ and $F$ are line bundles.
 \begin{enumerate}
  \item Follows from the definition of the first Chern class of a line bunde.
  \item Let $E$ be a line bundle over $M$ and let $f:N\longrightarrow M$ be a differentiable map. Let
        $\{U_{\alpha}\}$ be an open cover which trivializes the line bundle $E$ 
        with transition functions
        $g_{\alpha\beta}.$ Then $\{f^{-1}(U_{\alpha})\}$ is an open cover which trivializes the pull-back 
        bundle $f^{\ast}E,$ with transition functions 
        \linebreak$\tilde{g}_{\alpha\beta}=g_{\alpha\beta}\circ f.$
        From \eqref{cechcocycle} we deduce that:
        \begin{equation*}
        \begin{split}
         \{c_1(f^{\ast}E)\}_{\alpha\beta\gamma}&=\frac{1}{2\pi i}
         (\ln\tilde{g}_{\alpha\beta}+\ln\tilde{g}_{\beta\gamma}+\ln\tilde{g}_{\gamma\alpha})=\\
                                               &=\frac{1}{2\pi i}
         (\ln g_{\alpha\beta}\circ f+\ln g_{\beta\gamma}\circ f+\ln g_{\gamma\alpha}\circ f)=\\
                                               &=\frac{1}{2\pi i}
         f^{\sharp}(\ln g_{\alpha\beta}+\ln\ g_{\beta\gamma}+\ln g_{\gamma\alpha})=\\
                                               &=\{f^{\sharp}c_1(E)\}_{\alpha\beta\gamma}.
        \end{split}
        \end{equation*}
        and this completes the proof of $(2).$
  \item By the
        splitting principle we can assume $E=L_1\oplus\cdots\oplus L_r$ and 
        \linebreak$F=L_{r+1}\oplus\cdots\oplus F_{r+q}$
        are direct sum of line bundles. Then 
        \begin{equation*}
        \begin{split}
         c(E\oplus F)&=c(L_1\oplus\cdots\oplus L_r\oplus L_{r+1}\oplus\cdots\oplus L_{r+q})=\\
                     &=\sum_{i=0}^{r+q}\sigma_i(L_1,\ldots,L_r,L_{r+1},\ldots,L_{r+q})=\\
                     &=\prod_{i=1}^{r+q}(1+c_1(L_i))=\\
                     &=\left[\prod_{h=1}^{r}(1+c_1(L_h))\right]
                     \left[\prod_{k=1}^{q}(1+c_1(L_{r+k}))\right]=\\
                     &=c(E)\cup c(F).
        \end{split}
        \end{equation*}
   \item It follows from a general results: for any divisors $D\in\mathrm{Div}(X)$ in a compact Riemann surface
         \begin{equation*}
          \int_{X}c_1(D)=\text{deg}(D),
         \end{equation*}
         where $c_1(D)$ is the Chern class of the line bundle associated with the divisor $D.$
         Then, if $H$ is the hyperplane line bundle over $\mathbb{P}^1(\mathbb{C}),$ we find
         \begin{equation*}
          \int_{\mathbb{P}^1(\mathbb{C})}c_1(H)=\text{deg}H=1.
         \end{equation*}
 \end{enumerate}
\end{proof}

\begin{figure}[H]
      \begin{tikzpicture}[node distance=2.5cm, auto]
         \node (A) {$S_1\oplus\cdots\oplus S_{n-2}\oplus S_{n-1}\oplus Q_{n-1}$};
         \node (B) [below of=A] {$\mathbb{P}(Q_{n-2})$};
         \node (C) [right of=B] {$S_1\oplus S_2\oplus Q_2$};
         \node (D) [below of=C] {$\mathbb{P}(Q_1)$};
         \node (E) [right of=D] {$S_1\oplus Q_1$};
         \node (F) [below of=E] {$\mathbb{P}(E)$};
         \node (G) [right of=F] {$E$};
         \node (H) [below of=G] {$M$};
         \draw[->] (A) to node {$$} (B);
         \draw[->] (C) to node {$$} (D);
         \draw[->] (B) to node {$$} (D);
         \draw[->] (D) to node {$$} (F);
         \draw[->] (E) to node {$$} (F);
         \draw[->] (G) to node {$\tau$} (H);
         \draw[->] (F) to node {$$} (H); 
       \end{tikzpicture}
      \caption{Splitting Diagram\label{SplittingDiagram}}
      \end{figure}

\begin{prop}
 If $E=M\times\mathbb{C}$ is the trivial line bundle over the differentiable manifold $M,$ then 
 $c_1(E)=0.$
\end{prop}
\begin{proof}
 It follows immediately from \eqref{cechcocycle}.
\end{proof}

\begin{defin}
 Let $E$ be a complex vector bundle of rank $r$ over the differentiable manifold $M$ with transition functions
 $g_{\alpha\beta}.$ The determiant bundle of $E$ is the line bundle $\det(E)$ with transition functions 
 $\tilde{g}_{\alpha\beta}\det(g_{\alpha\beta}).$
\end{defin}

\begin{prop}
 Let $E$ be a complex vector bundle of rank $r$ over the differentiable manifold $M,$ and let 
 $\det(E)$ its determinant bundle. Then we have
 \begin{equation*}
  c_1(E)=c_1(\det(E)).
 \end{equation*}
\end{prop}
\begin{proof}
 \begin{enumerate}
  \item If $E$ is a line bundle the result is trivial, since $E=\det(E).$
  \item From te splitting principle we may assume $E=L_1\oplus\cdots\oplus L_r$ is direct sum of line 
        bundles. From $(1)$ we have 
        $\det(L_1\otimes\cdots\otimes L_r)=c_1(L_1\otimes\cdots\otimes L_r)$ and then we have
        \begin{equation*}
         \begin{split}
          c_1(\det(E))&=c_1(\det(L_1\otimes\cdots\otimes L_r))=\\
                            &=c_1(L_1\otimes\cdots\otimes L_r)=\\
                            &=c_1(L_1)+\cdots+c_1(L_r)=c_1(E).
         \end{split}
        \end{equation*}

 \end{enumerate}

\end{proof}

\begin{cor}
 Let $E$ be a complex line bundle over $M$ and let $E^{\ast}$ its dual bundle. Then 
 $c_1(E^{\ast})=-c_1(E).$
\end{cor}
\begin{proof}
 Obviously $E\otimes E^{\ast}$ is isomorphic to the trivial line bundle $M\times\mathbb{C}$ over $M.$
 Then $0=c_1(E\otimes E^{\ast})=c_1(E)+c_1(E^{\ast})$ and this completes the proof.
 \end{proof}
 
 \begin{cor}
  Let $E$ be a complex vector bundle of rank $r$ over the differentiable manifold $M,$ and let $E^{\ast}$
  be its dual bundle. Then for $i=1,\ldots,r$ we have $c_i(E^{\ast})=(-1)^{i}c_i(E).$
 \end{cor}
\begin{proof}
 From the splitting principle we may assume $E=L_1\oplus\cdots\oplus L_r$ is direct sum of line bundles
 over $M.$ Then $E^{\ast}=L_{1}^{\ast}\oplus\cdots\oplus L_{r}^{\ast}.$ From the previous corollary and
 elementary properites of symmetric functions we obtain
 \begin{equation*}
  \begin{split}
   c_i(E^{\ast})&=c_i(L_{1}^{\ast}\oplus\cdots\oplus L_{r}^{\ast})=\\
                &=\sigma_i(c_1(L_{1}^{\ast}),\ldots,c_1(L_{1}^{\ast}))=\\
                &=\sigma_i(-c_1(L_1),\ldots,-c_1(L_r))=\\
                &=(-1)^{i}\sigma_i(c_1(L_1),\ldots,c_1(L_r))=(-1)^{i}c_i(E).
  \end{split}
 \end{equation*}
\end{proof}

\section{Axiomatic approach to Chern classes}
In order to minimize topological prerequisites, in this section we take the axiomatic approach to Chern classes.
This enables us to separate differential geometry aspects of Chern classes from their topological aspects. 
We consider the cathegory of complex vector bundles over real manifolds. 
For more references see Kobayashi \cite{KOB} and Kobayashi-Nomizu vol. 1 \cite{KN1}.

\begin{axiom}
 For each complex vector bundle $E$ over $M$ and for each integer $0\leq i\leq\mathrm{rk}(E),$ the $i\text{-th}$
 Chern class $c_i(E)\in H^{2i}(M,\mathbb{R})$ is given and $c_0(E)=1.$ We set 
 \begin{equation*}
  c(E)=\sum_{i=0}^{\mathrm{rk}(E)}c_i(E),
 \end{equation*}
 and call $c(E)$ the total Chern class of $E.$
\end{axiom}

\begin{axiom}
 \emph{(Naturality)}
 Let $E$ be a complex vector bundle over $M$ and let $f:N\longrightarrow M$ be a $\mathcal{C}^{\infty}$
 mapping. Then
 \begin{equation*}
  c(f^{\ast}E)=f^{\ast}(c(E))\in H^{\ast}(M,\mathbb{R}).
 \end{equation*}
\end{axiom}

\begin{axiom}
 \emph{(Whitney sum formula)}
 Let $E_i,\ldots,E_q$ be complex line bundles over $M$ and let $E_1\oplus\cdots\oplus E_q$ be their
 Whitney sum. Then
 \begin{equation*}
  c(E_1\oplus\cdots\oplus E_q)=c(E_1)\cup\cdots\cup c(E_q).
 \end{equation*}
\end{axiom}

\begin{axiom}
 \emph{(Normalization)}
 If $L$ is the universal line bundle over $\mathbb{P}^1(\mathbb{C}),$ then $-c_1(L)$ is the positive generator
 of $H^2(\mathbb{P}^1(\mathbb{C}),\mathbb{Z}),$ that is, the one compatible with the orientation of 
 $\mathbb{P}^{1}(\mathbb{C}).$ In other words, $c_1(L)$ integrated on the foundamental
 2-cycle $\mathbb{P}^1(\mathbb{C})$ is equal to $-1.$
\end{axiom}

\section{Chern classes in terms of curvature}
In the previous section we introduced the $i\text{-th}$ Chern class $c_i(E)$ of a complex vector bundle as an 
element of $H^{2i}(M,\mathbb{R}).$ Via the de-Rham theory we should be able to represent $c_i(E)$ by a closed 
$2i\text{-form}$ $\gamma_i.$ In this section we shall construct such a $\gamma_i$ using the curvature form of a 
connection in $E.$ For convenience, we imbed $H^{2i}(M,\mathbb{R})$ into $H^{2i}(M,\mathbb{C})$ and
represent $c_i(E)$ by a closed complex $2i\text{-form}$ $\gamma_i.$

\begin{defin}
 Let $V$ be the Lie algebra $\frak{gl}(r,\mathbb{C})$ of the linear group $GL(r,\mathbb{C}),$ i.e., the
 Lie algebra of all $r\times r$ complex matrices. Let $G$ be $GL(r,\mathbb{C})$ acting on
 $\frak{gl}(r,\mathbb{C})$ by the adjoint action, i.e.,
 \begin{equation*}
  X\in\frak{gl}(r,\mathbb{C})\longrightarrow aXa^{-1}\in\frak{gl}(r,\mathbb{C}),\hspace{0.5cm}
  a\in GL(r,\mathbb{C}).
 \end{equation*}
Now we define homogeneous polynomials $f_k$ on $\frak{gl}(r,\mathbb{C})$ of degree
$k=1,\ldots,r$ by
\begin{equation}
\label{definepolynomials}
 \text{det}\left(I_r-\frac{1}{2\pi i}X\right)=1+f_1(X)+f_2(X)+\cdots+f_r(X),\hspace{0.5cm}
 X\in\frak{gl}(r,\mathbb{C}).
\end{equation}
Since 
\begin{equation*}
 \text{det}\left(I_r-\frac{1}{2\pi i}aXa^{-1}\right)=
 \text{det}\left(a\left(I_r-\frac{1}{2\pi i}X\right)a^{-1}\right)=
 \text{det}\left(I_r-\frac{1}{2\pi i}X\right)
\end{equation*}
the polynomials $f_1,\ldots,f_r$ are $GL(r,\mathbb{C})\text{-invariants}.$ It is known that these polynomials
generate the algebra of $GL(r,\mathbb{C})\text{-invariant}$ polynomials on $\frak{gl}(r,\mathbb{C}).$
\end{defin}

Since $GL(r,\mathbb{C})$ is a connected Lie group, the $GL(r,\mathbb{C})\text{-invariance}$ can be 
expressed infinitesimally. In fact, we have the following result:

\begin{prop}
A symmetric $\mathbb{C}\text{-multilinear}$ k-form $f$ on $\frak{gl}(r,\mathbb{C})$ is 
\linebreak$GL(r,\mathbb{C})\text{-invariant}$ if and only if
\begin{equation}
\label{lieinvariance}
 \sum_{j=1}^{k} f(X_1,\ldots,[Y,X_j],\ldots,X_k)=0\hspace{0.5cm}\text{ for }\hspace{0.5cm}
 X_j,Y\in\frak{gl}(r,\mathbb{C}).
\end{equation}
\end{prop}

\begin{defin}
 Let $E$ be a complex vector bundle of rank $r$ over a differentiable manifold $M$ of (real) dimension $n.$
 Let $D$ be a connection in $E$ and let $R$ be its curvature. Choosing a local frame field 
 $s=(s_1,\ldots,s_r),$ we denote the connection form and the curvature form of $D$ by 
 $\omega$ and $\Omega$ respectively. Given a $GL(r,\mathbb{C})\text{-invariant}$ symmetric multilinear
 form $f$ of degree $k$ on $\frak{gl}(r,\mathbb{C}),$ we set
 \begin{equation*}
  \gamma=f(\Omega)=f(\Omega,\ldots,\Omega).
 \end{equation*}
 \end{defin}

\begin{prop}
 $\gamma$ is indipendent of the choice of the local frame field $s$ and hence is a globally defined 
 differential form of degree $2k.$
\end{prop}
\begin{proof}
 If $s'=sa^{-1}$ is another frame field, then the corresponding curvature form is given by
 $a\Omega a^{-1}.$ Since $f$ is $GL(r,\mathbb{C})\text{-invariant},$ it immediately follows that $\gamma$ is a 
 globally defined differential form of degree $2k.$
\end{proof}

\begin{prop}
 $\gamma$ is closed, i.e., $\mathrm{d}\gamma=0.$ Then $\gamma$ represents a cohomology class in 
 $H^{2k}(M,\mathbb{C}).$
\end{prop}
\begin{proof}
 Using the Bianchi identity $\mathrm{d}\Omega=[\Omega,\omega]$ and \eqref{lieinvariance} we have
 \begin{equation*}
  \begin{split}
   \mathrm{d}\gamma&=\mathrm{d}f(\Omega,\ldots,\Omega)=\\
                   &=f(\mathrm{d}\Omega,\ldots,\Omega)+\cdots+f(\Omega,\ldots,\mathrm{d}\Omega)=\\
                   &=f([\Omega,\omega],\ldots,\Omega)+\cdots+f(\Omega,\ldots,[\Omega,\omega])=0.
  \end{split}
 \end{equation*}
\end{proof}

Now we show that the cohomology class is well defined.

\begin{prop}
 The cohomology class of $\gamma$ does not depend on the choice of the connection $D.$ 
\end{prop}
\begin{proof}
 We consider two connections $D_0$ and $D_1$ in $E$ and connect them by a segment of connections
 \begin{equation*}
  D_t=(1-t)D_0+tD_1,\hspace{0.5cm}0\leq t\leq1.
 \end{equation*}
Let $\omega_t$ and $\Omega_t$ be the connection form and the curvature form of $D_t$ with respect to a 
local frame field $s.$ Then we write
\begin{equation*}
 \omega_t=\omega_0+t\alpha,\hspace{0.5cm}\text{ where }\hspace{0.5cm}\alpha=\omega_1-\omega_0,
\end{equation*}
and
\begin{equation*}
 \Omega_t=\mathrm{d}\omega_t+\omega_t\wedge\omega_t.
\end{equation*}
Then
\begin{equation*}
 \frac{\mathrm{d}\Omega_t}{\mathrm{d}t}=
 \mathrm{d}\alpha+\alpha\wedge\omega_t+\omega_t\wedge\alpha=D_t\alpha.
\end{equation*}
Finally we set
\begin{equation*}
 \varphi=k\int_{0}^{1}f(\alpha,\Omega_t,\ldots,\Omega_t)\mathrm{d}t.
\end{equation*}
From \eqref{connectionchange} we see that the difference $\alpha$ of two connection forms 
transforms in the same way as the curvature form under a transformation of the local frame field $s.$
It follows that $f(\alpha,\Omega_t,\ldots,\Omega_t)$ is indipendent of $s$ and hence a globally defined
$(2k-1)\text{-form}$ on $M.$ Therefore, $\varphi$ is a $(2k-1)\text{-form}$ on $M.$
From Bianchi identity $D_t\Omega_t=0,$ we obtain
\begin{equation*}
 \begin{split}
  k\mathrm{d}f(\alpha,\Omega_t,\ldots,\Omega_t)&=kD_tf(\alpha,\Omega_t,\ldots,\Omega_t)=
  kf(D_t\alpha,\Omega_t,\ldots,\Omega_t)=\\
  &=kf\left(\frac{\mathrm{d}\Omega_t}{\mathrm{d}t},\Omega_t,\ldots,\Omega_t\right)=
  \frac{\mathrm{d}}{\mathrm{d}t}f(\Omega_t,\Omega_t,\ldots,\Omega_t).
 \end{split}
\end{equation*}
Hence,
\begin{equation*}
 \mathrm{d}\varphi=\int_{0}^{1}\frac{\mathrm{d}}{\mathrm{d}t}f(\Omega_t,\Omega_t,\ldots,\Omega_t)=
 f(\Omega_1,\ldots,\Omega_1)-f(\Omega_0,\ldots,\Omega_0),
\end{equation*}
which proves that the cohomology class of $\gamma$ does not depend on the connection $D.$
\end{proof}

\begin{defin}
 Using the $GL(r,\mathbb{C})\text{-invariant}$ polynomials $f_k$ defined by \eqref{definepolynomials}, we may
 define 
 \begin{equation}
  \label{CHERN1}
  \gamma_k=f_k(\Omega),\hspace{0.5cm}k=1,\cdots,r.
 \end{equation}
In other words,
\begin{equation}
 \label{CHERN2}
 \text{det}\left(I_r-\frac{1}{2\pi i}\Omega\right)=1+\gamma_1+\gamma_2+\cdots+\gamma_r.
\end{equation}
\end{defin}

\begin{note}
 After some linear algebra calculations we find
 \begin{equation*}
  \gamma_k=\frac{(-1)^k}{(2\pi i)^k}\sum \delta_{i_1\cdots i_k}^{j_1\cdots j_k}
  \Omega_{j_1}^{i_1}\wedge\cdots\wedge\Omega_{j_k}^{i_k}.
 \end{equation*}
In particular
\begin{equation*}
 \gamma_i=-\frac{1}{2\pi i}\mathrm{tr}(\Omega),
\end{equation*}
and 
\begin{equation*}
 \gamma_2=-\frac{1}{8\pi^2}(\mathrm{tr}(\Omega)\wedge\mathrm{tr}(\Omega)-\mathrm{tr}(\Omega\wedge\Omega)).
\end{equation*}
\end{note}

\begin{teo}
 Let $E$ be a complex vector bundle of rank $r$ over a differentiable manifold $M$ of (real) dimension $n.$ 
 The $k\text{-th}$ Chern class $c_k(E)$ of a complex vector bundle $E,$ as a cohomology class in 
 $H^{2k}(M,\mathbb{C}),$ is represented by the closed $2k\text{-form}$ $\gamma_k$ defined by
 \eqref{CHERN1} or \eqref{CHERN2}.
\end{teo}
\begin{proof}
 We have to show that the cohomology classes represented by the closed 2k-forms $\gamma_k$ satisfy the 
 four axioms given in the previous section.
 \begin{enumerate}
  \item Axiom 1 is trivially satisfied. We simply need to set $\gamma_0=1.$
  \item For Axiom 2, in the vector bundle $f^{\ast}E$ induced from $E$ by the $\mathcal{C}^{\infty}$
        mapping $f:N\longrightarrow M,$ we use the connection $f^{\ast}D$ induced from a connection $D$ 
        in $E.$ Then its curvature form is given by $f^{\ast}\Omega.$ Since
        \linebreak$f_k(f^{\ast}\Omega)=f^{\ast}(f_k(\Omega))=f^{\ast}\gamma_k,$ Axiom 2 is satisfied.
  \item To verify Axiom 3, let $D_1,\ldots,D_q$ be connections in line bundles 
        $E_1,\ldots,E_q$ respectively and let $\Omega_1,\ldots,\Omega_q$ be their curvature forms.
        We use the connection $D=D_1\otimes\cdots\otimes D_q$ in $E=E_1\otimes\cdots\otimes E_q,$ then
        its curvature form is diagonal with diagonal entries $\Omega_1,\ldots,\Omega_q.$
        Hence
        \begin{equation*}
         \text{det}\left(I_r-\frac{1}{2\pi i}\Omega\right)=\left(1-\frac{1}{2\pi i}\Omega_1\right)\wedge
         \cdots\wedge\left(1-\frac{1}{2\pi i}\Omega_q\right),
        \end{equation*}
         which establishes Axiom 3.
  \item We take a natural Hermitian structure in the tautological line bundle $L$ over 
        $\mathbb{P}^1(\mathbb{C}),$ i.e., the one arising from the natural inner product in 
        $\mathbb{C}^2.$ Since a fibre of $L$ is a complex line through the origin of $\mathbb{C}^2,$ each
        element $\zeta\in L$ is represented by a vector $(\zeta^0,\zeta^1)$ in $\mathbb{C}^2.$ Then the
        Hermitian structure $h$ is defined by:
        \begin{equation*}
         h(\zeta,\zeta)=|\zeta^0|^2+|\zeta^1|^2
        \end{equation*}
        Considering $[\zeta^0,\zeta^1]$ as a homogeneous coordinate in $\mathbb{P}^1(\mathbb{C}),$ let
        $z=\zeta^1/\zeta^0$ be the inhomogeneous coordinate in the local chart 
        $U_0=\mathbb{P}^1(\mathbb{C})-\{[0,1]\}.$ Let $s$ be the frame field of $L$ over $U_0$ defined by
        \begin{equation*}
         s(z)=(1,z)\in L_z\subseteq\mathbb{C}^2.
        \end{equation*}
        With respect to $s,$ $h$ is given by the function
        \begin{equation*}
         H(z)=h(s(z),s(z))=1+|z|^2.
        \end{equation*}
        Hence the connection form and the curvature form of the Hermitian connection $D$ are given by
        \begin{equation*}
         \omega=\frac{\overline{z}\mathrm{d}z}{1+|z|^2},\hspace{0.5cm}
         \Omega=-\frac{\mathrm{d}z\wedge\mathrm{d}\overline{z}}{(1+|z|^2)^2}.
        \end{equation*}
        So
        \begin{equation*}
         \gamma_1=\frac{\mathrm{d}z\wedge\mathrm{d}\overline{z}}{2\pi i(1+|z|^2)^2}.
        \end{equation*}
        Using polar coordinates $(r,t)$ defined by $z=re^{2\pi it},$ $0\leq t\leq 1$ we write
        \begin{equation*}
         \gamma_1=-\frac{2r\mathrm{d}r\wedge\mathrm{d}t}{(1+r^2)^2}\hspace{0.5cm}\text{ on }U_0.
        \end{equation*}
        Then integrating $\gamma_1$ over $\mathbb{P}^1(\mathbb{C}),$ we obtain
        \begin{equation*}
        \begin{split}
         \int_{\mathbb{P}^1(\mathbb{C})}\gamma_1&=\int_{U_0}\gamma_1=
         \int_{0}^{1}\int_{0}^{+\infty}-\frac{2r\mathrm{d}r\wedge\mathrm{d}t}{(1+r^2)^2}=\\
         &=\int_{0}^{1}\mathrm{d}t\left(\int_{0}^{+\infty}-\frac{2r\mathrm{d}r}{(1+r^2)^2}\right)=\\
         &=\int_{0}^{1}\mathrm{d}t\left[\frac{1}{1+r^{2}}\right]_{t=0}^{t=+\infty}
         =\int_{0}^{1}-\mathrm{d}t=-1.
        \end{split} 
        \end{equation*}
        and this verifies Axiom 4.
 \end{enumerate}

\end{proof}

\chapter{Algebraic and Analytic Tools}
In this chapter we present some algebraic notions such as coherent sheaves, torsion-free sheaves and 
locally-free sheaves
over a compact K\"ahler manifold $(X,\omega).$ Moreover, we introduce the notion of $\omega\text{-stable}$
and $\omega\text{-semistable}$ torsion-free sheaf of $\mathcal{O}_X\text{-modules}$ over a compact
K\"ahler manifold
$(X,\omega).$ For more detailed definitions and proofs see \cite{HAR}, \cite{KOB} and \cite{WEL}.

On the other hand, in the last section of this chapter we present some useful analytic tools 
which are involved in the proofs of the main results of this work: 
the Fredholm alternative Theorem and the Maximum Principle.
See \cite{EVA} and \cite{KOB} for more details.

\section{Torsion-free and locally-free analytic coherent sheaves}
\begin{defin}
 Let $X$ be complex manifold of (complex) dimension $n$ and let $\mathcal{O}=\mathcal{O}_X$ 
 be the structure sheaf of $X,$
 i.e., the sheaf of germs of holomorphic functions on $X.$ We write
 \begin{equation*}
  \mathcal{O}^p=\mathcal{O}\otimes\cdots\otimes\mathcal{O}
 \end{equation*}
An analytic sheaf over $X$ is a sheaf of $\mathcal{O}_X\text{-modules}$ over $X.$ 
\end{defin}

\begin{defin}
 Let $\mathcal{S}$ be an analytic sheaf over a complex manifold $X$ of (complex) dimension $n.$ 
 $\mathcal{S}$ is coherent if, for every point $P\in X,$ there exists a neighborhood $U$ of $P$ in $X$ and
 an exact sequence of sheaves
 \begin{equation*}
  \mathcal{O}^{q}\vert_{U}\longrightarrow\mathcal{O}^{p}\vert_{U}\longrightarrow\mathcal{S}\vert_{U}
  \longrightarrow0.
 \end{equation*}
\end{defin}

\begin{defin}
 Let $\mathcal{S}$ be an analytic sheaf over a complex manifold $X$ of (complex) dimension $n.$ We have the
 following definitions:
 \begin{enumerate}
  \item  $\mathcal{S}$ is free if it is isomorphic to a direct sum of copies of the structure sheaf 
         $\mathcal{O}_X,$
  \item  $\mathcal{S}$ is locally-free if $X$ can be covered by open sets $\{U\}$ for which
         $\mathcal{S}\vert_U$ is a free $\mathcal{O}\vert_U\text{-module}.$
 \end{enumerate}
 If $\mathcal{S}$ is locally-free, the rank of $\mathcal{S}$ on such an open set is the number of copies 
 of the structure sheaf needed (finite or infinite).
If $X$ is connected, the rank of a locally-free sheaf is the same everywhere. A locally-free sheaf of 
rank $1$ is also called an invertible sheaf.
\end{defin}

\begin{prop}
\label{finiterank}
 Let $\mathcal{S}$ be a coherent sheaf over a complex manifold $X$ of (complex) dimension $n.$
 Then $\mathcal{S}$ has finite rank.
\end{prop}
\begin{proof}
 It follows essentially from Oka Lemma and the Syzygy Theorem. (See Gunning-Rossi \cite{GR} for
 details).
\end{proof}

\begin{defin}
 Let $\mathcal{S}$ be an analytic sheaf over a complex manifold $X$ of (complex) dimension $n.$ 
 $\mathcal{S}$ is torsion-free if for every point $P\in X,$ the corresponding stalk $\mathcal{S}_P$ is 
 a torsion-free $\mathcal{O}_P\text{-module}.$
\end{defin}

\begin{prop}
 Every subsheaf of an analytic torsion-free sheaf is also torsion-free.
\end{prop}
\begin{proof}
 Let $\mathcal{F}\subseteq\mathcal{S}$ be a subsheaf of the analytic sheaf $\mathcal{S}.$ For every
 point $P\in X$ the stalk $\mathcal{F}_P$ is an $\mathcal{O}_P\text{-submodule}$ of $\mathcal{S}_P.$ 
 Since $\mathcal{S}_P$ is a torsion-free $\mathcal{O}_P\text{-module},$ we conclude that
 $\mathcal{F}_P$ is a torsion-free $\mathcal{O}_P\text{-module},$ and this completes the proof.
\end{proof}

\begin{prop}
 Let $\mathcal{S}$ be a coherent sheaf over a complex manifold $X$ of (complex) dimension $n.$ If
 $\mathcal{S}$ is locally-free then it is torsion-free.
\end{prop}
\begin{proof}
 Let $P\in X$ a point of $X.$ We only have to show that the stalk $\mathcal{S}_P$ is a torsion-free 
 $\mathcal{O}_P\text{-module}.$ Since $\mathcal{S}$ is locally-free, we can find  a neighborhood 
 $U\subseteq X$ of $P$ in $X$ such that 
 $\mathcal{S}\vert_U=\mathcal{O}^{q}\vert_{U}.$ Since $\mathcal{S}$ is coherent, according to 
 Proposition \ref{finiterank}, the rank on $U$ is finite. If we consider the stalk at the point $P$ we
 have $\mathcal{S}_P=\mathcal{O}_{P}^{q}$ and then, since $\mathcal{O}_P$ is UFD, we have that the stalk
 $\mathcal{S}_P$ is a torsion-free $\mathcal{O}_P\text{-module}.$
\end{proof}

The following result demonstrates a deep relationship between vector bundles and locally-free sheaves. For
the proof see Wells \cite{WEL}.

\begin{teo}
\label{correspondence}
 Let $X$ be a complex connected manifold of (complex) dimensions n. There is a one-to-one correspondence
 between isomorphism classes of holomorphic vector bundles over $X$ and isomorphism classes
 of locally-free coherent sheaves over $X.$
\end{teo}

\begin{defin}
 Let $\mathcal{S}$ be a coherent sheaf over a complex manifold $X$ of complex dimension $n.$
 The singularity set of the sheaf $\mathcal{S}$ is
 \begin{equation*}
  S=S_{n-1}(\mathcal{S})=\{P\in X|\mathcal{S}_P\text{ is not free}\}
 \end{equation*}
\end{defin}

\begin{teo}
 Let $\mathcal{S}$ be coherent torsion-free sheaf over a complex manifold $X$ of (complex) dimension
 $n.$ The singularity set $S$ is a closed analytic subset of $X$ of dimension
 $\mathrm{dim}_{\mathbb{C}}S\leq n-2.$
\end{teo}
\begin{proof}
 See Kobayashi \cite{KOB} p. 154-159 for a detailed proof.
\end{proof}

\begin{cor}
\label{riemannsurface}
 Let $\mathcal{S}$ be a coherent torsion-free sheaf over a complex manifold $X$ of (complex) dimension $1,$ i.e, 
 a Riemann surface. Then $\mathcal{S}$ is locally-free.
\end{cor}
\begin{proof}
 Clearly $\mathcal{S}$ is locally-free outside the singularity set $S.$ From the previous Theorem, since $X$ 
 is a Riemann surface, we conclude that $S$ is empty and this completes the proof.
\end{proof}

\section{Stable and semistable sheaves}

According to Theorem \ref{correspondence}, every exact sequance
\begin{equation}
\label{bundlesequence}
 0\longrightarrow E_m\longrightarrow\cdots\longrightarrow E_0\longrightarrow0
\end{equation}
of holomorphic vector bundles over a complex manifold $X$ of (complex) dimension $n,$ induces an exact
sequence of locally-free coherent sheaves over $X$
\begin{equation}
\label{sheavessequence}
 0\longrightarrow\mathcal{E}_m\longrightarrow\cdots\longrightarrow\mathcal{E}_0\longrightarrow0,
\end{equation}
where $\mathcal{E}_i$ denotes the sheaf $\mathcal{O}(E_i)$ of germs of holomorphic sections of $E_i.$
Conversely, every exact sequence \eqref{sheavessequence} of locally-free sheaves over a complex manifold
$X$ comes from an exact sequence \eqref{bundlesequence} of the corresponding holomorphic vector bundles.

\begin{defin}
 Given a coherent sheaf $\mathcal{S}$ over a complex manifold $X,$ we shall define its determinant bundle
 $\det\mathcal{S}.$ Let
 \begin{equation}
 \label{determinantresolution}
  0\longrightarrow\mathcal{E}_n\longrightarrow\cdots\longrightarrow\mathcal{E}_0\longrightarrow
  \mathcal{S}\vert_U\longrightarrow0
 \end{equation}
be a locally-free resolution of $\mathcal{S}\vert_U,$ where $U$ is a small open set in the
base manifold $X.$ Thanks to Syzygy Theorem and Oka Lemma such a resolution always exists. Let $E_i$
denote the holomorphic vector bundle corresponding to the sheaf $\mathcal{E}_i.$ We set
\begin{equation}
\label{localdet}
 \mathrm{det}\mathcal{S}\vert_U=\bigotimes_{i=0}^{n}(\mathrm{det}E_i)^{(-1)^i}.
\end{equation}
\end{defin}
   
One should check that $\det\mathcal{S}\vert_U$ is indipendent of the choice
of the resolution \eqref{determinantresolution}, for details see Kobayashi \cite{KOB}, p. 163-165.

\begin{prop}
 Let $\mathcal{S}$ be a coherent sheaf over a complex manifold $X$ of complex dimension $n.$ Then 
 $\mathrm{det}\mathcal{S}$ is a line bundle over $X.$
\end{prop}
\begin{proof}
 From \eqref{localdet} we know that $\mathrm{det}\mathcal{S}$ is locally a holomorphic line
 bundle over $X$ and then, thanks to Theorem \ref{correspondence}, it is a locally-free sheaf 
 over $X$ of rank $1.$  Hence, also from Theorem \ref{correspondence} we deduce that
 $\mathrm{det}\mathcal{S}$ is in one-to-one correspondence with an isomorphism class of holomorphic
 line bundle over $X.$
\end{proof}

\begin{defin}
 Let $\det\mathcal{S}$ be the determinant bundle of a coherent sheaf $\mathcal{S}$ over a
 complex manifold $X$ of complex dimension $n.$ We define the first Chern class $c_1(\mathcal{S})$ by
 \begin{equation*}
  c_1(\mathcal{S})=c_1(\det\mathcal{S})\in H^{2}(X,\mathbb{C}).
 \end{equation*}
\end{defin}

\begin{remark}
 If $\mathcal{S}$ and $\mathcal{F}$ are isomorphic coherent sheaves over the complex manifold $X,$ then 
 $\det\mathcal{S}$ and $\det\mathcal{F}$ are isomorphic. In particular $\mathcal{S}$ and $\mathcal{F}$ 
 have the same Chern class $c_1(\mathcal{S})=c_1(\mathcal{F}).$
\end{remark}

\begin{defin}
 Let $\mathcal{S}$ be a torsion-free coherent sheaf over a compact K\"ahler manifold $(X,\omega)$ of 
 (complex) dimension $n.$ The degree of $\mathcal{S}$ is defined by
 \begin{equation}
  \label{degree}
  \mathrm{deg}(\mathcal{S})=\int_{M}c_1(\mathcal{S})\wedge\omega^{n-1},
 \end{equation}
while the slope of $\mathcal{S}$ is defined by
 \begin{equation}
  \label{slope}
  \mu(\mathcal{S})=\frac{\mathrm{deg}(\mathcal{S})}{\text{rk}(\mathcal{S})}=
  \frac{\int_{M}c_1(\mathcal{S})\wedge\omega^{n-1}}{\text{rk}(\mathcal{S})}.
 \end{equation}
\end{defin}

\begin{defin}
 Let $\mathcal{S}$ be a torsion-free coherent sheaf over a compact K\"ahler manifold $(X,\omega)$
 of (complex) dimension $n.$ $\mathcal{S}$ is $\omega\text{-semistable}$ if for every
 coherent subsheaf $\mathcal{S}'\subseteq\mathcal{S}$ with 
 $0<\text{rk}(\mathcal{S}')<\text{rk}\mathcal{S},$ the inequality
 \begin{equation*}
  \mu(\mathcal{S'})\leq\mu(\mathcal{S})
 \end{equation*}
 holds. If moreover the strict inequality
 \begin{equation*}
  \mu(\mathcal{S'})<\mu(\mathcal{S})
 \end{equation*}
 holds for all coherent subsheaf $\mathcal{S'}$ with $0<\text{rk}(\mathcal{S'})<\text{rk}(\mathcal{S}),$
 we say that $\mathcal{S}$ is $\omega\text{-stable}.$
\end{defin}

\section{Analytic tools}
\begin{teo}
\label{Fredholm}
 \emph{(Fredholm alternative)} Let $H$ be a Hilbert space and let 
 \linebreak$K:H\longrightarrow H$ 
 be a compact bounded linear operator. Then
 \begin{enumerate}
  \item $\ker(I-K)$ is finite dimensional (and hence it is closed),
  \item $\mathrm{Im}(I-K)$ is closed,
  \item $\mathrm{Im}(I-K)=\ker(I-K^{\ast})^{\bot},$
  \item $\ker(I-K)=0$ if and only if $\mathrm{Im}(I-K)=H,$
  \item $\text{dim}\ker(I-K)=\text{dim}\ker(I-K^{\ast}).$
 \end{enumerate}
\end{teo}

\begin{teo}
\label{MPPE}
 \emph{(Maximum principle for parabolic equations)}
 Let $M$ be a compact Riemannian manifold and let $f:M\times[0,a)\longrightarrow\mathcal{R}$ a function
 of class $\mathcal{C}^1$ with continuous laplacian $\Delta f$ satifsfying the inequality
 \begin{equation*}
  \partial_{t}f+c\Delta f\leq0,\hspace{0.5cm}(c>0.)
 \end{equation*}
Set $F(t)=\max_{M}f(x,t).$ Then $F(t)$ is a monotone decreasing function of $t.$
\end{teo}

\begin{lemma}
 \label{liminfBanach}
 Let $(X,\|\cdot\|_{X})$ be a Banach space and let $\{x_m\}\subseteq X$ be a sequence such that 
 $x_{m}\rightharpoonup x$ in $X.$ Then
 \begin{equation*}
  \|x\|_{X}\leq\liminf_{m}\|x_{m}\|_{X}.
 \end{equation*}
\end{lemma}

\begin{lemma}
\label{efremcalzelunghe}
 Let $(X,\|\cdot\|_{X})$ and $(Y,\|\cdot\|_{Y})$ be Banach spaces and let $\{x_m\}\subseteq X,$ 
 $\{L_m\}\subseteq\mathcal{L}(X,Y)$ be sequences such that
 \begin{enumerate}
  \item $x_m\rightharpoonup x$ in $X,$
  \item $L_m\longrightarrow L$ in $\mathcal{L}(X,Y).$
 \end{enumerate}
 Then
 \begin{equation*}
  L_mx_m\rightharpoonup Lx\hspace{0.5cm}\text{ in }\hspace{0.2cm}Y.
 \end{equation*}
\end{lemma}
\begin{proof}
 First, from $x_m\longrightarrow x$ in $X$ we deduce that the sequence $\{x_m\}$ is bounded in $X,$ then 
 there exists $M>0$ such that
 \begin{equation*}
  \|x_m\|_{X}\leq M
 \end{equation*}
and $M$ does not depend on $m.$ 
Let $\varphi\in Y',$ then $\varphi L_{m}\longrightarrow\varphi L$ in $X',$ in fact
\begin{equation*}
 \|\varphi L_{m}-\varphi L\|_{X'}\leq\|\varphi\|_{Y'}\|L_{m}-L\|_{\mathcal{L}(X,Y)}\longrightarrow0.
\end{equation*}
So that, since $x_{m}\rightharpoonup x$ in $X$ and since $\varphi L,\varphi L_{m}\in X',$
\begin{equation*}
 \begin{split}
  |\langle\varphi,L_{m}x_{m}\rangle-\langle\varphi,Lx\rangle|&\leq 
  |\langle\varphi,L_{m}x_{m}\rangle-\langle\varphi,Lx_{m}\rangle|+
  |\langle\varphi,Lx_{m}\rangle-\langle\varphi,Lx\rangle|\leq\\
  &\leq|\langle\varphi L_{m},x_{m}\rangle-\langle\varphi L,x_{m}\rangle|+
  |\langle\varphi L,x_{m}\rangle-\langle\varphi L,x\rangle|\leq\\
  &\leq\|\varphi L_{m}-\varphi L\|_{X'}\|x_{m}\|_{X}+
  |\langle\varphi L,x_{m}\rangle-\langle\varphi L,x\rangle|\leq\\
  &\leq M\|\varphi L_{m}-\varphi L\|_{X'}+|\langle\varphi L,x_{m}\rangle-\langle\varphi L,x\rangle|
  \longrightarrow0.
 \end{split}
\end{equation*}
\end{proof}

\chapter{Approximate Hermitian-Yang-Mills Metrics and Semistability}

In this chapter we review the notions of (weak) Hermitian-Yang-Mills structure and approximate 
Hermitian-Yang-Mills structure for Higgs bundles. Then, we construct the Donaldson functional for Higgs bundles
over 
compact K\"ahler manifolds and present some basic properties of it. In particular, we study
the properites of the Donaldson heat flow and we establish a relation between this gradient flow 
and the mean curvature of the Hitchin-Simpson connection. We also study some properties of the solutions 
of the evolution equation associated with that functional. Finally, we study the problem of the 
existence of approximate Hermitian-Yang-Mills strucutres and its relation with the notion of semistability.

In particular in this chapter we show that for a Higgs bundle $\frak{E}=(E,\phi)$ over a compact Riemann surface
$X$ with K\"ahler form $\omega,$ thefollowing conditions are equivalent:
\begin{enumerate}
 \item There exists an approximate Hermitian-Yang-Mills metric structure,
 \item The Higgs bundle $\frak{E}=(E,\phi)$ is $\omega\text{-semistable}.$
\end{enumerate}

\section{Higgs sheaves and Higgs bundles}
We start this section with some basic definitions.
\begin{defin}
 Let $(X,\omega)$ be a compact K\"ahler manifold of (complex) dimension $n.$
 We define $\Omega_{X}^{1}$ the holomorphic cotangent bundle to $X.$ It is the dual of 
 the holomorphic tangent bundle to $X.$ 
 \begin{note}
   From the decomposition
 $TX^{\mathbb{C}}=TX^{\mathbb{C+}}\oplus TX^{\mathbb{C-}}$ we identify the holomorphic tangent
 bundle to $X$ with the sheaf of holomorphic sections of $TX^{\mathbb{C+}}.$
 \end{note}
\end{defin}

\begin{defin}
Let $(X,\omega)$ be a compact K\"ahler manifold of (complex) dimension $n.$
 A Higgs sheaf $\frak{E}$ over $X$ is a coherent sheaf $E$ over $X,$ together with a morphism of
 $\mathcal{O}_X\text{-modules}$ $\phi:E\longrightarrow E\otimes\Omega_{X}^{1},$ such that the morphism
 $\phi\wedge\phi:E\longrightarrow E\otimes\Omega_{X}^{2}$ vanishes, i.e., $\phi\wedge\phi=0.$
 The morphism $\phi$ is called the Higgs field of $\frak{E}.$
\end{defin}

\begin{defin}
 A Higgs sheaf $\frak{E}$ is said to be torsion-free if the sheaf $E$ is torsion free. A Higgs bundle
 $\frak{E}$ is just a Higgs sheaf in which the sheaf $E$ is locally-free.
\end{defin}

\begin{defin}
 A Higgs subsheaf $\frak{F}$ is a subsheaf $F$ of $E$ such that 
 \linebreak$\phi(F)\subseteq F\otimes\Omega_{X}^{1}.$
\end{defin}

\begin{defin}
 Let $(X,\omega)$ be a compact K\"ahler manifold and let $\frak{E}=(E,\phi)$ be a Higgs sheaf over $X.$ A 
 section $s\in\Gamma(X,E)$ is $\phi\text{-invariant}$ if there exists a section $\lambda$ of $\Omega_{X}^{1}$ 
 such that $\phi(s)=s\otimes\lambda.$
\end{defin}

\begin{defin}
Let $(X,\omega)$ be a K\"ahler manifold of (complex) dimension $n$ and
 let $\frak{E}_1$ and $\frak{E}_2$ be two Higgs sheaves over $X.$ 
 A morphism between $\frak{E}_1$ and $\frak{E}_2$ is a map $f:E_1\longrightarrow E_2$ such that the 
 diagram
 \begin{equation*}
 \begin{tikzpicture}[node distance=2cm, auto]
 \node (A) {$E_1$};
 \node (B) [right of=A] {$E_1\otimes\Omega_{X}^{1}$};
 \node (C) [below of=A] {$E_2$};
 \node (D) [below of=B] {$E_2\otimes\Omega_{X}^{1}$};
 \draw[->] (A) to node {$\phi_1$} (B);
 \draw[->] (A) to node [swap] {$f$} (C);
 \draw[->] (B) to node {$f\otimes\mathrm{Id}$} (D);
 \draw[->] (C) to node {$\phi_2$} (D);
\end{tikzpicture}
\end{equation*}
is commutative. We will denote such a morphism by $f:\frak{E}_1\longrightarrow\frak{E}_2.$
A sequence of Higgs sheaves is a sequence of their corresponding coherent sheaves where each map is a
morphism of Higgs sheaves. A short exact sequence of Higgs sheaves is defined in the obvious way.
\end{defin}

\begin{note}
 Using local coordinates on $X$ we can write $\phi=\phi_{\alpha}\mathrm{d}z^{\alpha},$ where the index takes
values $\alpha=1,\ldots,n$ and each $\phi_{\alpha}$ is an endomorphism of $E.$ The condition 
$\phi\wedge\phi=0$ is then equivalent to the commutativity of the endomorphisms $\phi_{\alpha}.$
\end{note}

\begin{note}
 Let $(X,\omega)$ be a compact K\"ahler manifold of (complex) dimension $n$ and let $\frak{E}=(E,\phi)$
 be a Higgs bundle of rnak $r$ over $X.$
 If $(s_{1},\ldots,s_{r})$ is a local frame field on $X$ and $\eta^{1},\ldots,\eta^{r}$ is its dual,
using local coordinates on $X$ one can write 
$\phi=\phi^{\gamma}_{\alpha\beta}s_{\gamma}\otimes\eta^{\beta}\otimes\mathrm{d}z^{\alpha},$
where the indexes take values $\alpha=1,\ldots,n$ and $\beta,\gamma=1,\ldots,r$ while 
$\phi^{\gamma}_{\alpha\beta}$ are functions locally defined on $X.$ 
\end{note}

We give the definitions of dual Higgs bundle and Higgs pull-back bundle.

\begin{defin}
 Let $(X,\omega)$ be a compact K\"ahler manifold of (complex) dimension $n$ and let $\frak{E}=(E,\phi)$ be a 
 Higgs bundle of rank $r$ over $X.$
 Let us consider the Higgs field $\phi$  as a section of $\mathrm{End}(E)\otimes\Omega_{X}^{1}.$ 
 Since $\mathrm{End}(E^{\ast})\cong\mathrm{End}(E)$
 there is a natural dual morphism $\phi^{\ast}:E^{\ast}\longrightarrow E^{\ast}\otimes\Omega_{X}^{1}.$
 From this it follows that $\frak{E}^{\ast}=(E^{\ast},\phi^{\ast})$ is a Higgs bundle, called the dual
 Higgs bundle of $\frak{E}=(E,\phi).$
\end{defin}

\begin{defin}
 Let $(X,\omega)$ be a compact K\"ahler manifold of (complex) dimension $n$ and let $\frak{E}=(E,\phi)$ be a 
 Higgs bundle of rank $r$ over $X.$ Let $Y$ be another compact K\"ahler manifold and  let $f:Y\longrightarrow X$
 be a holomorphic map. 
 Since in the cathegory of holomorphic bundles there is an isomorphism
\begin{equation*}
 f^{\ast}E\cong f^{-1}E\otimes_{f^{-1}\mathcal{O}_{X}}\mathcal{O}_{Y},
\end{equation*}
if $\phi$ is a Higgs field on $E,$ there is a pull-back Higgs field $f^{\ast}\phi$ on $f^{\ast}E$ defined by
\begin{equation*}
 (f^{\ast}\phi)(\sum\varphi_{j}s_{j})=\sum\varphi_{j}\phi(s_{j}),
\end{equation*}
where $\varphi_{j}\in A^{0}_{Y}$ are $\mathcal{C}^{\infty}$ functions on $Y$ and $s_{j}\in\Gamma(X,E)$ are
$\mathcal{C}^{\infty}$ sections of $E.$
 From this it follows that $f^{\ast}\frak{E}=(f^{\ast}E,f^{\ast}\phi)$ is a Higgs bundle on $Y,$ called th.e 
 pull-back bundle of $\frak{E}=(E,\phi).$ 
\end{defin}

\begin{defin}
Let $(X,\omega)$ be a compact K\"ahler manifold of (complex) dimension $n$ and 
 let $\frak{E}$ be a Higgs sheaf over $X$ of rank $r.$ The degree of $\frak{E}$ is defined 
 by
 \begin{equation*}
  \mathrm{deg}(\frak{E})=\int_{X}c_1(\frak{E})\wedge\omega^{n-1},
 \end{equation*}
 and the slope of $\frak{E}$ is defined by
 \begin{equation*}
  \mu(\frak{E})=\frac{\mathrm{deg}(\frak{E})}{\mathrm{rk}(\frak{E})}=
  \frac{\int_{X}c_1(\frak{E})\wedge\omega^{n-1}}{\mathrm{rk}(\frak{E})}.
 \end{equation*}
\end{defin}

As in the ordinary case (see Kobayashi \cite{KOB} for more details) there is a notion of stability
for Higgs sheaves, which depends on the K\"ahler form $\omega$ and makes reference only to Higgs subsheaves.
Namely we have:

\begin{defin}
Let $(X,\omega)$ be a compact K\"ahler manifold of (complex) dimension $n$ and let $\frak{E}$ be a Higgs sheaf 
over $X$ of rank $r.$
 $\frak{E}$ is $\omega\text{-stable}$ (resp. $\omega\text{-semistable}$) if it is torsion-free
 and for any Higgs subsheaf $\frak{F}$ with 
 \linebreak$0<\mathrm{rk}(\frak{F})<\mathrm{rk}(\frak{E})$ one has the inequality
 $\mu(\frak{F})<\mu(\frak{E})$ (resp. $\mu(\frak{F})\leq\mu(\frak{E})$).
\end{defin}

Let $(X,\omega)$ be a compact K\"ahler manifold of (complex) dimension $n$ and let
let $\frak{E}=(E,\phi)$ be a Higgs bundle of rank $r$ over $X.$ Let $h$ be an Hermitian metric 
on this bundle.
Let $D_h=D'_{h}+D''$ be the Hermitian connection on $E.$ From section \ref{compatibility} 
we already know that 
this connection exists and is the unique connection compatible 
with the metric $h$ and the holomorphic structure of the bundle $E.$ (For more
details see \cite{KOB}). Here $D'_{h}$ and $D''=\mathrm{d''}_{E}$ are the components of type $(1,0)$ and $(0,1).$
Using this decomposition
of $D_h$ and the Higgs field $\phi,$ Simpson \cite{SIM} introduced a connection on $\frak{E}$ in the 
following way:
\begin{equation*}
 \mathcal{D}''=D''+\phi,\hspace{0.5cm}\mathcal{D}'_h=D'_h+\overline{\phi}_h,
\end{equation*}
where $\overline{\phi}_h$ is the usual adjoint of the Higgs field with respect to the hermitian structure
$h,$ and it is defined by the formula
\begin{equation*}
 h(\overline{\phi}_hs,s')=h(s,\phi s'),
\end{equation*}
where $s$ and $s'$ are sections of the Higgs bundle.

\begin{note}
 $\mathcal{D}'_h$ and $\mathcal{D}''$ are not of type $(1,0)$ and $(0,1).$
\end{note}

\begin{defin}
 The resulting connection $\mathcal{D}_h=\mathcal{D}'_h+\mathcal{D}''$ is called the Hitchin-Simpson
 connection. Clearly
 \begin{equation*}
  \mathcal{D}_h=D_h+\phi+\overline{\phi}_h
 \end{equation*}
depends on the Higgs field $\phi$ and there is an extra dependence on $h$ via $\overline{\phi}_h.$
\end{defin}

\begin{defin}
 The curvature of the Hitchin-Simpson connection is defined by 
 $\mathcal{R}_h=\mathcal{D}_h\circ\mathcal{D}_h$ and we say that the pair $(\frak{E},h)$ is Hermitian flat
 if this curvature vanishes.
\end{defin}

From the previous definition we immediately have
\begin{equation*}
  \mathcal{R}_h=(D_h+\phi+\overline{\phi}_h)\wedge(D_h+\phi+\overline{\phi}_h),
\end{equation*}
then using the decomposition $D_h=D'_h+D''$ and defining
\begin{equation*}
 [\phi,\overline{\phi}_h]=\phi\wedge\overline{\phi}_h+\overline{\phi}_h\wedge\phi
\end{equation*}
we obtain the following formula of the Hitchin-Simpson curvature in terms of the curvature of the 
Hermitian connection $D_h$
\begin{equation*}
 \mathcal{R}_h=R_h+{D'}_h(\phi)+D''(\overline{\phi}_h)+[\phi,\overline{\phi}_h].
\end{equation*}
In the previous formula, ${D'}_h(\phi)$ and $D''(\overline{\phi}_h)$ are the components of type 
$(2,0)$ and $(0,2),$
respectively, while the $(1,1)$ component is given by
\begin{equation*}
 \mathcal{R}_{h}^{1,1}=R_h+[\phi,\overline{\phi}_h].
\end{equation*}

We denote by $\text{Herm}(\frak{E})$ the space of Hermitian forms in $\frak{E}$ and by 
$\text{Herm}^{+}(\frak{E})$ the space of Hermitian structures (i.e., the positive definite Hermitian forms) in
$\frak{E}.$ For any Hermitian structure $h$ it is possible to identify $\text{Herm}(\frak{E})$ with the 
tangent space of $\text{Herm}^{+}(\frak{E})$ at the point $h$ (see Kobayashi \cite{KOB} for more details). Hence
\begin{equation}
\label{hermitiantangent}
 \text{Herm}(\frak{E})\cong T_h\text{Herm}^{+}(\frak{E}).
\end{equation}
If $v$ denotes an element in $\text{Herm}(\frak{E}),$ one defines the endomorphism $h^{-1}v$
by setting $s'\mapsto h^{-1}vs',$
where $h^{-1}vs'$ is the unique section of $E$ such that
\begin{equation*}
 v(s,s')=h(s,h^{-1}vs')\hspace{0.5cm}\text{ for all }s\in\Gamma(X,\frak{E}).
\end{equation*}
We can also define a Riemann structure in $\text{Herm}^{+}(\frak{E}).$ 
From \eqref{hermitiantangent}, for any $v,v'\in\text{Herm}(\frak{E})$ we define the inner product 
\begin{equation*}
 (v,v')_h=\int_{X}\mathrm{tr}(h^{-1}v\cdot h^{-1}v')\frac{\omega^{n}}{n!}
\end{equation*}

We have some natural properties associated with tensor products and direct sums. In particular we have 

\begin{prop}
 Let $\frak{E}_1$ and $\frak{E}_2$ two Higgs bundles with Higgs fields $\phi_1$ and $\phi_2$ respectively.
 Then
 \begin{enumerate}
  \item The pair $\frak{E}_1\otimes\frak{E}_2=(E_1\otimes E_2,\phi)$ is a Higgs bundle with Higgs field
        \linebreak$\phi=\phi_1\otimes I_2+I_1\otimes\phi_2.$
  \item If $pr_i:E_1\oplus E_2\longrightarrow E_i$ with $i=1,2$ denote the natural projections, then
        $\frak{E}_1\oplus\frak{E}_2=(E_1\oplus E_2,\phi)$ is a Higgs bundle with 
        $\phi=pr_{1}^{\ast}\phi_1+pr_{2}^{\ast}\phi_2.$
 \end{enumerate}
\end{prop}

Let $(X,\omega)$ be a compact K\"ahler manifold of (complex) dimension $n$ and let $\frak{E}=(E,\phi)$ be a Higgs 
bundle of rank $r$ over $X.$
In a similar way as in the ordinary case (see Kobayashi, \cite{KOB}) we have a notion of 
Hermitian-Yang-Mills metric 
structure for the Higgs bundle $\frak{E}.$
Let us consider the usual operator $\ast:A^{p,q}\longrightarrow A^{p+1,q+1}$ and the operator
$L:A^{p,q}\longrightarrow A^{p+1,q+1}$ defined by $L\varphi=\omega\wedge\varphi,$ where 
$\varphi\in A^{p,q}$ is a form on $X$ of type $(p,q).$ Then we define as usual the operator 
$\varLambda=\ast\circ L\ast^{-1}:A^{p,q}\longrightarrow A^{p-1,q-1}.$

\begin{defin}
 Consider now a Hermitian metric $h\in\text{Herm}^{+}(\frak{E})$ and let $\mathcal{R}_h$ be its 
 Hitchin-Simpson curvature. We can define the mean curvature of the Hitchin-Simpson connection 
 $\mathcal{D}_h$ just by contraction of this curvature with the operator $i\varLambda.$ In other words,
 \begin{equation}
 \label{curvature1}
  \mathcal{K}_h=i\varLambda\mathcal{R}_h,
 \end{equation}
 or equivalently 
 \begin{equation}
 \label{curvature2}
  in\mathcal{R}_h\wedge\omega^{n-1}=\mathcal{K}\omega^{n}.
 \end{equation}
\end{defin}

\begin{note}
 $\mathcal{K}_{h}$ is selfadjoint with respect to the Hermitian metric $h,$ i.e.,
 \linebreak$h(\mathcal{K}_{h}s,s')=h(s,\mathcal{K}_{h}s')$ for any section $s,s'\in\Gamma(X,E).$
\end{note}

\begin{defin}
 We define the operators
 \begin{equation*}
  \Box_0=i\varLambda\mathrm{d''}\mathrm{d}'\hspace{0.5cm}\text{ and }\hspace{0.5cm}
  \tilde{\Box}_{h}=i\varLambda\mathcal{D}''\mathcal{D}'_h.
 \end{equation*}
Note that operator $\tilde{\Box}_{h}$ depends on the K\"ahler form $\omega$ via the the action of the operator 
$\varLambda$ and also on the metric $h,$ while $\Box_{0}$ depends on the K\"ahler form $\omega$ via the 
application of $\varLambda.$
\end{defin}

\begin{prop}
 $\mathcal{K}_h\in A^{0}(\mathrm{End}(E)).$ In particular the $(2,0)$ and $(0,2)$ components of 
 $\mathcal{R}_h$ do not contribute to $\mathcal{K}_h.$
\end{prop}
\begin{proof}
 We already know that $\mathcal{R}_h\in A^{2}(\mathrm{End}(E))$ and that the operator 
 $i\varLambda$ kills the components of $(2,0)$ and $(0,2)$ type in $\mathcal{R}_h.$ Then, from \eqref{curvature1}
 we have
 \begin{equation*}
 \begin{split}
      \mathcal{K}_h&=i\varLambda\mathcal{R}_h=
  i\varLambda(R_h+[\phi,\overline{\phi}_h]+D'_h(\phi)+D''(\overline{\phi}_h))=\\
               &=
  i\varLambda(R_h+[\phi,\overline{\phi}_h])+i\varLambda D'_h(\phi)+i\varLambda D''(\overline{\phi}_h)=\\
               &=i\varLambda(R_h+[\phi,\overline{\phi}_h])=i\varLambda\mathcal{R}_{h}^{1,1}.
 \end{split}
 \end{equation*}
\end{proof}

\begin{defin}
 We say that a Hermitian metric $h\in\text{Herm}^{+}(\frak{E})$ is a weak Hermitian-Yang-Mills structure
 with factor $\gamma$ for $\frak{E}$ if
 \begin{equation*}
  \mathcal{K}_h=\gamma I_E,
 \end{equation*}
where $\gamma$ is a real valued function on $X$ and $I_E$ is the identity endomorphism on $E.$
If $\gamma=c$ is a real positive constant, we say that $h$ is a Hermitian-Yang-Mills structure.
\end{defin}

\begin{note}
 The mean curvature can be considered also a Hermitian form, by defining 
 \begin{equation}
  \label{curvaturehermitianform}
  \mathcal{K}_h(s,s')=h(s,\mathcal{K}_hs')
 \end{equation}
where $s,s'$ are section of the Higgs bundle $\frak{E}=(E,\phi).$
\end{note}

\section{Weak Hermitian-Yang-Mills structures: elementary results}

Let $(X,\omega)$ be a compact K\"ahler manifold of (complex) dimension $n.$
From the expression of the curvature of tensor product and direct sum of Higgs bundles over $X,$ 
we immediately have the following
\begin{prop}
 \begin{enumerate}
  \item If $h_1$ and $h_2$ are two weak Hermitian-Yang-Mills structures with factors $\gamma_1$ and
        $\gamma_2$ for Higgs bundles $\frak{E}_1$ and $\frak{E}_2$  over $X,$ then $h_1\otimes h_2$ is a weak 
        Hermitian-Yang-Mills structure with factor $\gamma_1+\gamma_2$ for the tensor product bundle
        $\frak{E}_1\otimes\frak{E}_2.$
  \item The metric $h_1\oplus h_2$ is a weak Hermitian-Yang-Mills structure with factor $\gamma$ for the 
        Whitney sum $\frak{E}_1\oplus\frak{E}_2$ if and only if both metrics $h_1$ and $h_2$ are
        weak Hermitian-Yang-Mills structures with the same factor $\gamma$ for $\frak{E}_1$ and 
        $\frak{E}_2.$ 
 \end{enumerate}
\end{prop}
\begin{proof}
\begin{enumerate}
 \item As in the classical case condider the formula for the curvature of the Hitchin-Simpson
       connection in a tensor product
       \begin{equation*}
          \mathcal{R}_{1\otimes2}=\mathcal{R}_1\otimes I_2+I_1\otimes\mathcal{R}_2.
       \end{equation*}
       Hence, taking the trace with respect to $\omega$ 
       (that is, applying the operator $i\varLambda$) we have the following expression involving
       the mean curvature of a tensor product
       \begin{equation*}
        \mathcal{K}_{1\otimes2}=\mathcal{K}_1\otimes I_2+I_1\otimes\mathcal{K}_2.
       \end{equation*}
 \item To prove (2) we use the following identity
       \begin{equation*}
        \mathcal{K}_{1\oplus2}=\mathcal{K}_1\oplus\mathcal{K}_2.
       \end{equation*}
\end{enumerate}
\end{proof}

From the previous result and the definition of the dual Higgs bundle we have the following
\begin{cor}
 Let $h\in\text{Herm}^{+}(\frak{E})$ be a (weak) Hermitian-Yang-Mills structure with factor $\gamma$ 
 for the Higgs bundle $\frak{E}$ over $X.$ Then
 \begin{enumerate}
  \item The induced metric on the tensor product 
        $\frak{E}^{\otimes p}\otimes\frak{E}^{\ast\otimes q}$ is a (weak) Hermitian-Yang-Mills structure
        with factor $(p-q)\gamma,$
  \item The induced Hermitian metric on $\bigwedge^{p}\frak{E}$ is a (weak) Hermitian-Yang-Mills structure 
        with factor $p\gamma$ for every $0\leq p\leq r=\mathrm{rk}(\frak{E}).$
 \end{enumerate}
\end{cor}

In general, if $h$ is a weak Hermitian-Yang-Mills structure with factor $\gamma,$ the slope 
$\mu(\frak{E})$ can be written in terms of $\gamma.$ In fact, we have

\begin{prop}
 Let $(X,\omega)$ be a compact K\"ahler manifold of (complex) dimension $n$ and let $\frak{E}=(E,\phi)$ be 
 a Higgs bundle of rank $r$ over $X.$ If \linebreak$h\in\text{Herm}^{+}(\frak{E})$ is a weak Hermitian-Yang-Mills
 structure with factor $\gamma,$ then
 \begin{equation}
 \label{gammaslope}
  \mu(\frak{E})=\frac{1}{2n\pi}\int_{X}\gamma\omega^{n}.
 \end{equation}
\end{prop}
\begin{proof}
 Let $\mathcal{R}_h$ be the Hitchin-Simpson curvature of the Hitchin-Simpson connection 
 $\mathcal{D}_h$ associated with the Hermitian connection $D_h.$ From \eqref{curvature2} we have the 
 identity
 \begin{equation*}
  in\mathcal{R}_h\wedge\omega^{n-1}=\mathcal{K}_h\omega^n.
 \end{equation*}
 Taking the trace we obtain
 \begin{equation*}
  in\mathrm{tr}\mathcal{R}_h\wedge\omega^{n-1}=\mathrm{tr}\mathcal{K}_h\omega^n.
 \end{equation*}
By hypothesis $h$ is a weak Hermitian-Yang-Mills structure with factor $\gamma.$ Integrating over $X$ we
obtain
\begin{equation*}
 \begin{split}
  \mu(\frak{E})&=\frac{1}{r}\int_{X}c_1(\frak{E})\wedge\omega^{n-1}=
  \frac{1}{r}\int_{X}-\frac{1}{2\pi i}\mathrm{tr}(\mathcal{R}_h)\wedge\omega^{n-1}=\\
               &=-\frac{1}{2\pi ir}\int_{X}\mathrm{tr}(\mathcal{R}_h)\wedge\omega^{n-1}=
  -\frac{1}{2\pi ir}\int_{X}\frac{1}{in}\mathrm{tr}(\mathcal{K}_h)\omega^n=\\
               &=\frac{1}{2n\pi r}\int_{X}\mathrm{tr}(\mathcal{K}_h)\omega^n=
  \frac{1}{2n\pi r}\int_{X}r\gamma\omega^n=\frac{1}{2n\pi }\int_{X}\gamma\omega^n.
 \end{split}
\end{equation*}
\end{proof}

\begin{lemma}
 Let $h$ be a weak Hermitian-Yang-Mills structure with factor $\gamma$ for $\frak{E}$ and let $a=a(x)$
 be a real positive definite function on $X,$ then $h'=ah$ is a weak Hermitian-Yang-Mills structure 
 with factor $\gamma'=\gamma+\Box_0(\ln a).$
\end{lemma}
\begin{proof}
 Clearly $h'$ defines another Hermitian metric on $\frak{E}.$ Since $h'$ is a conformal change of $h,$
 we have in particular $\overline{\phi}_{h'}=\overline{\phi}_{h},$ in fact, for every sections
 $s,s'$ of $\frak{E}$ we have
 \begin{equation*}
  \begin{split}
   h(\phi s,s')&=h(s,\overline{\phi}_{h}s')\\
   h'(\phi s,s')&=h'(s,\overline{\phi}_{h'}s').
  \end{split}
 \end{equation*}
Since $h'=ah,$ we have
\begin{equation*}
 h(\phi s,s')=h(s,\overline{\phi}_{h'}s'),
\end{equation*}
and this proves that $\overline{\phi}_{h'}=\overline{\phi}_{h}.$
Then from \eqref{curvature2}, $K'=K+\Box_0(\ln a)$ (see Kobayashi \cite{KOB} for the classical
case). Since taking the wedge product with $\omega^{n-1}$ kills the $(2,0)$ and $(0,2)$ components, we obtain
\begin{equation*}
 \begin{split}
  \mathcal{K}'\omega^n&=in\mathcal{R}'\wedge\omega^{n-1}=in(R'+[\phi,\overline{\phi}_{h'}])\wedge\omega^{n-1}=\\
                      &=in(R'+[\phi,\overline{\phi}_h])\wedge\omega^{n-1}=inR'\wedge\omega^{n-1}+
                      in[\phi,\overline{\phi}_h]\wedge\omega^{n-1}=\\
                      &=K'\omega^n+in[\phi,\overline{\phi}_h]\wedge^{n-1}=(K+\Box_0(\ln a)I_E)\omega^n+
                      in[\phi,\overline{\phi}_h]\wedge\omega^{n-1}=\\
                      &=in(R+[\phi,\overline{\phi}_h])\wedge\omega^{n-1}+\Box_0(\ln a)I_E\omega^n=\\
                      &=\mathcal{K}\omega^n+\Box_0(\ln a)I_E\omega^n=(\mathcal{K}+\Box_0(\ln a)I_E)\omega^n=
                      (\gamma+\Box_0(\ln a))I_E\omega^n.
 \end{split}
\end{equation*}
\end{proof}

\begin{lemma}
\label{rescalinglemma}
 If $h\in\text{Herm}^{+}(\frak{E})$ is a weak Hermitian-Yang-Mills structure with factor $\gamma,$ then there 
 exists a conformal change $h'=ah$ such that $h'$ is a Hermitian-Yang-Mills structure with constant factor $c,$
 given by
 \begin{equation}
  \label{rescale}
  c\int_{X}\omega^n=\int_{X}\gamma\omega^n.
 \end{equation}
Such a conformal change is unique up to homotety.
\end{lemma}
\begin{proof}
 Since $X$ is compact the integrals $\int_{X}\gamma\omega^n$ and $\int_{X}\omega^n$ are finite real numbers. 
 Then there exists a finite real number $c$ such that
 \begin{equation*}
  c\int_{X}\omega^n=\int_{X}\gamma\omega^n.
 \end{equation*}
Hence,
\begin{equation*}
  \int_{X}(c-\gamma)\omega^n=0.
 \end{equation*}
It is sufficient to prove that there exists a function $u$ satisying the equation
\begin{equation}
\label{firstpoisson}
 \Box_0u=c-\gamma.
\end{equation}
In fact, setting $h'=e^{u}h,$ from the previous Lemma it follows that $h'$ is a weak Hermitian-Yang-Mills 
structure with factor $\gamma'=\gamma+(c-\gamma)=c.$
Now from Hodge theory and Fredholm alternative Theorem we know that \eqref{firstpoisson} has a solution 
if and only if $c-\gamma$ is orthogonal to all $\Box_0\text{-harmonic}$ functions.
Since $X$ is a compact complex manifold, a function on $X$ is $\Box_0\text{-harmonic}$ if and only if
it is constant.
So \eqref{firstpoisson} has solution if and only if 
\begin{equation*}
 \int_{X}(c-\gamma)\omega^{n}=0.
\end{equation*}
But this equality always holds from the choice of the constant $c,$ and this completes the proof.
\end{proof}

From the previous Lemma we immediately see that if a Higgs bundle admits a weak Hermitian-Yang-Mills structure,
then it also admits, by an appropriate conformal change of the metric, a Hermitian-Yang-Mills structure. In 
particular, if the Higgs bundle has rank $1,$ we have the following

\begin{cor}
\label{always1}
Let $(X,\omega)$ be a compact K\"ahler manifold of (complex) dimension $n$ and 
 let $\frak{E}=(E,\phi)$ be a Higgs bundle of rank $1$ over $X.$ Then $\frak{E}$ admits a Hermitian-Yang-Mills
 structure.
\end{cor}
\begin{proof}
 Let $h\in\text{Herm}^{+}(\frak{E})$ be a Hermitian metric on $E,$ which always exists from a partition of
 unity argument. Since $\frak{E}$ has rank $1$ and since $\mathcal{K}_h$ is selfadjoint, 
 we have 
 \begin{equation*}
  \mathcal{K}_h=\gamma I_E,
 \end{equation*}
 where $\gamma$ is a positive real function on the manifold $X.$ Then the thesis comes from Lemma
 \ref{rescalinglemma}.
\end{proof}

\section{Approximate Hermitian-Yang-Mills metrics}

In this section we define approximate Hermitian-Yang-Mills metric structures on Higgs bundles, and study
some of their properties.

\begin{defin}
Let $(X,\omega)$ be a compact K\"ahler manifold of (complex) dimension $n$ and 
 let $\frak{E}=(E,\phi)$ be a Higgs bundle of rank $r$ over $X.$
 We define a positive real constant $c$ as
 \begin{equation}
 \label{cdefinition}
 c=\frac{2\pi\mu(\frak{E})}{(n-1)!\mathrm{Vol}(X)}.
\end{equation}
\end{defin}

\begin{defin}
Let $(X,\omega)$ be a compact K\"ahler manifold of (complex) dimension $n$ and
 let $\frak{E}=(E,\phi)$ be a Higgs bundle of rank $r$ over $X.$ Let $h$ be a Hermitian metric on $\frak{E}$ and 
 let $\mathcal{K}$ be its Hitchin-Simpson mean 
 curvature. We define the lenght of the endomorphism $\mathcal{K}-cI_E$ by the formula
 \begin{equation}
  \label{Klenght}
  |\mathcal{K}-cI_E|^2=\mathrm{tr}[(\mathcal{K}-cI_E)\cdot(\mathcal{K}-cI_E)].
 \end{equation}
Since $\mathcal{K}$ is selfadjoint with respect to the metric $h$ and $c$ is real, 
$|\mathcal{K}-cI_E|^2$ is a real function
on $X$ and $|\mathcal{K}-cI_E|^2\geq0.$
\end{defin}

\begin{defin}
 In the hypotheses of the previous definition, we introduce the following norms:
 \begin{equation}
  \label{Knorm}
  \begin{split}
   \|\mathcal{K}-cI_E\|_{L^1}&=\int_{X}|\mathcal{K}-cI_E|\frac{\omega^n}{n!}\\
   \|\mathcal{K}-cI_E\|_{L^2}^{2}&=\int_{X}|\mathcal{K}-cI_E|^2\frac{\omega^n}{n!}\\
   \|\mathcal{K}-cI_E\|_{L^{\infty}}&=\max_{X}|\mathcal{K}-cI_E|.
  \end{split}
 \end{equation}
\end{defin}

\begin{defin}
Let $(X,\omega)$ be a compact K\"ahler manifold of (complex) dimension $n$ and let $\frak{E}=(E,\phi)$ be a 
Higgs bundle 
of rank $r$ over $X.$
 We say that $\frak{E}$ admits an approximate
 Hermitian-Yang-Mills structure if for any $\epsilon>0$ there exists a metric $h_{\epsilon}$ such that
 \begin{equation*}
  \|\mathcal{K}_{h_{\epsilon}}-cI_E\|_{L^{\infty}}=\max_{X}|\mathcal{K}_{h_{\epsilon}}-cI_E|<\epsilon.
 \end{equation*}
Here $\mathcal{K}_{h_{\epsilon}}$ is the mean curvature of the Hitchin-Simpson connection associated with
$h_{\epsilon}.$
\end{defin}

This notion satisfies some simple properties with respect to tensor product an direct sums.

\begin{prop}
\label{approximatetensorproduct}
 Let $(X,\omega)$ be a compact K\"ahler manifold of (complex) dimension $n.$ If the Higgs bundles $\frak{E}_1$
 and $\frak{E}_2$ over $X$ admit 
 approximate Hermitian-Yang-Mills structures, so does their tensor product $\frak{E}_1\otimes\frak{E}_2.$
 Furthermore, if $\mu(\frak{E}_1)=\mu(\frak{E}_2),$ so does their Whitney sum.
\end{prop}
\begin{proof}
\begin{enumerate}
 \item Assume that $\frak{E}_1$ and $\frak{E}_2$ admit approximate Hermitian-Yang-Mills structures with factors $c_1$
 and $c_2,$ respectively. Let $\epsilon>0.$ Then there exist $h_1$ and $h_2$ such that
 \begin{equation*}
  \max_{X}|\mathcal{K}_1-c_1I_{E_1}|<\epsilon/2,\hspace{0.5cm}\max_{X}|\mathcal{K}_2-c_2I_{E_2}|<\epsilon/2,
 \end{equation*}
 where $\mathcal{K}_{1}$ and $\mathcal{K}_{2}$ are the mean curvature endomorphisms of the Hitchin-Simpson
 connection associated with $h_1$ and to $h_2,$ respectively. 
Defining $h=h_1\otimes h_2$ and setting $c=c_1+c_2,$ $I=I_{E_1}\otimes I_{E_2},$ the mean curvature
\begin{equation*}
 \mathcal{K}=\mathcal{K}_1\otimes I_{E_2}+I_{E_1}\otimes\mathcal{K}_2
\end{equation*}
satisfies the inequality
\begin{equation*}
\begin{split}
 \max_{X}|\mathcal{K}-cI|& = \max_{X}|\mathcal{K}_1\otimes I_{E_2}+I_{E_1}\otimes\mathcal{K}_2-
                                c_1I_{E_1}\otimes I_{E_2}-c_2I_{E_1}\otimes I_{E_2}|\leq\\
                         & \leq \max_{X}|\mathcal{K}_1\otimes I_{E_2}-c_1I_{E_1}\otimes I_{E_2}|+\\
                         &   +  \max_{X}|I_{E_1}\otimes\mathcal{K}_2-c_2I_{E_1}\otimes I_{E_2}|\leq\\
                         & \leq \max_{X}|\mathcal{K}_1-c_1I_{E_1}|+\max_{X}|\mathcal{K}_2-c_2I_{E_2}|<\\
                         & <\epsilon/2+\epsilon/2=\epsilon,
\end{split}
\end{equation*}
and this completes the proof.
\item If $\mu(\frak{E}_1)=\mu(\frak{E}_2),$ necessarly $c_1=c_2=c.$ Defining this time $h=h_1\oplus h_2,$ from 
$\mathcal{K}=\mathcal{K}_1\oplus\mathcal{K}_2$ we have
\begin{equation*}
\begin{split}
\max_{X}|\mathcal{K}-cI_{E_1\oplus E_2}|& = \max_{X}|\mathcal{K}_1\oplus\mathcal{K}_2-cI_{E_1\oplus E_2}|\leq\\
                                &  \leq  \max_{X}|\mathcal{K}_1-cI_{E_1}|+\max_{X}|\mathcal{K}_2-cI_{E_2}|=\\
                                &= \max_{X}|\mathcal{K}_1-c_1I_{E_1}|+\max_{X}|\mathcal{K}_2-c_2I_{E_2}|<\\
                                &< \epsilon/2+\epsilon/2=\epsilon.
\end{split}
\end{equation*}
\end{enumerate}
\end{proof}

\begin{cor}
 If $\frak{E}$ admits an approximate Hermitian-Yang-Mills stucture, so do the tensor product 
 $\frak{E}^{\otimes p}\otimes\frak{E}^{\ast\otimes q}$ and the exterior product bundle $\bigwedge^p\frak{E}$ 
 whenever $0\leq p\leq r=\mathrm{rk}(\frak{E}).$
\end{cor}

\begin{cor}
 \label{corBruOter}
 Let $(X,\omega)$ be a compact K\"ahler manifold of (complex) dimension $n$ and let $\frak{E}=(E,\phi)$ be a 
 Higgs bundle of rank $r$ over $X.$ If $\mathrm{deg}(\frak{E})<0,$ then $\frak{E}$ has no nonzero 
 $\phi\text{-invariant}$ sections.
\end{cor}
\begin{proof}
 Suppose $\frak{E}$ admits an approximate Hermitian-Yang-Mills structure, then for each $\epsilon$ there exists 
 a Hermitian metric $h_{\epsilon}$ on $\frak{E}$ such that
 \begin{equation*}
  \max_{X}|\mathcal{K}_{h_{\epsilon}}-cI_E|<\epsilon,
 \end{equation*}
 where $\mathcal{K}_{h_{\epsilon}}$ is the Hitchin-Simpson mean curvature endomorphism associated with 
 $h_{\epsilon}.$

This implies that the mean curvature, seen as a Hermitian form, satisfies the following inequality
\begin{equation*}
 -\epsilon h_{\epsilon}<\mathcal{K}_{h_{\epsilon}}-cI_E<\epsilon h_{\epsilon}.
\end{equation*}
If we assume that $\mathrm{deg}(\frak{E})<0,$ then $c<0$ in the above inequality and then for some sufficient 
small $\epsilon$ the mean curvature $\mathcal{K}_{h_{\epsilon}}$ is negative definite. Then 
the result follows from a vanishing Theorem for Higgs bundles. (See Theorem $3.2$ in \cite{CAS} for details).
\end{proof}

\section{The Donaldson functional}

Let $(X,\omega)$ be a compact K\"ahler manifold of (complex) dimension $n$ and
let $\frak{E}=(E,\phi)$ a be Higgs bundle of rank $r$ $X.$ Let $h,k\in\text{Herm}^{+}(\frak{E}).$ Since 
$\text{Herm}^{+}(\frak{E})$ is a connected Riemannian manifold (see Chapter VI in \cite{KOB} for more details),
we connect $h$ and $k$ by a curve $h_t,0\leq t\leq1,$ in $\text{Herm}^{+}(\frak{E})$ so that
$h_0=k$ and $h_1=h.$ We set
\begin{equation*}
 \mathcal{Q}_1=\ln(\det(k^{-1}h)),\hspace{0.5cm}\mathcal{Q}_2=i\int_{0}^{1}
 \mathrm{tr}(v_t\cdot\mathcal{R}_t)\mathrm{d}t.
\end{equation*}
where $v_t=h_{t}^{-1}\partial_th_t$ and $\mathcal{R}_t$ denotes the curvature of the Hitchin-Simpson connection
associated with $h_t.$ Since $h$ and $k$ are Hermitian structures, $\det(k^{-1}h)$ is a strictly positive real 
function.

Notice that $\mathcal{Q}_1(h,k)$ does not involve the curve $h_t.$ On the other hand to define 
$\mathcal{Q}_2(h,k)$ we use explicitly the curve $h_t.$

\begin{defin}
Let $(X,\omega)$ be a compact K\"ahler manifold of (complex) dimension $n$ and 
 let $\frak{E}=(E,\phi)$ be a Higgs bundle of rank $r$ over $X.$
Let \linebreak$h,k\in\text{Herm}^{+}(\frak{E}).$
We define the Donaldson functional by
\begin{equation}
 \label{donaldsonfunctional}
 \mathcal{L}(h,k)=\int_{X}\left[\mathcal{Q}_2(h,k)-\frac{c}{n}\mathcal{Q}_1(h,k)\omega\right]
 \wedge\frac{\omega^{n-1}}{(n-1)!},
\end{equation}
where $c$ is, as usual, the constant given by
\begin{equation*}
 c=\frac{2\pi\mu(\frak{E})}{(n-1)!\mathrm{Vol}(X)}.
\end{equation*}
\end{defin}

\begin{note}
 Notice that the $(2,0)$ and $(0,2)$ components of $\mathcal{R}_t$ do not contribute to $\mathcal{L}(h,k).$
 In fact, if $\mathcal{R}_h$ is the curvature of the Hitchin-Simpson connection associated with the Hermitian 
 metric $h,$ from $D'_{h}(\phi)\wedge\omega^{n-1}=D''(\overline{\phi_{h}})\wedge\omega^{n-1}=0$ we have
 \begin{equation*}
 \begin{split}
  \mathcal{R}_h\wedge\omega^{n-1}&=(R_h+D'_h(\phi)+D''(\overline{\phi}_h)+[\phi,\overline{\phi}_h])
  \wedge\omega^{n-1}\\
                                 &=(R_h+[\phi,\overline{\phi}_h])\wedge\omega^{n-1}=
                                 \mathcal{R}_{h}^{1,1}\wedge\omega^{n-1}.
 \end{split}
 \end{equation*}
\end{note}

We want to prove that the Donaldson functional does not depend on the curve joining the metrics $h$ an $k.$
First of all we prove the following result

\begin{lemma}
\label{holomorphicstokes}
Let $(X,\omega)$ be a compact K\"ahler manifold of (complex) dimension $n.$ Let 
 $\eta\in\mathrm{d}'A^{0,1}+\mathrm{d''}A^{1,0},$ one has
 \begin{equation*}
  \int_{X}\eta\wedge\omega^{n-1}=0.
 \end{equation*}
\end{lemma}
\begin{proof}
 We write $\eta=\mathrm{d}(\alpha+\beta)=\mathrm{d}'\alpha+\mathrm{d''}\beta$ where $\alpha\in A^{1,0}$ and
 $\beta\in A^{0,1}.$ Since $X$ is K\"ahler, i.e., $\mathrm{d}\omega=0,$ we have
 \begin{equation*}
  \mathrm{d}[(\alpha+\beta)\wedge\omega^{n-1}]=[\mathrm{d}(\alpha+\beta)]\wedge\omega^{n-1}=
  \eta\wedge\omega^{n-1}.
 \end{equation*}
From Stokes' Theorem, since $\partial X=\emptyset$ we conclude
\begin{equation*}
 \int_{X}\eta\wedge\omega^{n-1}=\int_{X}\mathrm{d}[(\alpha+\beta)\wedge\omega^{n-1}]=
 \int_{\partial X}(\alpha+\beta)\wedge\omega^{n-1}=0.
\end{equation*}
\end{proof}

The following Lemma and the subsequent Propositions are straightforward generalizations of a result of Kobayashi
(see \cite{KOB}, Chapter VI) to the Higgs case. Part of the proof is similar to the proof presented in
\cite{KOB}. However, some differences arise because of the term involving the commutator of the Higgs field
in the Hitchin-Simpson curvature.

\begin{lemma}
 Let $h_t$ (for $a\leq t\leq b$) be any differentiable curve in $\text{Herm}^{+}(\frak{E})$ and $k$ any fixed
 Hermitian structure of $\frak{E}.$ Then the $(1,1)$ component of
 \begin{equation*}
  i\int_{0}^{1}\mathrm{tr}(v_t\mathcal{R}_t)\mathrm{d}t+\mathcal{Q}_2(h_a,k)-\mathcal{Q}_2(h_b,k)
 \end{equation*}
is an element in $\mathrm{d}'A^{0,1}+\mathrm{d''}A^{1,0}.$
\end{lemma}
\begin{proof}
 Following \cite{KOB}, we consider the domain $\Delta$ in $\mathbb{R}^2$ defined by
 \begin{equation*}
  \Delta=\{(t,s)|a\leq t\leq b,0\leq s\leq1\},
 \end{equation*}
and let $h:\Delta\longrightarrow\text{Herm}^{+}(\frak{E})$ be a smooth mapping such that $h(t,0)=k$ and
$h(t,1)=h_t$ for $a\leq t\leq b.$ Let $h(a,s)$ and $h(b,s)$ line segments from $k$ to $h_a$ and from
$h_b$ to $k,$ respectively. There is a simple expression for $h$ given by
$h(t,s)=sh_t+(1-s)k.$
Define the endomorphisms $u=h^{-1}\partial_sh,$ $v=h^{-1}\partial_th$ and we put
\begin{equation*}
 \mathcal{R}=\mathrm{d''}(h^{-1}\mathrm{d}'h)+[\phi,\overline{\phi}_h]
\end{equation*}
and
\begin{equation*}
 \Psi=i\mathrm{tr}[h^{-1}\tilde{\mathrm{d}}h\mathcal{R}],
\end{equation*}
where $\tilde{\mathrm{d}}=\partial_sh\mathrm{d}s+\partial_th\mathrm{d}t$ is the exterior 
differential of the smooth mapping $h$ in the domain $\Delta.$
Hence $\Psi$ can be written in the form
\begin{equation}
\label{psistokes}
 \Psi=i\mathrm{tr}[(u\mathrm{d}s+v\mathrm{d}t)\mathcal{R}].
\end{equation}
Applying Stokes' Theorem to $\Psi$ (which is considered here as a $1\text{-form}$ in the domain $\Delta$) we get
\begin{equation*}
 \int_{\Delta}\tilde{\mathrm{d}}\Psi=\int_{\partial\Delta}\Psi.
\end{equation*}
The right hand side of the above expression can be computed straightforwardly from definition.
In fact, after a short computation we obtain
\begin{equation*}
 \begin{split}
  \int_{\partial\Delta}\Psi&=
   -\int_{t=a}^{t=b}\left.\Psi\right|_{s=0}+\int_{s=0}^{s=1}\left.\Psi\right|_{t=a}+
   \int_{t=a}^{t=b}\left.\Psi\right|_{s=1}-\int_{s=0}^{s=1}\left.\Psi\right|_{t=b}=\\
                          &=i\int_{a}^{b}\mathrm{tr}(v_t\mathcal{R}_t)\mathrm{d}t+
                          \mathcal{Q}_2(h_a,k)-\mathcal{Q}_2(h_b,k).
 \end{split}
\end{equation*}
Therefore, we need to show that the left hand side of \eqref{psistokes} is an element of 
$\mathrm{d}'A^{0,1}+\mathrm{d''}A^{1,0},$ and hence, it suffices to show that
$\tilde{\mathrm{d}}\Psi\in\mathrm{d}'A^{0,1}+\mathrm{d''}A^{1,0}.$
Now, from the definition of $\Psi$ we have
\begin{equation*}
 \begin{split}
  \tilde{\mathrm{d}}\Psi&=i\mathrm{tr}[\tilde{\mathrm{d}}(u\mathrm{d}s+v\mathrm{d}t)\mathcal{R}-
                                    (u\mathrm{d}s+v\mathrm{d}t)\tilde{\mathrm{d}}\mathcal{R}]=\\
                        &=i\mathrm{tr}[(\partial_sv-\partial_tu)\mathcal{R}-u\partial_t\mathcal{R}+
                           v\partial_s\mathcal{R}]\mathrm{d}s\wedge\mathrm{d}t.
 \end{split}
\end{equation*}
On the other hand, a simple computation shows that
\begin{equation*}
 \begin{split}
  \partial_tu&=-vu+h^{-1}\partial_t\partial_sh,\hspace{0.5cm}\partial_sv=-uv+h^{-1}\partial_s\partial_th,\\
  \partial_t\mathcal{R}&=\mathrm{d''}D'v+[\phi,\partial_t\overline{\phi}_h],\hspace{0.5cm}
  \partial_t\mathcal{R}=\mathrm{d''}D'u+[\phi,\partial_s\overline{\phi}_h].
 \end{split}
\end{equation*}
Replacing these expressions in the formula for $\tilde{\mathrm{d}}\Psi$ and writing 
$\mathcal{R}=R+[\phi,\overline{\phi}_h]$ (since the  $(2,0)$ and $(0,2)$ components do not contribute), we have
\begin{equation*}
 \begin{split}
  \tilde{\mathrm{d}}\Psi&=i\mathrm{tr}[(vu-uv)R-u\mathrm{d''}D'v+v\mathrm{d''}D'u]\mathrm{d}s\wedge\mathrm{d}t+\\
                        &+i\mathrm{tr}[v[\phi,\partial_s\overline{\phi}_h]-u[\phi,\partial_t\overline{\phi}_h]]+
                          (vu-uv)[\phi,\overline{\phi}_h]\mathrm{d}s\wedge\mathrm{d}t.
 \end{split}
\end{equation*}
The first trace in the expression above does not depend on the Higgs field $\phi$ (in fact, it is the same 
expression that it is found in \cite{KOB} for the classical case). Then we only have to show that the second trace
is identically zero.

First of all, we need explicit expressions for $\partial_t\overline{\phi}_h$ and $\partial_s\overline{\phi}_h.$
Omitting the parameter $t$ for simplicity, from \cite{SIM} we know that 
\begin{equation*}
 \overline{\phi}_{h_s+\delta s}=u_{0}^{-1}\overline{\phi}_{h_s}u_0=\overline{\phi}_{h_s}+
 u_{0}^{-1}[\overline{\phi}_{h_s},u_0]
\end{equation*}
where $u_0$ is a selfadjoint endomorphism such that $h_{s+\delta s}=h_s+u_0.$ Now
\begin{equation*}
 h_{s+\delta s}=h_s+\partial_sh_s+\mathcal{O}(\delta s^2)
\end{equation*}
and hence, at the first order in $\delta s,$ we obtain $u_0=1+\delta s$ and consequently 
$\partial_s\overline{\phi}_h=[\overline{\phi}_h,u].$ In a similar way we obtain the formula
$\partial_t\overline{\phi}_h=[\overline{\phi}_h,v].$ Therefore, using these relations, the Jacobi identity and 
the cyclic property of the trace, we see that the second trace is identically zero. On the other hand, the term
involving the curvature $R$ can be written in terms of $u,$ $v$ and their covariant derivatives. So, finally, we 
get
\begin{equation*}
 \tilde{\mathrm{d}}\Psi=-i\mathrm{tr}[vD'\mathrm{d''}u+u\mathrm{d''}D'v]\mathrm{d}s\wedge\mathrm{d}t.
\end{equation*}
As it is shown in \cite{KOB} in the classical case, defining the $(0,1)\text{-form}$
$\alpha=i\mathrm{tr}[v\mathrm{d''}u]$ we finally obtain
\begin{equation*}
 \tilde{\mathrm{d}}\Psi=-[\mathrm{d}'\alpha+\mathrm{d''}\overline{\alpha}+
                          i\mathrm{d''}\mathrm{d'}\mathrm{tr}(vu)]
 \mathrm{d}s\wedge\mathrm{d}t
\end{equation*}
and hence $\tilde{\mathrm{d}}\Psi$ is an element of $\mathrm{d}'A^{0,1}+\mathrm{d''}A^{1,0}.$
\end{proof}

As a consequence of the above Lemma we have an important result for piecewise differentiable closed curves.
Namely, we have

\begin{prop}
 Let $h_t,\alpha\leq t\leq\beta,$ be a piecewise differentiable closed curve in $\text{Herm}^{+}(\frak{E}).$
 Then
 \begin{equation*}
  i\int_{\alpha}^{\beta}\mathrm{tr}(v_t\cdot\mathcal{R}_{t}^{1,1})\mathrm{d}t=0\hspace{0.5cm}\text{ mod }
  \mathrm{d}'A^{0,1}+\mathrm{d''}A^{1,0}.
 \end{equation*}
\end{prop}
\begin{proof}
 Let $\alpha=a_0<a_1<\cdots<a_p=\beta$ be the values of $t$ where the curve $h_t$ is not differentiable. Now take
 a fixed metric $k$ in $\text{Herm}^{+}(\frak{E}).$ We have
 \begin{equation*}
  i\int_{\alpha}^{\beta}\mathrm{tr}(v_t\cdot\mathcal{R}_{t}^{1,1})\mathrm{d}t=
  \sum_{j=1}^{p} i\left(\int_{a_{j-1}}^{a_j}\mathrm{tr}(v_t\cdot\mathcal{R}_{t}^{1,1})\mathrm{d}t\right).
 \end{equation*}
From the previous Lemma, for each $j=1,\ldots,p,$ we have 
\begin{equation*}
 \int_{a_{j-1}}^{a_j}\mathrm{tr}(v_t\cdot\mathcal{R}_{t}^{1,1})\mathrm{d}t=0\hspace{0.5cm}\text{ mod }
 \mathrm{d}'A^{0,1}+\mathrm{d''}A^{1,0},
\end{equation*}
and this completes the proof.
\end{proof}

\begin{cor}
 The Donaldson functional $\mathcal{L}(h,k)$ does not depend on the curve joining the Hermitian metrics $h$ and
 $k$ in $\text{Herm}^{+}(\frak{E}).$
\end{cor}
\begin{proof}
 Let $\gamma_1$ and $\gamma_2$ be two differentiable curves from $h$ to $k,$ and let 
 $\mathcal{L}_{\gamma_1}(h,k)$ and $\mathcal{L}_{\gamma_2}(h,k)$ be the Donaldson functionals computed along
 the curves $\gamma_1$ and $\gamma_2,$ respectively. We have to show that
 \begin{equation*}
  \mathcal{L}_{\gamma_1}(h,k)=\mathcal{L}_{\gamma_2}(h,k).
 \end{equation*}
 From the definition of Donaldson functional we have
\begin{equation*}
 \begin{split}
  \mathcal{L}_{\gamma_1}(h,k)-\mathcal{L}_{\gamma_2}(h,k)&=
  \int_{X}\left[\mathcal{Q}_{2}^{\gamma_1}(h,k)-\frac{c}{n}\mathcal{Q}_{1}^{\gamma_1}(h,k)\omega\right]
  \wedge\frac{\omega^{n-1}}{(n-1)!}+\\
                                                         &-
  \int_{X}\left[\mathcal{Q}_{2}^{\gamma_2}(h,k)-\frac{c}{n}\mathcal{Q}_{1}^{\gamma_2}(h,k)\omega\right]
  \wedge\frac{\omega^{n-1}}{(n-1)!}=\\                                                
                                                         &=
  \int_{X}\left[(\mathcal{Q}_{2}^{\gamma_1}-\mathcal{Q}_{2}^{\gamma_2})-
  \frac{c}{n}(\mathcal{Q}_{1}^{\gamma_1}-\mathcal{Q}_{1}^{\gamma_2})\omega\right]
  \wedge\frac{\omega^{n-1}}{(n-1)!}.
 \end{split}
\end{equation*}
Since 
\begin{equation*}
\mathcal{Q}_{1}^{\gamma_1}(h,k)=\mathcal{Q}_{1}^{\gamma_2}(h,k)=\ln(\det(k^{-1}h)),
\end{equation*}
we have
\begin{equation*}
\begin{split}
  \mathcal{L}_{\gamma_1}(h,k)-\mathcal{L}_{\gamma_2}(h,k)&=\int_{X}[(\mathcal{Q}_{2}^{\gamma_1}(h,k)-
 \mathcal{Q}_{2}^{\gamma_2}(h,k))-]\wedge\frac{\omega^{n-1}}{(n-1)!}=\\
							   &=\int_{X}\left[
i\int_{\gamma_1}\mathrm{tr}(v_t\cdot\mathcal{R}_t)\mathrm{d}t-i\int_{\gamma_2}\mathrm{tr}
(v_t\cdot\mathcal{R}_t)\mathrm{d}t\right]\wedge\frac{\omega^{n-1}}{(n-1)!}=\\
&=
\int_{X}\left[i\int_{\gamma_1-\gamma_2}\mathrm{tr}(v_t\cdot\mathcal{R}_t)\mathrm{d}t\right]
\wedge\frac{\omega^{n-1}}{(n-1)!}.
\end{split}
\end{equation*}
But $\gamma_1-\gamma_2$ is a piecewise differentiable closed curve in $\text{Herm}^{+}(\frak{E}),$ then by the
previous Proposition 
\begin{equation*}
i\int_{\gamma_1-\gamma_2}\mathrm{tr}(v_t\cdot\mathcal{R}_t)=0\hspace{0.5cm}\text{ mod }
\mathrm{d}'A^{0,1}+\mathrm{d''}A^{1,0},
\end{equation*}
so from Lemma \ref{holomorphicstokes} we finally have
\begin{equation*}
 \mathcal{L}_{\gamma_1}(h,k)-\mathcal{L}_{\gamma_2}(h,k)=
 \int_{X}\left[i\int_{\gamma_1-\gamma_2}\mathrm{tr}(v_t\cdot\mathcal{R}_t)\mathrm{d}t\right]\wedge
 \frac{\omega^{n-1}}{(n-1)!}=0,
\end{equation*}
and this completes the proof.
\end{proof}

\begin{prop}
 For any Hermitian metric $h\in\text{Herm}^{+}(\frak{E})$ and any real constant $a>0,$ the Donaldson functional
 satisfies $\mathcal{L}(h,ah)=0.$
\end{prop}
\begin{proof}
 Let $r=\mathrm{rk}(\frak{E})$ be the rank of the Higgs bundle $\frak{E}=(E,\phi).$ Clearly
 \begin{equation*}
  \mathcal{Q}_1(h,ah)=\ln\det[(ah)^{-1}h]=-r\ln a.
 \end{equation*}
Now, let us consider the curve $h_t=e^{\ln a(1-t)}h$ from $ah$ to $h.$ For this curve 
$v_t=h_{t}^{-1}\partial_th_t=-\ln aI_E$ and
\begin{equation*}
 \mathcal{R}_{t}^{1,1}=\mathrm{d''}(h_{t}^{-1}\mathrm{d}'h_t)+[\phi,\overline{\phi}_t]=
 \mathrm{d''}(h^{-1}\mathrm{d}'h)+[\phi,\overline{\phi}_t],
\end{equation*}
where $\overline{\phi}_t=\overline{\phi}_{h_t}$ is just an abbreviation. Therefore the $(1,1)$ component of
$\mathcal{Q}_2(h,ah)$ becomes
\begin{equation*}
 \mathcal{Q}_{2}^{1,1}(h,ah)=i\int_{0}^{1}\mathrm{tr}(v_t\cdot\mathcal{R}_{t}^{1,1})\mathrm{d}t=
 i\int_{0}^{1}\mathrm{tr}[-\ln a(R+[\phi,\overline{\phi}_t])]\mathrm{d}t=-i(\ln a)\mathrm{tr}R.
\end{equation*}
and hence, from the above formula we obtain
\begin{equation*}
 \begin{split}
  \mathcal{L}(h,ah)&=\int_{X}\left[\mathcal{Q}_{2}(h,ah)-\frac{c}{n}\mathcal{Q}_{1}(h,ah)\right]
  \wedge\frac{\omega^{n-1}}{(n-1)!}=\\
  &=\int_{X}\left[\mathcal{Q}_{2}^{1,1}(h,ah)-\frac{c}{n}\mathcal{Q}_{1}(h,ah)\right]
  \wedge\frac{\omega^{n-1}}{(n-1)!}=\\
  &=\int_{X}\left[-i(\ln{a})\mathrm{tr}R+\frac{c}{n}r\ln{a}\right]\wedge\frac{\omega^{n-1}}{(n-1)!}=\\
  &=-\frac{in\ln{a}}{n!}\int_{X}\mathrm{tr}R\wedge\omega^{n-1}+cr(\ln{a})\mathrm{Vol}(X)=\\
  &=-\frac{2\pi r}{(n-1)!}\ln{a}\frac{1}{r}\int_{X}\frac{i}{2\pi}\mathrm{tr}R\wedge\omega^{n-1}+
  cr(\ln{a})\mathrm{Vol}(X)=\\
  &=-\frac{2\pi r\ln{a}}{(n-1)!}\frac{1}{r}\int_{X}c_1(\frak{E})\wedge\omega^{n-1}+
  cr(\ln{a})\mathrm{Vol}(X)=\\
  &=-\frac{2\pi r\ln{a}}{(n-1)!}\mu(\frak{E})+\frac{2\pi\mu(\frak{E})}{(n-1)!\mathrm{Vol}(X)}
  r\ln{a}\mathrm{Vol}(X)=0.
 \end{split}
\end{equation*}
\end{proof}

\begin{lemma}
 For any differentiable curve $h_t$ in in the Riemannian manifold $\text{Herm}^{+}(\frak{E})$ and any fixed 
 point $k\in\text{Herm}^{+}(\frak{E})$ we have
 \begin{equation*}
  \begin{split}
   \partial_t\mathcal{Q}_1(h_t,k)&=\mathrm{tr}(v_t),\\
   \partial_t\mathcal{Q}_{2}^{1,1}(h_t,k)&=i\mathrm{tr}(v_t\cdot\mathcal{R}_{t}^{1,1})\hspace{0.5cm}\text{ mod }
   \mathrm{d}'A^{0,1}+\mathrm{d''}A^{1,0}.
  \end{split}
 \end{equation*}
\end{lemma}
\begin{proof}
\begin{enumerate}
 \item Since $k$ does not depend on $t,$ we get
       \begin{equation}
      \partial_t\mathcal{Q}_1(h_t,k)=\partial_t\ln(\det k^{-1})+\partial_t\ln(\det h_t)=\partial_t\ln(\det h_t)=
       \mathrm{tr}(v_t).
       \end{equation}
 \item It suffices to consider $b$ as a variable in the equation
       \begin{equation*}
        i\int_{0}^{1}\mathrm{tr}(v_t\cdot\mathcal{R}_t)\mathrm{d}t+\mathcal{Q}_2(h_a,k)-\mathcal{Q}_2(h_b,k)=0
        \hspace{0.5cm}\text{ mod }\mathrm{d}'A^{0,1}+\mathrm{d''}A^{1,0},
       \end{equation*}
       and differentiate this expression with respect to $b.$ 
\end{enumerate}
\end{proof}

Let $h$ be a metric in $\text{Herm}^{+}(\frak{E})$ and let $\mathcal{K}$ be the mean curvature of the Hitchin-
Simpson connection associated with $h.$ Notice that the endomorphism 
$\mathcal{K}\in A^0(\mathrm{End}(E))$ can 
be written as $\mathcal{K}=h^{-1}\mathcal{K}(\cdot,\cdot),$ where $\mathcal{K}(\cdot,\cdot)$ denotes
the mean curvature as a Hermitian form.

\begin{teo}
Let $(X,\omega)$ be a compact K\"ahler manifold of (complex) dimension $n$ and
 let $\frak{E}=(E,\phi)$ be a Higgs bundle of rank $r$ over $X.$ Let $\text{Herm}^{+}(\frak{E})$ be the 
 Riemannian manifold of the Hermitian
 metric on $E,$ and let $k$ be a fixed element in $\text{Herm}^{+}(\frak{E}).$ Then, $h$ is a critical point of
 $\mathcal{L},$ i.e., a critical point of the function $\mathcal{L}(\ast,k)\longrightarrow\mathbb{R},$ if and
 only if $\mathcal{K}-ch=0,$ if and only if $h$ is an Hermitian-Yang-Mills structure for $\frak{E}.$
\end{teo} 
\begin{proof}
 By using the above Lemma, from $in\mathcal{R}_{t}^{1,1}\wedge\omega^{n-1}=\mathcal{K}_t\omega^n$
 we have the following formula for the derivative with respect to $t$ of the Donaldson functional
 \begin{equation*}
  \begin{split}
   \frac{\mathrm{d}}{\mathrm{d}t}\mathcal{L}(h_t,k)&=
   \frac{\mathrm{d}}{\mathrm{d}t}\int_{X}\left[\mathcal{Q}_2(h_t,k)-\frac{c}{n}\mathcal{Q}_1(h_t,k)\omega\right]
 \wedge\frac{\omega^{n-1}}{(n-1)!}=\\
 &=\int_{X}\left[\partial_t\mathcal{Q}_2(h_t,k)-\frac{c}{n}\partial_t\mathcal{Q}_1(h_t,k)\omega\right]
 \wedge\frac{\omega^{n-1}}{(n-1)!}=\\
 &=\int_{X}\left[i\mathrm{tr}(v_t\cdot\mathcal{R}_{t}^{1,1})-\frac{c}{n}\mathrm{tr}(v_t)\omega\right]
 \frac{\omega^{n-1}}{(n-1)!}=\\
 &=\int_{X}[\mathrm{tr}(v_t\cdot\mathcal{K}_t)-c\mathrm{tr}(v_t)]\frac{\omega^n}{n!}=\\
 &=\int_{X}\mathrm{tr}[(\mathcal{K}_t-cI_E)v_t]\frac{\omega^n}{n!}.
  \end{split}
\end{equation*}
We consider the endomorphism $\mathcal{K}_t$ as a Hermitian form by defining
\linebreak$\mathcal{K}_t(s,s')=h_t(s,\mathcal{K}_ts'),$ where $s$ and $s'$ are section of the Higgs bundle 
$\frak{E}=(E,\phi).$ 
Since $v_t=h_{t}^{-1}\partial_{t}h_{t},$ 
for any fixed Hermitian metric \linebreak$k\in\text{Herm}^{+}(\frak{E})$
and any differentiable curve $h_t$ in $\text{Herm}^{+}(\frak{E})$ we obtain
\begin{equation*}
 \frac{\mathrm{d}}{\mathrm{d}t}\mathcal{L}(h_t,k)=(\mathcal{K}_t-ch_t,\partial_th_t),
\end{equation*}
where $\mathcal{K}_t$ is considered here as a form and $(\cdot,\cdot)$ is the inner product in the Riemannian
manifold $\text{Herm}^{+}(\frak{E}).$ For each $t,$ we can consider $\partial_th_t\in\text{Herm}(\frak{E})$ as
a tangent vector of $\text{Herm}^{+}(\frak{E})$ at $h_t.$ (See \cite{KOB}, Chapter VI for more details).
Therefore, the differential $\mathrm{d}\mathcal{L}$ of the functional evaluated at $\partial_th_t$
is given by
\begin{equation*}
 \mathrm{d}\mathcal{L}(\partial_{t}h_t,k)=\frac{\mathrm{d}}{\mathrm{d}t}\mathcal{L}(h_t,k).
\end{equation*}
Then, the gradient of $\mathcal{L}(\ast,k)$ is given by $\nabla\mathcal{L}=\mathcal{K}-ch,$ and this
completes the proof. Notice that here
$\mathcal{K}$ still denotes the mean curvature as an Hermitian form.
\end{proof}

Now, using the decomposition $\mathcal{D}=\mathcal{D}'_h+\mathcal{D}'',$ we show that all critical points of 
$\mathcal{L}$ correspond to an absolute minimum.

\begin{teo}
\label{teorema47}
 Let $k$ be a fixed Hermitian structure of the Higgs bundle $\frak{E}$ and let $h_0$ be a critical point of 
 $\mathcal{L}(h,k).$ The Donaldson functional attains an absolute minimum at $h_0.$
\end{teo}
\begin{proof}
 The second derivative of $\mathcal{L}$ is
 \begin{equation*}
  \begin{split}
   \partial_{t}^{2}\mathcal{L}(h_t,k)&=\partial_t\int_{X}\mathrm{tr}[(\mathcal{K}_t-cI_E)v_t]\frac{\omega^n}{n!}=\\
    &=\int_{X}\mathrm{tr}[\partial\mathcal{K}_t\cdot v_t+(\mathcal{K}_t-cI_E)\partial_tv_t]\frac{\omega^n}{n!}.
  \end{split}
 \end{equation*}
Here $\mathcal{K}_t$ is an element of $A^0(\mathrm{End}(E)).$
Since $h_0$ is a critical point of the Donaldson functional, $\mathcal{K}_t-cI_E=0$ at $t=0,$ hence
\begin{equation*}
 \left.\partial_{t}^{2}\mathcal{L}(h_t,k)\right|_{t=0}=
 \int_{X}\left.\mathrm{tr}(\partial_t\mathcal{K}_t\cdot v_t)\frac{\omega^n}{n!}\right|_{t=0}.
\end{equation*}
On the oher hand, $\partial_t\mathcal{K}_t$ can be written in terms of $v_t.$ In fact, we have
\begin{equation*}
 \begin{split}
  \mathcal{D}''\mathcal{D}'_{h_t}v_t&=\mathcal{D}''(D'_{h_t}v_t+[\overline{\phi}_{h_t},v_t])=\\
  &=D''D'_{h_t}v_t+[\phi,D'_{h_t}v_t]+D''[\overline{\phi}_{h_t},v_t]+[\phi,[\overline{\phi}_{h_t},v_t]].
 \end{split}
\end{equation*}
Since $\partial_t\phi_{h_t}=[\overline{\phi}_{h_t},v_t]$ we get
\begin{equation*}
 \partial_t\mathcal{R}_{t}^{1,1}=\partial_tR_t+[\phi,\partial_t\overline{\phi}_{h_t}]=
 D''D'_{h_t}v_t+[\phi,[\overline{\phi}_{h_t},v_t]].
\end{equation*}
Therefore, since the operator $i\varLambda$ kills the $(2,0)$ and $(0,2)$ components, we have
\begin{equation*}
 i\varLambda[\phi,D'_{h_t}v_t]=0\hspace{0.5cm}\text{ and }\hspace{0.5cm}
 i\varLambda D''[\overline{\phi}_{h_t},v_t]=0.
\end{equation*}
From the previous equations and since the linear operators $i\varLambda$ and $\partial_t$ commute, we obtain
\begin{equation*}
 \begin{split}
  \partial_t\mathcal{K}_t&=\partial_ti\varLambda\mathcal{R}_{t}^{1,1}=
  i\varLambda\partial_t\mathcal{R}_{t}^{1,1}=\\
  &=i\varLambda(D''D'_{h_t}v_t+[\phi,[\overline{\phi}_{h_t},v_t]])=\\
  &=i\varLambda(D''D'_{h_t}v_t+[\phi,[\overline{\phi}_{h_t},v_t]])+i\varLambda[\phi,D'_{h_t}v_t]+
  i\varLambda D''[\overline{\phi}_{h_t},v_t]=\\
  &=i\varLambda(D''D'_{h_t}v_t+[\phi,D'_{h_t}v_t]
  +D''[\overline{\phi}_{h_t},v_t]+[\phi,[\overline{\phi}_{h_t},v_t]])=\\
  &=i\varLambda\mathcal{D}''\mathcal{D}'_{h_t}v_t.
 \end{split}
\end{equation*}
Hence, replacing this in the expression for the second derivative of $\mathcal{L}$ we find
\begin{equation*}
 \left.\partial_{t}^{2}\mathcal{L}(h_t,k)\right|_{t=0}=
  \int_{X}\left.\mathrm{tr}(i\varLambda\mathcal{D}''\mathcal{D}'_{h_t}v_t\cdot v_t)
  \frac{\omega^n}{n!}\right|_{t=0}=\|\mathcal{D}'_{h_t}v_t\|_{t=0}^{2}.
\end{equation*}
That is, $h_0$ must be a local minimum of $\mathcal{L}$. Now, suppose in addition that $h_1$ is an 
arbitrary element of $\text{Herm}^{+}(\frak{E})$ and 
assume $h_t$ is a geodesic which joins the point $h_{0}$ and $h_{1}.$
Hence, $\partial_tv_t=0$ (see \cite{KOB} for more details). Therefore,
for such a geodesic we have
\begin{equation*}
  \partial_{t}^{2}\mathcal{L}(h_t,k)=
   \int_{X}\mathrm{tr}(\partial_t\mathcal{K}_t\cdot v_t)\frac{\omega^n}{n!}.
\end{equation*}
Following the same procedure we have done before, but this time for $t$ arbitrary, we get for $0\leq t\leq1$
\begin{equation*}
 \partial_{t}^{2}\mathcal{L}(h_t,k)=\|\mathcal{D}'_{h_t}v_t\|_{L^2}^{2}\geq0.
\end{equation*}

Note that the right hand side implicity depends on $t$ via $\mathcal{D}'_{h_t}.$ It follows 
that $\mathcal{L}(h_0,k)\leq\mathcal{L}(h_1,k).$ If we assume $h_1$ is also a critical point of 
$\mathcal{L},$ we necessarily obtain the equality $\mathcal{L}(h_0,k)=\mathcal{L}(h_1,k),$ so the local minimum
defined for any critical point of $\mathcal{L}$ is an absolute minimum.
\end{proof}

\section{The evolution equation}

For the construction of approximate Hermitian-Yang-Mills structures, the standard procedure is to start with a 
fixed Hermitian metric $h_0$ and try to find from it an approximate Hermitian-Yang-Mills metric structure using
a curve $h_t,0\leq t<+\infty.$ In other words, we try to find the approximate structure by deforming $h_0$ through
the $1\text{-parameter}$ family of Hermitian metrics $h_t.$

\begin{teo}
\label{EXISTANCE}
 Given a Hermitian metric $h_0$ on the Higgs bundle $\frak{E},$ the non-linear evolution problem
 \begin{equation*}
  \left\lbrace\begin{aligned}
         \partial_th_t&=-(\mathcal{K}_t-ch_t)\\
           h(0)&=h_0
        \end{aligned}
  \right.
 \end{equation*}
has a unique smooth solution defined for every positive time $0\leq t<+\infty.$ Notice that here $\mathcal{K}_t$
is the Hermitian form associated with the mean curvature of the Hitchin-Simpson connection of $h_t.$
\end{teo}
\begin{proof}
 See Kobayashi \cite{KOB}, p. 205-223 for details.
\end{proof}

\begin{note}
 Let $k$ be a fixed Hermitian metric on the Higgs bundle $\frak{E}.$ We have
 \linebreak$\nabla\mathcal{L}=\mathcal{K}-ch,$ where
 $\mathcal{K}$ is the Hermitian form associated with the mean Hitchin-Simpson curvature of $h.$ Hence,
 we can rewrite the evolution problem in the gradient form
 \begin{equation*}
  \left\lbrace\begin{aligned}
         \partial_th_t&=-\nabla\mathcal{L}\\
           h(0)&=h_0
        \end{aligned}
  \right..
 \end{equation*}
 Here $\nabla\mathcal{L}$ is a vector field on the Riemannian manifold $\text{Herm}^{+}(\frak{E}).$
 This non-linear evolution problem is called the Donaldson heat flow problem.
\end{note}

In this section we study some properties of the solutions of the Donaldson heat flow problem. In 
particular, we are interested in the study of the mean curvature when the parameter $t$ goes to infinity.

\begin{prop}
\label{proposizione52}
 Let $h_t,0\leq t<+\infty$ be the solution of the Donaldson heat flow with initial condition $h_0.$ Then:
 \begin{enumerate}
  \item For any fixed Hermitian metric $k\in\text{Herm}^{+}(\frak{E}),$ the functional 
        $\mathcal{L}(h_t,k)$ is a monotone decreasing function of $t;$ that is,
        \begin{equation}
        \label{firstderivative}
         \frac{\mathrm{d}}{\mathrm{d}t}\mathcal{L}(h_t,k)=-\|\mathcal{K}_t-cI_E\|_{L^2}^{2}\leq0,
        \end{equation}
  \item $\max_{X}|\mathcal{K}_t-cI_E|^{2}=\|\mathcal{K}_t-cI_E\|_{L^{\infty}}^{2}$ is a monotone decreasing 
        function,
  \item If $\mathcal{L}(h_t,k)$ is bounded from below, i.e., there exists a real constant $A$ such that 
        $\mathcal{L}(h_t,k)\geq A>-\infty$ for $0\leq t<+\infty,$ then
        \begin{equation*}
         \max_{X}|\mathcal{K}_t-cI_E|^{2}\longrightarrow0\hspace{0.5cm}\text{ as }\hspace{0.5cm}
         t\longrightarrow+\infty.
        \end{equation*}        
 \end{enumerate}
 Here $\mathcal{K}_t\in A^0(\mathrm{End}(E))$ is the Hitchin-Simpson mean curvature endomorphism.
\end{prop}
\begin{proof}
 The proofs of $(2)$ and $(3)$ are similar to the proofs in the classical case \cite{KOB}, but this time 
 we need to 
 work with the operator $\tilde{\Box}_h=i\varLambda\mathcal{D}''\mathcal{D}'_h$ instead of the operator
 $\Box_h=i\varLambda D''D'_h.$
 \begin{enumerate} 
  \item From a previous calculation we already know that
        \begin{equation*}
         \frac{\mathrm{d}}{\mathrm{d}t}\mathcal{L}(h_t,k)=(\mathcal{K}_t-ch_t,\partial_th_t),
        \end{equation*}
        where $\mathcal{K}_t$ is now thought as a Hermitian form.
        Since $h_t$ is the solution of the Donaldson heat flow, 
        from the definition of the Riemannian structure in $\text{Herm}^{+}(\frak{E})$ we get
        \begin{equation*}
         \frac{\mathrm{d}}{\mathrm{d}t}\mathcal{L}(h_t,k)=-(\mathcal{K}_t-ch_t,\mathcal{K}_t-ch_t)=
         -\|\mathcal{K}_t-cI_E\|_{L^2}^{2}.
        \end{equation*}
  \item Let $\mathcal{K}_t\in A^{0}(\mathrm{End}(E))$ be the mean curvature of the Hitchin-Simpson connection 
  associated with the metric $h_t$ and let $v_t=h_{t}^{-1}\partial_th_t.$
  Consider the operator $\tilde{\Box}_{h}=i\varLambda\mathcal{D''}\mathcal{D'}_{h}$ which depends on the 
 K\"ahler form $\omega$ via the adjoint multiplication $\varLambda$ and also on the metric $h.$ Using
 $\tilde{\Box}_{h},$ we can rewrite
 \begin{equation*}
  i\varLambda\mathcal{D''}\mathcal{D'}_{h_t}v_t=i\varLambda\partial_{t}\mathcal{R}_{t}^{1,1}=
  \partial_{t}\mathcal{K}_t
 \end{equation*}
as
\begin{equation*}
 \partial_{t}\mathcal{K}_{t}=\tilde{\Box}_{t}v_t,
\end{equation*}
where $\tilde{\Box}_{t}=\tilde{\Box}_{h_t}$ and the subscript $t$ remember us the dependence on the metric 
$h_t.$
From the evolution equation we have $v_t=-(\mathcal{K}_t-cI_E),$ and hence we get
$\tilde{\Box}_{t}v_t=-\tilde{\Box}_{t}\mathcal{K}_t.$ Therefore we obtain
\begin{equation*}
 \partial_{t}\mathcal{K}_t=-\tilde{\Box}_{t}\mathcal{K}_t
\end{equation*}
or
\begin{equation*}
 (\partial t+\tilde{\Box}_{t})\mathcal{K}_t=0.
 \end{equation*}
        On the other hand, since $\mathrm{tr}(AB)=\mathrm{tr}(BA),$ 
        \begin{equation*}
         \begin{split}
          \mathcal{D''}\mathcal{D'}_{h_t}|\mathcal{K}_t-cI_E|^2&=
          \mathcal{D''}\mathcal{D'}_{h_t}\mathrm{tr}[(\mathcal{K}_t-cI_E)\cdot(\mathcal{K}_t-cI_E)]=\\
          &=2\mathrm{tr}[(\mathcal{K}_t-cI_E)\cdot\mathcal{D''}\mathcal{D'}_{h_t}\mathcal{K}_t]+
            2\mathrm{tr}[(\mathcal{D''}\mathcal{K}_t\cdot\mathcal{D'}_{h_t}\mathcal{K}_t].
         \end{split}
        \end{equation*}
        Applying the $i\varLambda$ operator (this will kill the $(2,0)$ and $(0,2)$ components),
        from $\mathrm{tr}(AB)=\mathrm{tr}(BA)$ and since the operator $i\varLambda$ and the trace commute we get
        \begin{equation*}
         \begin{split}
           \tilde{\Box}_{t}|\mathcal{K}_t-cI_E|^2&=
           i\varLambda\mathcal{D''}\mathcal{D'}_{h_t}|\mathcal{K}_t-cI_E|^2=\\
           &=i\varLambda\mathcal{D''}\mathcal{D'}_{h_t}
           \mathrm{tr}[(\mathcal{K}_t-cI_E)\cdot(\mathcal{K}_t-cI_E)]=\\   
           &=2\mathrm{tr}[(\mathcal{K}_t-cI_E)\cdot\tilde{\Box}_{t}\mathcal{K}_t]+\\
           &+i\varLambda\mathrm{tr}[\mathcal{D'}_{h_t}(\mathcal{K}_t-cI_E)\cdot\mathcal{D''}
           (\mathcal{K}_t-cI_E)]+\\
           &+i\varLambda\mathrm{tr}[\mathcal{D''}(\mathcal{K}_t-cI_E)\cdot\mathcal{D'}_{h_t}
           (\mathcal{K}_t-cI_E)]=\\
           &=2\mathrm{tr}[(\mathcal{K}_t-cI_E)\cdot\tilde{\Box}_{t}\mathcal{K}_t]+
             i\varLambda\mathrm{tr}[\mathcal{D'}_{h_t}\mathcal{K}_t\cdot\mathcal{D''}\mathcal{K}_t]+\\
           &+i\varLambda\mathrm{tr}[\mathcal{D''}\mathcal{K}_t\cdot\mathcal{D'}_{h_t}\mathcal{K}_t]=\\
           &=2\mathrm{tr}[(\mathcal{K}_t-cI_E)\cdot\tilde{\Box}_{t}\mathcal{K}_t]+
             2i\varLambda\mathrm{tr}[\mathcal{D'}_{h_t}\mathcal{K}_t\cdot\mathcal{D''}\mathcal{K}_t]=\\
           &=-2\mathrm{tr}[(\mathcal{K}_t-cI_E)\cdot\partial_t\mathcal{K}_t]-2|\mathcal{D}''\mathcal{K}_t|^2=\\
           &=-\partial_t|\mathcal{K}_t-cI_E|^2-2|\mathcal{D''}\mathcal{K}_t|^2,
         \end{split}
        \end{equation*}
        where $|\mathcal{D''}\mathcal{K}_t|^{2}=-i\varLambda\mathrm{tr}[\mathcal{D'}_{h_t}\mathcal{K}_t\cdot
        \mathcal{D''}\mathcal{K}_t]$ is a positive real valued function on $X.$ (See p. 225 in \cite{KOB} for 
        details in the classical case).
         So, finally we obtain
         \begin{equation}
          \label{HSCardona}
          \begin{split}
           (\partial_t+\tilde{\Box}_{t})|\mathcal{K}_t-cI_E|^2&=\partial_t|\mathcal{K}_t-cI_E|^2+
           \tilde{\Box}_{t}|\mathcal{K}_t-cI_E|^2=\\
           &=\partial_t|\mathcal{K}_t-cI_E|^2-\partial_t|\mathcal{K}_t-cI_E|^2-2|\mathcal{D}''\mathcal{K}_t|^2=\\
           &=-2|\mathcal{D}''\mathcal{K}_t|^2\leq0,
          \end{split}
         \end{equation}
         and (2) follows from the Maximum Principle \ref{MPPE}.
 \item Finally (3) follows from (1) and (2) in a similar way to the classical case \cite{KOB}.
       Integrating the equality \eqref{firstderivative} from $0$ to $s,$ we obtain
       \begin{equation*}
        \mathcal{L}(h_s,k)-\mathcal{L}(h_0,k)=-\int_{0}^{s}\|\mathcal{K}_t-cI_E\|_{L^2}^{2}\mathrm{d}t.
       \end{equation*}
       Since $\mathcal{L}(h_s,k)$ is bounded below by a constant $A$ indipendent of $s$ and since is a monotone 
       decreasing of $s,$ there exists the limit
       \begin{equation*}
        \lim_{s\rightarrow+\infty}\mathcal{L}(h_s,k)=L,
       \end{equation*}
       where $L$ is a finite real number.
       Hence
       \begin{equation*}
        \int_{0}^{+\infty}\|\mathcal{K}_t-cI_E\|_{L^2}^{2}\mathrm{d}t=
        \lim_{s\rightarrow+\infty}\mathcal{L}(h_s,k)-\mathcal{L}(h_0,k)=L-\mathcal{L}(h_0,k)<+\infty.
       \end{equation*}
        In particular we deduce
        \begin{equation}
         \label{norm2tendto0}
         \|\mathcal{K}_t-cI_E\|_{L^2}^{2}\longrightarrow0\hspace{0.5cm}\text{ as}\hspace{0.5cm}
         s\longrightarrow+\infty.
        \end{equation}
        Let $\chi=\chi(x,y,t)$ be the heat kernel for the differential operator $\partial_t+\tilde{\Box}_{t}.$
        Set
 \begin{equation*}
  f(x,t)=(|\mathcal{K}_t-cI_E|^2)(x)\hspace{0.5cm}\text{ for }\hspace{0.5cm}(x,t)\in X\times[0,+\infty).
 \end{equation*}
Now fix $t_0\in[0,+\infty)$ and set
\begin{equation*}
 u(x,t)=\int_X\chi(x,y,t-t_0)(|\mathcal{K}_t-cI_E|^2)(y)\mathrm{d}y,
\end{equation*}
where $\mathrm{d}y$ is the volume form $\mathrm{d}y=\frac{\omega^n}{n!}.$
Then $u(x,y)$ is of class $\mathcal{C}^{\infty}$ on \linebreak$X\times(t_0,+\infty)$ and extends to a 
continuous function on $X\times[t_0,+\infty).$ 
From the definition of the heat kernel we immediately have
\begin{equation*}
 (\partial_t+\tilde{\Box}_{t})u(x,t)=0\hspace{0.5cm}\text{ for }\hspace{0.5cm}(x,t)\in X\times(t_0,+\infty),
\end{equation*}
and
\begin{equation*}
 u(x,t_0)=f(x,t_0)=(|\mathcal{K}_{t_0}-cI_E|^2)(x).
\end{equation*}
Combined with the inequality \eqref{HSCardona}
this yields
\begin{equation*}
 (\partial_t+\tilde{\Box}_{t})(|\mathcal{K}_t-cI_E|^2-u(x,t))\leq0\hspace{0.5cm}\text{ for }\hspace{0.5cm}
 (x,t)\in X\times(t_0,+\infty).
\end{equation*}
By the Maximun Principle \ref{MPPE} and the properties of $u(x,t)$ we find
\begin{equation*}
 \max_{X}(|\mathcal{K}_t-cI_E|^2-u(x,t))\leq\max_{X}(|\mathcal{K}_{t_0}-cI_E|^2-u(x,t_0))=0,\hspace{0.2cm}
 t\geq t_0.
\end{equation*}
Hence
\begin{equation*}
\begin{split}
 \max_{X}|\mathcal{K}_{t_0+a}-cI_E|^2 & \leq\max_{X}u(x,a,t_0+a)=\\
                                      &  =  
 \max_{X}\int_X\chi(x,y,a)|\mathcal{K}_{t_0}-cI_E|^2(y)\mathrm{d}y\leq\\
                                      & \leq
 C_a\int_X|\mathcal{K}_{t_0}-cI_E|^2(y)\mathrm{d}y=\\
                                       &  =
 C_a\|\mathcal{K}_{t_0}-cI_E\|_{L^2}^{2},
 \end{split}
\end{equation*}
where 
\begin{equation*}
 C_a=\max_{X\times X}\chi(x,y,a).
\end{equation*}
Fix $a,$ say $a=1,$ and let $t_0\longrightarrow+\infty.$
Since $\mathcal{L}(h_t,k)$ is bounded below, using \eqref{norm2tendto0} we conclude
\begin{equation*}
 \max_{X}|\mathcal{K}_{t_0+1}-cI_E|^2\leq C_1\|\mathcal{K}_{t_0}-cI_E\|_{L^2}^{2}\longrightarrow0,
\end{equation*}
and this completes the proof.
  \end{enumerate}
\end{proof}

\begin{cor}
 The Donaldson functional is a real valued function.
\end{cor}
\begin{proof}
 Let $h_t,$ $0\leq t\leq1,$ be a curve in $\mathrm{Herm}^{+}(E)$ such that $h_{0}=k$ and $h_{1}=h.$ From a 
 previous calculation we already know that
 \begin{equation*}
  \frac{\mathrm{d}}{\mathrm{d}t}\mathcal{L}(h_{t},k)=(\mathcal{K}_{t}-ch_{t},\partial_{t}h_{t})
 \end{equation*}
where $\mathcal{K}_{t}$ is now thought as a Hermitian form and $(\cdot,\cdot)$ is the inner product of 
$\mathrm{Herm}^{+}(E).$ So that
$(\mathcal{K}_{t}-ch_{t},\partial_{t}h_{t})\in\mathbb{R}$ for any $t\in[0,1]$ 
(see p.196-197 in \cite{KOB} for details), and then
\begin{equation*}
 \mathcal{L}(h,k)=\int_{0}^{1}\frac{\mathrm{d}}{\mathrm{d}t}\mathcal{L}(h_{t},k)\mathrm{d}t\in\mathbb{R}.
\end{equation*}
\end{proof}

At this point we introduce the main result of this section. This establishes a relation among the boundedness
property of Donaldson functional, semistability and the existence of approximate Hermitian-Yang-Mills
metric structures on the Higgs bundle $\frak{E}.$

\begin{teo}
\label{teorema53}
Let $(X,\omega)$ be a compact K\"ahelr manifold of (complex) dimension $n$ and
let $\frak{E}=(E,\phi)$ be a Higgs bundle of rank $r$ over $X.$
We have implications $(1)\longrightarrow(2)\longrightarrow(3)$ for the following statements:
 \begin{enumerate}
  \item for any fixed Hermitian metric structure $k\in\text{Herm}^{+}(\frak{E}),$ there exists a constant $B$
        such that 
        $\mathcal{L}(h,k)\geq B$ for all Hermitian metrics $h\in\text{Herm}^{+}(\frak{E}),$
  \item $\frak{E}$ admits an approximate Hermitian-Yang-Mills structure, i.e., for each $\epsilon>0$ there 
        exists an Hermitian metric $h$ in $\frak{E}$ such that
        \begin{equation*}
         \max_{X}|\mathcal{K}-cI_E|<\epsilon,
        \end{equation*}
        where $h$ depends on $\epsilon$ and $\mathcal{K}$ is the mean curvature endomorphism of the 
        Hitchin-Simpson connection associated with the metric $h,$
  \item $\frak{E}$ is $\omega\text{-semistable}.$
 \end{enumerate}
\end{teo}
\begin{proof}
 \begin{enumerate}
  \item Assume $(1).$ The Donaldson functional is bounded below by a constant $B.$ Let $h_0$
        be a fixed Hermitian metric in $\frak{E}$ and let $h_t$ be the solution of the Donaldson heat flow with 
        initial condition $h_0.$ Then from the previous Proposition there exists a real consant $B$ such that
        \linebreak$\mathcal{L}(h_t,h_0)\geq B>-\infty$ for every positive time $0\leq t<+\infty.$
        From Theorem \ref{teorema53},
        \begin{equation*}
         \max_{X}|\mathcal{K}_t-cI_E|^{2}\longrightarrow0\hspace{0.5cm}\text{ as }\hspace{0.5cm}
         t\longrightarrow+\infty,
        \end{equation*}   
        where $\mathcal{K}_t$ is the mean curvature endomorphism $\mathcal{K}_t\in A^0(End(\frak{E})).$
        Hence, there exists an approximate Hermitian-Yang-Mills structure.
 \item  On the other hand, that $(2)$ impies $(3)$ has been proved in \cite{BRU} by Bruzzo and Gra\~{n}a Otero. 
        Here we reproduce their proof.
        Assume $(2)$ and let $\frak{F}$ be a proper nontrivial Higgs subsheaf of $\frak{E}.$ Then 
        $\mathrm{rk}(F)=p$ for some \linebreak$0<p<r=\mathrm{rk}(E)$ and the 
        inclusion $\frak{F}\longrightarrow\frak{E}$
        induces a morphism $\det\frak{F}\longrightarrow\bigwedge^{p}\frak{E}.$
        Tensoring by $(\det\frak{F})^{-1}$ we have a nonzero section $s$ of the Higgs bundle
        \begin{equation*}
         \frak{G}=\bigwedge^{p}\frak{E}\otimes(\det\frak{F})^{-1}.
        \end{equation*}
        If $\psi$ represents the Higgs field naturally defined on $\frak{G}$ by the Higgs field of $\frak{E}$ and
        $\frak{F},$ then $s$ is $\psi\text{-invariant},$ i.e., $\psi(s)=s.$ Now, since by hypothesis $\frak{E}$
        admits an 
        approximate Hermitian-Yang-Mills structure, from Proposition \ref{approximatetensorproduct} we know 
        that so does $\frak{G}$ and in particular
        \begin{equation*}
         c_{\frak{G}}=\frac{2p\pi(\mu(\frak{E})-\mu(\frak{F}))}{(n-1)!\mathrm{Vol}(X)}.
        \end{equation*}
        From \ref{corBruOter}, since $s$ is a nonzero $\psi\text{-invariant}$ section, $\deg(\frak{G})\geq0$ 
        and so $c_{\frak{G}}\geq0.$  Hence $\mu(\frak{E})\geq\mu(\frak{F}),$ showing that $\frak{E}$ is
        $\omega\text{-semistable}.$ 
 \end{enumerate}
\end{proof}

\section{The one-dimensional case}

In this section we establish a boundedness property for the Donaldson functional for semistable Higgs bundles 
over compact Riemann surfaces. As a consequence we get that in the one-dimensional case all three conditions
in Theorem \ref{teorema53} are equivalent.

We introduce some properites that will be useful in proving some statements using induction on the rank of 
Higgs bundles.
This section is essentially an extension to Higgs bundles of the classical case. (See 
p. 226-233 in \cite{KOB} for more details).

Now, let $(X,\omega)$ be a compact K\"ahler manifold of (complex) dimension $n$ and let
\begin{equation}
\label{sequenzaesattahiggs}
 0\longrightarrow\frak{E}'\longrightarrow\frak{E}\longrightarrow\frak{E}''\longrightarrow0
\end{equation}
be an exact sequence of Higgs bundles over $X.$ As in the ordinary case 
(see Chapter I, \S 6 in \cite{KOB}), a Hermitian metric $h$ in $\frak{E}$ induces Hermitian structures $h'$ and
$h''$ in $\frak{E}'$ and $\frak{E}'',$ respectively. We have also a second foundamental form 
\linebreak$A_h\in A^{1,0}(\mathrm{Hom}(E',E''))$ and its adjoint 
$B_h\in A^{0,1}(\mathrm{Hom}(E'',E')),$ where 
\linebreak$B_{h}^{\ast}=-A_h.$
As usual, some properties which hold in the ordinary case, also hold in the Higgs case.

\begin{prop}
 Given an exact sequence \eqref{sequenzaesattahiggs} and a pair of Hermitian structures $h$ and $k$ in 
 $\frak{E},$ the function $\mathcal{Q}_1(h,k)$ and the form $\mathcal{Q}_2(h,k)$ satisfy the following
 relations:
 \begin{equation}
  \label{relation1}
  \mathcal{Q}_1(h,k)=\mathcal{Q}_1(h',k')+\mathcal{Q}_1(h'',k''),
 \end{equation}
\begin{equation}
 \label{relation2}
 \begin{split}
    \mathcal{Q}_2(h,k)&=\mathcal{Q}_2(h',k')+\mathcal{Q}_2(h'',k'')-
  i\mathrm{tr}[B_h\wedge B_{h}^{\ast}-B_k\wedge B_{k}^{\ast}]\\
                      &\text{mod}\hspace{0.5cm}\mathrm{d}'A^{0,1}+\mathrm{d''}A^{1,0}.
 \end{split}
\end{equation} 
\end{prop}
\begin{proof}
 \begin{enumerate}
  \item \eqref{relation1} is straightforward from the definition of $\mathcal{Q}_1.$ Since $h=h'\oplus h''$ and
       $k=k'\oplus k'',$ we immediately have
       \begin{equation*}
       \begin{split}
         \mathcal{Q}_1(h,k)&=\ln(\det(k^{-1}h))=\\
                           &=\ln(\det(k'^{-1}h'\oplus k''^{-1}h''))=\\
                           &=\ln(\det(k'^{-1}h'))+\ln(\det(k''^{-1}h''))=
                           \mathcal{Q}_1(h',k')+\mathcal{Q}_1(h'',k'').
       \end{split}
       \end{equation*}
  \item On the other hand, \eqref{relation2} follows from an analysis similar to the ordinary case. Since the 
        sequence \eqref{sequenzaesattahiggs} in particular is an exact sequence of holomorphic vector bundles
        over the complex manifold $X,$ for any Hermitian metric $h$ we have a splitting of the exact sequence
        by $\mathcal{C}^{\infty}$ homomorphisms $\mu_h:E\longrightarrow E'$ and
        $\lambda_h:E''\longrightarrow E.$ In particular
        \begin{equation*}
         B_h=\mu_h\circ\mathrm{d''}\circ\lambda_h.
        \end{equation*}
         We consider now a curve of Hermitian structures $h_t$ for $0\leq t\leq1$ such that $h_0=k$ and
         $h_1=h.$ Corresponding to $h_t$ we have a family of homomorphisms $\mu_t$ and $\lambda_t.$ We define the 
         homomorphism $S_t:E''\longrightarrow E'$ given by
         \begin{equation*}
          S_t=\lambda_t-\lambda_0.
         \end{equation*}
         A short computation shows that $\partial_tB_t=\mathrm{d''}(\partial_tS_t).$ Choosing convenient 
         orthonormal local frame fields for $\frak{E}'$ and $\frak{E}''$ (see \cite{KOB} for more details), the 
         endomorphism $v_t=h_{t}^{-1}\partial_th_t$ can be presented by the matrix
         \begin{equation*}
             v_t=\left(\begin{array}{cc}
	          v'_t		 	                 &-\partial_tS_t\\
                  -(\partial_tS_t)^{\ast}                  &v''_t
                          \end{array}\right).
         \end{equation*}
         Here $v'_t$ and $v''_t$ are the natural endomorphisms associated with $h'_t$ and $h''_t,$ respectively.
         Now, from the ordinary case we have
         \begin{equation*}
             R_t=\left(\begin{array}{cc}
	          R'_t-B_t\wedge B_{t}^{\ast}		 	                 &D'B_t\\
                  -D''B_{t}^{\ast}                 &R''_t-B_{t}^{\ast}\wedge B_t
                          \end{array}\right)
         \end{equation*}
         where $R'_t$ and $R''_t$ are the Hermitian curvatures of $\frak{E}'$ and $\frak{E}''$ associated with 
         the metrics $h'_t$ and $h''_t,$ respectively.
         
         Now, $\mathcal{R}_{t}^{1,1}=R_t+[\phi,\overline{\phi}_{h_t}].$ Since $\frak{E}'$ and $\frak{E}''$ are 
         Higgs subbundles of $\frak{E},$ we obtain an expression for the (1,1)-component of the Hitchin-Simpson
         curvature
          \begin{equation*}
             \mathcal{R}_{t}^{1,1}=\left(\begin{array}{cc}
	          \mathcal{R'}_{t}^{1,1}-B_t\wedge B_{t}^{\ast}               &D'B_t\\
                  -D''B_{t}^{\ast}                      &\mathcal{R''}_{t}^{1,1}-B_{t}^{\ast}\wedge B_t
                          \end{array}\right),
         \end{equation*}
         where $\mathcal{R'}_{t}^{1,1}=R'_t+[\phi,\overline{\phi}_{h_t}]_{E'}$ and
         $\mathcal{R''}_{t}^{1,1}=R''_t+[\phi,\overline{\phi}_{h_t}]_{E''}.$ Hence we can compute the trace
         \begin{equation*}
          \begin{split}
           \mathrm{tr}(v_t\cdot\mathcal{R}_{t}^{1,1})&=
           \mathrm{tr}(v'_t\cdot\mathcal{R'}_{t}^{1,1})+\mathrm{tr}(v''_t\cdot\mathcal{R''}_{t}^{1,1})+\\
           &+\mathrm{tr}(\partial_tS_t\cdot D''B_{t}^{\ast})-\mathrm{tr}((\partial_tS_t)^{\ast}\cdot D'B_t)+\\
           &+\mathrm{tr}(v'_t\cdot B_t\wedge B_{t}^{\ast})-\mathrm{tr}(v''_t\cdot B_{t}^{\ast}\wedge B_t).
          \end{split}
         \end{equation*}
         The last four terms are the same as in the ordinary case \cite{KOB}. After a short computation we finally
         get that
         \begin{equation*}
         \begin{split}
         \mathrm{tr}(v_t\cdot\mathcal{R}_{t}^{1,1})&=\mathrm{tr}(v'_t\cdot\mathcal{R'}_{t}^{1,1})+
              \mathrm{tr}(v''_t\cdot\mathcal{R''}_{t}^{1,1})-\partial_t\mathrm{tr}(B_t\wedge B_{t}^{\ast})\\
              &\text{ mod }\hspace{0.5cm}\mathrm{d}'A^{0,1}+\mathrm{d''}A^{1,0}.
         \end{split}
         \end{equation*}
         Then, multiplying the last expression by $i$ and integrating from $t=0$ to $t=1$ we finally obtain
         \eqref{relation2}.
 \end{enumerate}
\end{proof}

\begin{defin}
 In the hypothesis of the previous Proposition, we define $|B|$ as the nonnegative real function on the compact 
 K\"ahler manifold $(X,\omega)$ that satisfies
 \begin{equation*}
  |B|^2\omega^n=-in\mathrm{tr}(B\wedge B^{\ast})\wedge\omega^{n-1}.
 \end{equation*}
\end{defin}

\begin{lemma}
\label{SplittingLemma}
Let $(X,\omega)$ be a compact K\"ahler manifold of (complex) dimension $n.$
 Given an exact sequence of Higgs bundles
 \begin{equation*}
 0\longrightarrow\frak{E}'\longrightarrow\frak{E}\longrightarrow\frak{E}''\longrightarrow0
\end{equation*}
 over $X,$ with $\mu(\frak{E})=\mu(\frak{E}')$ and a pair of Hermitian 
 structures $h$ and $k$ in $\frak{E},$ we have the identity
 \begin{equation}
  \label{Ldecomposition}
  \mathcal{L}(h,k)=\mathcal{L}(h',k')+\mathcal{L}(h'',k'')+\|B_h\|_{L^2}^{2}-\|B_k\|_{L^2}^{2}.
 \end{equation}
\end{lemma}
\begin{proof}
 First of all, from $\mu(\frak{E})=\mu(\frak{E}')$ we have $c=c'=c''.$ From Lemma \ref{holomorphicstokes} and 
 from
 \eqref{relation1} and \eqref{relation2} we obtain
 \begin{equation*}
  \begin{split}
   \mathcal{L}(h,k)&=\int_{X}\left[\mathcal{Q}_2(h,k)-\frac{c}{n}\mathcal{Q}_1(h,k)\omega\right]
 \wedge\frac{\omega^{n-1}}{(n-1)!}=\\
                   &=\int_{X}\left[\mathcal{Q}_2(h',k')-\frac{c'}{n}\mathcal{Q}_1(h',k')
                   \omega\right]\wedge\frac{\omega^{n-1}}{(n-1)!}+\\
                   &+\mathcal{L}(h,k)=\int_{X}\left[\mathcal{Q}_2(h'',k'')-\frac{c''}{n}\mathcal{Q}_1(h'',k'')
                   \omega\right]\wedge\frac{\omega^{n-1}}{(n-1)!}+\\
                   &+\int_{X}-i\mathrm{tr}[B_h\wedge B_{h}^{\ast}-B_k\wedge B_{k}^{\ast}]\wedge
                   \frac{\omega^{n-1}}{(n-1)!}=\\
                   &=\mathcal{L}(h',k')+\mathcal{L}(h'',k'')+\\
                   &+\int_{X}|B_h|^2\frac{\omega^n}{n!}-\int_{X}|B_k|^2\frac{\omega^n}{n!}=\\
                   &=\mathcal{L}(h',k')+\mathcal{L}(h'',k'')+\|B_h\|_{L^2}^{2}-\|B_k\|_{L^2}^{2}.
  \end{split}
 \end{equation*}
\end{proof}

In dimension greater or equal than two, the notion stability 
(resp. semistability) depends on the K\"ahler 
form, as the degree depends on it. Now, in dimension one, the degree does not depend on the 
K\"ahler form and hence the notion of stability (resp. semistability) does not depend on it and we can 
establish all our results without any explicit reference to $\omega.$
Since the degree and the rank of any Higgs sheaf is the same degree and rank of the corresponding coherent sheaf,
we have the following (see \cite{KOB}, Ch. V, Lemma 7.3).

\begin{lemma}
\label{Pippone1}
 Let consider the exact sequence of Higgs sheaves
  \begin{equation*}
0\longrightarrow\frak{F}\longrightarrow\frak{E}\longrightarrow\frak{G}\longrightarrow0,
\end{equation*}
then
\begin{equation*}
 \mathrm{rk}(\frak{F})(\mu(\frak{E})-\mu(\frak{F}))+\mathrm{rk}(\frak{G})(\mu(\frak{E})-\mu(\frak{G}))=0.
\end{equation*}
\end{lemma}

From Lemma \ref{Pippone1} it follows that the condition of stability (resp. semistability) can be written in 
terms of quotient Higgs sheaves instead of Higgs subsheaves. To be precise we have

\begin{cor}
 Let $(X,\omega)$ be a compact K\"ahler manifold of (complex) dimension $n$ and let $\frak{E}$ be a torsion-free 
 Higgs sheaf over $X.$ Then $\frak{E}$ is $\omega\text{-stable}$ (resp. semistable) if for every quotient 
 Higgs sheaf $\frak{G}$ with $0<\mathrm{rk}(\frak{G})<\mathrm{rk}(\frak{E})$ it follows 
 $\mu(\frak{G})<\mu(\frak{E})$ (resp. $\mu(\frak{G})\leq\mu(\frak{E})$).
\end{cor}

From the definition of degree, one has that any torsion Higgs sheaf $\frak{T}$ has 
$\mathrm{deg}(\frak{T})\geq0.$ Therefore, in a similar way to the classical case, this implies that in the 
definition of stability (resp. semistability) we do not have to consider all quotient Higgs sheaves. To be 
precise we have

\begin{prop}
 Let $(X,\omega)$ be a compact K\"ahler manifold of (complex) dimension $n$ and let $\frak{E}$ be a torsion-free 
 Higgs sheaf over $X.$ Then
 \begin{enumerate}
  \item $\frak{E}$ is $\omega\text{-stable}$ (resp. semistable) if and only if $\mu(\frak{F})<\mu(\frak{E})$
        (resp. $\leq$) for any Higgs subsheaf $\frak{F}$ with 
        $0<\mathrm{rk}(\frak{F})<\mu(\frak{E})$ and such that the quotient $\faktor{\frak{E}}{\frak{F}}$ is 
        torsion-free.
  \item $\frak{E}$ is $\omega\text{-stable}$ (resp. semistable) if and only if $\mu(\frak{E})<\mu(\frak{G})$
        (resp. $\leq$) for any torsion-free quotient Higgs sheaf $\frak{G}$ with 
        $0<\mathrm{rk}(\frak{G})<\mu(\frak{E}).$ 
 \end{enumerate}
\end{prop}
\begin{proof}
 Is true in one direction. For the converse, suppose the inequality between slopes in $(1)$
 (resp. in $(2)$) holds for proper Higgs subsheaves with torsion-free quotient (resp. for torsion-free quotient
 Higgs sheaves) and let consider an exact sequence of Higgs sheaves
   \begin{equation*}
0\longrightarrow\frak{F}\longrightarrow\frak{E}\longrightarrow\frak{G}\longrightarrow0.
\end{equation*}
 Let $\frak{E}=(E,\phi)$ and denote by $\phi_{F}$ and $\phi_{G}$ the Higgs fields of $\frak{F}$ and $\frak{G}$
 respectively. That is $\frak{F}=(F,\phi_{F})$ and $\frak{G}=(G,\phi_{G}).$ Now, let $T$ be the torsion 
 subsheaf of $G.$ Since the Higgs field satisfies $\phi(T)\subseteq T\otimes\Omega_{X}^{1},$ the pair 
 $\frak{T}=(T,\left.\phi\right|_{T})$ is a Higgs subsheaf of $\frak{E}$ with Higgs quotient, say $\frak{G}_{1}.$
 Then if we define $\frak{F}_{1}$ by the kernel of the Higgs morphism $\frak{E}\longrightarrow\frak{G}_{1},$ we 
 have the following commutative diagram of Higgs sheaves
    \begin{equation*}
 \begin{tikzpicture}[node distance=1.5cm, auto]
 \node (A) {$0$};
 \node (B) [below of=A] {$\frak{T}$};
 \node (C) [below of=B] {$\frak{G}$};
 \node (D) [right of=C] {$0$};
 \node (E) [left of=C] {$\frak{E}$};
 \node (F) [left of=E] {$\frak{F}$};
 \node (Z) [left of=B] {$$};
 \node (G) [left of=Z] {$0$};
 \node (H) [left of=F] {$0$};
 \node (I) [below of=H] {$0$};
 \node (L) [right of=I] {$\frak{F}_{1}$};
 \node (M) [right of=L] {$\frak{E}$};
 \node (N) [right of=M] {$\frak{G}_{1}$};
 \node (O) [right of=N] {$0$};
 \node (P) [below of=L] {$\faktor{\frak{F}_{1}}{\frak{F}}$};
 \node (Q) [below of=P] {$0$};
 \node (R) [below of=N] {$0$};
 \draw[->] (A) to node {$$} (B);
 \draw[->] (B) to node {$$} (C);
 \draw[->] (C) to node {$$} (D);
 \draw[->] (E) to node {$$} (C);
 \draw[->] (F) to node {$$} (E);
 \draw[->] (G) to node {$$} (F);
 \draw[->] (H) to node {$$} (F);
 \draw[->] (F) to node {$$} (L);
 \draw[->] (E) to node {$Id$} (M);
 \draw[->] (C) to node {$$} (N);
 \draw[->] (I) to node {$$} (L);
 \draw[->] (L) to node {$$} (M);
 \draw[->] (M) to node {$$} (N);
 \draw[->] (N) to node {$$} (O);
 \draw[->] (L) to node {$$} (P);
 \draw[->] (P) to node {$$} (Q);
 \draw[->] (N) to node {$$} (R);
\end{tikzpicture}
\end{equation*}
in which all rows and columns are exact. From this diagram we have that $\frak{F}$ is a Higgs subsheaf of 
$\frak{F}_{1}$ with $\frak{T}=\faktor{\frak{F}_{1}}{\frak{F}}.$ Since $\frak{T}$ is a torsion Higgs sheaf
$\mathrm{deg}(\frak{T})\geq0$ and we also obtain
\begin{equation*}
 \mathrm{deg}(\frak{G})=\mathrm{deg}(\frak{T})+\mathrm{deg}(\frak{G}_{1})\geq\mathrm{deg}(\frak{G}_{1})
\end{equation*}
and
\begin{equation*}
 \mathrm{deg}(\frak{F}_{1})=\mathrm{deg}(\frak{F})+\mathrm{deg}(\frak{T})\geq\mathrm{deg}(\frak{F}).
\end{equation*}
Now, since $\frak{T}$ is torsion we have $\mathrm{rk}(\frak{G})=\mathrm{rk}(\frak{G}_{1})$ and 
$\mathrm{rk}(\frak{F}_{1})=\mathrm{rk}(\frak{F})$ and hence finally we obtain
\begin{equation*}
 \mu(\frak{F})\leq\mu(\frak{F}_{1})\hspace{0.5cm}\text{ and }\hspace{0.5cm}
 \mu(\frak{G}_{1})\leq\mu(\frak{G}).
\end{equation*}
At this point, the converse direction in $(1)$ and $(2)$ follows from the hypothesis and the last two 
inequalities.
 \end{proof}

Now we can establish a boundedness property for the Donaldson functional for semistable Higgs bundles in the 
one-dimensional case, i.e., for compact Riemann surfaces. Namely we have

\begin{teo}
Let $(X,\omega)$ be a compact Riemann surface and
 let $\frak{E}=(E,\phi)$ be a Higgs bundle of rank $r$ over $X.$
 If $\frak{E}$ is $\omega\text{-semistable},$ then for any fixed Hermitian metric $k\in\text{Herm}^{+}(\frak{E})$
 the set $\{\mathcal{L}(h,k)|h\in\text{Herm}^{+}(\frak{E})\}$ is bounded below.
\end{teo}
\begin{proof}
Fix $k$ and assume $\frak{E}$ is $\omega\text{-semistable}.$ The proof runs by induction on the rank of 
$\frak{E}.$
\begin{enumerate}
 \item If $\mathrm{rk}(\frak{E})=1,$ from Corollary \ref{always1} there exists a
       Hermitian-Yang-Mills structure $h_0.$ The Donaldson functional must attain an absolute minimum
        at $h_0,$ i.e., for any other Hermitian metric $h$
        \begin{equation*}
         \mathcal{L}(h,k)\geq\mathcal{L}(h_0,k).
        \end{equation*}
        Then the set $\{\mathcal{L}(h,k)|h\in\text{Herm}^{+}(\frak{E})\}$ is bounded below.
 \item Let us assume $\mathrm{rk}(\frak{E})\geq2$ and let us assume the thesis true for every Higgs bundle 
       $\frak{F}$
       over $X$ such that $0<\mathrm{rk}(\frak{F})<\mathrm{rk}(\frak{E}).$
       
        We have to distinguish between two cases.
    \begin{enumerate}
         \item If $\frak{E}$ is $\omega\text{-stable},$ there exists a Hermitian-Yang-Mills structure $h_0$
         on it (see \cite{SIM} for a detailed proof) and the Donaldson functional must 
         attain an absolute minimum at $h_0,$ i.e., for any other Hermitian metric $h$
        \begin{equation*}
         \mathcal{L}(h,k)\geq\mathcal{L}(h_0,k).
        \end{equation*}
         Hence the set $\{\mathcal{L}(h,k)|h\in\}\text{Herm}^{+}(\frak{E})$ is bounded below.
         \item Suppose $\frak{E}$ is $\omega\text{-semistable},$ but not $\omega\text{-stable},$ with 
              $\mathrm{rk}(\frak{E})\geq2.$
              Among all proper nontrivial Higgs subsheaves with 
              $0<\mathrm{rk}(\frak{F})<\mathrm{rk}(\mathrm{E}),$ torsion-free quotien and the same slope of 
              $\frak{E}$ we choose one, say $\frak{E'},$ with minimal rank. 
              Since $\mu(\frak{E'})=\mu(\frak{E}),$ the sheaf $\frak{E'}$ is necessarly $\omega\text{-stable}.$
              If not, there exists a proper Higgs subsheaf $\frak{F'}$ of
              $\frak{E'}$ with $\mu(\frak{F'})\geq\mu(\frak{E'}).$ Since $\frak{F'}$ is clearly a subsheaf of 
              $\frak{E}$ and $\frak{E}$ is $\omega\text{-semistable},$ we obtain 
              \linebreak$\mu(\frak{E'})\leq\mu(\frak{F'})\leq\mu(\frak{E})=\mu(\frak{E'}),$ which
              is a contradiction, because  $\frak{E'}$ was chosen of minimal rank.
              Now let $\frak{E''}=\faktor{\frak{E}}{\frak{E'}}.$ Using Lemma 7.3 in \cite{KOB} it follows 
              that $\mu(\frak{E})=\mu(\frak{E'})=\mu(\frak{E''})$ and $\frak{E''}$ is 
              $\omega\text{-semistable}.$
              Hence we have the following exact sequence of Higgs sheaves
              \begin{equation*}
               0\longrightarrow\frak{E'}\longrightarrow\frak{E}\longrightarrow\frak{E''}\longrightarrow0,
              \end{equation*}
              where both $\frak{E'}$ and $\frak{E''}$ are torsion-free.
              Since $\text{dim}_{\mathbb{C}}X=1,$ from Corollary \ref{riemannsurface} we deduce that 
              they are also locally-free. Hence the above sequence is in fact an exact sequence of Higgs bundles. 
              
              Assume now $h$ is an arbitrary metric on $\frak{E},$ by applying Lemma 
              \ref{SplittingLemma} to the
              metrics $h$ an $k$ we obtain
              \begin{equation}
              \label{boundedness}
              \begin{split}
               \mathcal{L}(h,k)&=\mathcal{L}(h',k')+\mathcal{L}(h'',k'')+
                                 \|B_h\|_{L^2}^{2}-\|B_k\|_{L^2}^{2}\geq\\
                                &\geq\mathcal{L}(h',k')+\mathcal{L}(h'',k'')-\|B_k\|_{L^2}^{2},
              \end{split}
              \end{equation}
              where $h',$ $k'$ and $h'',$ $k''$ are the Hermitian structures induced by $h$ and $k$ in 
              $\frak{E'}$ and $\frak{E''},$ respectively. By induction, $\mathcal{L}(h',k')$ and
              $\mathcal{L}(h'',k'')$ are bounded below by consants depending only on $k'$ and $k''.$ Then,
              from \eqref{boundedness} it follows that $\mathcal{L}(h,k)$ is bounded below by a 
              constant depending only on $k.$
  \end{enumerate}
\end{enumerate}
\end{proof}

As consequence of this Theorem we obtain the following

\begin{teo}
\label{correspondenceinRiemannsurfaces}
 Let $(X,\omega)$ be a compact Riemann surface and let $\frak{E}=(E,\phi)$ be a Higgs sheaf of rank $r$ over $X.$
 The following are equivalent:
 \begin{enumerate}
  \item $\frak{E}$ is $\omega\text{-semistable},$
  \item $\frak{E}$ admits an approximate Hermitian-Yang-Mills structure.
 \end{enumerate}
\end{teo}
\begin{proof}
Since $X$ is a Riemann surface, from \ref{riemannsurface} we deduce that $(E,\phi)$ is a 
Higgs bundle over $X.$ Hence, the thesis comes from the previous result and Theorem \ref{teorema53}. 
\end{proof}

As a consequence of this Theorem we immediately deduce that in the one-dimensional case, many results about Higgs 
bundles written in terms of approximate Hermitian-Yang-Mills structures can be translated in terms of 
semistability. In particular we have the following

\begin{cor}
 If $(X,\omega)$ is a compact Riemann surface and $\frak{E}_1,$ $\frak{E}_2$ are 
 $\omega\text{-semistable}$ Higgs bundles over $X,$ then so is their tensor product 
 $\frak{E}_1\otimes\frak{E}_2.$ Furthermore, if $\mu(\frak{E}_1)=\mu(\frak{E}_2),$ so is the Whitney sum 
 $\frak{E}_1\oplus\frak{E}_2.$
\end{cor}
\begin{proof}
 \begin{enumerate}
  \item Let $\frak{E}_1$ and $\frak{E}_2$ be $\omega\text{-semistable}$ Higgs bundles over $X.$ From the previous 
        Theorem
        $\frak{E}_1$ and $\frak{E}_2$ admit approximate Hermitian-Yang-Mills metric structures, so does 
        their tensor product $\frak{E}_1\otimes\frak{E}_2$ (Proposition \ref{approximatetensorproduct}).
        Hence, using the above Theorem we conclude that
        $\frak{E}_1\otimes\frak{E}_2$ is $\omega\text{-semistable}.$
  \item It is similar to $(1),$ using the second part of Proposition \ref{approximatetensorproduct}.
 \end{enumerate}
\end{proof}

\begin{cor}
 If $\frak{E}$ is $\omega\text{-semistable},$ then so is the tensor product bundle
 $\frak{E}^{\otimes p}\otimes\frak{E}^{\ast\otimes q}$ and the exterior product bundle 
 $\bigwedge^{p}\frak{E}$ whenever $0\leq p\leq r=\mathrm{rk}(\frak{E}).$
\end{cor}

\chapter{General Case}
In the previous chapter we have proved the equivalence between semistability and the existence of 
approximate Hermitian-Yang-Mills metric structures for Higgs bundles over compact Riemann surfaces. 
In this chapter we extend this result to higher dimensions.
Let $(X,\omega)$ be a compact K\"ahler manifold of (complex) dimension $n$ and let $\frak{E}=(E,\phi)$ be a 
Higgs bundle of rank $r$ over $X.$
Let $h_t$ be the solution of the Donaldson heat flow with initial metric 
$h_0,$ i.e., the solution of the non-linear evolution problem
 \begin{equation*}
  \left\lbrace \begin{aligned}
                \partial_th_t&=-\nabla\mathcal{L}\\
                         h(0)&=h_0
               \end{aligned}
\right..
 \end{equation*}
From Theorem \ref{EXISTANCE} we know that this problem has a unique solution
$h_t$ defined for every positive time $0\leq t<+\infty,$ and $\mathcal{L}(h_t,h_0)$ is a real monotone 
decreasing function on $t$ for $0\leq t<+\infty.$ 
If we assume that $(E,\phi)$ is $\omega\text{-semistable},$ we can distinguish between two cases:
\begin{enumerate}
 \item $\mathcal{L}(h_t,h_0)$ is bounded below. From Theorem \ref{teorema53} 
       there exists an approximate Hermitian-Yang-Mills structure.
 \item $\mathcal{L}(h_t,h_0)$ is not bounded below, i.e., $\mathcal{L}(h_t,h_0)\longrightarrow-\infty$
       for $t\longrightarrow+\infty.$ Under the assumption that $(E,\phi)$ is 
       $\omega\text{-semistable},$ we can still prove the existence of approximate 
       Hermitian-Yang-Mills metrics.
\end{enumerate}

\section{Some useful results}

\begin{lemma}
 Let $(X,\omega)$ be a compact K\"ahler manifold of (complex) dimension $n$ and let $\frak{E}=(E,\phi)$ be
 a Higgs bundle of rank $r$ over $X.$ Let $h\in\text{Herm}^{+}(E)$ be a Hermitian metric and let
 $\mathcal{R}$ and $\mathcal{K}$
 be its Hitchin-Simpson curvature and mean curvature, respectively. Then
 \begin{equation}
  \label{nulltrace}
  \int_X\mathrm{tr}(\mathcal{K}-cI_E)\frac{\omega^{n}}{n!}=0
 \end{equation}
\end{lemma}
\begin{proof}
First we note that $i\varLambda\mathrm{tr}(\mathcal{R})=\mathrm{tr}(i\varLambda\mathcal{R}),$ 
since $i\varLambda$ and the trace are linear differential operator.
From Stokes' formula and the representation of Chern classes in terms of curvature we have:
 \begin{equation*}
  \begin{split}
   \frac{1}{2\pi}\int_X\mathrm{tr}(\mathcal{K}-cI_E)\frac{\omega^{n}}{n!}&=
   \frac{1}{2\pi}\int_X\mathrm{tr}(\mathcal{K})\frac{\omega^{n}}{n!}-\frac{cr\mathrm{Vol}(X)}{2\pi}=\\
                                                                       &=
   \frac{1}{2\pi}\int_X\mathrm{tr}(i\varLambda\mathcal{R})\frac{\omega^{n}}{n!}-
   \frac{\mu(E)r\mathrm{Vol}(X)}{(n-1)!\mathrm{Vol}(X)}=\\
                                                                       &=
   \frac{1}{2\pi}\int_Xi\varLambda\mathrm{tr}(\mathcal{R})\frac{\omega^{n}}{n!}-
   \frac{r}{(n-1)!}\frac{1}{r}\int_Xc_1(E)\wedge\omega^{n-1}=\\
                                                                       &=
   \int_X\frac{i}{2\pi}\mathrm{tr}(\mathcal{R})\wedge\varLambda\omega^n-
    \int_Xc_1(E)\wedge\frac{\omega^{n-1}}{(n-1)!}=\\
                                                                       &=
   \int_X-\frac{1}{2\pi i}\mathrm{tr}(\mathcal{R})\wedge\frac{\omega^{n-1}}{(n-1)!}-
    \int_Xc_1(E)\wedge\frac{\omega^{n-1}}{(n-1)!}=\\
                                                                       &=
   \int_Xc_1(E)\wedge\frac{\omega^{n-1}}{(n-1)!}-
   \int_Xc_1(E)\wedge\frac{\omega^{n-1}}{(n-1)!}=0.
  \end{split}
 \end{equation*}
\end{proof}

Given a metric $h\in\text{Herm}^{+}(E),$ let $\mathcal{K}$ be the mean curvature of the Hitchin-Simpson 
connection associated with $h.$ There always exists a real positive function $a=a(x)$ on $X$ such that, setting 
$h'=ah,$ $\mathrm{tr}(\mathcal{K'}-cI_{E})=0.$ To be more precise, we have the following:

\begin{lemma}
\label{conformalchangetrace}
 Let $(X,\omega)$ be a compact K\"ahler manifold of (complex) dimension $n$ and let $\frak{E}=(E,\phi)$ be a 
 Higgs bundle of rank $r$ over $X.$
 Let $h\in\text{Herm}^{+}(E)$ be a Hermitian metric on $E,$ and let $\mathcal{K}$ be the mean
 curvature of the Hitchin-Simpson connection associated with $h.$ By an appropriate conformal
 change of the metric $h,$ we can assume $\mathrm{tr}(\mathcal{K}-cI_E)=0.$
\end{lemma}
\begin{proof}
 Consider a real positive function $a=a(x)$ on $X$ and set $h'=ah.$ Then $h'=ah$ defines another Hermitian
 metric $h'\in\text{Herm}^{+}(E).$
 From the identity
 \begin{equation}
  \label{firstlaplacian}
  \mathcal{K}'\omega^{n}=\mathcal{K}\omega^{n}+\Box_0\ln(a)I_E\omega^{n}
 \end{equation}
and since the K\"ahler form is nowhere vanishing, taking the trace in the previous identity we obtain
\begin{equation}
 \label{traceidentity}
 \mathrm{tr}(\mathcal{K}'-cI_E)=\mathrm{tr}(\mathcal{K}-cI_E)+r\Box_0\ln(a).
\end{equation}
Hence, it suffices to prove that there is a function $u$ satisfying the equation
\begin{equation}
 \label{secondpoisson}
 \Box_0u=-\frac{1}{r}\mathrm{tr}(\mathcal{K}-cI_E).
\end{equation}
In fact, setting $h'=e^{u}h,$ from \eqref{traceidentity} it follows that $\mathrm{tr}(\mathcal{K}'-cI_E)=0.$
Now from Hodge theory and Fredholm alternative Theorem we know that \eqref{secondpoisson} has a solution 
if and only if $\mathrm{tr}(\mathcal{K}-cI_E)$ is orthogonal to all $\Box_0\text{-harmonic}$ functions.
Since $X$ is a compact complex manifold, a function on $X$ is $\Box_0\text{-harmonic}$ if and only if it
is constant.
So \eqref{secondpoisson} has solution if and only if 
\begin{equation*}
 \int_X\mathrm{tr}(\mathcal{K}-cI_E)\frac{\omega^{n}}{n!}=0.
\end{equation*}
But this equality always holds from the previous Lemma and this completes the proof.
\end{proof}

\begin{lemma}
 Let $h_t$ be a solution of the Donaldson heat flow. Then
 \linebreak$\|\mathcal{K}_t-cI_E\|_{L^2}^{2}$ is a monotone decreasing function of $t$ for $0\leq t<+\infty.$
\end{lemma}
\begin{proof}
 From a previous calculation we have
 \begin{equation*}
 \frac{\mathrm{d}}{\mathrm{d}t}\mathcal{L}(h_t,h_0)=-\|\mathcal{K}_t-cI_E\|_{L^2}^{2} 
 \end{equation*}
and
\begin{equation*}
 \partial_{t}^{2}\mathcal{L}(h_t,h_0)=\|\mathcal{D}'v_t\|_{h_t}^{2}\geq0.
\end{equation*}
Hence,
\begin{equation*}
 \frac{\mathrm{d}}{\mathrm{d}t}\|\mathcal{K}_t-cI_E\|_{L^2}^{2}=-\|\mathcal{D}'v_t\|_{h_t}^{2}\leq0
\end{equation*}
and this completes the proof.
\end{proof}

\begin{lemma}
 \label{heatnull}
 Let $h_t,$ $0\leq t<+\infty,$ be the solution of the Donaldson heat flow. 
 We have 
 \begin{equation*}
  (\partial t+\tilde{\Box}_{t})\mathrm{tr}(\mathcal{K}_t-cI_E)=0,
 \end{equation*}
where $\mathcal{K}_t$ is the mean curvature of the Hitchin-Simpson connection associated with $h_t$ and
$\tilde{\Box}_{t}=\tilde{\Box}_{h_t}.$ Here the subscript $t$ remember us the dependence on the metric $h_t.$
\end{lemma}
\begin{proof}
 We already know that
 \begin{equation*}
  (\partial t+\tilde{\Box}_{t})\mathcal{K}_t=0.
 \end{equation*}
Since the trace and the linear operators $\tilde{\Box}_{t}$ and $\partial_{t}$ commute, we obtain
\begin{equation*}
 \begin{split}
  (\partial t+\tilde{\Box}_{t})\mathrm{tr}(\mathcal{K}_t-cI_E)&=
  \mathrm{tr}[(\partial t+\tilde{\Box}_{t})(\mathcal{K}_t-cI_E)]=\\
  &=\mathrm{tr}[(\partial t+\tilde{\Box}_{t})\mathcal{K}_t]=0.
 \end{split}
\end{equation*}
\end{proof}

Let $\text{Herm}_{\mathrm{int}}^{+}(E)$ denote the set of all Hermitian metrics $h$ satifsfying
\begin{equation*}
 \|\mathcal{K}_{h}\|_{L^1}=\int_{X}|\mathcal{K}_{h}|\frac{\omega^{n}}{n!}<+\infty
\end{equation*}
where $\mathcal{K}_h$ is the mean curvature of the Hitchin-Simpson connection of $h.$ If $X$ is compact this 
space coincides with $\text{Herm}^{+}(E).$ The space $\text{Herm}_{\mathrm{int}}^{+}(E)$ was studied 
by Simpson in \cite{SIM}. It is an analytic manifold, which in general is not connected, and has 
the following properties. If $k\in\text{Herm}_{\mathrm{int}}^{+}(E)$ is a fixed element, then any other 
metric in the same connected component is given by $h=k\exp(a)$ with $a$ a smooth endomorphism of $E$ which is 
selfadjoint with respect to $k.$ Moreover Simpson showed in \cite{SIM} that for any Higgs 
bundle $\frak{E}=(E,\phi)$ over a compact K\"ahler manifold $X$ the solution of the Donaldson heat flow remains 
in the same connected component of the initial metric. To be precise, if $k$ is a fixed metric in 
$\text{Herm}_{\mathrm{int}}^{+}(E),$ then the unique solution of the Donaldson heat flow $h_t$ with $h_0=k$ is 
contained in the same connected component as $k.$
 
Now, let $(X,\omega)$ be a compact K\"ahler manifold of (complex) dimension $n$ and let 
$\frak{E}=(E,\phi)$ be a Higgs bundle of rank $r$ over $X.$ Let $h_0\in\text{Herm}^{+}(E)$ be an initial metric 
with the condition $\mathrm{tr}(\mathcal{K}_{0}-cI_E)=0,$ and let $h_t,$ $0\leq t<+\infty,$ be the 
corresponding solution of 
the Donaldson heat flow. 
Then $h_t$ will be of the form $h_t=h_{0}\exp(S(t))$ for some section $S(t)$ of $\mathrm{End}(E)$
over $X,$ and $S(t)$ 
will be selfadjoint with respect to $h_0.$

\begin{defin}
Let $(X,\omega)$ be a compact K\"ahler manifold of (complex) dimension $n$ and let $\frak{E}=(E,\phi)$ be a Higgs 
bundle of rank $r$ over $X.$ Let $h_0\in\text{Herm}^{+}(E)$ be an initial metric 
with the condition $\mathrm{tr}(\mathcal{K}_{0}-cI_E)=0,$ where $\mathcal{K}_0$ is the mean curvature of the 
Hitchin-Simpson connection associated with $h_0.$ Let $h_t,$ $0\leq t<+\infty,$ be the solution of 
the Donaldson heat flow with initial condition $h_0,$
and let $S(t)$ be a section of $\mathrm{End}(E)$ such that $h_t=h_0\exp(S(t))$ and $S(t)$ is selfadjoint with 
respect to $h_0.$ Then $|S(t)|^2=\mathrm{tr}(S(t)\cdot S(t))$
 is a nonnegative real valued function on $X$ and we have the following norms:
 \begin{enumerate}
  \item $\|S(t)\|_{L^1}=\int_{X}|S(t)|\frac{\omega^n}{n!},$
  \item $\|S(t)\|_{L^2}=\left(\int_{X}|S(t)|^2\frac{\omega^n}{n!}\right)^{\frac{1}{2}},$
  \item $\|S(t)\|_{L^{\infty}}=\max_{X}|S(t)|.$
 \end{enumerate}
 Moreover, since $h_t$ is differentiable for $0<t<+\infty$, so does $S(t).$
\end{defin}

From the initial condition $\mathrm{tr}(\mathcal{K}_0-cI_E)=0$ and Lemma \ref{heatnull}, using the theory 
of heat equations on Riemannian manifolds, we deduce that
\linebreak$\mathrm{tr}(\mathcal{K}_t-cI_E)=0.$ From $v_t=-(\mathcal{K}_t-cI_E)$ and  
\begin{equation*}
 \begin{split}
  \frac{\partial}{\partial t}e^{\mathrm{tr}S(t)}&=\frac{\partial}{\partial t}\det e^{S(t)}=
  \frac{\partial}{\partial t}\det(h_{0}^{-1}h_t)=\\
  &=\frac{\partial}{\partial t}\frac{\det h_t}{\det h_0}=\frac{1}{\det h_0}\frac{\partial}{\partial t}\det h_t=
  \frac{1}{\det h_0}\mathrm{tr}(h_{t}^{-1}\partial_{t}h_t)=\\
  &=\frac{1}{\det h_0}\mathrm{tr}v_t=-\frac{1}{\det h_0}\mathrm{tr}(\mathcal{K}_t-cI_E)=0
 \end{split}
\end{equation*}
we deduce that $\mathrm{tr}S(t)$ is constant.
Since 
\begin{equation*}
 e^{\mathrm{tr}S(0)}=\det e^{S(0)}=\det(h_{0}^{-1}h_0)=1,
\end{equation*}
one has $\mathrm{tr}S(0)=0,$ and hence
\begin{equation}
\label{traceSnull}
 \mathrm{tr}S(t)=0\hspace{0.5cm}\text{ for }\hspace{0.5cm}0\leq t<+\infty.
\end{equation}

Now, after introducing two useful results we prove one of the most important inequalities of 
this section. 

\begin{lemma}
 \label{lemmelemme0}
 Let $A,B\in M_{n,n}(\mathbb{C})$ such that $A$ and $B$ are selfadjoint and assume that $B$ is 
 positive semi-definite. Then
 \begin{equation*}
  |\mathrm{tr}(AB)|\leq\sqrt{\mathrm{tr}(AA)}\left|\mathrm{tr}(B)\right|.
 \end{equation*}
\end{lemma}
\begin{proof}
 Since $B$ is selfadjoint and positive semi-definite we obtain
 \begin{equation*}
  \mathrm{tr}(B^{2})\leq\mathrm{tr}^{2}(B).
 \end{equation*}
In fact, let $\lambda_{1},\ldots,\lambda_{n}\in[0,+\infty]$ be the eigenvalues of $B.$ Hence,
\begin{equation*}
 \mathrm{tr}(B^{2})=\sum_{1\leq i\leq n}\lambda_{i}^{2}\leq\left(\sum_{1\leq i\leq n}\lambda_{i}\right)^{2}=
 \mathrm{tr}^{2}(B).
\end{equation*}
Finally, from the Cauchy-Schwarz inequality and since $A$ and $B$ are selfadjoint we conclude
\begin{equation*}
 \begin{split}
  \left|\mathrm{tr}(AB)\right|&=\left|\mathrm{tr}(AB^{\ast})\right|\leq 
  \sqrt{\mathrm{tr}(AA^{\ast})}\sqrt{\mathrm{tr}(BB^{\ast})}=\\
  &=\sqrt{\mathrm{tr}(AA)}\sqrt{\mathrm{tr}(BB)}\leq\sqrt{\mathrm{tr}(AA)}|\mathrm{tr}(B)|.
 \end{split}
\end{equation*}
\end{proof}

\begin{lemma}
 \label{lemmelemme2}
 Let $a_{1},\ldots,a_{n}\in\mathbb{R}.$ Then
 \begin{equation*}
  \sqrt{a_{1}^{2}+\cdots+a_{n}^{2}}\leq 
  n^{1/2}\ln(e^{a_{1}}+\cdots+e^{a_{n}}+e^{-a_{1}}+\cdots+e^{-a_{n}}).
 \end{equation*}
\end{lemma}
\begin{proof}
 Since $f(x)=\ln(x)$ is a monotone increasing function of $x$ on $(0,+\infty),$ for $1\leq i\leq n$ we have
 \begin{equation*}
  |a_{i}|=\ln(e^{|a_{i}|})\leq 
  \ln(e^{a_{1}}+\cdots+e^{a_{n}}+e^{-a_{1}}+\cdots+e^{-a_{n}}).
 \end{equation*}
Hence,
\begin{equation*}
 \sqrt{a_{1}^{2}+\cdots+a_{n}^{2}}\leq n^{1/2}\max_{1\leq i\leq n}|a_{i}|\leq 
 n^{1/2}\ln(e^{a_{1}}+\cdots+e^{a_{n}}+e^{-a_{1}}+\cdots+e^{-a_{n}}).
\end{equation*}
\end{proof}

\begin{lemma}
Let $(X,\omega)$ be a compact K\"ahler manifold of (complex) dimension $n$ and let $\frak{E}=(E,\phi)$ be a 
Higgs bundle of rank $r$ over $X.$
Let \linebreak$h_0\in\text{Herm}^{+}(E)$ be a metric with the condition 
$\mathrm{tr}(\mathcal{K}_{0}-cI_E)=0,$ where $\mathcal{K}_0$ is the mean curvature of the Histchin-Simpson 
connection of $h_0.$ 
Let $h_t$ be the solution of the Donaldson heat flow with initial condition $h_0$ and let $S(t)$ be a section
of $\mathrm{End}(E)$ such that $h_t=h_0\exp(S(t))$ and $S(t)$ is selfadjoint with respect to $h_0.$
The following inequality holds:
\begin{equation}
\label{superinequality}
 \left(\frac{1}{\sqrt{\mathrm{rk}(E)}}\|S(t)\|_{L^1}-\mathrm{Vol}(X)\ln(2\mathrm{rk}(E))\right)
 \|\mathcal{K}_t-cI_E\|_{L^2}\leq-\sqrt{\mathrm{Vol}(X)}\mathcal{L}(h_t,h_0).
\end{equation}
\end{lemma}
\begin{proof}
 Let $V=\mathrm{Vol}(X)$ and $r=\mathrm{rk}(E).$  
 First of all, from the H\"older inequality we have
 \begin{equation}
 \label{Holder1}
  \begin{split}
   \|\mathcal{K}_s-cI_E\|_{L^1}&=\int_{X}|\mathcal{K}_s-cI_E|\frac{\omega^{n}}{n!}=
   \int_{X}1\cdot|\mathcal{K}_s-cI_E|\frac{\omega^{n}}{n!}\leq\\
   &\leq(V)^{1/2}\left(\int_{X}|\mathcal{K}_s-cI_E|^{2}\frac{\omega^{n}}{n!}\right)^{1/2}=
   (V)^{1/2}\|\mathcal{K}_s-cI_E\|_{L^2}
  \end{split}
 \end{equation}
 
Set $H_{t}=e^{S(t)}=h_{0}^{-1}h_{t}$ and 
let $\lambda_{1}(t),\ldots,\lambda_{r}(t)$ be the eigenvalues of $S(t).$ Since $S(t)$ is selfadjoint with 
respect to $h_{0}$ we deduce that, for each $t\geq0,$ they are real valued functions on $X$ and then, using 
Lemma \ref{lemmelemme2}, we obtain
 \begin{equation}
  \label{Pippa2}
   \begin{split}
    |S(t)|&=\sqrt{\lambda_{1}^{2}(t)+\cdots+\lambda_{r}^{2}(t)}\leq\\
          &\leq r^{1/2}\ln(e^{\lambda_{1}(t)}+\cdots+e^{\lambda_{r}(t)}+e^{-\lambda_{1}(t)}+\cdots+
          e^{-\lambda_{r}(t)})=\\ 
          &=r^{1/2}\ln(\mathrm{tr}H_t+\mathrm{tr}H_{t}^{-1}),
   \end{split}
 \end{equation}
 
so that
\begin{equation}
 \label{Pippa3}
 r^{-1/2}|S(t)|\leq\ln(\mathrm{tr}H_{t}+\mathrm{tr}H_{t}^{-1}).
\end{equation} 
 
Since $S(t)$ is selfadjoint with respect to $h_{0},$ it follows that
$H_{t}=e^{S(t)}$ and $H_{t}^{-1}=e^{-S(t)}$ are also selfadjoint with respect to $h_{0}.$ Moreover, 
from the definition of $H_{t}$  and the properties of the exponential of matrices 
we deduce that $H_t$ and $H_{t}^{-1}$ are positive definite and then 
$\mathrm{tr}H_{t}>0$ and $\mathrm{tr}H_{t}^{-1}>0.$
In fact
\begin{equation*}
 \mathrm{tr}H_{t}=e^{\lambda_{1}(t)}+\cdots+e^{\lambda_{r}(t)}>0
\end{equation*}
and
\begin{equation*}
 \mathrm{tr}H_{t}^{-1}=e^{-\lambda_{1}(t)}+\cdots+e^{-\lambda_{r}(t)}>0.
\end{equation*}

We also notice that 
$H_{t}^{-1}\partial_{t}H_{t}=h_{t}^{-1}h_{0}\partial_{t}(h_{0}^{-1}h_{t})=h_{t}^{-1}\partial_{t}h_{t}$ and hence
$H_{t}^{-1}\partial_{t}H_{t}$ is selfadjoint with respect to $h_{0}.$

By direct calculation, from Lemma \ref{lemmelemme0}, since $\partial_{t}$ and the 
trace commute and since the trace is $GL(r,\mathbb{C})\text{-invariant}$ we find
\begin{equation}
 \label{Pippa1}
 \begin{split}
  \frac{\partial}{\partial t}\ln(\mathrm{tr}H_t+\mathrm{tr}H_{t}^{-1})&=
   \frac{\partial_{t}(\mathrm{tr}H_{t}+\mathrm{tr}H_{t}^{-1})}{\mathrm{tr}(H_t+H_{t}^{-1})}=
  \frac{\partial_{t}\mathrm{tr}H_{t}+\partial_{t}\mathrm{tr}H_{t}^{-1}}{\mathrm{tr}(H_t+H_{t}^{-1})}=\\
  &=\frac{\mathrm{tr}(\partial_{t}H_{t})+\mathrm{tr}(\partial_{t}H_{t}^{-1})}
  {\mathrm{tr}H_t+\mathrm{tr}H_{t}^{-1}}=\\
  &=\frac{\mathrm{tr}(H_{t}^{-1}\cdot\partial_{t}H_{t}^{-1}\cdot H_{t})-
  \mathrm{tr}(H_{t}^{-1}\cdot\partial_{t}H_{t}\cdot H_{t}^{-1})}
  {\mathrm{tr}H_t+\mathrm{tr}H_{t}^{-1}}\leq\\
  &\leq\frac{\left|\mathrm{tr}(H_{t}^{-1}\partial_{t}H_{t}H_{t})-
  \mathrm{tr}(H_{t}^{-1}\partial_{t}H_{t}H_{t}^{-1})\right|}
  {\mathrm{tr}H_t+\mathrm{tr}H_{t}^{-1}}\leq\\
  &\leq\frac{\left|\mathrm{tr}(H_{t}^{-1}\partial_{t}H_{t}H_{t})\right|+
  \left|\mathrm{tr}(H_{t}^{-1}\partial_{t}H_{t}H_{t}^{-1})\right|}
  {\mathrm{tr}H_t+\mathrm{tr}H_{t}^{-1}}=\\
  &=\frac{\left|\mathrm{tr}(h_{t}^{-1}\partial_{t}h_{t}H_{t})\right|+
  \left|\mathrm{tr}(h_{t}^{-1}\partial_{t}h_{t}H_{t}^{-1})\right|}
  {\mathrm{tr}H_t+\mathrm{tr}H_{t}^{-1}}\leq\\
  &\leq\frac{\sqrt{\mathrm{tr}(h_{t}^{-1}\partial_{t}h_{t}\cdot h_{t}^{-1}\partial_{t}h_{t})}
  (\left|\mathrm{tr}H_{t}\right|+
  \left|\mathrm{tr}H_{t}^{-1}\right|)}
  {\mathrm{tr}H_t+\mathrm{tr}H_{t}^{-1}}=\\
  &=\frac{\sqrt{\mathrm{tr}(h_{t}^{-1}\partial_{t}h_{t}\cdot h_{t}^{-1}\partial_{t}h_{t})}
  (\mathrm{tr}H_{t}+\mathrm{tr}H_{t}^{-1})}
  {\mathrm{tr}H_t+\mathrm{tr}H_{t}^{-1}}=\\
  &=\sqrt{\mathrm{tr}(h_{t}^{-1}\partial_{t}h_{t}\cdot h_{t}^{-1}\partial_{t}h_{t})}=
  \left|\mathcal{K}_t-cI_E\right|.
  \end{split}
\end{equation}

On one hand, since $\|\mathcal{K}_{t}-cI_E\|_{L^2}$ is a nonnegative real valued monotone decreasing
function 
on $t,$ $0\leq t<+\infty,$ we have
\begin{equation}
 \label{Grossa2}
 t\|\mathcal{K}_{t}-cI_E\|_{L^2}^{2}\leq\int_{0}^{t}\|\mathcal{K}_{s}-cI_E\|_{L^2}^{2}\mathrm{d}s=
 -\mathcal{L}(h_{t},h_{0}),
\end{equation}
and then, for $t>0$
\begin{equation}
 \label{Grossa3}
 \|\mathcal{K}_{t}-cI_E\|_{L^2}\leq t^{-1/2}(-\mathcal{L}(h_{t},h_{0}))^{1/2}.
\end{equation}

On the other hand, since $X$ is compact we can use the Fubini-Tonelli Theorem.
Integrating over $X,$ from \eqref{Pippa1} and \eqref{Pippa3} and from the H\"older inequality we have
\begin{equation}
 \label{Grossa1}
 \begin{split}
  r^{-1/2}\|S(t)\|_{L^1}-V\ln(2r)&=\int_{X}r^{-1/2}|S(t)|-\ln(2r)\frac{\omega^{n}}{n!}\leq\\
  &\leq\int_{X}[\ln(\mathrm{tr}H_{t}+\mathrm{tr}H_{t}^{-1})-\ln(2r)]\frac{\omega^{n}}{n!}=\\
  &=\int_{X}\frac{\omega^{n}}{n!}\int_{0}^{t}\frac{\partial}{\partial s}
  \ln(\mathrm{tr}H_{s}+\mathrm{tr}H_{s}^{-1})\mathrm{d}s=\\
  &=\int_{0}^{t}\mathrm{d}s\int_{X}\frac{\partial}{\partial s}
  \ln(\mathrm{tr}H_{s}+\mathrm{tr}H_{s}^{-1})\frac{\omega^{n}}{n!}\leq\\
  &\leq\int_{0}^{t}\mathrm{d}s\int_{X}|\mathcal{K}_{s}-cI_E|\frac{\omega^{n}}{n!}=\\
  &=\int_{0}^{t}\|\mathcal{K}_{s}-cI_E\|_{L^1}\mathrm{d}s=\\
  &=\int_{0}^{t}1\cdot\|\mathcal{K}_{s}-cI_E\|_{L^1}\mathrm{d}s\leq\\
  &\leq\left(\int_{0}^{t}\mathrm{d}s\right)^{1/2}\left(\int_{0}^{t}\|\mathcal{K}_{s}-cI_E\|_{L^1}^{2}
  \mathrm{d}s\right)^{1/2}=\\
  &=t^{1/2}\left(\int_{0}^{t}\|\mathcal{K}_{s}-cI_E\|_{L^1}^{2}\mathrm{d}s\right)^{1/2}\leq\\
  &\leq t^{1/2}\left(\int_{0}^{t}V\|\mathcal{K}_{s}-cI_E\|_{L^2}^{2}\mathrm{d}s\right)^{1/2}=\\
  &=V^{1/2}t^{1/2}\left(\int_{0}^{t}\|\mathcal{K}_{s}-cI_E\|_{L^2}^{2}\mathrm{d}s\right)^{1/2}=\\
  &=V^{1/2}t^{1/2}(-\mathcal{L}(h_{t},h_{0}))^{1/2}. 
 \end{split}
\end{equation}

Combining \eqref{Grossa1} and \eqref{Grossa3}, since $\|\mathcal{K}_{t}-cI_E\|_{L^2}\geq0$ finally we obtain
\begin{equation*}
 \begin{split}
  (r^{-1/2}\|S(t)\|_{L^1}-V\ln(2r))\|\mathcal{K}_{t}-cI_E\|_{L^2}&\leq
  (V^{1/2}t^{1/2}(-\mathcal{L}(h_{t},h_{0}))^{1/2})\cdot\\
  &\cdot t^{-1/2}(-\mathcal{L}(h_{t},h_{0}))^{1/2}=\\
  &=-V^{1/2}\mathcal{L}(h_{t},h_{0})
 \end{split}
\end{equation*}
and this proves \eqref{superinequality}.
\end{proof}

Now we review some constructions involving Hermitian matrices. Let $E$ be a Higgs bundle with a fixed
Hermitian metric $k,$ and let $S=S(E)$ be the real vector bundle of selfadjoint endomorphisms of $E.$ Suppose 
$\varphi:\mathbb{R}\longrightarrow\mathbb{R}$ is a smooth function. Then we define a map of fibre bundles over 
$X$
\begin{equation*}
 \varphi:S\longrightarrow S
\end{equation*}
as follows: suppose $s\in S,$ then, at each point in $X,$ choose an orthonormal frame $\{e_{i}\}$ for $E$ with 
$s(e_{i})=\lambda_{i}e_{i},$ and set 
\begin{equation*}
 \varphi(s)(e_{i})=\varphi(\lambda_{i})e_{i}.
\end{equation*}
Suppose $\Psi:\mathbb{R}\times\mathbb{R}\longrightarrow\mathbb{R}$ is a smoot function of two variables. Then we 
define a map of fibre bundles
\begin{equation*}
 \Psi:S\longrightarrow S(\mathrm{End}(E)),
\end{equation*}
where $S(\mathrm{End}(E))$ consists of elements of $\mathrm{End}(\mathrm{End}(E))$ which are selfadjoint with respect to
the usual metric 
$\mathrm{tr}(A\cdot B^{\ast}).$ 
The function $\Psi$ is described as follows. Suppose $s\in S$ and $A\in\mathrm{End}(E).$
Choose an 
orthonormal frame field $\{e_{i}\}$ of eigenvectors of $s$ with eigenvalues $\lambda_{i}.$ Let $\{\hat{e}_{i}\}$ 
be the dual frame field in $E^{\ast},$ and write $A=\sum_{i,j}A_{ij}\hat{e}_{i}\otimes e_{j}.$ Then set 
\begin{equation*}
\Psi(s)(A)=\sum_{i,j}\Psi(\lambda_{i},\lambda_{j})A_{ij}\hat{e}_{i}\otimes e_{j}.
\end{equation*}
Again this is well defined, smooth and linear in $A.$

If the functions $\varphi$ and $\Psi$ are analytic, then we can express the constructions above as a power series. 
If
\begin{equation*}
 \varphi(\lambda)=\sum a_{n}\lambda^{n}
\end{equation*}
then
\begin{equation*}
 \varphi(s)=\sum a_{n}s^{n}.
\end{equation*}
If
\begin{equation*}
 \Psi(\lambda_{1},\lambda_{2})=\sum b_{mn}\lambda_{1}^{m}\lambda_{2}^{n}
\end{equation*}
then
\begin{equation*}
 \Psi(s)(A)=\sum b_{mn}s^{m}As^{n}.
\end{equation*}
Now, suppose $\varphi:\mathbb{R}\longrightarrow\mathbb{R}$ is a smooth function. Define
$\mathrm{d}\varphi:\mathbb{R}\times\mathbb{R}\longrightarrow\mathbb{R}$ by 
\begin{equation*}
 \mathrm{d}\varphi(\lambda_{1},\lambda_{2})=\frac{\varphi(\lambda_1)-\varphi(\lambda_2)}{\lambda_{1}-\lambda_{2}},
\end{equation*}
which is taken as $(\mathrm{d}\varphi/\mathrm{d}\lambda)(\lambda_1)$ if $\lambda_{1}=\lambda_{2}.$ If $s\in S,$
then
$\mathcal{D''}\varphi(s)=\mathrm{d}\varphi(s)(\mathcal{D''}s)$
where the right side uses the obvious extension to form-coefficient in the second variable. To see this for 
example when $\varphi$ is analytic, note that if $\varphi(\lambda)=\lambda^{n}$ then
\begin{equation*}
 \mathrm{d}\varphi(\lambda_{1},\lambda_{2})=\sum_{i+j=n-1}\lambda_{1}^{i}\lambda_{2}^{j}
\end{equation*}
whereas
\begin{equation*}
 \mathcal{D''}(s^{n})=\sum_{i+j=n-1}s^{i}\mathcal{D''}(s)s^{j}.
\end{equation*}
The construction $\varphi(s)$ and $\Psi(S)$ retain the same positivity properties as $\varphi$ and $\Psi.$ For 
example if $\varphi(\lambda)>0$ for all $\lambda,$ then $\varphi(s)$ is positive definite for all $s.$ And
if $\Psi(\lambda_{1},\lambda_{2})>0$ for all $\lambda_{1},\lambda_{2}$ then 
$\mathrm{tr}(\Psi(A)\cdot A^{\ast})>0$ for all $s$ and all $A\in\mathrm{End}(E).$
\\

We will describe how these constructions behave with respect to Sobolev norms.
Fix a Hermitian metric $k$ on a Higgs bundle 
$E.$ Using the metric we can define the space 
\begin{equation*}
 L_{0}^{p}(S)=\{s\in S|\int_{X}|s|^{p}\frac{\omega^{n}}{n!}<+\infty\},
\end{equation*}
where $s$ is selfadjoint with respect to the metric $k$ and $|s|^{2}=\mathrm{tr}(s\cdot s).$
In particular $L_{0}^{2}(S)$ is a Hilbert space with the inner product
\begin{equation}
\label{Hilbert1}
 \langle s,u\rangle=\int_{X}\mathrm{tr}(s\cdot u)\frac{\omega^n}{n!}
\end{equation}
and then
\begin{equation}
 \label{Hilbert2}
 \|s\|_{L^2}^{2}=\int_{X}\mathrm{tr}(s\cdot s)\frac{\omega^n}{n!}.
\end{equation}

\begin{note}
 In the above formulas $s$ and $u$ are selfadjoint with respect to the metric $k$ since they are 
element of $S.$ Then $\int_{X}\mathrm{tr}(s\cdot u)\frac{\omega^n}{n!}\in\mathbb{R},$
in fact, since the K\"ahler form is real, from $\mathrm{tr}(A)=\mathrm{tr}(A^{t})$ and 
$\mathrm{tr}(AB)=\mathrm{BA},$ we have
\begin{equation*}
\begin{split}
 \overline{\int_{X}\mathrm{tr}(s\cdot u)\frac{\omega^n}{n!}}&=
 \int_{X}\overline{\mathrm{tr}(s\cdot u)}\frac{\overline{\omega}^n}{n!}=
 \int_{X}\mathrm{tr}(\overline{s}\cdot\overline{u})\frac{\omega^n}{n!}=
 \int_{X}\mathrm{tr}((\overline{s}\cdot\overline{u})^{t})\frac{\omega^n}{n!}=\\
 &=\int_{X}\mathrm{tr}(u^{\ast}\cdot s^{\ast})\frac{\omega^n}{n!}=
 \int_{X}\mathrm{tr}(u\cdot s)\frac{\omega^n}{n!}=
 \int_{X}\mathrm{tr}(s\cdot u)\frac{\omega^n}{n!}.
\end{split}
\end{equation*}
\end{note}

From \eqref{Hilbert1} and \eqref{Hilbert2} we can deduce the following

\begin{lemma}
\label{FUNZIONALI}
 Let $k$ be a fixed Hermitian metric and let $\mathcal{K}_k$ be the mean curvature of the Hitchin-Simpson 
 connection associated with the metric $k.$ If \linebreak$u_m\rightharpoonup u_{\infty}$ in $L_{0}^{2}(S),$ then
 \begin{equation}
 \label{innerhilbert}
  \int_{X}\mathrm{tr}(u_m\cdot\mathcal{K}_{k})\frac{\omega^n}{n!}\longrightarrow
  \int_{X}\mathrm{tr}(u_{\infty}\cdot\mathcal{K}_{k})\frac{\omega^n}{n!}.
 \end{equation}
Moreover, if $\mathrm{tr}(u_m)=0$ for each $m,$ then $\mathrm{tr}(u_{\infty})=0.$
\end{lemma}
\begin{proof}
 From \eqref{Hilbert1} and \eqref{Hilbert2} we immediately deduce that
 \begin{equation*}
  s\longmapsto\int_{X}\mathrm{tr}(s\cdot\mathcal{K}_{k})\frac{\omega^{n}}{n!}
 \end{equation*}
is a continuous linear functional on $L_{0}^{2}(S)$ and then \eqref{innerhilbert} follows from the definition of 
weak convergence in Hilbert spaces.
In order to prove that \linebreak$\mathrm{tr}(u_{\infty})=0$ if $\mathrm{tr}(u_m)=0$ for each $m$
let us consider the set 
\begin{equation*}
 \Omega=\{x\in X|\mathrm{tr}(u_{\infty}(x))>0\}.
\end{equation*}
Since $u_{\infty}$ is a continuous section of $\mathrm{End}(E)$ and it is selfadjoint with respect to the 
metric $k$
we have that $\mathrm{tr}(u_{\infty})$ is a continuous 
real valued function on $X,$ so that $\Omega$ is an open subset of $X.$
Then, since $X$ is compact and the volume form $\frac{\omega^n}{n!}$ defines a finite measure on $X,$ if we apply 
\eqref{innerhilbert} to the open set $\Omega$ we obtain
\begin{equation*}
\int_{\Omega}\mathrm{tr}(u_m)\frac{\omega^n}{n!}=\int_{\Omega}\mathrm{tr}(u_m\cdot I_E)\frac{\omega^n}{n!}
\longrightarrow\int_{\Omega}\mathrm{tr}(u_{\infty}\cdot I_E)\frac{\omega^n}{n!}=
\int_{\Omega}\mathrm{tr}(u_{\infty})\frac{\omega^n}{n!}.
\end{equation*}
Since $\mathrm{tr}(u_m)=0$ we deduce that
\begin{equation*}
 \int_{\Omega}\mathrm{tr}(u_{\infty})\frac{\omega^n}{n!}=0.
\end{equation*}
But the continuous function $\mathrm{tr}(u_{\infty})$ is strictly positive on the open set $\Omega$ and then
we conclude that $\Omega=\emptyset.$ In a similar way the open set 
\linebreak$\Theta=\{x\in X|\mathrm{tr}(u_{\infty}(x))<0\}$
is empty, so $\mathrm{tr}(u_{\infty})=0$ on $X.$
\end{proof}

Let $L_{1}^{p}(S)$ denote the space of sections $s\in S$ such that $s\in L^{p}(S)$ and 
$\mathcal{D''}s\in L^{p}(E).$ Note that this is a condition on $\mathrm{d''}_{E}s=D''s$ and also a growth 
condition involving the Higgs field $\phi$ if $X$ is noncompact.
For a given number $b$ denote the closed subspaces of sections $s\in S$ with $|s|\leq b$ by 
$L_{0,b}^{p}$ and $L_{1,b}^{p}.$
Finally let $P(S)$ be the normed space of smooth sections $s\in S$ with norm 
\begin{equation*}
 \|s\|_{p}=\max_{X}|s|+\|\mathcal{D''}s\|_{L^2}+\|\mathcal{D'}_{k}s\|_{L^1}.
\end{equation*}
The construcions $\varphi$ and $\Psi$ behave in a rather delicate fashion on $L^{p}(S)$ and $L_{1}^{p}(S)$
as it is shown in the following Proposition. They behave better on $P(S)$ since their 
$\|\cdot\|_{L^{\infty}}\text{-norm}$ is controlled.

\begin{prop}
 \label{proposizione41}
 Let $\varphi$ and $\Psi$ be functions as above.
 \begin{enumerate}
  \item The map $\varphi$ extends to a continuous nonlinear map 
        \begin{equation*}
         \varphi:L_{0,b}^{p}(S)\longrightarrow L_{0,b'}^{p}(S)
        \end{equation*}
        for some $b'.$
  \item The map $\Psi$ extends to a map 
        \begin{equation*}
         \Psi:L_{0,b}^{p}(S)\longrightarrow\mathrm{Hom}(L^{p}(\mathrm{End}(E)),L^{q}(\mathrm{End}(E)))
        \end{equation*}
        for $q\leq p,$ and for $q<p$ it is continuous in the norm operator topology. 
  \item The map $\varphi$ extends to a map 
        \begin{equation*}
         \varphi:L_{1,b}^{p}(S)\longrightarrow L_{1,b'}^{p}(S)
        \end{equation*}
        for $q\leq p,$ and it is continuous for $q<p.$ The formula 
        $\mathcal{D''}\varphi(s)=\mathrm{d}\varphi(s)(\mathcal{D''}s)$ holds in this context.
  \item If $\varphi$ and $\Psi$ are analytic with infinite radius of convergence, the maps
        \begin{equation*}
         \varphi:P(S)\longrightarrow P(S),
        \end{equation*}
        \begin{equation*}
         \Psi:P(S)\longrightarrow P(\mathrm{End}(\mathrm{End}(E)))
        \end{equation*}
        are analytic.
 \end{enumerate}
\end{prop}
For the proof see Proposition $4.1$ in \cite{SIM}.
\\

Now, let $k\in\text{Herm}^{+}(E)$ be a fixed Hermitian structure. We already know that any Hermitian metric $h$ 
will be of the form $k\exp(v)$ for some section $v$ of $\mathrm{End}(E)$ over $X.$ Moreover, $v$ 
is selfadjoint with 
respect to $k.$ We can join $k$ to $h$ by the 
geodesic $h_{\tau}=k\exp(\tau v)$ where $0\leq\tau\leq1.$ Note that here 
$v_{\tau}=h_{\tau}^{-1}\partial_{\tau}h_{\tau}=v$ is constant,
i.e., it does not depend on $\tau.$ (See Chapter VI, \S 1 and \S 2 in \cite{KOB} for more details).
Now, in the proof of Theorem \ref{teorema47} we got an expression for the second derivative 
$\partial_{\tau}^{2}\mathcal{L}(h_{\tau},k)$ for any curve $h_{\tau}=k\exp(\tau v),$ namely:
\begin{equation*}
 \partial_{\tau}^{2}\mathcal{L}(h_{\tau},k)=\int_{X}\mathrm{tr}[\partial_{\tau}\mathcal{K}_{\tau}\cdot v_{\tau}+
 (\mathcal{K}_{\tau}-cI_E)\cdot\partial_{\tau}v_{\tau}]\frac{\omega^{n}}{n!},
\end{equation*}
where $\mathcal{K}_{\tau}$ is the mean curvature endomorphism of the Hitchin-Simpson connection associated with 
the 
metric $h_{\tau}=k\exp(\tau v).$
Notice that in our case, the chosen curve is such that $h_{0}=k,$ since it is also a geodesic 
$\partial_{\tau}v_{\tau}=0$ we have
\begin{equation}
\label{equazione53cardona}
 \partial_{\tau}^{2}\mathcal{L}(h_{\tau},k)=\int_{X}\mathrm{tr}(\partial_{\tau}\mathcal{K}_{\tau}\cdot v)=
 \frac{\omega^{n}}{n!}=\|\mathcal{D'}_{h_{\tau}}v\|_{h_{\tau}}^{2}.
\end{equation} 
Therefore, following \cite{SIU}, the idea is to find a simple expression for
$\|\mathcal{D'}_{h_{\tau}}v\|_{h_{\tau}}^{2}$ 
or equivalently for $\|\mathcal{D''}v\|_{h_{\tau}}^{2}$ and to integrate it twice with respect to $\tau.$ 
We can do 
this using local coordinates, indeed, at any point of $X$ we can choose a local frame field so that $h_{0}=I$
and $v=\mathrm{diag}(\beta_{1},\ldots,\beta_{r}).$
In particular, using such a local frame field we have, for 
$h_{\tau}=k\exp(\tau v),$ $h_{\tau}^{ij}=e^{-\beta_{j}\tau}\delta_{ij},$ and hence, (after a short computation)
we obtain
\begin{equation*}
 \|\mathcal{D''}v\|_{h_{\tau}}^{2}=\int_{X}\sum_{i,j=1}^{r}e^{(\beta_{i}-\beta_{j})\tau}
 |\mathcal{D''}v_{j}^{i}|^{2}\frac{\omega^{n-1}}{(n-1)!}.
\end{equation*}
Now, at $\tau=0$ the functional $\mathcal{L}(h_{\tau},k)$ vanishes and since $h_{0}=k$ is not a 
Hermitian-Yang-Mills structure, we have
\begin{equation*}
 \left.\partial_{\tau}\mathcal{L}(h_{\tau},k)\right|_{\tau=0}=\int_{X}\mathrm{tr}[(\mathcal{K}_{0}-cI_E)\cdot v]
 \frac{\omega^{n}}{n!}.
\end{equation*}
Then, integrating \eqref{equazione53cardona} twice we obtain

\begin{equation}
 \label{quasidonaldson}
 \mathcal{L}(h_{\tau},k)=\tau\int_{X}\mathrm{tr}[(\mathcal{K}_{0}-cI_E)\cdot v]\frac{\omega^{n}}{n!}+
 \int_{X}\sum_{i,j=1}^{r}\varPsi_{\tau}(\beta_{i},\beta_{j})|\mathcal{D''}v_{j}^{i}|^{2}
 \frac{\omega^{n-1}}{(n-1)!}
\end{equation}
where $\varPsi_{\tau}$ is the analyitic function given by
\begin{equation*}
 \varPsi_{\tau}(\beta_{i},\beta_{j})
 =\frac{e^{(\beta_{j}-\beta_{i})\tau}-(\beta_{j}-\beta_{i})\tau-1}{(\beta_{j}-\beta_{i})^{2}}.
\end{equation*}

In particular, at $\tau=1$ the expression \eqref{quasidonaldson} corresponds (up to a constant term) to the 
definition of the Donaldson functional given by Simpson \cite{SIM}. 
In fact, setting $\Psi(x_1,x_2)=\varPsi_{1}(x_1,x_2)$ for $(x_1,x_2)\in\mathbb{R}\times\mathbb{R},$ for two 
metrics in the same component $k$ and $h=e^{s},$ Simpson defines in \cite{SIM} the Donaldson functional as
\begin{equation*}
 \mathcal{L}(h,k)=\int_{X}\mathrm{tr}(s\cdot\mathcal{K}_{k})\frac{\omega^{n}}{n!}+
 \int_{X}\langle\Psi(s)(\mathcal{D''}s),\mathcal{D''}s\rangle_{k}\frac{\omega^{n-1}}{(n-1)!}.
\end{equation*}

Notice also that if the initial metric 
$k=h_{0}$ is Hermitian-Yang-Mils, the first term of the right hand side of \eqref{quasidonaldson} vanishes and 
the functional coincides with the Donaldson functional in \cite{SIU}.
\\

Now, let $h_{0}\in\text{Herm}^{+}(E)$ be a metric with the condition $\mathrm{tr}(\mathcal{K}_{0}-cI_E)=0$ and 
let $h_{t}$ be the solution of the Donaldson heat flow with initial metric $h_{0}.$
Since $\mathrm{tr}S(t)=0$ and since $\mathrm{tr}(AB)=\mathrm{tr}(BA),$ from \eqref{quasidonaldson} we have 
\begin{equation}
\label{quasidonaldson2}
  \mathcal{L}(h_{t},h_{0})=\int_{X}\mathrm{tr}(S(t)\cdot\mathcal{K}_{0})\frac{\omega^{n}}{n!}+
 \int_{X}\sum_{i,j=1}^{r}\varPsi_{1}(\beta_{i}(t),\beta_{j}(t))|\mathcal{D''}S(t)_{j}^{i}|^{2}
 \frac{\omega^{n-1}}{(n-1)!}
\end{equation}
where $\beta_{1}(t),\ldots,\beta_{r}(t)$ are the eigenvalues of $S(t).$ Following Simpson in \cite{SIM} 
we can rewrite \eqref{quasidonaldson2} in the equivalent form
\begin{equation}
\label{quasidonaldson3}
 \mathcal{L}(h_{t},h_{0})=\int_{X}\mathrm{tr}(S(t)\cdot\mathcal{K}_{0})\frac{\omega^{n}}{n!}+
 \int_{X}\langle\Psi(S(t))(\mathcal{D''}S(t)),\mathcal{D''}S(t)\rangle_{h_{0}}
 \frac{\omega^{n-1}}{(n-1)!}.
\end{equation}

From the previous constructions it follows that if the Donaldson functional $\mathcal{L}(h_t,h_0)$ is not 
bounded below, then $\|S(t)\|_{L^{\infty}}\longrightarrow+\infty$ as $t\longrightarrow+\infty.$ Namely we have

\begin{lemma}
 \label{NormaSinftygoesinfty}
 Let $(X,\omega)$ be a compact K\"ahler manifold of (complex) dimension $n$ and let $\frak{E}=(E,\phi)$ be a 
 Higgs bundle of rank $r\geq2$ over $X.$ Let $h_0\in\text{Herm}^{+}(E)$ be a Hermitian metric such that 
 $\mathrm{tr}(\mathcal{K}_0-cI_E)=0,$ where $\mathcal{K}_0$ is the mean curvature of the Hitchin-Simpson
 connection associated with $h_0.$ 
 Let $h_{t}$ be the solution of the Donaldson heat flow with initial 
 metric $h_{0}$ and assume $h_t=h_0e^{S(t)},$ where $S(t)$ is a section of $\mathrm{End}(E)$ and it is 
 selfadjoint with respect to $h_0.$ If 
 \begin{equation*}
  \mathcal{L}(h_t,h_0)\longrightarrow-\infty\hspace{0.5cm}\text{ as }\hspace{0.5cm}t\longrightarrow+\infty,
 \end{equation*}
then
\begin{equation*}
 \|S(t)\|_{L^{\infty}}\longrightarrow+\infty\hspace{0.5cm}\text{ as }\hspace{0.5cm}t\longrightarrow+\infty.
\end{equation*}
 \end{lemma}
\begin{proof}
 We will show that if the limit does not hold, there is a sequence 
 \linebreak$t_m\longrightarrow+\infty$ such that 
 $\mathcal{L}(h_{t_m},h_0)$ is bounded below.\\
 Suppose the required limit does not hold. Hence, we can find a positive constant $B>0$ and a sequence 
 $t_m\longrightarrow+\infty$ such that
 \begin{equation*}
  \|S(t_m)\|_{L^{\infty}}\leq B
 \end{equation*}
From the theory of the Gerschgorin circles (see \cite{BIN} for details), there exists a positive radius $R>0,$ 
which does not depend on $m,$ such that
\begin{equation*}
 |\beta_{i}^{(m)}|\leq R\hspace{0.5cm}\text{ for all }\hspace{0.5cm}m\geq0,\hspace{0.5cm}1\leq i\leq r,
\end{equation*}
where $\beta_{i}^{(m)}$ are the eigenvalues of $S(t_m).$
Since
\begin{equation*}
 \Psi(x_1,x_2)=\frac{e^{(x_2-x_1)}-(x_2-x_1)-1}{(x_2-x_1)^2}
\end{equation*}
is analytic with infinite radius of convergence, there exists a constant $C$ such that
\begin{equation*}
 \sum_{i,j=1}^{r}|\Psi(\beta_{i}^{(m)},\beta_{j}^{(m)})|\leq C
\end{equation*}
and $C$ does not depend on $m.$
Therefore, from \eqref{quasidonaldson2} and since $X$ is compact and we can locally consider 
$\mathcal{D''}=D''+\phi$ as 
a matrix of forms of degree $1$ which does not depend on $m,$ we have 
\begin{equation*}
 \begin{split}
  \mathcal{L}(h_{t_m},h_0)&\geq-|\mathcal{L}(h_{t_m},h_0)|=\\
  &=-\left|\int_{X}\mathrm{tr}(S(t_m)\cdot\mathcal{K}_0)\frac{\omega^{n}}{n!}\right.+\\
  &+\left.\int_{X}\sum_{i,j=1}^{r}\varPsi_{1}(\beta_{i}^{(m)},\beta_{j}^{(m)})|\mathcal{D''}S(t_m)_{j}^{i}|^{2}
 \frac{\omega^{n-1}}{(n-1)!}\right|\geq\\
 &\geq-\left|\int_{X}\mathrm{tr}(S(t_m)\cdot\mathcal{K}_0)\frac{\omega^{n}}{n!}\right|-\\
 &+\left|\int_{X}\sum_{i,j=1}^{r}\varPsi_{1}(\beta_{i}^{(m)},\beta_{j}^{(m)})|\mathcal{D''}S(t_m)_{j}^{i}|^{2}
 \frac{\omega^{n-1}}{(n-1)!}\right|\geq\\
 &\geq-\int_{X}|\mathrm{tr}(S(t_m)\cdot\mathcal{K}_0)|\frac{\omega^{n}}{n!}-\\
 &+\int_{X}\left|\sum_{i,j=1}^{r}\varPsi_{1}(\beta_{i}^{(m)},\beta_{j}^{(m)})\right|
 |\mathcal{D''}S(t_m)_{j}^{i}|^{2}\frac{\omega^{n-1}}{(n-1)!}\geq\\
 &\geq-\int_{X}|\mathrm{tr}(S(t_m)\cdot\mathcal{K}_0)|\frac{\omega^{n}}{n!}-\\
 &+\int_{X}\sum_{i,j=1}^{r}|\varPsi_{1}(\beta_{i}^{(m)},\beta_{j}^{(m)})|
 |\mathcal{D''}S(t_m)_{j}^{i}|^{2}\frac{\omega^{n-1}}{(n-1)!}\geq\\
 &\geq-C_1\|S(t_m)\|_{L^{\infty}}-C\int_{X}|\mathcal{D''}S(t_m)_{j}^{i}|^{2}\frac{\omega^{n-1}}{(n-1)!}\geq\\
 &\geq-C_1\|S(t_m)\|_{L^{\infty}}-CC_2C_1\|S(t_m)\|_{L^{\infty}}\geq\\
 &\geq-(C_1+CC_2)B
 \end{split}
\end{equation*}
and this estimate does not depend on $m.$ Therefore $\mathcal{L}(h_{t_m},h_0)$ is bounded below as 
$t_m\longrightarrow+\infty,$ which is a contradiction.
\end{proof}

Now, let $\{s_m\}$ be a sequence of sections in $S$ with $\mathrm{tr}(s_m)=0$ such that 
\begin{equation*}
 \|s_m\|_{L^1}\longrightarrow+\infty
\end{equation*}
and let us assume 
\begin{equation*}
 \max_{X}|s_m|\leq C_1\|s_m\|_{L^1}+C_2
\end{equation*}
where $C_1$ and $C_2$ do not depend on $m.$
Set $l_m=\|s_m\|_{L^1}$ and $u_m=l_{m}^{-1}s_m,$ so $\|u_m\|_{L^1}=1.$ Since 
$l_m\longrightarrow+\infty,$ from 
\begin{equation*}
 \|u_j\|_{L^{\infty}}=\frac{\|S(t_j)\|_{L^{\infty}}}{\|S(t_j)\|_{L^1}}\leq C_1+\frac{C_2}{\|S(t_j)\|_{L^1}}
\end{equation*}
we conclude that $\max_{X}|u_m|\leq C,$ where $C$ does not depend on $m.$ Moreover, since $\mathrm{tr}$ is 
linear, from $u_m=l_{m}^{-1}s_m$ and $\mathrm{tr}(s_m)=0$ we deduce that $\mathrm{tr}(u_m)=0.$

\begin{lemma}
 \label{Simpsonliminf}
 Up to extracting a subsequence, $u_{m}\rightharpoonup u_{\infty}$ weakly in $L_{1}^{2}(S).$ 
 The limit $u_{\infty}$ is not $0.$
 If $\Phi:\mathbb{R}\times\mathbb{R}\longrightarrow\mathbb{R}$ is a nonnegative smooth function with compact 
 support such that 
 $\Phi(x_1,x_2)\leq(x_{1}-x_{2})^{-1}$ whenever $x_1>x_2,$ then
 \begin{equation}
 \label{MACROSCOPICA}
  \begin{split}
  &\quad\int_{X}\mathrm{tr}(u_{\infty}\cdot\mathcal{K}_{k})\frac{\omega^{n}}{n!}+
  \int_{X}\langle\Phi(u_{\infty})(\mathcal{D''}u_{\infty}),\mathcal{D''}u_{\infty}\rangle_{k}
  \frac{\omega^{n-1}}{(n-1!)}\leq\\
  &\leq\liminf_{m}
  \left[\int_{X}\mathrm{tr}(u_{m}\cdot\mathcal{K}_{k})\frac{\omega^{n}}{n!}+
  \int_{X}\langle l_{m}\Psi(l_{m}u_{m})(\mathcal{D''}u_{m}),\mathcal{D''}u_{m}\rangle_{k}
  \frac{\omega^{n-1}}{(n-1!)}\right],
  \end{split}
 \end{equation}
where $k$ is a Hermitian metric on the Higgs bundle $(E,\phi).$ 
\end{lemma}
\begin{proof}
In Proposition $5.3$ and Lemma $5.4$ in \cite{SIM} Simpson proved that, up to considering a subsequence, 
$u_{m}\rightharpoonup u_{\infty}$ weakly in $L_{1}^{2}(S)$ and $u_{\infty}\neq0.$
Hence, we have to prove the estimate \eqref{MACROSCOPICA}.
First, assume $\Phi:\mathbb{R}\times\mathbb{R}\longrightarrow\mathbb{R}$ is a nonnegative smooth function with 
compact support $K$ such that $\Psi(x_{1},x_{2})<(x_{1}-x_{2})^{-1}$ whenever $x_{1}>x_{2}.$
Then, there exists $l\gg0$ such that
\begin{equation}
\label{cuscricc}
 \Psi(x_{1},x_{2})<l\Psi(lx_{1},lx_{2})\hspace{0.5cm}\text{ for all }\hspace{0.5cm}(x_{1},x_{2})\in
 \mathbb{R}\times\mathbb{R}.
\end{equation}
To see this fix $(a,b)\in\mathbb{R}\times\mathbb{R}.$ Hence, since as $l\longrightarrow+\infty$ the 
quantity $l\Psi(la,lb)$ increases monotonically to $(a-b)^{-1}$ if $a>b$ and to $+\infty$ if $a\leq b,$ there 
exists $l_{(a,b)}\gg0$ such that
\begin{equation*}
 \Phi(a,b)<l_{(a,b)}\Psi(l_{(a,b)}a,l_{(a,b)}b)
\end{equation*}
so that, in order to prove \eqref{cuscricc}, it suffices to take 
\begin{equation*}
 l=\max_{(a,b)\in K}l_{(a,b)}.
\end{equation*}

From \eqref{cuscricc} and since $l_{m}\longrightarrow+\infty,$ for $m\gg0$ 
 \begin{equation*}
  \begin{split}
  \|\Phi^{1/2}(u_{m})\mathcal{D''}u_{m}\|_{L^2}^{2}&=
  \int_{X}\langle\Phi(u_{\infty})(\mathcal{D''}u_{\infty}),\mathcal{D''}u_{\infty}\rangle_{k}
  \frac{\omega^{n-1}}{(n-1!)}\leq\\
  &\leq
  \int_{X}\langle l_{m}\Psi(l_{m}u_{m})(\mathcal{D''}u_{m}),\mathcal{D''}u_{m}\rangle_{k}
  \frac{\omega^{n-1}}{(n-1!)}
  \end{split}
 \end{equation*}
 and then
 \begin{equation}
  \label{ottomana}
  \liminf_{m}\|\Phi^{1/2}(u_{m})\mathcal{D''}u_{m}\|_{L^2}^{2}\leq 
  \liminf_{m}\int_{X}\langle l_{m}\Psi(l_{m}u_{m})(\mathcal{D''}u_{m}),\mathcal{D''}u_{m}\rangle_{k}
  \frac{\omega^{n-1}}{(n-1!)}.
 \end{equation}

 We already know that $\max|u_{m}|\leq C,$ where $C$ does not depend on $m.$ Then, since $X$ is compact,
 $\{u_{m}\}\subseteq L_{0,b}^{2}(S).$
 
 Now, consider the compact immersion of Sobolev spaces 
 \begin{equation}
  \label{immersionesobolev}
  L_{1}^{2}(S)\hookrightarrow L_{0}^{2}(S).
 \end{equation}
From $u_{m}\rightharpoonup u_{\infty}$ weakly in $L_{1}^{2}(S)$ and from \eqref{immersionesobolev} 
we can deduce 
that, up to consider a subsequence,
$u_{m}\longrightarrow u_{\infty}$ strongly in $L_{0,b}^{2}(S).$

So, we can apply \ref{proposizione41} (b) to conclude that for any $q<2$
\begin{equation*}
 \Phi^{1/2}(u_{m})\longrightarrow\Phi^{1/2}(u_{\infty})\hspace{0.5cm}\text{ strongly in }
 \mathrm{Hom}(L^{2},L^{q}).
\end{equation*}

Moreover, from $u_{m}\rightharpoonup u_{\infty}$ weakly in $L_{1}^{2}(S),$ we have 
$\mathcal{D''}u_{m}\rightharpoonup\mathcal{D''}u_{\infty}$ weakly in $L_{0}^{2}(S)$ and hence, from 
Lemma \ref{efremcalzelunghe} we find that
\begin{equation*}
 \Phi^{1/2}(u_{m})\mathcal{D''}u_{m}\rightharpoonup\Phi^{1/2}(u_{\infty})\mathcal{D''}u_{\infty}
 \hspace{0.5cm}\text{ weakly in }L_{0}^{q}(S)
\end{equation*}
for any $q<2.$ So that, from Lemma \ref{liminfBanach} 
\begin{equation}
 \label{referenzaA}
 \|\Phi^{1/2}(u_{\infty})\mathcal{D''}u_{\infty}\|_{L^{q}}^{2}\leq\liminf_{m}
 \|\Phi^{1/2}(u_{m})\mathcal{D''}u_{m}\|_{L^{q}}^{2}
\end{equation}
for any $q<2.$

On the other hand, from Lemma \ref{FUNZIONALI} we know that
\begin{equation}
 \label{referenzaB}
 \int_{X}\mathrm{tr}(u_m\cdot\mathcal{K}_{k})\frac{\omega^n}{n!}\longrightarrow
 \int_{X}\mathrm{tr}(u_{\infty}\cdot\mathcal{K}_{k})\frac{\omega^n}{n!}.
\end{equation}

Set $V=\mathrm{Vol}(X).$ From \eqref{ottomana}, \eqref{referenzaA} and 
\eqref{referenzaB} we have
\begin{equation*}
 \begin{split}
  &\quad\int_{X}\mathrm{tr}(u_{\infty}\cdot\mathcal{K}_{k})\frac{\omega^n}{n!}+
        \|\Phi^{1/2}(u_{\infty})\mathcal{D''}u_{\infty}\|_{L^{q}}^{2}\leq\\
  & \leq\int_{X}\mathrm{tr}(u_{\infty}\cdot\mathcal{K}_{k})\frac{\omega^n}{n!}+
        \liminf_{m}\|\Phi^{1/2}(u_{m})\mathcal{D''}u_{m}\|_{L^{q}}^{2}\leq\\ 
  & \leq\int_{X}\mathrm{tr}(u_{\infty}\cdot\mathcal{K}_{k})\frac{\omega^n}{n!}+
        \liminf_{m}V^{\frac{2-q}{q}}\|\Phi^{1/2}(u_{m})\mathcal{D''}u_{m}\|_{L^{2}}^{2}\leq\\
  &    =\int_{X}\mathrm{tr}(u_{\infty}\cdot\mathcal{K}_{k})\frac{\omega^n}{n!}+
        V^{\frac{2-q}{q}}\liminf_{m}\|\Phi^{1/2}(u_{m})\mathcal{D''}u_{m}\|_{L^{2}}^{2}\leq\\
  & \leq\int_{X}\mathrm{tr}(u_{\infty}\cdot\mathcal{K}_{k})\frac{\omega^n}{n!}+
        V^{\frac{2-q}{q}}
        \liminf_{m}\int_{X}\langle l_{m}\Psi(l_{m}u_{m})(\mathcal{D''}u_{m}),\mathcal{D''}u_{m}\rangle_{k}
                   \frac{\omega^{n-1}}{(n-1!)}=\\
  &=\left(1-V^{\frac{2-q}{q}}\right)\int_{X}\mathrm{tr}(u_{\infty}\cdot\mathcal{K}_{k})\frac{\omega^n}{n!}+\\
  &+V^{\frac{2-q}{q}}\left[\int_{X}\mathrm{tr}(u_{\infty}\cdot\mathcal{K}_{k})\frac{\omega^n}{n!}+
         \liminf_{m}\int_{X}\langle l_{m}\Psi(l_{m}u_{m})(\mathcal{D''}u_{m}),\mathcal{D''}u_{m}\rangle_{k}
         \frac{\omega^{n-1}}{(n-1!)}\right]=\\
  &=\left(1-V^{\frac{2-q}{q}}\right)\int_{X}\mathrm{tr}(u_{\infty}\cdot\mathcal{K}_{k})\frac{\omega^n}{n!}+\\
  &+V^{\frac{2-q}{q}}\liminf_{m}
  \left[\int_{X}\mathrm{tr}(u_{m}\cdot\mathcal{K}_{k})\frac{\omega^{n}}{n!}+
  \int_{X}\langle l_{m}\Psi(l_{m}u_{m})(\mathcal{D''}u_{m}),\mathcal{D''}u_{m}\rangle_{k}
  \frac{\omega^{n-1}}{(n-1!)}\right].
 \end{split}
\end{equation*}

This works for any $q<2$ and then, in the limit, using some continuity and boundedness arguments, 
this implies inequality \eqref{MACROSCOPICA}.

Since \eqref{MACROSCOPICA} holds for all nonnegative smooth functions 
$\Phi:\mathbb{R}\times\mathbb{R}\longrightarrow\mathbb{R}$ with compact support such that 
$\Phi(x_{1},x_{2})<(x_{1}-x_{2})^{-1}$ whenever $x_{1}>x_{2},$ again using some continuity and boundedness
arguments, we can conclude that the inequality in the Lemma 
also holds if we assume $\Phi(x_{1},x_{2})\leq(x_{1}-x_{2})^{-1}$ whenever $x_{1}>x_{2}.$ 
\end{proof} 

\begin{lemma}
 \label{pippopollina}
 In the hypothesis of the previous Lemma, let $\tilde{\Phi}:\mathbb{R}\times\mathbb{R}\longrightarrow\mathbb{R}$ 
 be a nonnegative smooth function such that 
 \begin{equation*}
  \tilde{\Phi}(\lambda_{i},\lambda_{j})=\frac{1}{\lambda_{i}-\lambda_{j}}\hspace{0.5cm}\text{ if }
  \hspace{0.5cm}\lambda_{i}>\lambda_{j},
 \end{equation*}
where $\lambda_{1}<\cdots<\lambda_{l}$ are the distinct eigenvalues of $u_{\infty}.$
Then
 \begin{equation*}
  \begin{split}
  &\quad\int_{X}\mathrm{tr}(u_{\infty}\cdot\mathcal{K}_{k})\frac{\omega^{n}}{n!}+
  \int_{X}\langle\tilde{\Phi}(u_{\infty})(\mathcal{D''}u_{\infty}),\mathcal{D''}u_{\infty}\rangle_{k}
  \frac{\omega^{n-1}}{(n-1!)}\leq\\
  &\leq\liminf_{m}
  \left[\int_{X}\mathrm{tr}(u_{m}\cdot\mathcal{K}_{k})\frac{\omega^{n}}{n!}+
  \int_{X}\langle l_{m}\Psi(l_{m}u_{m})(\mathcal{D''}u_{m}),\mathcal{D''}u_{m}\rangle_{k}
  \frac{\omega^{n-1}}{(n-1!)}\right].
  \end{split}
 \end{equation*}
\end{lemma}
\begin{proof}
 We can construct a nonnegative smooth function ${\Phi}:\mathbb{R}\times\mathbb{R}\longrightarrow\mathbb{R}$ 
 with compact support such that
 \begin{enumerate}
  \item ${\Phi}(\lambda_{i},\lambda_{j})=\tilde{\Phi}(\lambda_{i},\lambda_{j}),$ where 
         $\lambda_{1}<\cdots<\lambda_{l}$ are the distinct eigenvalues of $u_{\infty},$
  \item $\Phi(x_{1},x_{2})\leq(x_{1}-x_{2})^{-1}$ whenever $x_{1}>x_{2}.$
 \end{enumerate}
Hence, from $(1)$ and the above Lemma we conclude that
\begin{equation*}
 \begin{split}
  &\quad\int_{X}\mathrm{tr}(u_{\infty}\cdot\mathcal{K}_{k})\frac{\omega^{n}}{n!}+
  \int_{X}\langle\tilde{\Phi}(u_{\infty})(\mathcal{D''}u_{\infty}),\mathcal{D''}u_{\infty}\rangle_{k}
  \frac{\omega^{n-1}}{(n-1!)}=\\
  &=\int_{X}\mathrm{tr}(u_{\infty}\cdot\mathcal{K}_{k})\frac{\omega^{n}}{n!}+
  \int_{X}\langle{\Phi}(u_{\infty})(\mathcal{D''}u_{\infty}),\mathcal{D''}u_{\infty}\rangle_{k}
  \frac{\omega^{n-1}}{(n-1!)}\leq\\
  &\leq\liminf_{m}
       \left[\int_{X}\mathrm{tr}(u_{m}\cdot\mathcal{K}_{k})\frac{\omega^{n}}{n!}+
             \int_{X}\langle l_{m}\Psi(l_{m}u_{m})(\mathcal{D''}u_{m}),\mathcal{D''}u_{m}\rangle_{k}
             \frac{\omega^{n-1}}{(n-1!)}\right].
 \end{split}
\end{equation*}

\end{proof}

The notion of weak subbundle of a holomorphic vector bundle was introduced in \cite{YAU}, and we can make a 
similar definition. 

\begin{defin}
 Let $(X,\omega)$ be a compact K\"ahler manifold of (complex) dimension $n$ and let $\frak{E}=(E,\phi)$ be a 
 Higgs sheaf of rank $r$ over $X.$ Let $k$ be a Hermitian metric on $E$ and let 
 $\mathcal{D}_{k}=\mathcal{D'}_{k}+\mathcal{D''}$ be the Hitchin-Simpson connection associated with  $k.$  
 A $L_{1}^{2}(S)\text{-subbundle}$ of $E$ is a section of $\pi\in L_{1}^{2}(S)$ such that 
 \begin{enumerate}
  \item $\pi^{2}=\pi=\pi^{\ast k},$
  \item $(I_E-\pi)\mathcal{D''}(\pi)=0.$
 \end{enumerate}
\end{defin}
Following Uhlenbeck and Yau in \cite{YAU} one can prove that $\pi$ is smooth outside a subvariety of complex 
codimension greater than or equal to $2$ and that it defines a Higgs subsheaf of $E.$
In fact, since $\mathcal{D''}=D''+\phi,$ we can separate the components of type $(0,1)$ from the components of 
type $(1,0).$ Hence we have
\begin{enumerate}
 \item $(I_E-\pi)D''(\pi)=0,$
 \item $(I_E-\pi)\phi(\pi)=0.$
\end{enumerate}
From \cite{YAU}, $(1)$ can be identified as the holomorphic condition, while $(2)$ can be identified as the 
Higgs subsheaf condition.

\begin{lemma}
\emph{(Key Lemma)}
\label{keylemma}
Let $(X,\omega)$ be a compact K\"ahler manifold of (complex) dimension $n$ and let 
 $\frak{E}=(E,\phi)$ be a Higgs bundle of rank $r\geq2$ over $X.$ 
 Let $h_t$ be the solution of the Donaldson heat flow with initial condition $h_0,$ and suppose
 $\mathrm{tr}(\mathcal{K}_0-cI_E)=0,$ where $\mathcal{K}_{0}$ is the mean curvature endomorphism of the 
 Hitchin-Simpson connection associated with $h_{0}.$ Let us assume $(E,\phi)$ is $\omega\text{-semistable}$ and
 $\mathcal{L}(h_t,h_0)$ is not bounded below, i.e. 
 $\mathcal{L}(h_t,h_0)\longrightarrow-\infty$ as $t\longrightarrow+\infty.$ Then
 \begin{equation*}
  -\frac{\mathcal{L}(h_t,h_0)}{\|S(t)\|_{L^1}}\longrightarrow0\hspace{0.5cm}\text{ as }\hspace{0.5cm}
  t\longrightarrow+\infty,
 \end{equation*}
 where $S(t)$ is a section of $\mathrm{End}(E)$ such that $h_{t}=h_{0}e^{S(t)}$ and $S(t)$ is selfdajoint with
 respect to 
 $h_{0}.$
\end{lemma}
\begin{proof}
 Let $\mathcal{K}_{t}$ be the mean curvature of the Hitchin-Simpson connection associated with the metric $h_{t}.$
 We will show that if the estimate does not hold, there is a Higgs subsheaf which condradicts semistability.
 Suppose the required estimate does not hold. Hence, we can find a positive constant $\epsilon>0$ and a 
 sequence $t_{m}\longrightarrow+\infty$ such that
 \begin{equation}
  \label{Grossina2}
  -\frac{\mathcal{L}(h_{t_{m}},h_{0})}{\|S(t_{m})\|_{L^{1}}}\geq\epsilon>0.
 \end{equation}
 Since $\max_{X}|\mathcal{K}_t-cI_E|$ is a monotone decreasing function 
 (see Proposition \ref{proposizione52} for details), from 
 \begin{equation*}
  \max_{X}|\mathcal{K}_t|\leq\max_{X}|\mathcal{K}_t-cI_E|+c
 \end{equation*}
we deduce that $\max_{X}|\mathcal{K}_{t}|$ is uniformly 
 bounded with respect to $t.$ So from Lemma $(3.1)\text{(d)}$ in \cite{SIM} and the hypotesis that $\max_{X}|\mathcal{K}_{t}|$ are 
 uniformly bounded, we have the following Simpson's estimate (p. 885 in \cite{SIM})
 \begin{equation}
  \label{Grossina3}
  \|S(t)\|_{L^{\infty}}\leq C_{1}\|S(t)\|_{L^{1}}+C_{2},
 \end{equation}
where $C_{1}$ and $C_{2}$ depend only on the curvature of the initial metric $h_{0}$ and the K\"aher form 
$\omega.$ 
Since $\mathcal{L}(h_{t},h_{0})$ is not bounded below, $\mathcal{L}(h_{t_{m}},h_{0})\longrightarrow-\infty$ as 
$t_{m}\longrightarrow+\infty$ and then 
$\|S(t_{m})\|_{L^{1}}\longrightarrow+\infty.$ In fact, from Lemma \ref{NormaSinftygoesinfty} and
\eqref{Grossina3} we have 
\begin{equation*}
 \frac{\|S(t_{m})\|_{L^{\infty}}-C_{2}}{C_{1}}\leq\|S(t_{m})\|_{L^{1}}\longrightarrow+\infty\hspace{0.5cm}
 \text{ as }\hspace{0.5cm}t_m\longrightarrow+\infty.
\end{equation*}
Set $u_m=l_{m}^{-1}S(t_m),$ where $l_m=\|S(t_m)\|_{L^1},$ then $\|u_m\|_{L^1}=1.$ Since $S(t)$ is selfadjoint 
with respect to $h_0,$ the $u_{m}$ are also selfadjoint with respect to $h_0.$ 
From the hypothesis
$\mathrm{tr}(\mathcal{K}_0-cI_E)=0$ and from \eqref{traceSnull} we deduce that
\begin{equation}
 \label{tracciaumnulla}
 \mathrm{tr}u_m=\mathrm{tr}[l_{m}^{-1}S(t_m)]=l_{m}^{-1}\mathrm{tr}S(t_m)=0.
\end{equation}
Since $\|S(t_m)\|_{L^1}\longrightarrow+\infty,$ from \eqref{Grossina3} we find
\begin{equation*}
 \|u_m\|_{L^{\infty}}=\frac{\|S(t_m)\|_{L^{\infty}}}{\|S(t_m)\|_{L^1}}\leq C_1+\frac{C_2}{\|S(t_m)\|_{L^1}}
\end{equation*}
and then for $m$ largely enough, $\|u_j\|_{L^{\infty}}\leq C_1+1.$
\\

Hence, up to extracing a subsequence, from Lemma $5.4$ in \cite{SIM}
$u_m\rightharpoonup u_{\infty}$ weakly in $L_{1}^{2},$ and the limit is not $0.$

Moreover, from \eqref{tracciaumnulla} and from Lemma \ref{FUNZIONALI} we deduce that 
$\mathrm{tr}u_{\infty}=0.$ From Lemma $5.5$ in \cite{SIM} we know that the eigenvalues of $u_{\infty}$
are real and constant almost everywhere, in other words, there are 
$\lambda_1<\cdots<\lambda_l$ which are the distinc eigenvalues of $u_{\infty}(x)$ for almost all $x\in X.$

Since 
$\mathrm{tr}u_{\infty}=0$ and since $u_{\infty}$ is not $0$ we must have $l\geq2,$ otherwise
if $\mathrm{tr}u_{\infty}=0$ and $l=1$ it follows that $u_{\infty}=0$ contraddicting the fact that
$u_{\infty}\neq0.$ 
\\

The weak limit $u_{\infty}$ gives rise to a flag of $L_{1}^{2}(S)\text{-subbundles}.$
For any integer $1\leq\alpha<l,$ define $\mathcal{C}^{\infty}$ functions 
$P_{\alpha}:\mathbb{R}\longrightarrow\mathbb{R}$ sucht that 
 \begin{equation*}
        P_{\alpha}=\left\lbrace \begin{aligned}
                                 1\hspace{0.2cm}&\text{if}\hspace{0.2cm}x\leq\lambda_{\alpha}\\
                                 0\hspace{0.2cm}&\text{if}\hspace{0.2cm}x\geq\lambda_{\alpha+1}                      
                                  \end{aligned}
                   \right.
 \end{equation*}
and set
\begin{equation*}
 \pi_{\alpha}=P_{\alpha}(u_{\infty}).
 \end{equation*}
 From the definition of $P_{\alpha}$ it follows that, if $\lambda_{1}<\cdots<\lambda_{l}$ are the distinct 
 eigenvalues of $u_{\infty},$
 
 \begin{equation}
 \label{Psuautovalori}
        P_{\alpha}(\lambda_{i})=\left\lbrace \begin{aligned}
                                 1\hspace{0.2cm}&\text{if}\hspace{0.2cm}i\leq\alpha\\
                                 0\hspace{0.2cm}&\text{if}\hspace{0.2cm}i\geq\alpha+1                      
                                  \end{aligned}
                   \right.\hspace{0.7cm}1\leq i\leq l.
 \end{equation}

We contend that the $\pi_{\alpha},$ $1\leq\alpha\leq l,$ are $L_{1}^{2}(S)\text{-subbundles}$ of $E.$ In fact we have
\begin{enumerate}
 \item The $\pi_{\alpha}$ are in $L_{1}^{2}(S)$ by Proposition $4.1\text{(c)}$ in \cite{SIM},
 \item From \eqref{Psuautovalori} $P_{\alpha}^{2}-P_{\alpha}$ vanishes at $\lambda_{1},\ldots,\lambda_{l}$
       and then $\pi_{\alpha}^{2}=\pi_{\alpha},$
 \item From \S $4$ in \cite{SIM} 
       $\mathcal{D''}(\pi_{\alpha})=\mathrm{d}P_{\alpha}(u_{\infty})(\mathcal{D''}u_{\infty}).$
       Set $\Phi_{\alpha}(y_1,y_2)=(1-P_{\alpha})(y_2)\cdot\mathrm{d}P_{\alpha}(y_1,y_2).$ It is easy to see that 
       $(I_E-\pi_{\alpha})\mathcal{D''}(\pi_{\alpha})=\Phi_{\alpha}(u_{\infty})(\mathcal{D''}u_{\infty}).$ 
       On the other hand, $\Phi_{\alpha}(\lambda_{i},\lambda_{j})=0$ if $\lambda_{i}>\lambda_{j},$ in fact
       \begin{enumerate}
        \item If $\lambda_{j}\leq\lambda_{\alpha},$ then $(1-P_{\alpha})(\lambda_{j})=0,$
        \item If $\lambda_{i}>\lambda_{j}\geq0,$ then $P_{\alpha}(\lambda_{i})=P_{\alpha}(\lambda_{j})=0$ and
              then 
              \linebreak$\mathrm{d}P_{\alpha}(\lambda_{i},\lambda_{j})=
              \frac{P_{\alpha}(\lambda_{i}-P_{\alpha}(\lambda_{j})}{\lambda_{i}-\lambda{j}}=0.$ 
       \end{enumerate}
       By Lemma $5.6$ in \cite{SIM}, $\Phi_{\alpha}(u_{\infty})(\mathcal{D''}u_{\infty})=0,$ so we conclude that 
       \linebreak
       $(I_E-\pi_{\alpha})\mathcal{D''}(\pi_{\alpha})=\Phi_{\alpha}(u_{\infty})(\mathcal{D''}u_{\infty})=0,$
\end{enumerate}  
and then the $\pi_{\alpha}$ are $L_{1}^{2}(S)\text{-subbundles}.$
\\

By Uhlenbeck and Yau's regularity result of $L_{1}^{2}\text{-subbundles}$ 
(see \cite{YAU} for a detailed proof), $\pi_{\alpha}$ represents a coherent torsion-free Higgs subsheaf 
$E_{\alpha}$ of $(E,\phi).$

Set 
\begin{equation*}
 \nu=\lambda_{l}\mathrm{deg}(E)-\sum_{\alpha=1}^{l-1}(\lambda_{\alpha+1}-\lambda_{\alpha})
 \mathrm{deg}(E_{\alpha}).
\end{equation*}

Now, choose an orthonormal frame field $\{e_j\}$ of eigenvectors of $u_{\infty}$ with eigenvalues $\lambda_{j}.$
Here $\lambda_{1}\leq\cdots\leq\lambda_{r}$ are the all eigenvalues of $u_{\infty}.$ From \eqref{Psuautovalori}
we have
\begin{equation*}
 \begin{split}
  \left[\lambda_{l}I_E-\sum_{\alpha=1}^{l-1}(\lambda_{\alpha+1}-\lambda_{\alpha})\pi_{\alpha}\right](e_j)&=
  \lambda_{l}e_j-\sum_{\alpha=1}^{l-1}(\lambda_{\alpha+1}-\lambda_{\alpha})P_{\alpha}(\lambda_j)e_j=\\
  &=\lambda_{l}e_j-\sum_{\alpha=j}^{l-1}(\lambda_{\alpha+1}-\lambda_{\alpha})e_j=\\
  &=\lambda_{j}e_j=u_{\infty}(e_j).
 \end{split}
\end{equation*}

Then we can write
\begin{equation}
\label{limitedeboleugualeproiettori}
 u_{\infty}=\lambda_{l}I_E-\sum_{\alpha=1}^{l-1}(\lambda_{\alpha+1}-\lambda_{\alpha})\pi_{\alpha}. 
\end{equation}

From $\mathrm{tr}u_{\infty}=0,$  
taking the trace in \eqref{limitedeboleugualeproiettori} we have
\begin{equation*}
 \lambda_{l}\mathrm{rk}(E)-\sum_{\alpha=1}^{l-1}(\lambda_{\alpha+1}-\lambda_{\alpha})\mathrm{rk}(E_{\alpha})=0
\end{equation*}
and then
\begin{equation}
\label{nupositivo}
 \begin{split}
  \nu&=\nu-\lambda_{l}\mathrm{deg}(E)+\lambda_{l}\mathrm{deg}(E)=\\
     &=[\nu-\lambda_{l}\mathrm{deg}(E)]+\lambda_{l}\mathrm{rk}(E)\mu(E)=\\
     &=[\nu-\lambda_{l}\mathrm{deg}(E)]+
       \sum_{\alpha=1}^{l-1}(\lambda_{\alpha+1}-\lambda_{\alpha})\mathrm{rk}(E_{\alpha})\mu(E)=\\
     &=-\sum_{\alpha=1}^{l-1}(\lambda_{\alpha+1}-\lambda_{\alpha})\mathrm{deg}(E_{\alpha})+
       \sum_{\alpha=1}^{l-1}(\lambda_{\alpha+1}-\lambda_{\alpha})\mathrm{rk}(E_{\alpha})\mu(E)=\\
     &=\sum_{\alpha=1}^{l-1}(\lambda_{\alpha+1}-\lambda_{\alpha})\mathrm{rk}(E_{\alpha})\mu(E)-
       \sum_{\alpha=1}^{l-1}(\lambda_{\alpha+1}-\lambda_{\alpha})\mathrm{rk}(E_{\alpha})\mu(E_{\alpha})=\\
     &=\sum_{\alpha=1}^{l-1}(\lambda_{\alpha+1}-\lambda_{\alpha})\mathrm{rk}(E_{\alpha})[\mu(E)-\mu(E_{\alpha})].  
 \end{split}
\end{equation}

On the other hand, from the Chern-Weil formula (Lemma $3.2$ in \cite{SIM}) we have
\begin{equation*}
 \mathrm{deg}(E_{\alpha})=\int_{X}
 \mathrm{tr}(\pi_{\alpha}\mathcal{K}_{0})\frac{\omega^{n}}{n!}-
 \int_{X}|\mathcal{D''}\pi_{\alpha}|_{h_{0}}^{2}\frac{\omega^{n-1}}{(n-1)!},
\end{equation*}

while from the definition of Chern classes in terms of curvature and from \eqref{curvature2} we can write

\begin{equation*}
 \mathrm{deg}(E)=\int_{X}\mathrm{tr}(\mathcal{K}_0)\frac{\omega^n}{n!}.
\end{equation*}

Therefore, from \eqref{limitedeboleugualeproiettori} and since 
$\mathcal{D''}(\pi_{\alpha})=\mathrm{d}P_{\alpha}(u_{\infty})(\mathcal{D''}u_{\infty}),$ we find 

\begin{equation}
 \label{Grassa1}
 \nu=\int_{X}\mathrm{tr}(u_{\infty}\mathcal{K}_{0})\frac{\omega^{n}}{n!}+\int_{X}\langle\sum_{\alpha=1}^{l-1}
 (\lambda_{\alpha+1}-\lambda_{\alpha})(\mathrm{d}P_{\alpha})^{2}(u_{\infty})(\mathcal{D''}u_{\infty}),
 \mathcal{D''}u_{\infty}\rangle_{h_{0}}\frac{\omega^{n-1}}{(n-1)!}.
\end{equation}
 
In fact, after a straightforward computation, we have

\begin{equation*}
 \begin{split}
  \nu&=\lambda_{l}\mathrm{deg}(E)-\sum_{\alpha=1}^{l-1}(\lambda_{\alpha+1}-\lambda_{\alpha})
       \mathrm{deg}(E_{\alpha})=\\
  &=\lambda_{l}\int_{X}\mathrm{tr}(\mathcal{K}_0)\frac{\omega^n}{n!}-
  \sum_{\alpha=1}^{l-1}(\lambda_{\alpha+1}-\lambda_{\alpha})
  \left[\mathrm{tr}(\pi_{\alpha}\mathcal{K}_{0})\frac{\omega^{n}}{n!}-
 \int_{X}|\mathcal{D''}\pi_{\alpha}|_{h_{0}}^{2}\frac{\omega^{n-1}}{(n-1)!}\right]=\\
 &=\int_{X}\mathrm{tr}(\lambda_{l}I_E\mathcal{K}_0)\frac{\omega^n}{n!}-
 \int_{X}\sum_{\alpha=1}^{l-1}(\lambda_{\alpha+1}-\lambda_{\alpha})
 \mathrm{tr}(\pi_{\alpha}\mathcal{K}_{0})\frac{\omega^{n}}{n!}+\\
 &+\int_{X}\sum_{\alpha=1}^{l-1}(\lambda_{\alpha+1}-\lambda_{\alpha})
 |\mathcal{D''}\pi_{\alpha}|_{h_{0}}^{2}\frac{\omega^{n-1}}{(n-1)!}=\\
 &=\int_{X}\mathrm{tr}(\lambda_{l}I_E\mathcal{K}_0)\frac{\omega^n}{n!}-
 \int_{X}\sum_{\alpha=1}^{l-1}
 \mathrm{tr}((\lambda_{\alpha+1}-\lambda_{\alpha})\pi_{\alpha}\mathcal{K}_{0})\frac{\omega^{n}}{n!}+\\
 &+\int_{X}\sum_{\alpha=1}^{l-1}(\lambda_{\alpha+1}-\lambda_{\alpha})
 |\mathcal{D''}\pi_{\alpha}|_{h_{0}}^{2}\frac{\omega^{n-1}}{(n-1)!}=\\
 &=\int_{X}\mathrm{tr}((\lambda_{l}I_E-\sum_{\alpha=1}^{l-1}(\lambda_{\alpha+1}-\lambda_{\alpha})\pi_{\alpha})
    \cdot\mathcal{K}_{0})\frac{\omega^{n}}{n!}+\\
 &+\int_{X}\langle\sum_{\alpha=1}^{l-1}
 (\lambda_{\alpha+1}-\lambda_{\alpha})(\mathrm{d}P_{\alpha})^{2}(u_{\infty})(\mathcal{D''}u_{\infty}),
 \mathcal{D''}u_{\infty}\rangle_{h_{0}}\frac{\omega^{n-1}}{(n-1)!}=\\
 &=\int_{X}\mathrm{tr}(u_{\infty}\cdot\mathcal{K}_{0})\frac{\omega^{n}}{n!}+\\
 &+\int_{X}\langle\sum_{\alpha=1}^{l-1}
 (\lambda_{\alpha+1}-\lambda_{\alpha})(\mathrm{d}P_{\alpha})^{2}(u_{\infty})(\mathcal{D''}u_{\infty}),
 \mathcal{D''}u_{\infty}\rangle_{h_{0}}\frac{\omega^{n-1}}{(n-1)!}.
 \end{split} 
\end{equation*}

From \eqref{Psuautovalori}, if $\lambda_{i}>\lambda_{j},$ we have 
\begin{equation*}
\begin{split}
  \sum_{\alpha=1}^{l-1}
 (\lambda_{\alpha+1}-\lambda_{\alpha})(\mathrm{d}P_{\alpha})^{2}(\lambda_{i},\lambda_{j})&=
 \sum_{\alpha=j}^{i}(\lambda_{\alpha+1}-\lambda_{\alpha})
  \left(\frac{P_{\alpha}(\lambda_{i})-P_{\alpha}(\lambda_{j})}{\lambda_{i}-\lambda_{j}}\right)^{2}=\\
 &=\frac{1}{(\lambda_{i}-\lambda_{j})^{2}}\sum_{\alpha=j}^{i}(\lambda_{\alpha+1}-\lambda_{\alpha})=\\
 &=\frac{1}{(\lambda_{i}-\lambda_{j})^{2}}\cdot(\lambda_{i}-\lambda_{j})=
 \frac{1}{(\lambda_{i}-\lambda_{j})}.
\end{split}
\end{equation*}

Finally, if we apply Lemma \ref{pippopollina} to the function 
\begin{equation*}
 \tilde{\Phi}(x_{1},x_{2})=\sum_{\alpha=1}^{l-1}
 (\lambda_{\alpha+1}-\lambda_{\alpha})(\mathrm{d}P_{\alpha})^{2}(x_1,x_2),
\end{equation*}
from \eqref{quasidonaldson3}
we obtain
\begin{equation*}
 \begin{split}
   0&>-\epsilon\geq\liminf_{m}\frac{\mathcal{L}(h_{t_m},h_0)}{\|S(t_m)\|_{L^1}}=\\
  &=\liminf_{m}\left[\frac{1}{\|S(t_m)\|_{L^1}}
  \left(\int_{X}\mathrm{tr}(S(t_m)\cdot\mathcal{K}_{0})\frac{\omega^{n}}{n!}\right.\right.+\\
 &+\left.\left.\int_{X}\langle\Psi(S(t_m))(\mathcal{D''}S(t_m)),\mathcal{D''}S(t_m)\rangle_{h_{0}}
 \frac{\omega^{n-1}}{(n-1)!}\right)\right]=\\
 &=\liminf_{m}\left[\frac{1}{l_m}
  \left(\int_{X}\mathrm{tr}(S(t_m)\cdot\mathcal{K}_{0})\frac{\omega^{n}}{n!}\right.\right.+\\
 &+\left.\left.\int_{X}\langle\Psi(l_mu_m)(\mathcal{D''}S(t_m)),\mathcal{D''}S(t_m)\rangle_{h_{0}}
 \frac{\omega^{n-1}}{(n-1)!}\right)\right]=\\
 &=\liminf_{m}
  \left(\int_{X}\mathrm{tr}(\frac{S(t_m)}{l_m}\cdot\mathcal{K}_{0})\frac{\omega^{n}}{n!}+\right.\\
 &\left.+\int_{X}\langle\frac{1}{l_m}\Psi(l_mu_m)(\mathcal{D''}S(t_m)),\mathcal{D''}S(t_m)\rangle_{h_{0}}
 \frac{\omega^{n-1}}{(n-1)!}\right)=\\
 &=\liminf_{m}
  \left(\int_{X}\mathrm{tr}(u_m\cdot\mathcal{K}_{0})\frac{\omega^{n}}{n!}+\right.\\
 &+\left.\int_{X}\langle l_m\Psi(l_mu_m)(\mathcal{D''}\frac{S(t_m)}{l_m}),\mathcal{D''}\frac{S(t_m)}{l_m}
 \rangle_{h_{0}}\frac{\omega^{n-1}}{(n-1)!}\right)=\\
 &=\liminf_{m}
  \left(\int_{X}\mathrm{tr}(u_m\cdot\mathcal{K}_{0})\frac{\omega^{n}}{n!}+
 \int_{X}\langle l_m\Psi(l_mu_m)(\mathcal{D''}u_m),\mathcal{D''}u_m\rangle_{h_{0}}
 \frac{\omega^{n-1}}{(n-1)!}\right)\geq\\
&\geq\int_{X}\mathrm{tr}(u_{\infty}\mathcal{K}_{0})\frac{\omega^{n}}{n!}+\\
&+\int_{X}\langle\sum_{\alpha=1}^{l-1}
 (\lambda_{\alpha+1}-\lambda_{\alpha})(\mathrm{d}P_{\alpha})^{2}(u_{\infty})(\mathcal{D''}u_{\infty}),
 \mathcal{D''}u_{\infty}\rangle_{h_{0}}\frac{\omega^{n-1}}{(n-1)!}=\nu.
 \end{split}
\end{equation*}

On the other hand, \eqref{nupositivo} and the $\omega\text{-semistability}$ imply $\nu\geq0,$ so we get a 
contradiction.
\end{proof}

\begin{lemma}
\label{normL20}
Let $(X,\omega)$ be a compact K\"ahler manifold of (complex) dimension $n$ and let 
 $\frak{E}=(E,\phi)$ be a Higgs bundle of rank $r$ over $X.$ 
 Let $h_t$ be the solution of the Donaldson heat flow with initial condition $h_0,$ and suppose
 $\mathrm{tr}(\mathcal{K}_0-cI_E)=0.$
 Let us assume $(E,\phi)$ is $\omega\text{-semistable}$ and $\mathcal{L}(h_t,h_0)$ is not bounded below, i.e., 
 $\mathcal{L}(h_t,h_0)\longrightarrow-\infty$ as $t\longrightarrow+\infty.$ Then
 \begin{equation*}
  \|\mathcal{K}_t-cI_E\|_{L^2}\longrightarrow0\hspace{0.5cm}\text{ as }\hspace{0.5cm}
  t\longrightarrow+\infty.
 \end{equation*}
\end{lemma}
\begin{proof}
 From the hypothesis we know that $\lim_{t\rightarrow+\infty}\mathcal{L}(h_t,h_0)=-\infty,$ then from 
 Lemma \ref{keylemma} $\|S(t)\|_{L^1}\longrightarrow+\infty$ as $t\longrightarrow+\infty.$
 Otherwise there exists $\epsilon>0$ and a sequence $t_j\longrightarrow+\infty$ such that 
 $\|S(t_j)\|_{L^1}\leq\epsilon$ and then 
 \begin{equation*}
  -\frac{\mathcal{L}(h_{t_j},h_0)}{\|S(t_j)\|_{L^1}}\geq
  -\frac{\mathcal{L}(h_{t_j},h_0)}{\epsilon}
  \longrightarrow+\infty\hspace{0.5cm}\text{ as }
  \hspace{0.2cm}j\longrightarrow+\infty,
 \end{equation*}
contraddicting \ref{keylemma}.
 For $t$ large enough,  
 \begin{equation}
 \label{divide}
  \left(\frac{1}{\sqrt{\mathrm{rk}(E)}}\|S(t)\|_{L^1}-\mathrm{Vol}(X)\ln(2\mathrm{rk}(E))\right)
 \end{equation}
is positive, so we can divide both terms of \eqref{superinequality} by \eqref{divide} obtaining
 \begin{equation*}
   \|\mathcal{K}_t-cI_E\|_{L^2}\leq
   \frac{-\sqrt{\mathrm{Vol}(X)}\mathcal{L}(h_t,h_0)}
   {\left(\frac{1}{\sqrt{\mathrm{rk}(E)}}\|S(t)\|_{L^1}-\mathrm{Vol}(X)\ln(2\mathrm{rk}(E))\right)}.
 \end{equation*}
Now, applying again Lemma \ref{keylemma}
 \begin{equation*}
   0\leq
   \|\mathcal{K}_t-cI_E\|_{L^2}\leq
   \frac{-\sqrt{\mathrm{Vol}(X)}\mathcal{L}(h_t,h_0)}
   {\left(\frac{1}{\sqrt{\mathrm{rk}(E)}}\|S(t)\|_{L^1}-\mathrm{Vol}(X)\ln(2\mathrm{rk}(E))\right)}
   \longrightarrow0
 \end{equation*}
 as $t\longrightarrow+\infty$ and this completes the proof.
\end{proof}

\section{Proof of the main Theorem}
In this section we prove the equivalence between semistability and the existance of approximate 
Hermitian-Yang-Mills metric structures for Higgs bundles in every dimension.
First of all we need to prove a preliminary result under the assumption 
$\mathrm{tr}(\mathcal{K}_{0}-cI_E)=0.$
\begin{prop}
\label{casetracenull}
Let $(X,\omega)$ be a compact K\"ahler manifold of (complex) dimension $n$ and let 
 $\frak{E}=(E,\phi)$ be a Higgs bundle of rank $r$ over $X.$  
 Let $h_t$ be the solution of the Donaldson heat flow with initial condition $h_0$ and let 
 $\mathcal{K}_0$ be the mean curvature of the Hitchin-Simpson connection associated with $h_0.$
 Let us assume $h_0$ satisfies
 the condition $tr(\mathcal{K}_0-cI_E)=0$ and the Donaldson functional 
 $\mathcal{L}(h_t,h_0)$ is not bounded below, i.e., $\mathcal{L}(h_t,h_0)\longrightarrow-\infty$ as
 $t\longrightarrow+\infty.$ If $(E,\phi)$ is $\omega\text{-semistable},$ then
 \begin{equation*}
  \max_{X}|\mathcal{K}_t-cI_E|\longrightarrow0.
 \end{equation*}
So there exists an approximate Hermitian-Yang-Mills metric structure on the semistable Higgs bundle
$(E,\phi).$
\end{prop}
\begin{proof}
 We follow Kobayashi's argument (see \cite{KOB}, p.224-226 for details).
 Let $\chi=\chi(x,y,t)$ be the heat kernel for the differential operator $\partial_t+\tilde{\Box}_{t},$
 where $\tilde{\Box}_{t}=\tilde{\Box}_{h_t}$ and the subscript $t$ remember us the dependence on the metric 
 $h_t.$ Set
 \begin{equation*}
  f(x,t)=(|\mathcal{K}_t-cI_E|^2)(x)\hspace{0.5cm}\text{ for }\hspace{0.5cm}(x,t)\in X\times[0,+\infty).
 \end{equation*}
Now fix $t_0\in[0,+\infty)$ and set
\begin{equation*}
 u(x,t)=\int_X\chi(x,y,t-t_0)(|\mathcal{K}_t-cI_E|^2)(y)\mathrm{d}y
\end{equation*}
where $\mathrm{d}y$ is the volume form $\mathrm{d}y=\frac{\omega^n}{n!}.$
Then $u(x,y)$ is of class $\mathcal{C}^{\infty}$ on \linebreak$X\times(t_0,+\infty)$ and extends to a 
continuous function on $X\times[t_0,+\infty).$ 
From the definition of the heat kernel we immediately have
\begin{equation*}
 (\partial_t+\tilde{\Box}_{t})u(x,t)=0\hspace{0.5cm}\text{ for }\hspace{0.5cm}(x,t)\in X\times(t_0,+\infty),
\end{equation*}
and
\begin{equation*}
 u(x,t_0)=f(x,t_0)=(|\mathcal{K}_{t_0}-cI_E|^2)(x).
\end{equation*}
Combined with the inequality \eqref{HSCardona} this yields
\begin{equation*}
 (\partial_t+\tilde{\Box}_{t})(|\mathcal{K}_t-cI_E|^2-u(x,t))\leq0\hspace{0.5cm}\text{ for }\hspace{0.5cm}
 (x,t)\in X\times(t_0,+\infty).
\end{equation*}
By the Maximum Principle \ref{MPPE} and the properties of $u(x,t)$ we find
\begin{equation*}
 \max_{X}(|\mathcal{K}_t-cI_E|^2-u(x,t))\leq\max_{X}(|\mathcal{K}_{t_0}-cI_E|^2-u(x,t_0))=0,\hspace{0.2cm}
 t\geq t_0.
\end{equation*}
Hence,
\begin{equation*}
\begin{split}
 \max_{X}|\mathcal{K}_{t_0+a}-cI_E|^2 & \leq\max_{X}u(x,a,t_0+a)=\\
                                      &  =  
 \max_{X}\int_X\chi(x,y,a)|\mathcal{K}_{t_0}-cI_E|^2(y)\mathrm{d}y\leq\\
                                      & \leq
 C_a\int_X|\mathcal{K}_{t_0}-cI_E|^2(y)\mathrm{d}y=\\
                                       &  =
 C_a\|\mathcal{K}_{t_0}-cI_E\|_{L^2}^{2},
 \end{split}
\end{equation*}
where 
\begin{equation*}
 C_a=\max_{X\times X}\chi(x,y,a).
\end{equation*}
Fix $a,$ say $a=1,$ and let $t_0\longrightarrow+\infty.$ Using Lemma \ref{normL20} we conclude
\begin{equation*}
 \max_{X}|\mathcal{K}_{t_0+1}-cI_E|^2\leq C_1\|\mathcal{K}_{t_0}-cI_E\|_{L^2}^{2}\longrightarrow0,
\end{equation*}
and this completes the proof.
\end{proof}

Now, using the previous Proposition and Lemma \ref{conformalchangetrace} we can give the proof of the main 
Theorem of this section:
\begin{teo}
\label{MAINTHEOREM}
 Let $(X,\omega)$ be a compact K\"ahler manifold of (complex) dimension $n$ and let 
 $\frak{E}=(E,\phi)$ be a Higgs bundle of rank $r$ over $X.$
 If $(E,\phi)$ is $\omega\text{-semistable},$ then it admits an approximate Hermitian-Yang-Mills structure. 
\end{teo}
\begin{proof}
 Let $h_0\in\text{Herm}^{+}(E)$ be a fixed Hermitian metric structure on $E,$ and let 
 $\mathcal{K}_0$ be the mean curvature of the Hitchin-Simpson connection associated with $h_0.$ From Lemma
 \ref{conformalchangetrace} we may assume $\mathrm{tr}(\mathcal{K}_0-cI_E)=0.$ Let $h_t$ be a
 solution of the Donaldson heat flow with initial condition $h_0.$ We already know that $h_t$ is 
 defined for every positive time $0\leq t<+\infty$ and $\mathcal{L}(h_t,h_0)$ is a real monotone 
 decreasing function of $t$ for  $0\leq t<+\infty.$
 Now we can distinguish between three cases:
 \begin{enumerate}
  \item If $\mathrm{rk}(E)=1,$ from $\mathrm{tr}(\mathcal{K}_0-cI_E)=0$ we deduce that $\mathcal{K}_{0}=cI_{E}.$
        So that $h_{0}$ is a Hermitian-Yang-Mills metric.
  \item If $\mathrm{rk}(E)\geq2$ and $\mathcal{L}(h_t,h_0)$ is bounded below, the thesis comes from 
        Theorem \ref{teorema53}. 
  \item If $\mathrm{rk}(E)\geq2$ and $\mathcal{L}(h_t,h_0)$ is not bounded below, since
        $\mathrm{tr}(\mathcal{K}_0-cI_E)=0,$ the thesis comes from Proposition \ref{casetracenull}.
 \end{enumerate}
\end{proof}

Hence, from the previous result and Theorem \ref{teorema53} we have the following
\begin{cor}
 \label{EQTOT}
 Let $(X,\omega)$ be a compact K\"ahler manifold of (complex) dimension $n$ and let $\frak{E}=(E,\phi)$ be a 
 Higgs bundle of rank $r$ over $X.$ The following conditions are equivalent:
 \begin{enumerate}
  \item $\frak{E}$ is $\omega\text{-semistable},$
  \item $\frak{E}$ admits an approximate Hermitian-Yang-Mills structure.
 \end{enumerate}

\end{cor}

As a consequence of this we deduce that many results about Higgs 
bundles written in terms of approximate Hermitian-Yang-Mills structures can be translated in terms of 
semistability. In particular we have the following:

\begin{cor}
 If $(X,\omega)$ is a compact K\"ahler manifold of (complex) dimension $n$ and $\frak{E}_1,$ $\frak{E}_2$ are 
 $\omega\text{-semistable}$ Higgs bundles over $X,$ then so is their tensor product 
 $\frak{E}_1\otimes\frak{E}_2.$ Furthermore, if $\mu(\frak{E}_1)=\mu(\frak{E}_2),$ so is the Whitney sum 
 $\frak{E}_1\oplus\frak{E}_2.$
\end{cor}
\begin{proof}
 \begin{enumerate}
  \item Let $\frak{E}_1$ and $\frak{E}_2$ be $\omega\text{-semistable}$ Higgs bundles over $X.$ From Theorem 
        \ref{MAINTHEOREM}
        $\frak{E}_1$ and $\frak{E}_2$ admit approximate Hermitian-Yang-Mills metric structures, so does 
        their tensor product $\frak{E}_1\otimes\frak{E}_2$ (Proposition \ref{approximatetensorproduct}).
        Hence, using the above corollary we conclude that
        $\frak{E}_1\otimes\frak{E}_2$ is $\omega\text{-semistable}.$
  \item It is similar to $(1),$ using the second part of Proposition \ref{approximatetensorproduct}.
 \end{enumerate}
\end{proof}

\begin{cor}
 If $\frak{E}$ is $\omega\text{-semistable},$ then so is the tensor product bundle
 $\frak{E}^{\otimes p}\otimes\frak{E}^{\ast\otimes q}$ and the exterior product bundle 
 $\bigwedge^{p}\frak{E}$ whenever $0\leq p\leq r=\mathrm{rk}(\frak{E}).$
\end{cor}


\begin{thebibliography}{9}
\addcontentsline{toc}{chapter}{Bibliography}

\bibitem{B14} S. Bando, Y. T. Siu, \emph{Stable sheaves and Einstein-Hermitian metrics}, in Geometry and 
              analysis on complex manifolds, World Scientific Publising. River Edge, NJ, (1994), 39-50.
\bibitem{BIN} D. A. Bini, M. Capovani, O. Menchi, \emph{Metodi Numerici per l'Algebra
              Lineare}, Zanichelli, (1988).
\bibitem{B24} I. Biswas, G. Schumacher, \emph{Yang-Mills equations for stable Higgs sheaves}, 
              International Journal of Mathematics Vol. 20, No 5 (2009), 541-556.
              Springer-Verlag, (1980). 
\bibitem{BOT} R. Bott, L.W. Tu, \emph{Differential Forms in Algebraic Topology},
              Springer-Verlag, (1982).
\bibitem{BRE} H. Brezis, \emph{Analisi Funzionale Teoria e Applicazioni}, Liguori Editore, (1986).              
\bibitem{BRU} U. Bruzzo, B. Gra\~{n}a Otero, \emph{Metrics on semistable and numerically effective Higgs 
              bundles}, J. reine ang. Math., \textbf{612} (2007), 59-79
              Springer-Verlag, (1980). 
\bibitem{B48} N. M. Buchdahl, \emph{Hermitian-Einstein connections and stable vector bundles over compact 
              complex surfaces}, Math. Ann., \textbf{280} (1988), 626-648.
\bibitem{B49} N. M. Buchdahl, \emph{Blowups and gauge fields}, Pacific Jpurnal of Math., \textbf{196} (2000),
              69-111.              
\bibitem{CAS} S. A. H. Cardona, \emph{Approximate Hermitian-Yang-Mills structures and semistability
              for Higgs bundles. I:Generalities and the one dimensional case},
              Ann. Glob. Anal. Geom., \textbf{42} (2012), 349-370
\bibitem{CAR} H. Cartan, \emph{Elementary Theory of Analytic Functions of One or Several Variables},
              Addison-Wesley, (1963).              
\bibitem{CHE} S. Chern, \emph{Complex Manifolds without Potential Theory},
              Springer-Verlag, (1979).
\bibitem{D50} S. K. Donaldson, \emph{A new proof of a theorem of Narasimhan and Seshadri}, 
              J. Diff. Geom. Soc., \textbf{18} (1983), 269-278.
\bibitem{D51} S. K. Donaldson, \emph{Anti-self dual Yang-Mills connections on complex algebraic surfaces and 
              stable vector bundles}, Proc. Lond. Math. Soc., \textbf{3} (1985), 1-26.
\bibitem{D52} S. K. Donaldson., \emph{Infinite determinants, stable bundles and curvature}, 
              Duke Math. J., \textbf{54} (1987), 231-247.
\bibitem{EVA} L. C. Evans, \emph{Partial Differential Equations},
              American Mathematical, (1998).              
\bibitem{GR}  R. C. Gunning, H. Rossi, \emph{Analytic Function of Several Complex Variables},
              Prentice-Hall, (1965).
\bibitem{CIN} J. Li, X. Zhang, \emph{Existence of approximate Hermitian-Einstein structures 
              on semi-stable Higgs bundles}, arXiv: 1206.6676v1.              
\bibitem{NAK} M. Nakahara, \emph{Geometry, Topology and Physics},
              Institute of Physics Publishing Bristol and Philadelphia, (1990).
\bibitem{HAR} R. Hartshorne, \emph{Algebraic Geometry},
              Springer, (1977).       
\bibitem{HIZ} F. Hirzebruch, \emph{Topological Methods in Algebraic Geometry},
              Springer-Verlag, (1966).
\bibitem{H16} N. J. Hitchin, \emph{The self-duality equations on a Riemann surface},
              Proc. Lond. Math., \textbf{55}, (1987), 59-126.
\bibitem{J46} A. Jacob, \emph{Existence of approximate Hermitian-Einstein structures on semistable bundles}, 
              arXiv: 1012.1888v1, (2010).
\bibitem{KO3} S. Kobayashi, \emph{First Chern class and holomorphic tensor fields},
              Nagoya Math. J., \textbf{77} (1980), 5-11.              
\bibitem{KO4} S. Kobayashi, \emph{Curvature and stability of vector bundles},
              Proc. Jap. Acad., \textbf{58} (1982), 158-162.
\bibitem{KOB} S. Kobayashi, \emph{Differential Geometry of Complex Vector Bundles},
              Iwanami Shoten Publishers and Princeton University Press, (1987).
\bibitem{KN1} S. Kobayashi, K. Nomizu, \emph{Foundations of Differential Geometry},
              John Wiley and Sons, Vol I (1963).
\bibitem{KN2} S. Kobayashi, K. Nomizu, \emph{Foundations of Differential Geometry},
              John Wiley and Sons, Vol II (1969).
\bibitem{L40} M. L\"ubke, \emph{The Kobayashi-Hitchin correspondence}, World Scientific Publishin Co. Pte. Ltd.,
              (1995).
\bibitem{L41} M. L\"ubke, \emph{Stability of Einstein-Hermitian vector bundles}, 
              Manusctipta Math., \textbf{42} (1983), 245-257.              
\bibitem{N12} M. Narasimhan, C. Seshadri, \emph{Stable and unitary bundles on a compact Riemann surface},
              Math. Ann.,\textbf{82} (1965), 540-564.         
\bibitem{SIM} C. T. Simpson, \emph{Constructing variations of Hodge strucures using Yang-Mills
              connections and application to uniformization}, Journal of American Mathematical Society,
              \textbf{1} (1988), 867-918.
\bibitem{S18} C. T. Simpson, \emph{Higgs bundles and local systems}, Publ. Math. I.H.E.S.,
              \textbf{75} (1992), 5-92.
\bibitem{SIU} Y. T. Siu, \emph{Lectures on Hermitian-Einstein metrics for stable bundles and K\"ahler-Einstein
              metrics}, Birk\"auhser, Basel-Boston (1987).              
\bibitem{YAU} K. Uhlenbeck, S. T. Yau, \emph{On the existence of Hermitian-Yang-Mills connections in stable
              vector bundles}, Comm. Pure Appl. Math.,
              \textbf{39} (1986), S257-S293.              
\bibitem{WEL} R. O. Wells, \emph{Differential Analysis on Complex Manifolds},
              Springer-Verlag, (1980).
              
\end{thebibliography}
\end{document}